\documentclass{CUP-JNL-FMS}%
\usepackage{amsmath}
\usepackage{amsthm,amscd,graphicx}
\usepackage{dsfont}
\usepackage{amssymb}
\usepackage{amsfonts}
\usepackage{mathrsfs}
\usepackage{fullpage}
\usepackage{colortbl}
\usepackage{tikz}
\usetikzlibrary{calc,matrix}
\usepackage{enumitem}
\usepackage[all]{xy}
\usetikzlibrary{decorations.pathmorphing,decorations.markings,decorations.pathreplacing,arrows,shapes}
\usetikzlibrary{knots}
\usepackage[bookmarks=true, bookmarksopen=true,%
bookmarksdepth=3,bookmarksopenlevel=2,%
colorlinks=true,%
linkcolor=blue,%
citecolor=blue,%
filecolor=blue,%
menucolor=blue,%
urlcolor=blue]{hyperref}

\newtheorem{thm}{Theorem}[section]
\newtheorem{prop}[thm]{Proposition}
\newtheorem{cor}[thm]{Corollary}
\newtheorem{lem}[thm]{Lemma}
\newtheorem{conj}[thm]{Conjecture}
\newtheorem{prob}[thm]{Problem}

\theoremstyle{definition}
\newtheorem{defn}[thm]{Definition}

\newtheorem{notn}[thm]{Notation}
\newtheorem{rmk}[thm]{Remark}
\newtheorem{exmp}[thm]{Example}

\makeatletter
\let\c@equation\c@thm
\makeatother
\numberwithin{equation}{section}

\newcommand{\conf}{\mathrm{Conf}}

\newcommand{\PGL}{\mathrm{PGL}}
\newcommand{\SL}{\mathrm{SL}}

\newcommand{\Hom}{\mathrm{Hom}}

\newcommand{\DT}{\mathrm{DT}}
\newcommand{\ad}{\mathrm{ad}}
\newcommand{\diag}{\mathrm{diag}}

\newcommand{\codim}{\mathrm{codim}}
\renewcommand{\sc}{\mathrm{sc}}
\newcommand{\ord}{\mathrm{ord}}

\newcommand{\up}{\mathrm{up}}
\newcommand{\doverline}[1]{\overline{\overline{#1}}}
\newcommand{\fr}{\mathrm{fr}}
\newcommand{\uf}{\mathrm{uf}}
\newcommand{\Frac}{\mathrm{Frac}}
\newcommand{\prin}{\mathrm{prin}}
\newcommand{\lcm}{\mathrm{lcm}}
\newcommand{\spec}{\mathrm{Spec} \ }
\newcommand{\Za}{\mathrm{Za}}
\newcommand{\Id}{\mathrm{Id}}

\newcommand{\A}{\mathsf{A}}
\newcommand{\B}{\mathsf{B}}
\newcommand{\C}{\mathsf{C}}
\newcommand{\D}{\mathsf{D}}
\newcommand{\G}{\mathsf{G}}
\renewcommand{\L}{\mathsf{L}}
\newcommand{\N}{\mathsf{N}}
\renewcommand{\P}{\mathsf{P}}
\newcommand{\Q}{\mathsf{Q}}
\renewcommand{\S}{\mathsf{S}}
\newcommand{\T}{\mathsf{T}}
\newcommand{\U}{\mathsf{U}}
\newcommand{\W}{\mathsf{W}}
\newcommand{\Z}{\mathsf{Z}}
\newcommand{\Br}{\mathsf{Br}}
\newcommand{\E}{\mathsf{E}}

\renewcommand{\vec}[1]{\mathbf{#1}}
\newcommand{\inprod}[2]{\left\langle#1,#2\right\rangle}

\setenumerate[1]{label={(\arabic*)}}
\setlist[enumerate]{itemsep=0mm}

\allowdisplaybreaks

\articletype{RESEARCH ARTICLE}
\jyear{2021}
\jdoi{10.1017/fms.2021.X}

\begin{document}

\begin{Frontmatter}

\title{Cluster Structures on Double Bott-Samelson Cells}

\author[1]{Linhui Shen}\orcid{0000-0001-7424-8621} \author[2]{Daping Weng}\orcid{0000-0002-7858-5323}

\authormark{Linhui Shen \textit{et al}.}

\address[1]{\orgname{Michigan State University}, \orgaddress{\street{619 Red Cedar Road, East Lansing, MI}, \postcode{48824}, \country{U.S.}};\email{linhui@math.msu.edu}}
\address[2]{\orgname{Michigan State University}, \orgaddress{\street{619 Red Cedar Road, East Lansing, MI}, \postcode{48824}, \country{U.S.}};\email{wengdap1@msu.edu}}

\received{30 Jul 2020}
\accepted{14 Aug 2021}



\abstract{Let $\C$ be a symmetrizable generalized Cartan matrix. We introduce four different versions of double Bott-Samelson cells for every pair of positive braids in the generalized braid group associated to $\C$. We prove that the decorated double Bott-Samelson cells are smooth affine varieties, whose coordinate rings are naturally isomorphic to upper cluster algebras. 

\hspace{0.25cm} We explicitly describe the Donaldson-Thomas transformations on double Bott-Samelson cells and prove that they are cluster transformations. 
As an application, we complete the proof of the Fock-Goncharov  duality conjecture in these cases.
We discover a periodicity phenomenon of the Donaldson-Thomas transformations on a family of double Bott-Samelson cells.
We give a (rather simple) geometric proof of Zamolodchikov's periodicity conjecture in the cases of $\Delta\square \mathrm{A}_r$.

\hspace{0.25cm} When $\C$ is of type $\mathrm{A}$, the double Bott-Samelson cells are isomorphic to Shende-Treumann-Zaslow's moduli spaces of microlocal rank-1 constructible sheaves associated to Legendrian links. By counting their $\mathbb{F}_q$-points  we obtain rational functions which are  Legendrian link invariants. 
}

\end{Frontmatter}

\localtableofcontents

\vspace*{14pt}

\section{Introduction}

\subsection{Cluster Structures on Double Bott-Samelson Cells}

Bott-Samelson varieties were introduced by Bott and Samelson \cite{BS} in the context of compact Lie groups and were reformulated by Hansen \cite{Han} and Demazure \cite{Dem} independently in the reductive algebraic group setting. Bott-Samelson varieties give resolutions of singularities of Schubert varieties, and have many applications in geometric representation theory.
Webster and Yakimov \cite{WY} considered the product of two Bott-Samelson varieties and gave a stratification whose strata are parametrized by a triple of Weyl group elements and observed that a family of strata are isomorphic to double Bruhat cells introduced by Fomin and Zelevinsky \cite{FZ}. Lu and Mouquin \cite{LM} introduced a Poisson variety called \emph{generalized double Bruhat cells}, which is defined by a conjugate class in a semisimple Lie group together with two $n$-tuples of Weyl group elements. Elek and Lu \cite{EL} further studied the special case of \emph{generalized Bruhat cells} where one of the $n$-tuples was trivial, and proved that their coordinate rings, as Poisson algebras, are examples of symmetric Poisson CGL extension defined by Goodearl and Yakimov \cite{GoY}.

Motivated by the positivity phenomenon on double Bruhat cells, Fomin and Zelevinsky \cite{FZI} introduced a class of commutative algebras called \emph{cluster algebras}. Fock and Goncharov \cite{FGensemble} introduced \emph{cluster varieties} as the geometric counterparts of cluster algebras, and conjectured that the coordinate rings of cluster varieties admit canonical bases parametrized by the integral tropical set of their dual cluster varieties. The cluster structures on double Bruhat cells have been studied extensively in \cite{BFZ, FGamalgamation, GSVcp}.

\vskip 1mm

In this paper, we introduce a new family of varieties called \emph{double Bott-Samelson cells}, as a natural generalization of double Bruhat cells, and study their cluster structures. 

Our generalization goes in two directions: first, we extend the groups from semisimple types to Kac-Peterson groups, whose double Bruhat cells have been studied by Williams \cite{Wilkacmoody}; second, we replace a pair of Weyl group elements $(u,v)$ by a pair of positive braids $(b,d)$, which we believe is a new construction. In particular, our double Bott-Samelson cells further generalize Lu and Mouquin's generalized double Bruhat cells associated to the identity conjugacy class \cite{LM} by dropping the additional data of partitioning the positive braids $b$ and $d$ as two $n$-tuples of Weyl group elements and extending the family to include the Kac-Peterson cases.

We present three versions of double Bott-Samelson cells, an undecorated one  
$\conf^b_d(\mathcal{B})$, and two decorated ones $\conf^b_d\left(\mathcal{A}_\sc\right)$ and 
$\conf^b_d\left(\mathcal{A}_\ad\right)$. The difference between the two decorated versions are similar to the 
difference between double Bruhat cells associated to simply connected forms and adjoint forms. There is one more version of double Bott-Samelson cell  $\conf^b_d\left(\mathcal{A}_\sc^\fr\right)$, but it will not play a significant role in the present paper.

We prove the following result on cluster structures of double Bott-Samelson cells.

\begin{thm}[Theorems \ref{affineness of conf},  \ref{3.42}, $\&$ \ref{3.43}]
\label{mainlcsafv}
 The double Bott-Samelson cells $\conf^b_d\left(\mathcal{A}_\sc\right)$ and $\conf^b_d\left(\mathcal{A}_\ad\right)$ are smooth affine varieties. The coordinate ring $\mathcal{O}\left(\conf^b_d\left(\mathcal{A}_\sc\right)\right)$ is an upper cluster algebra, and  $\mathcal{O}\left(\conf^b_d\left(\mathcal{A}_\ad\right)\right)$ is a cluster Poisson algebra\footnote{See Definition \ref{cdvneq0}}. The pair $\left(\conf^b_d\left(\mathcal{A}_\sc\right),\conf^b_d\left(\mathcal{A}_\ad\right)\right)$ form a cluster ensemble.
\end{thm}

Since double Bruhat cells are special cases of double Bott-Samelson cells, it follows from our result that the cluster structures on double Bruhat cells in the symmetrizable cases are canonical in the sense that they do not depend on the choice of reduced words (initial seeds), solving a conjecture of Berenstein, Fomin, and Zelevinsky \cite[Remark 2.14]{BFZ}.

\vskip 1mm

Inside a given upper cluster algebra, the subalgebra generated by all cluster variables is called its \emph{cluster algebra}\footnote{In this paper, we always assume that frozen cluster variables are invertible.} \cite{FZI,BFZ}. An interesting question to ask is whether an upper cluster algebra coincides with its cluster algebra. Sufficient conditions to prove this equality include acyclicity \cite{BFZ}, local acyclicity \cite{Muller_acyclic}, and CGL extensions \cite{GoY}. In this paper, we provide a new family of cluster varieties for which this equality holds.

\begin{thm}[Theorem \ref{up=ord}]\label{upper=ordinary} The upper cluster alegbra $\mathcal{O}(\conf^b_d(\mathcal{A}_\sc))$ coincides with its cluster algebra.
\end{thm}

\subsection{Donaldson-Thomas Transformation and Periodicity Conjecture}

On every cluster variety there is  a special formal automorphism called the \emph{Donaldson-Thomas transformation}, which is closely related to the Donaldson-Thomas invariants of certain 3d Calabi-Yau category with stability conditions considered by Kontsevich and Soibelman \cite{KS}. Following the work of Gross, Hacking, Keel, and Kontsevich \cite{GHKK}, if the Donaldson-Thomas transformation is a cluster transformation, then the Fock-Goncharov cluster duality conjecture holds. The cluster nature of Donaldson-Thomas transformations have been verified on many examples of cluster ensembles, including moduli spaces of $\G$-local systems \cite{GS2}, Grassmannians \cite{Weng}, and double Bruhat cells \cite{Wengdb}. As a direct consequence, the cluster duality conjecture holds in those cases.

In the present paper we  explicitly realize the Donaldson-Thomas transformation of the double Bott-Samelson cell $\conf^b_d(\mathcal{B})$ as a sequence of \emph{reflection maps} followed by a \emph{transposition map} (see Section \ref{simplereflection} for their definitions). We prove the following statement.

\begin{thm}[Theorems \ref{4.9}, \ref{4.11}] The  Donaldson-Thomas transformation of $\conf^b_d(\mathcal{B})$ is a cluster transformation. The Fock-Goncharov  duality conjecture\footnote{See Conjecture \ref{duality conjecture}} holds for  $\left(\conf^b_d\left(\mathcal{A}_\sc\right), \conf^b_d\left(\mathcal{A}_\ad\right)\right)$. 
\end{thm}

The key ingredients for constructing Donaldson-Thomas transformations  are four reflection maps, $^ir$, $_ir$, $r^i$, and $r_i$, which are biregular isomorphisms between double Bott-Samelson cells that differ by the placement of  $s_i$:
\begin{equation}
\label{refl.maps}
\xymatrix{\conf^{s_ib}_d(\mathcal{A}) \ar@<0.5ex>[r]^{^ir} & \conf^b_{s_id}(\mathcal{A}) \ar@<0.5ex>[l]^{_ir}}  \quad \quad \quad \xymatrix{\conf^{bs_i}_d(\mathcal{A}) \ar@<0.5ex>[r]^{r^i} & \conf^b_{ds_i}(\mathcal{A}) \ar@<0.5ex>[l]^{r_i}}.
\end{equation}
We prove the following result on these reflection maps.

\begin{thm}[Corollary \ref{4.13}] Reflection maps are quasi-cluster transformations and hence are Poisson maps.
\end{thm}

In Section \ref{PeriodicityConje}, we investigate the periodicity of  Donaldson-Thomas transformations for a class of double Bott-Samelson cells associated to semisimple algebraic groups. We prove 

\begin{thm}[Theorem \ref{5.1}] 
\label{1.3thm}
If $\G$ is semisimple and the positive braids $(b,d)$ satisfy $\left(db^\circ\right)^m=w_0^{2n}$, then the Donaldson-Thomas transformation of ${\conf^b_d(\mathcal{B})}$ is of a finite order dividing $2(m+n)$.
\end{thm}

Zamolodchikov's periodicity conjecture asserts that the solution of the $Y$-system associated to a pair of Dynkin diagrams is periodic with period relating to the Coxeter numbers of the two Dynkin diagrams. Keller gave a categorical proof of the conjecture in full generality in \cite{Kelperiod}. 

Let $\Delta$ be a Dynkin diagram of finite type and let $\G$ be a group of type $\Delta$.
In this paper we relate the product $\Delta\square \mathrm{A}_n$ to a double Bott-Samelson cell associated to  $\G$, and give a new geometric proof of Zamolodchikov's periodicity conjecture (Corollary \ref{5.9}).

As explained in \cite[Section 5.7]{Keldilog}, Zamolodchikov's periodicity implies a result on the periodicity of the Donaldson-Thomas transformation. Weng \cite{Weng} gave a direct geometric proof of the periodicity of $\DT$ in the case of $\mathrm{A}_m\square \mathrm{A}_n$ by realizing the Donaldson-Thomas transformation as a biregular automorphism on a configuration space of lines. 

Theorem \ref{1.3thm} gives a new geometric proof of the periodicity of $\DT$ in the cases of $\Delta\square \mathrm{A}_n$.

\begin{thm}[Corollary \ref{5.3}] Let $\Delta$ be a Dynkin quiver of finite type. Then $\DT_{\Delta\square \mathrm{A}_n}$ is of a finite order dividing $\frac{2(h+n+1)}{\gcd(h,n+1)}$ where $h$ is the Coxeter number of $\Delta$.
\end{thm}

\subsection{Positive Braids Closures}

Let $(b,d)$ be a pair of positive braids in the braid group of type $\mathrm{A}_r$.  Every word $(\vec{i},\vec{j})$ of $(b, d)$ encodes two sequences of crossings at the top and at the bottom of a \emph{Legendrian link} $\Lambda^\vec{i}_\vec{j}$ embedded in the standard contact $\mathbb{R}^3$ (see Section \ref{cadsbhvo}). 
Legendrian links obtained from different words of $(b,d)$ are related by Legendrian Reidemeister moves and therefore are Legendrian isotopic. Abusing notations we denote the corresponding isotopic class of Legendrian links by $\Lambda_d^b$. 

The reflection maps \eqref{refl.maps} correspond to Legendrian isotopies that move a crossing from top to bottom or vice versa at the two ends of the link diagram. Below is a picture depicting such a move for the reflection maps $^1r\circ r^2:\conf^{s_1s_2}_{s_1}(\mathcal{A})\rightarrow \conf_{s_1s_1s_2}^e(\mathcal{A})$ of Dynkin type $\mathrm{A}_2$.

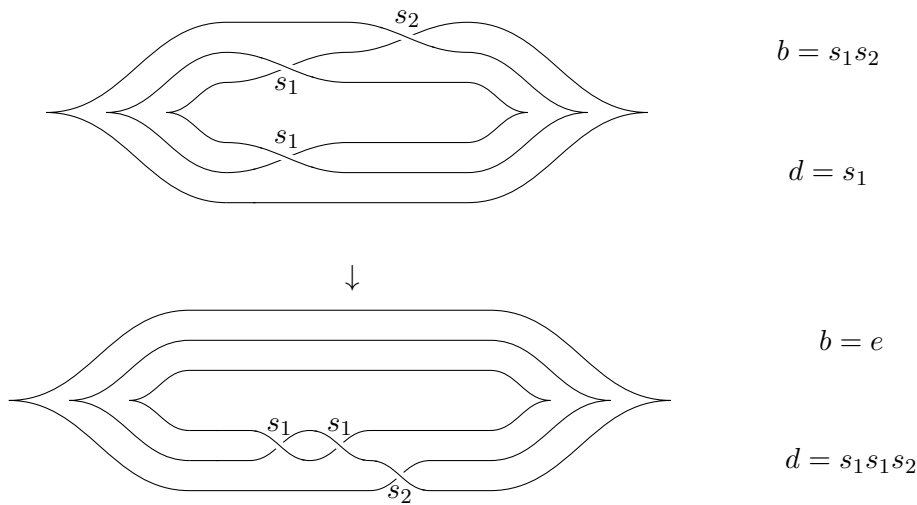
\[
\begin{tikzpicture}[scale=.8]
\begin{knot}[
consider self intersections, 
ignore endpoint intersections=false, 
]
\strand (0,0) -- (4,0) to [out=0,in=180] (7,1.5) to [out=180,in=0] (4,3) to [out=180,in=0] (2,2.5) to [out=180, in=0] (0,2) to [out=180, in=0] (-1,1.5) to [out=0,in=180] (0,1) to [out=0,in=180] (2,0.5) to [out=0, in=180] (4,0.5) to [out=0, in =180] (6,1.5) to [out=180,in=0] (4,2.5) to [out=180, in=0] (2,3) to [out=180, in =0] (0,3) to [out=180, in=0] (-3,1.5) to [out=0,in=180] (0,0);
\strand (0,0.5) to [out=0,in=180] (2,1) to (4,1) to [out=0, in=180] (5,1.5) to [out=180, in=0] (4,2) to (2,2) to [out=180, in=0] (0,2.5) to [out=180, in=0] (-2,1.5) to [out=0,in=180] (0,0.5);
\flipcrossings{2,3};
\end{knot}
\node at (10,2.5) [] {$b=s_1s_2$};
\node at (10,0.5) [] {$d=s_1$};
\node at (1,0.75) [above] {$s_1$};
\node at (1,2.25) [below] {$s_1$};
\node at (3,2.75) [above] {$s_2$};
\end{tikzpicture}
\]
\[
\downarrow \quad \quad \quad \quad \quad \quad \quad \quad
\]
\[
\begin{tikzpicture}[scale=.8]
\begin{knot}[
consider self intersections, 
ignore endpoint intersections=false, 
]
\strand (0,1) -- (1,1) to [out=0, in=180] (2,0.5) to [out=0, in=180] (3,1) -- (5,1) to [out=0, in=180] (6,1.5) to [out=180, in=0] (5,2) -- (0,2) to [out=180,in=0] (-1,1.5) to [out=0,in=180] (0,1);
\strand (0,0.5) -- (1,0.5) to [out=0, in=180] (2,1) to [out=0, in=180] (3,0.5) to [out=0, in=180] (4,0) -- (5,0) to [out=0, in=180] (8,1.5) to [out=180,in=0] (5,3) -- (0,3) to [out=180,in=0] (-3,1.5) to [out=0,in=180] (0,0) -- (3,0) to [out=0, in=180] (4,0.5) -- (5,0.5) to [out=0,in=180] (7,1.5) to [out=180,in=0] (5,2.5) -- (0,2.5) to [out=180,in=0] (-2,1.5) to [out=0,in=180] (0,0.5);
\flipcrossings{3};
\end{knot}
\node at (11,2.5) [] {$b=e$};
\node at (11,0.5) [] {$d=s_1s_1s_2$};
\node at (1.5,0.75) [above] {$s_1$};
\node at (2.5,0.75) [above] {$s_1$};
\node at (3.5,0.25) [below] {$s_2$};
\end{tikzpicture}
\]

Shende, Treumann, and Zaslow \cite{STZ} introduced a moduli space of microlocal rank-1 sheaves $\mathcal{M}_1\left(\Lambda\right)$ associated to any Legendrian link $\Lambda$. By a result of Guillermou, Kashiwara, and Schapira \cite{GKS}, the moduli spaces $\mathcal{M}_1\left(\Lambda\right)$ and $\mathcal{M}_1\left(\Lambda'\right)$ are isomorphic if $\Lambda$ and $\Lambda'$ are Legendrian isotopic \cite[Theorem 1.1]{STZ}. However, one should keep in mind that the isomorphisms between such moduli spaces depends on the Legendrian isotopies. 

By comparing the definitions of $\mathcal{M}_1\left(\Lambda^\vec{i}_\vec{j}\right)$ and $\conf^b_d(\mathcal{B})$ we obtain the following result.

\begin{thm}[Theorem \ref{caosdbvco}] 
\label{vfqve1} There is a natural isomorphism $\mathcal{M}_1\left(\Lambda^\vec{i}_\vec{j}\right)\cong \conf^b_d(\mathcal{B})$.
\end{thm}

Theorem \ref{vfqve1} implies that the automorphisms on the moduli spaces $\mathcal{M}_1\left(\Lambda^b_d\right)$ induced by braid moves are all trivial; therefore  one can canonically identify $\mathcal{M}_1\left(\Lambda^\vec{i}_\vec{j}\right)$ for different choices of words for $(b,d)$ and define the moduli space $\mathcal{M}_1\left(\Lambda^b_d\right)$ for a pair of positive braids $(b,d)$.

\vskip 2mm

The cells $\conf^b_d(\mathcal{A})$ associated to any generalized Cartan matrices are well defined over any finite field $\mathbb{F}_q$. Let
\[
f^b_d(q):= \left|\conf^b_d(\mathcal{A})\left(\mathbb{F}_q\right)\right|.
\]
In Section \ref{computation}, we provide an algorithm for computing $f^b_d(q)$.
The cell  $\conf^b_d(\mathcal{B})$ is isomorphic to $\conf^b_d(\mathcal{A})$ modulo a $\T\times \T$ action. Let $r$ be the rank of the Cartan subgroup $\T$. The orbifold counting of $\mathbb{F}_q$-points of  $\conf^b_d(\mathcal{B})$ is 
\[
g^b_d(q):= \left|\conf^b_d(\mathcal{B})\left(\mathbb{F}_q\right)\right|=\frac{\left|\conf^b_d(\mathcal{A})\left(\mathbb{F}_q\right)\right|}{\left|\T\times \T \left(\mathbb{F}_q\right)\right|}=\frac{f^b_d(q)}{(q-1)^{2r}}.
\]
In general $g^b_d(q)$ is a rational function, with possible poles  at $q=1$.

\begin{thm}[Corollary \ref{link_inv'}] Let $(b,d)$ be a pair of positive braids in the braid group of type $\mathrm{A}_r$. The double Bott-Samelson cell $\conf^b_d(\mathcal{B})$ (as an algebraic stack) and the rational function $g^b_d(q)$  are Legendrian link invariants for the positive braid closure $\Lambda^b_d$.
\end{thm}

\subsection{Further Questions}

\noindent\textbf{Comparison with Generalized Double Bruhat Cells.} Let $\vec{u}=\left(u_1,u_2,\dots, u_n\right)$ and $\vec{v}=\left(v_1,v_2,\dots, v_n\right)$ be two $n$-tuples of Weyl group elements and let $C$ be a conjugacy class in $\G$. Define
\[
\B_+\vec{u}\B_+:=\left\{\left[x_1,\dots, x_n\right]\in \G\underset{\B_+}{\times}\dots \underset{\B_+}{\times}\G \ \middle| \ x_i\in \B_+u_i\B_+\right\}
\]
and define $\B_-\vec{v}\B_-$ similarly. Lu and Mouquin \cite{LM} defined a \emph{generalized double Bruhat cell} as
\[
\G^{\vec{u},\vec{v}}_C:=\left\{\begin{array}{c}
[x_1,\dots, x_n,\\
y_1,\dots, y_n]\end{array} \ \middle| \ \begin{array}{c} \left[x_i\right]\in \B_+\vec{u}\B_+, \left[y_i\right]\in \B_-\vec{v}\B_-, \\ \left(x_1\dots x_n\right)\left(y_1\dots y_n\right)^{-1}\in C\end{array}\right\}.
\]
Note that when $C=\{e\}$ and $n=1$, it coincides with the ordinary double Bruhat cells. 

Let us lift $u_i$ and $v_j$ to positive braid elements and set  
$b=u_1\dots u_n$ and $d=v_1\dots v_n$. The generalized double Bruhat cell $\G_{\{e\}}^{\vec{u},\vec{v}}$ is biregularly isomorphic to our decorated double Bott-Samelson cell $\conf^b_d\left(\mathcal{A}\right)$ via the following map\footnote{The isomorphism is pointed to us by J.H. Lu.}
\begin{align*}
\G^{\vec{u},\vec{v}}_{\{e\}}&\longrightarrow \conf^b_d(\mathcal{A})\\
\begin{array}{c}
[x_1,\dots, x_n,\\
y_1,\dots, y_n]\end{array} & \longmapsto 
\left[\vcenter{\vbox{
\xymatrix{\U_+ \ar[r]^{u_1} \ar@{-}[d] & x_1\B_+ \ar[r]^{u_2} & x_1x_2\B_+\ar[r]^{u_3} & \dots \ar[r]^(0.4){u_n} & x_1\dots x_n\B_+ \ar@{-}[d] \\
\B_- \ar[r]_{v_1} & y_1\B_- \ar[r]_{v_2} & y_1y_2\B_- \ar[r]_{v_3} & \dots \ar[r]_(0.4){v_n} & y_1\dots y_n\U_-}}}
\right]
\end{align*}
In particular, this isomorphism shows that the generalized double Bruhat cells $\G_{\{e\}}^{\vec{u},\vec{v}}$ admit natural cluster structures. It further implies the following new result on generalized double Bruhat cells.

\begin{cor} Let $\vec{u}=(u_1,\ldots,  u_n)$, $\vec{v}=(v_1, \ldots, v_n)$, $\vec{u'}=(u_1', \ldots, u_m')$ and $\vec{v'}=(v_1', \ldots, v_m')$. If $u_1\dots u_n=u'_1\dots u'_m$ and $v_1\dots v_n=v'_1\dots v'_m$ in the braid group, then there is a canonical isomorphism between $\G_{\{e\}}^{\vec{u},\vec{v}}$ and  $\G_{\{e\}}^{\vec{u}',\vec{v}'}$. \end{cor} 

\begin{rmk}
We conjecture that the same statement holds for other conjugacy classes $C\neq \{e\}$.
In \cite{LM}, Lu and Mouquin defined a Poisson structure on $\G_{\{e\}}^{\vec{u},\vec{v}}$ by pushing forward the Poisson structure on products of flag varieties. In the adjoint form cases $\G=\G_\ad$, the space $\conf^b_d\left(\mathcal{A}_\ad\right)$ carries a natural Poisson structure from its cluster Poisson structure. We believe these two Poisson structures coincide, but a detailed checking is needed before we draw any definite conclusion.

In a recent work, Mouquin \cite{Mou} proved that the generalized double Bruhat cell $\G^{\vec{u},\vec{u}}_{\{e\}}$ is a Poisson groupoid over the generalized Bruhat cell $\G^{\vec{e},\vec{u}}_{\{e\}}$. We believe that the Poisson groupoid structure coincides with Fock-Goncharov's \emph{symplectic double} for cluster varieties \cite{FGrep}. 
We further observe that the inverse map of this Poisson groupoid resembles the Donaldson-Thomas transformation on $\conf^e_b\left(\mathcal{A}_\ad\right)$, and we would like to see a further investigation in these directions.
\end{rmk}

In general the decorated double Bott-Samelson cells do not cover the cases when $C \neq \{e\}$. Therefore we post the following question.

\begin{prob} Is there a way to generalize the decorated double Bott-Samelson construction to include all generalized double Bruhat cells? If yes, how do the Poisson structures arisen from the two approaches compare to each other?
\end{prob}

We expect that such a generalization (if it exists) is related to the \emph{braid cell} defined below.

\vspace{11pt}

\noindent\textbf{Braid Cell.} Let $\G$ be a split semisimple algebraic group. The general position condition between Borel subgroups $\xymatrix{\B \ar@{-}[r] & \B'}$ can be rewritten as $\xymatrix{\B \ar[r]^{w_0} & \B'}$ (see Notation \ref{tits distance notation}). In this case, a double Bott-Samelson cell can be defined as a configuration space of Borel subgroups satisfying the following relative position relation 
\[
\xymatrix{\B_0 \ar@{-->}[r]^b  \ar[d]_{w_0} & \B_2 \ar[d]^{w_0} \\ \B_1 \ar@{-->}[r]_d & \B_3}
\]
where the top and bottom dashed arrows represent a chain of flags with relative position conditions imposed by the positive braids $b$ and $d$ respectively.
When the words of $b$ and $d$ are reduced, then the double Bott-Samelson cell $\conf^u_v(\mathcal{A})$ is naturally isomorphic to the double Bruhat cells $\G^{u,v}$.

In \cite{WY}, Webster and Yakimov introduced a variety $\mathcal{P}^u_{v,w}$ associated to a triple of Weyl group elements $(u,v,w)$, which can be defined as the configuration space of  Borel subgroups satisfying the following relative position relation
\[
\xymatrix{\B_0 \ar[d]_{w_0} \ar[r]^v & \B_2\ar[d]^{uw_0} \\
\B_1 \ar[r]_{w^*} & \B_3}
\]
where $w^*:=w_0ww_0^{-1}$ in the Weyl group. Since $\xymatrix{\B_1\ar[r]^{w^*}& \B_3}$ is equivalent to $\xymatrix{\B_1 & \B_3\ar[l]_{w^{*-1}}}$, the above relative position relation diagram is equivalent to the following one
\[
\xymatrix{\B_0 \ar[d]_{w_0} \ar[r]^v & \B_2\ar[d]^{uw_0} \\
\B_1  & \B_3 \ar[l]^{w^{*-1}}}
\]
The chain $\xymatrix{\B_0 \ar[r]^v & \B_2 \ar[r]^{uw_0} & \B_3 \ar[r]^{w^{*-1}} & \B_1}$ can be treated as a chain of Borel subgroups with relative position condition imposed by a braid $b$, where $b$ is the concatenation of any triple of reduced words of $v$, $uw_0$, and $w^{*-1}$. The above relative position relation diagram reduces to the following one
\begin{equation}
\label{acsafvdf}
\xymatrix{\B_0 \ar[d]_{w_0} \ar@{-->}@/^10ex/[d]^b \\ \B_1}
\end{equation}

Let us take one step further by allowing $b$ to be  any positive braid. The moduli space parametrizing the configurations \eqref{acsafvdf} is called a \emph{braid cell} $\conf_b(\mathcal{B})$. By putting decorations on $\B_0$ and $\B_1$ we can define its decorated version $\conf_b(\mathcal{A})$. The cells $\conf_b(\mathcal{A})$ generalize the \emph{open Richardson varieties} \cite{LamS}. Following the proof of Theorem \ref{mainlcsafv}, one can show that $\conf_b(\mathcal{A})$ is an affine variety. We make the following conjecture.

\begin{conj} The coordinate ring of $\conf_b(\mathcal{A}_{\rm sc})$ is an upper cluster algebra.
\end{conj}

\vspace{11pt}

\noindent\textbf{Legendrian Link Invariants.} As stated earlier, the double Bott-Samelson cell $\conf^b_d(\mathcal{B})$ associated to Dynkin type $\mathrm{A}_r$ is a Legendrian link invariant for positive braids closures. Furthermore, Legendrian isotopies between the positive braid closures $\Lambda_d^b$ and $\Lambda_{d'}^{b'}$ give rise to isomorphisms between $\conf^b_d(\mathcal{B})$ and $\conf^{b'}_{d'}(\mathcal{B})$. We propose the following conjecture.

\begin{conj} \label{vqnbrwb} The isomorphisms $\conf^b_d(\mathcal{B})\overset{\cong}{\longrightarrow}\conf^{b'}_{d'}(\mathcal{B})$ associated to Legendrian isotopies between $\Lambda^b_d$ and $\Lambda^{b'}_{d'}$ are cluster Poisson transformations.
\end{conj}

One strategy to prove this conjecture is to equip $\mathcal{M}_1(\Lambda)$ with a cluster Poisson structure for any Legendrian link $\Lambda$, and then show that the Legendrian versions of the Reidemeister moves induce isomorphisms that preserve the cluster Poisson structures. Note that this is already true for the third Legendrian Reidemeister moves, i.e., the braid moves; but defining a cluster Poisson structure on $\mathcal{M}_1(\Lambda)$ and showing the cluster-ness of the remaining two Legendrian Reidemeister moves still seem to be quite difficult tasks.

Theorem \ref{vfqve1} naturally induces a cluster Poisson structure on $\mathcal{M}_1\left(\Lambda^b_d\right)$. In \cite{STWZ}, Shende, Treumann, Williams, and Zaslow studied cluster Poisson structures on $\mathcal{M}_1(\Lambda)$ for Lengendrian links $\Lambda\subset T^\infty \mathbb{R}^2$ that come from conormal lifts of immersed curves in $\mathbb{R}^2$. Although their ambient contact manifold $T^\infty \mathbb{R}^2$ is different from ours (which is the standard contact $\mathbb{R}^3$), it is still worthwhile to compare these two set-ups and the resulting cluster Poisson structures. Therefore we post the following question.

\begin{prob} How much does $\mathcal{M}_1(\Lambda)$ depend on the ambient contact manifold of $\Lambda$? Do the cluster Poisson structures obtained from double Bott-Samelson cells coincide with those in \cite{STWZ} for Legendrian links that can be embedded in both ways?
\end{prob}

Shende, Treumann, and Zaslow \cite{STZ} introduced a category ${\bf Sh}_{\Lambda}^\bullet(\mathbb{R}^2)$ of constructible  sheaves with singular support controlled by $\Lambda$, which can be viewed as a ``categorification'' of $\conf^b_d(\mathcal{B})\cong \mathcal{M}_1\left(\Lambda^b_d\right)$. 
Conjecture \ref{vqnbrwb} implies that the cluster $\mathrm{K}_2$ counterpart of $\conf^b_d(\mathcal{B})$, namely $\conf^b_d\left(\mathcal{A}_\sc^\fr\right)$, is a Legendrian link invariant as well. We further ask 

\begin{prob} Is there a categorification of $\conf^b_d\left(\mathcal{A}_\sc^\fr\right)$ associated to $\Lambda_d^b$?
\end{prob}

As observed from examples, we conjecture that the number of components in $\Lambda^b_d$ is equal to  $1-\ord_{q=1}g^b_d(q)$. In particular, $g^b_d(q)$ is  a polynomial when $\Lambda^b_d$ is a knot.

\section{Double Bott-Samelson Cells}

\subsection{Flags, Decorated Flags, Relative Position, and Compatibility} \label{flags}

In this Section, we fix notations and investigate several elementary properties of flag varieties.

Let $\C$ be an $r\times r$ symmetrizable generalized Cartan matrix of corank $l$. Let $\tilde{r}:=r+l$.

Let $(\G, \B_+, \B_-, \N, \S)$ be a \emph{twin Tits system}\footnote{We include a summary of twin Tits system in the appendix; see \cite{Ab} and \cite{Kum} for more details. } associated to $\C$. Here $\G$ is a Kac-Peterson group, $\B_+$ and $\B_-$ are opposite Borel subgroups of $\G$, $\N$ is the normalizer of $\T:=\B_+\cap \B_-$ in $\G$, and $\S$ is a set of Coxeter generators for the Weyl group $\W:=\N/\T$. 

Let $\varepsilon\in \{+,-\}$. We define two \emph{flag varieties} 
\[
\mathcal{B}_\varepsilon:=\left\{\text{Borel subgroups of $\G$ that are conjugate to $\B_\varepsilon$}\right\}.
\]
The group $\G$ acts transitively on $\mathcal{B}_\varepsilon$ by conjugation, with $\B_\varepsilon$ self-stabilizing. Therefore, we obtain natural isomorphisms
\[
\mathcal{B}_\varepsilon \cong \G/\B_\varepsilon \cong \B_\varepsilon\backslash \G.
\]
which identify Borel subgroups conjugate to $\B_\varepsilon$ with left and right cosets of $\B_\varepsilon$. When switching left and right cosets, we get $x\B_\varepsilon=\B_\varepsilon x^{-1}.$ Abusing notation, we shall use the terms ``Borel subgroups'' and ``flags'' interchangeably throughout this paper.

Let $\U_\varepsilon=\left[\B_\varepsilon, \B_\varepsilon\right]$ be the maximal unipotent subgroups inside $\B_\varepsilon$. Define \emph{decorated flag varieties}
\[
\mathcal{A}_+:=\G/\U_+ \quad \text{and} \quad \mathcal{A}_-:=\U_-\backslash \G.
\]
The inclusions $\U_\varepsilon\hookrightarrow \B_\varepsilon$ give rise to natural projections
\[
\mathcal{A}_+=\G/\U_+\rightarrow \G/\B_+ \cong \mathcal{B}_+ \quad \text{and} \quad \mathcal{A}_-=\U_-\backslash \G \rightarrow  \B_-\backslash \G \cong \mathcal{B}_-.
\]
We say $\A\in \mathcal{A}_\varepsilon$ is a {\it decorated flag} over $\B\in \mathcal{B}_\varepsilon$ if $\B$ is the image of $\A$ under the above projections.

All $\G$-actions in this paper are left actions unless otherwise specified. For example, $g\in \G$ acts on $\mathcal{A}_-$ by $g.\left(\U_-x\right):= \U_-xg^{-1}$.

The \emph{transposition} is an anti-involution of $\G$ that swaps $\B_+$ and $\B_-$. It induces biregular isomorphisms between (decorated) flag varieties:
\[
\mathcal{B}_+\overset{^t}{\longleftrightarrow} \mathcal{B}_- \quad \text{and} \quad \mathcal{A}_+\overset{^t}{\longleftrightarrow} \mathcal{A}_-.
\]
The images of $\B$ and $\A$ under transposition are denoted by $\B^t$ and $\A^t$ respectively.

\begin{notn} We use superscripts for elements in  $\mathcal{B}_+$, subscripts for elements in  $\mathcal{B}_-$, and parenthesis notations for elements in either flag variety, e.g., 
\begin{enumerate}
   \item Elements of $\mathcal{B}_+$: $\B^0, \B^1, \B^2, \dots$
   \item Elements of $\mathcal{B}_-$: $\B_0, \B_1, \B_2, \dots$
   \item Elements that are in either $\mathcal{B}_+$ or $\mathcal{B}_-$: $\B, \B(0), \B(1), \B(2),\dots$
\end{enumerate}
The same rule applies to decorated flags. 
\end{notn}

In this paper we focus on a pair of Kac-Peterson groups $\G_\sc$ and $\G_\ad$. For semisimple cases, $\G_\sc$ and $\G_\ad$ are the simply connected and adjoint semisimple algebraic groups respectively. In general, when the Cartan matrix $\C$ is not invertible, the construction of $\G_\sc$ and $\G_\ad$ depends on the choices of a lattice $\P\subset \mathfrak{h}^*$ and a basis $\left\{\omega_i\right\}_{i=1}^{\tilde{r}}$ of $\P$. See Appendix A for details.

The center of $\G_\sc$ contains a finite subgroup $\Z$ such that $\G_\ad\cong \G_\sc/\Z$. Note that $\Z\subset \B_\varepsilon$. Therefore the flag varieties $\mathcal{B}_\varepsilon$ associated to either group are isomorphic. For the decorated flag varieties, the covering map $\G_\sc\rightarrow \G_\ad$ induces a $|\Z|$-to-1 covering map $\pi:\mathcal{A}_{\sc,\varepsilon}\rightarrow \mathcal{A}_{\ad,\varepsilon}$ respectively. 

Let $\G$ be either $\G_\sc$ or $\G_\ad$. The group $\G$ admits Bruhat decompositions
\[
\G=\bigsqcup_{w\in \W} \B_+ w\B_+=\bigsqcup_{w\in \W} \B_- w\B_-
\]
and a Birkhoff decomposition
\[
\G=\bigsqcup_{w\in \W} \B_-w\B_+.
\]
Every $x\in \B_-\B_+=\U_-\T\U_+$ admits a unique decomposition (a.k.a. \emph{the Gaussian decomposition})
\[
x=[x]_-[x]_0[x]_+
\]
with $[x]_\varepsilon\in \U_\varepsilon$ and $[x]_0\in \T$. Such an element $x$ is called \emph{Gaussian decomposable}.

The above decompositions induce two $\W$-valued ``distance'' functions and a $\W$-valued ``codistance'' function, which are invariant under $\G$-diagonal actions.

\begin{defn} A pair of flags $\left(x\B_\varepsilon, y\B_\varepsilon\right)$ is of \emph{Tits distance} $d_\varepsilon\left(x\B_\varepsilon, y\B_\varepsilon\right)=w$ if $x^{-1}y\in \B_\varepsilon w\B_\varepsilon$. 

A pair $\left(x\B_-,y\B_+\right)$ is of \emph{Tits codistance} $d\left(x\B_-,y\B_+\right)=w$ if $x^{-1}y\in \B_-w\B_+$. 

A pair  $(\B_0, \B^0)$ is said to be in \emph{general positio}n (or \emph{opposite} to each other) if $d\left(\B_0,\B^0\right)=e$. 
\end{defn}

\begin{notn}\label{tits distance notation} We shall use the following notations to encode the Tits (co)distances between flags:
\begin{enumerate}
    \item $\xymatrix{\B^0\ar[r]^w& \B^1}$ means $d_+\left(\B^0,\B^1\right)=w$.
    \item $\xymatrix{\B_0 \ar[r]^w & \B_1}$ means $d_-\left(\B_0,\B_1\right)=w$.
    \item $\xymatrix{\B_0 \ar@{-}[r]^w & \B^0}$ means $d\left(\B_0,\B^0\right)=w$.
\end{enumerate}
We often omit $w$ in the diagrams if $w=e$. Similar diagrams with decorated flags placed at one or both of the ends imply that the pair of underlying flags are of the indicated Tits (co)distances.
\end{notn}

\begin{lem}\label{undecorated transposition} \begin{enumerate} \item $\xymatrix{\B \ar[r]^w & \B'}$ if and only if $\xymatrix{\B'^t \ar[r]^{w^{-1}} & \B^t}$.
\item $\xymatrix{\B \ar@{-}[r]^w & \B'}$ if and only if $\xymatrix{\B'^t \ar@{-}[r]^{w^{-1}} & \B^t}$.
\end{enumerate}
\end{lem}
\begin{proof} Obvious from definition.
\end{proof}

The following Lemma will be used many times. Its proof has been included in the appendix.

\begin{lem}\label{unique} Let $u, v, w$ be Weyl group elements such that $uv=w$ and $l(u)+l(v)=l(w)$. In each of the following triangles, the black relative position holds if and only if the blue relative position holds. Furthermore, each blue flag is uniquely determined by the pair of black flags.
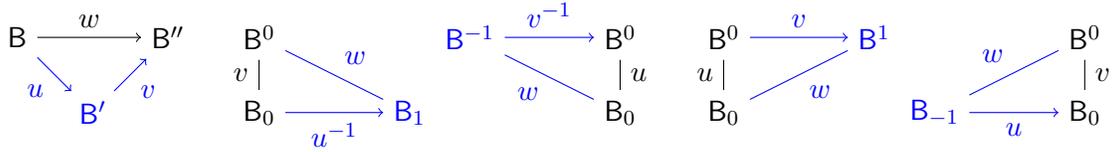
\[
\begin{tikzpicture}[baseline=0ex]
\node (b) at (0,1) [] {$\B$};
\node [blue] (b') at (1,0) [] {$\B'$};
\node (b'') at (2,1) [] {$\B''$};
\draw [->] (b) -- node [above] {$w$} (b'');
\draw [->,blue] (b) -- node [below left] {$u$} (b');
\draw [->,blue] (b') -- node [below right] {$v$} (b'');
\end{tikzpicture} 
\quad
\begin{tikzpicture}[baseline=0ex]
\node (b_0) at (0,1) [] {$\B^0$};
\node (b^0) at (0,0) [] {$\B_0$};
\node [blue] (b^1) at (2,0) [] {$\B_1$};
\draw [->, blue] (b^0) -- node [below] {$u^{-1}$} (b^1);
\draw [blue] (b^1) -- node [above right] {$w$} (b_0);
\draw (b_0) -- node [left] {$v$} (b^0);
\end{tikzpicture}
\begin{tikzpicture}[baseline=0ex]
\node (b_0) at (0,1) [] {$\B^0$};
\node (b^0) at (0,0) [] {$\B_0$};
\node [blue] (b_-1) at (-2,1) [] {$\B^{-1}$};
\draw [->, blue] (b_-1) -- node [above] {$v^{-1}$} (b_0);
\draw [blue] (b_-1) -- node [below left] {$w$} (b^0);
\draw (b_0) -- node [right] {$u$} (b^0);
\end{tikzpicture}
\]
\[
\begin{tikzpicture}[baseline=0ex]
\node (b_0) at (0,1) [] {$\B^0$};
\node (b^0) at (0,0) [] {$\B_0$};
\node [blue] (b_1) at (2,1) [] {$\B^1$};
\draw [->, blue] (b_0) -- node [above] {$v$} (b_1);
\draw [blue] (b_1) -- node [below right] {$w$} (b^0);
\draw (b_0) -- node [left] {$u$} (b^0);
\end{tikzpicture}
\begin{tikzpicture}[baseline=0ex]
\node (b_0) at (0,1) [] {$\B^0$};
\node (b^0) at (0,0) [] {$\B_0$};
\node [blue] (b^-1) at (-2,0) [] {$\B_{-1}$};
\draw [->, blue] (b^-1) -- node [below] {$u$} (b^0);
\draw [blue] (b^-1) -- node [above left] {$w$} (b_0);
\draw (b_0) -- node [right] {$v$} (b^0);
\end{tikzpicture}
\]
\end{lem}

Every dominant weight $\lambda$ of $\G$ gives rise to  a regular function $\Delta_\lambda$ on $\G$ such that $\Delta_{\lambda}(x)=\lambda([x]_0)$ for every $x\in \U_-\T\U_+$. They induce $\G$-invariant functions
\begin{align*}
\Delta_{\lambda}:\mathcal{A}_{-}\times \mathcal{A}_{+}&\rightarrow \mathbb{A}^1\\
\left(\U_-x,~y\U_+\right) & \mapsto \Delta_{\lambda}\left(xy\right).
\end{align*}
When $\G=\G_{\sc}$, we take the fundamental weights $\omega_1, \ldots, \omega_{\tilde{r}}$, and set $\Delta_i:=\Delta_{\omega_i}$.

The following result is an easy consequence of the fact that $\Delta_{\lambda}$ is invariant under transposition.

\begin{lem}\label{2.23} $\Delta_{\lambda}\left(\A_0,\A^0\right)=\Delta_{\lambda}\left(\left(\A^0\right)^t, \left(\A_0\right)^t\right)$.
\end{lem}

A result of Geiss, Leclerc, and Schr\"{o}er (Theorem \ref{gls}) allows us to detect general position of decorated flags based on the $\Delta$ functions.

\begin{thm}\label{gaussian} A pair $(\A_0, \A^0)$ is in general position if and only if $\Delta_{\lambda}\left(\A_0,\A^0\right)\neq 0$ for every dominant $\lambda$.
\end{thm}

\begin{rmk}\label{gaussian'} It suffices to check the non-vanishing of a finite set of $\Delta_\lambda$. For example, when $\G=\G_\sc$, it suffices to check $\Delta_i\neq 0$ for all $i$.
\end{rmk}

Every $w\in \W$ admit two special lifts to $\G$ denoted as $\overline{w}$ and $\doverline{w}$.
The following are refined versions of Bruhat and Birkhoff decomposition: 
\[
\G=\bigsqcup_{w\in \W} \U_+\T\overline{w}\U_+=\bigsqcup_{w\in \W}\U_-\doverline{w}\T\U_-=\bigsqcup_{w\in \W} \U_-\T\overline{w}\U_+.
\]
The factor $t\in \T$ is uniquely determined for every $g\in \G$ in the above decompositions. 

\begin{defn} \begin{enumerate}\item A pair of decorated flags $\xymatrix{x\U_+\ar[r]^w & y\U_+}$ is \emph{compatible} if $x^{-1}y\in \U_+\overline{w}\U_+$. 
\item A pair of decorated flags $\xymatrix{\U_-x \ar[r]^w & \U_-y}$ is \emph{compatible} if $xy^{-1}\in \U_-\doverline{w}\U_-$. 
\end{enumerate}
\end{defn}

\begin{lem}\label{unique compatible} For $\xymatrix{\B \ar[r]^w & \B'}$, a decoration on $\B$ uniquely determines a compatible decoration on $\B'$ and vice versa.
\end{lem}
\begin{proof} It follows from the uniqueness of the $\T$-factor in the refined version of Bruhat decompositions and Birkhoff decomposition.
\end{proof}

The following Lemma is an analogy of the first case of Lemma \ref{unique} for decorated flags.

\begin{lem}\label{unique decorated} Suppose $uv=w$ and $l(u)+l(v)=l(w)$.
\begin{enumerate} 
\item If $\xymatrix{\A \ar[r]^u & \A'}$ and $\xymatrix{\A' \ar[r]^v & \A''}$ are compatible, then so is $\xymatrix{\A \ar[r]^w & \A''}$.
\item If $\xymatrix{\A \ar[r]^w & \A''}$ is compatible, then there is a unique $\A'$ such that $\xymatrix{\A \ar[r]^u & \A'}$ and $\xymatrix{\A' \ar[r]^v & \A''}$ are compatible. 
\end{enumerate}
\end{lem}
\begin{proof} It follows from the fact that $\overline{w}=\overline{u}\,\overline{v}$ and $\doverline{w}=\doverline{u}\,\doverline{v}$.
\end{proof}

\begin{defn} A \emph{pinning} is a pair of decorated flags $(\U_-x, y\U_+)$ such that $xy\in \U_-\U_+$.
\end{defn}

\begin{lem}\label{decorated general position} The following conditions  are equivalent:
\begin{enumerate}
\item the pair $(\A_0, \A^0)$ is a pinning.
\item there exists a unique $z\in \G$ such that $(\A_0, \A^0)=(\U_-z^{-1}, z\U_+)$.
\item we have $\Delta_{\lambda}\left(\A_0,\A^0\right)=1$ for every dominant weight $\lambda$.
\end{enumerate}
Moreover, condition (2) implies that the action of $\G$ on the space of pinnings is free and transitive. When $\G=\G_\sc$, condition (3) can be replaced by showing that $\Delta_{\omega_i}\left(\A_0,\A^0\right)=1$ for $1\leq i \leq \tilde{r}$.
\end{lem}
\begin{proof} (1)$\implies$(2). Suppose $(\A_0, \A^0)=({\U_-x, y\U_+})$ is a pinning. Then $xy=[xy]_-[xy]_+\in \U^-\U^+$. Let $z:=x^{-1}[xy]_-=y[xy]_+^{-1}$. Then $\U_-z^{-1}=\U_-[xy]_-^{-1}x=\U_-x$ and $z\U_+=y[xy]_+^{-1}\U_+=y\U_+$. Note that $\U_-z^{-1}=\U_-z'^{-1}$ implies that $z^{-1}z'\in \U_-$ and $z\U_+=z'\U_+$ implies that $z^{-1}z'\in \U_+$. Therefore $z^{-1}z'\in \U_-\cap \U_+=\{e\}$ and $z=z'$. The uniqueness of $z$ follows.

(2)$\implies$(3). This is obvious from the definition of $\Delta_\lambda$.

(3)$\implies$(1). By Theorem \ref{gaussian}, $\Delta_\lambda\left(\A_0,\A^0\right)=1$ implies that $\A_0,\A^0)$ is in general opposition. Let $\A_0=\U_-x$ and $\A^0=y\U_+$. The product $xy$ is Gaussian decomposable, i.e., $xy=[xy]_-[xy]_0[xy]_+$. The condition $1=\Delta_\lambda\left(\A_0,\A^0\right)=\Delta_\lambda([xy]_0)$ implies that $[xy]_0=e$. Therefore $xy\in \U_-\U_+$.
\end{proof}

\begin{cor}\label{induced pinning} There is a one-to-one correspondence between pinnings and opposite pairs of flags in $\mathcal{A}_-\times \mathcal{B}_+$ (resp. $\mathcal{B}_-\times \mathcal{A}_+$) given by the forgetful map.
\end{cor}
\begin{proof}  Lemma \ref{decorated general position} asserts that every pinning is of the form $\left(\U_-z^{-1}, z\U_+\right)$. Therefore the forgetful map is surjective. For injectivity, if $\left(\U_-z^{-1},z\B_+\right)=\left(\U_-z'^{-1},z'\B_+\right)$ in $\mathcal{A}_-\times \mathcal{B}_+$, then $z^{-1}z'\in \U_-\cap \B_+=\{e\}$ and hence $z=z'$. A similar proof can be applied to the $\mathcal{B}_-\times \mathcal{A}_+$ cases.
\end{proof}

The next lemma shows that notions of compatibility and pinnings respect the transposition.

\begin{lem}\label{decorated transposition} \begin{enumerate}\item  $\xymatrix{\A \ar[r]^w & \A'}$ is compatible if and only if $\xymatrix{\A'^t\ar[r]^{w^{-1}} & \A^t}$ is compatible. 
\item $\xymatrix{\A \ar@{-}[r] & \A'}$ is a pinning if and only if $\xymatrix{\A'^t\ar@{-}[r] & \A^t}$ is a pinning.
\end{enumerate} 
\end{lem}
\begin{proof} (1) follows from the fact that $\overline{w}^t=\doverline{w^{-1}}$ and $\doverline{w}^t=\overline{w^{-1}}$. (2) is trivial.
\end{proof}

\subsection{Double Bott-Samelson Cells}
The semigroup $\Br_+$ of \emph{positive braids} is generated by symbols $s_i$ subject to the braid relations 
\begin{equation}
\label{braid.rel.snc}
\underbrace{s_i s_j \dots}_{m_{ij}}=\underbrace{s_j s_i\dots}_{m_{ij}},
\end{equation}
where $m_{ij}=2,3,4,6$ or $\infty$ according to whether $\C_{ij}\C_{ji}$ is $0,1,2,3$ or $\geq 4$.

A \emph{word} for a positive braid $b\in \Br_+$ is a sequence ${\bf i}=\left(i_1, i_2,\dots,i_n\right)$ such that $b=s_{i_1}s_{i_2}\dots s_{i_n}$. Denote by ${\bf H}(b)$ the set of all words for $b$.

For an arbitrary Weyl group element, its reduced words are related by braid relations. Hence, there is a set-theoretic lift $\W\hookrightarrow \Br_+$.

\begin{defn} Let ${\bf i}=(i_1, \ldots, i_n) \in {\bf H}(b)$. A ${\bf i}$-chain of flags is a sequence of flags:
\[
\B(\vec{i}):=\xymatrix{\B(0) \ar[r]^{s_{{i_1}}} & \B(1) \ar[r]^{s_{{i_2}}} & \cdots \ar[r]^{s_{{i_n}}} & \B(n)}.
\]
Denote by $\mathcal{C}({\bf i})$ the set of ${\bf i}$-chains of flags. 
\end{defn}

Lemma \ref{unique} allows us to do local changes to a chain of flags with prescribed relative positions. If $w=u_1u_2=v_1v_2$ in the Weyl group and  $l(w)=l\left(u_1\right)+l\left(u_2\right)=l\left(v_1\right)+l\left(v_2\right)$, then for every chain
$\xymatrix{\B(0) \ar[r]^{u_1} & \B(1) \ar[r]^{u_2} & \B(2)}$,
there is a unique flag $\B(3)$ such that 
$
\xymatrix{\B(0) \ar[r]^{v_1} & \B(3) \ar[r]^{v_2}  & \B(2)}.
$
\emph{Braid moves} are a class of special local changes: whenever the braid relation \eqref{braid.rel.snc} holds, one can change a chain of flags uniquely from
\[
\xymatrix@C=12pt{\dots \ar[r] \B(k) \ar[r]^{s_i} & \B(k+1) \ar[r]^{s_j} & \B(k+2) \ar[r]^(0.4){s_{i}} & \dots   \B\left(k+m_{ij}-1\right) \ar[rr]^{\text{$s_i$ or $s_j$}} & & \B\left(k+m_{ij}\right) \dots }
\]
to 
\[
\xymatrix@C=12pt{\dots \ar[r] \B(k) \ar[r]^{s_j} & \B'(k+1) \ar[r]^{s_i} & \B'(k+2) \ar[r]^(0.4){s_{j}} & \dots   \B'\left(k+m_{ij}-1\right) \ar[rr]^{\text{$s_j$ or $s_i$}} & & \B\left(k+m_{ij}\right) \dots }.
\]

Let ${\bf i},{\bf j}\in {\bf H}(b)$. Let $\tau$ be a sequence of braid moves taking ${\bf i}$ to ${\bf j}$. It further induces a bijection 
\[
\tau_{\bf i}^{\bf j}:~ \mathcal{C}({\bf i}) \longrightarrow \mathcal{C}({\bf j}).
\] 
\begin{thm}\label{bott-samelson uniqueness}
If ${\bf i}={\bf j}$, then $\tau_{\bf i}^{\bf j}$ is the identity map.
\end{thm}
\begin{proof} Let  ${\bf i}=(i_1, \ldots, i_n)$ and ${\bf j}=(j_1, \ldots, j_n)$. The map $\tau_{\bf i}^{\bf j}$ takes an ${\bf i}$-chain $\B({\bf i})=\left(\B(0), \ldots, \B(n)\right)$ 
to a ${\bf j}$-chain $\B({\bf j})=\left(\B'(0), \ldots, \B'(n)\right)$. It suffices to prove that if ${i_k}=j_k$ for all $m< k\leq n$ then $\B'(m)=\B(m)$.

Without loss of generality we assume that $\B(i)=\B^i\in \mathcal{B}_+$. Let $\B_0$ be a flag opposite to $\B^0$.
By the second case of Lemma \ref{unique} for $(u, v, w)=(s_i, e, s_i)$, we get a unique $\B_1$ such that 
\[
\begin{tikzpicture}[baseline=0ex]
\node (b_0) at (0,1.5) [] {$\B^0$};
\node (b^0) at (0,0) [] {$\B_0$};
\node [blue] (b^1) at (1.5,0) [] {$\B_1$};
\draw [->, blue] (b^0) -- node [below] {$s_{i_1}$} (b^1);
\draw [blue] (b^1) -- node [above right] {$s_{i_1}$} (b_0);
\draw (b_0) -- node [left] {} (b^0);
\end{tikzpicture}.
\] 
The third case of Lemma \ref{unique} implies that $(\B_1, \B^1)$ is in general position. Repeating the same construction through the whole sequence, we obtain an ${\bf i}$-chain $\B({\bf i})^{op}=\left(\B_0, \B_1, \ldots, \B_n\right)$, called an  \emph{opposite chain} of $\B({\bf i})$, such that  $(\B_k, \B^k)$ is in general position for $0\leq k \leq n$. The chain $\B({\bf i})^{op}$ is uniquely determined by the choice of $\B_0$.

Let us apply a braid move $\tau$ to $\B({\bf i})$, obtaining a $\tau({\bf i})$-chain $\B(\tau({\bf i}))$. 
Applying the same braid move $\tau$ to $\B({\bf i})^{op}$, we claim that the obtained chain is an opposite chain of $\B(\tau({\bf i}))$. We prove the claim for an $\mathrm{A}_2$ type braid move below.  The same argument works for $\mathrm{B}_2$ and $\mathrm{G}_2$ type braid moves. 
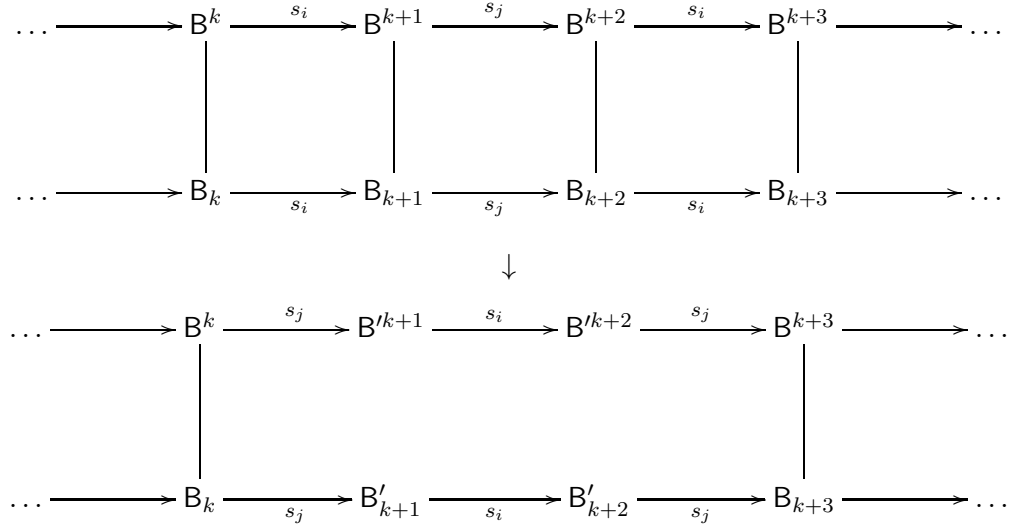
\[
\xymatrix@=20pt{
\dots \ar[r]& \B^k \ar[r]^{s_i} \ar@{-}[d] & \B^{k+1} \ar[r]^{s_j} \ar@{-}[d] & \B^{k+2} \ar[r]^{s_i} \ar@{-}[d] & \B^{k+3} \ar[r] \ar@{-}[d] & \dots\\
\dots \ar[r] & \B_k \ar[r]_{s_i} & \B_{k+1} \ar[r]_{s_j} & \B_{k+2} \ar[r]_{s_i} & \B_{k+3} \ar[r] & \dots
}
\]
\[
\downarrow
\]
\[
\xymatrix@=20pt{
\dots \ar[r]& \B^k \ar[r]^{s_j} \ar@{-}[d] & \B'^{k+1} \ar[r]^{s_i} & \B'^{k+2} \ar[r]^{s_j} & \B^{k+3} \ar[r] \ar@{-}[d] & \dots\\
\dots \ar[r] & \B_k \ar[r]_{s_j} & \B'_{k+1} \ar[r]_{s_i} & \B'_{k+2} \ar[r]_{s_j} & \B_{k+3} \ar[r] & \dots
}
\]
Note that a priori it is not obvious that the two pairs of flags in the middle are opposite. On the other hand, we may run the opposite chain construction with the new top chain starting with the same choice of $\B_0$, obtaining the following chain
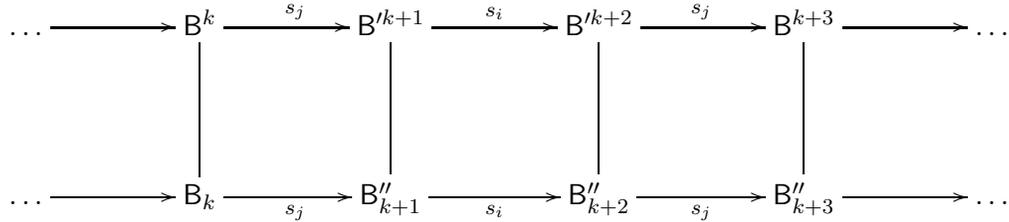
\[
\xymatrix@=20pt{
\dots \ar[r]& \B^k \ar[r]^{s_j} \ar@{-}[d] & \B'^{k+1} \ar[r]^{s_i} \ar@{-}[d] & \B'^{k+2} \ar[r]^{s_j} \ar@{-}[d] & \B^{k+3} \ar[r] \ar@{-}[d] & \dots\\
\dots \ar[r] & \B_k \ar[r]_{s_j} & \B''_{k+1} \ar[r]_{s_i} & \B''_{k+2} \ar[r]_{s_j} & \B''_{k+3} \ar[r] & \dots
}
\]
To show that $\B'_{k+1}=\B''_{k+1}$, $\B'_{k+2}=\B''_{k+2}$, and $\B_{k+3}=\B''_{k+3}$, it suffices to show $\B_{k+3}=\B''_{k+3}$, since the other two follows from the uniqueness part of Lemma \ref{unique}. Note that $w:=s_i s_j s_i=s_j s_i s_j$ is a Weyl group element, and during the opposite chain construction, $\B_{k+3}$ and $\B''_{k+3}$ are both determined by the relative position relations
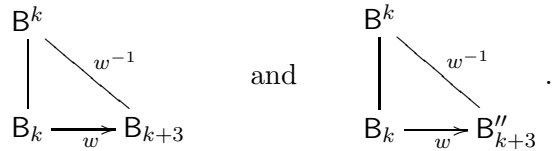
\[
\vcenter{\vbox{\xymatrix{\B^k \ar@{-}[d] \ar@{-}[dr]^{w^{-1}} & \\ \B_k \ar[r]_w & \B_{k+3} }}} \quad \quad \text{and} \quad \quad \vcenter{\vbox{\xymatrix{\B^k \ar@{-}[d] \ar@{-}[dr]^{w^{-1}} & \\ \B_k \ar[r]_w & \B''_{k+3} }}}.
\]
Therefore we can conclude $\B_{k+3}=\B''_{k+3}$ from the uniqueness part of Lemma \ref{unique}. This concludes the proof of the claim that the opposite chain construction is braid move equivariant. 

\vskip 2mm

Now after a sequence of braid moves from ${\bf i}$ to ${\bf j}$, we reach a ${\bf j}$-chain $\B({\bf j})= (\B'^0, \ldots, \B'^n)$ and its opposite chain $\B({\bf j})^{op}= (\B'_0, \ldots, \B'_n)$. If $i_k= j_k$ for $m<k\leq n$, we prove $\B^m=\B'^m$ by induction.  

For $m=n$, it is clear since braid moves do not change the flags at both ends of the chains. Suppose inductively we know that $\B^{m+1}=\B'^{m+1}$. Applying the inductive hypothesis to the opposite chain under the same sequence of braid moves we have $\B_{m+1}=\B'_{m+1}$. Since the opposite sequence construction is braid move equivariant, we know that the bottom primed sequence can be constructed from the top primed chain as an opposite sequence. By construction $\B'^m$ is the unique flag such that
$\vcenter{\vbox{\xymatrix{\B'^m \ar[r]^{s_{j_{m+1}}} \ar@{-}[dr]_{s_{j_{m+1}}} & \B'^{m+1} \ar@{-}[d] \\ & \B'_{m+1}}}}
$,
and  $\B^m$ is the unique flag such that
$
\vcenter{\vbox{\xymatrix{\B^m \ar[r]^{s_{i_{m+1}}} \ar@{-}[dr]_{s_{i_{m+1}}} & \B^{m+1} \ar@{-}[d] \\ & \B_{m+1}}}}
$.
Note that $\B'^{m+1}=\B^{m+1}$ and $\B'_{m+1}=\B_{m+1}$. If $i_{m+1}=j_{m+1}$, then $\B'^m=\B^m$ by Lemma \ref{unique}.
\end{proof}

Theorem \ref{bott-samelson uniqueness} implies that $\tau_{\bf i}^{\bf j}$ is canonical and does not depend on the choice of $\tau$. 
This allows us to replace the word $\vec{i}$ by the positive braid $b$.

\begin{defn}
Given a positive braid $b$, the set of \emph{$b$-chains} of flags is defined as
\[
\mathcal{C}(b)=\bigsqcup_{{\bf i}\in {\bf H}(b)} \mathcal{C}({\bf i}) \left/ \mbox{(bijections $\tau_{\bf i}^{\bf j}$)}\right.
\]
where ${\bf i}$ runs through all the words of $b$. We denote a $b$-chain as $\B(b)$. Similarly, let $\mathcal{C}_\pm(b)\subset \mathcal{C}(b)$ be the set of $b$-chains of flags in $\mathcal{B}_\pm$ respectively. 
\end{defn}

Following the proof of Theorem \ref{bott-samelson uniqueness}, braid moves keep the initial and terminal flags of every chain intact. For $\B(b) \in \mathcal{C}(b)$, we denote its \emph{initial flag} by $\B(b)_i$ and its \emph{terminal flag} by $\B(b)_t$. Since braid moves are equivarient under the $\G$-actions on $\vec{i}$-chains, one can lift the $\G$-action to $\mathcal{C}(b)$. 

\begin{defn}\label{conf(B)} Let $b$ and $d$ be  positive braids. The  \emph{double Bott-Samelson cell} $\conf^b_d(\mathcal{B})$ is 
\[
\conf^{b}_{d}(\mathcal{B}):=\G \left\backslash \left\{\left(\B(b), \B(d)\right) \in \mathcal{C}_+(b)\times \mathcal{C}_-(d) \middle | ~ \begin{array}{l}\text{$\B(d)_i$ is opposite to $\B(b)_i$}\\ \text {$\B(d)_t$ is opposite to $\B(b)_t$}
\end{array} 
\right\}. \right.
\]
\end{defn}

After fixing words $\vec{i}=\left(i_1, i_2,\dots, i_m\right)$ and $\vec{j}=\left(j_1, j_2,\dots, j_n\right)$ for $b$ and $d$, one can write a point in $\conf^b_d(\mathcal{B})$ as the $\G$-orbit of a collection of flags
\[
\xymatrix@=10ex{ \B^0 \ar[r]^{s_{i_1}} \ar@{-}[d] & \B^1 \ar[r]^{s_{i_2}} & \B^2 \ar[r] & \dots \ar[r] & \B^{m-1} \ar[r]^{s_{i_m}} & \B^m \ar@{-}[d] \\
\B_0 \ar[r]_{s_{j_1}} & \B_1 \ar[r]_{s_{j_2}} & \B_2 \ar[r] & \dots \ar[r] & \B_{n-1} \ar[r]_{s_{j_n}} & \B_n}
\] 

\begin{defn}
The \emph{decorated double Bott-Samelson cell} $\conf^b_d\left(\mathcal{A}_\sc\right)$ or $\conf^b_d\left(\mathcal{A}_\ad\right)$ is  defined to be the moduli space of configurations of flags that are points in $\conf^b_d(\mathcal{B})$ together with a decorated flag $\A^0$ over $\B^0:= \B(b)_i$ and a decorated flag $\A_n$ over $\B_n=\B(d)_t$.
\[
\xymatrix@=10ex{ \A^0 \ar[r]^{s_{i_1}} \ar@{-}[d] & \B^1 \ar[r]^{s_{i_2}} & \B^2 \ar[r] & \dots \ar[r] & \B^{m-1} \ar[r]^{s_{i_m}} & \B^m \ar@{-}[d] \\
\B_0 \ar[r]_{s_{j_1}} & \B_1 \ar[r]_{s_{j_2}} & \B_2 \ar[r] & \dots \ar[r] & \B_{n-1} \ar[r]_{s_{j_n}} & \A_n}
\]
\end{defn}

\begin{rmk}
Alternatively, one may think of $\conf^b_d\left(\mathcal{A}_\sc\right)$ as having chains of compatible decorated flags along the horizontal edges, with the top chain determined by $\A^0$ and the bottom chain determined by $\A_n$, and think of $\conf^b_d\left(\mathcal{A}_\ad\right)$ as having two pinnings along the vertical edges, with the left pinning determined  by $\A^0$ and the right pinning determined by $\A_n$, as shown below. See Sections \ref{section.ajod} and \ref{section.ajodad} for more details.
\[
\begin{tikzpicture}[scale=0.8]
\foreach \i in {0,1,3,4}
    {
    \node (t\i) at (\i*1.5,1.5) [] {$\bullet$};
    \node (b\i) at (\i*1.5,0) [] {$\bullet$};
    }
\node (t2) at (3,1.5) [] {$\cdots$};
\node (b2) at (3,0) [] {$\cdots$};
\draw (t0) -- (t1) -- (t2) -- (t3) -- (t4) -- (b4) -- (b3) -- (b2) -- (b1) -- (b0) -- (t0);
\draw [dashed] (-0.5,1) rectangle (6.5,2);
\draw [dashed] (-0.5,-0.5) rectangle (6.5,0.5);
\node at (3,-1) [] {$\conf^b_d\left(\mathcal{A}_\sc\right)$};
\end{tikzpicture}\quad \quad \quad 
\begin{tikzpicture}[scale=0.8]
\foreach \i in {0,1,3,4}
    {
    \node (t\i) at (\i*1.5,1.5) [] {$\bullet$};
    \node (b\i) at (\i*1.5,0) [] {$\bullet$};
    }
\node (t2) at (3,1.5) [] {$\cdots$};
\node (b2) at (3,0) [] {$\cdots$};
\draw (t0) -- (t1) -- (t2) -- (t3) -- (t4) -- (b4) -- (b3) -- (b2) -- (b1) -- (b0) -- (t0);
\draw [dashed] (5.5,-0.5) rectangle (6.5,2);
\draw [dashed] (-0.5,-0.5) rectangle (0.5,2);
\node at (3,-1) [] {$\conf^b_d\left(\mathcal{A}_\ad\right)$};
\end{tikzpicture}
\]
\end{rmk}

\begin{defn}
The \emph{framed double Bott-Samelson cell} $\conf^b_d\left(\mathcal{A}_\sc^\fr\right)$ is defined to be the subspace of $\conf^b_d\left(\mathcal{A}_\sc\right)$ whose decorated flags $\A^0$ and $\A_n$ induce a decorated flag over each flag via Lemma \ref{unique compatible} and Corollary \ref{induced pinning}, so that any two consecutive decorated flags on the perimeter are either compatible or forming a pinning. 
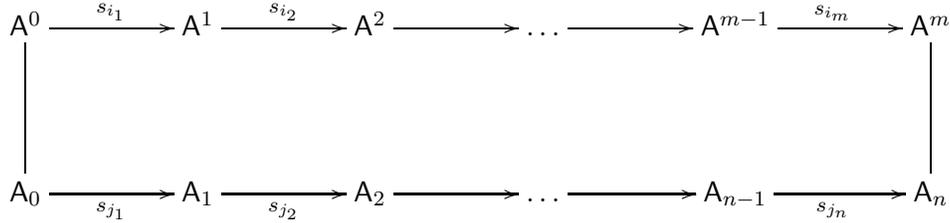
\[
\xymatrix@=10ex{ \A^0 \ar[r]^{s_{i_1}} \ar@{-}[d] & \A^1 \ar[r]^{s_{i_2}} & \A^2 \ar[r] & \dots \ar[r] & \A^{m-1} \ar[r]^{s_{i_m}} & \A^m \ar@{-}[d] \\
\A_0 \ar[r]_{s_{j_1}} & \A_1 \ar[r]_{s_{j_2}} & \A_2 \ar[r] & \dots \ar[r] & \A_{n-1} \ar[r]_{s_{j_n}} & \A_n}
\]
\end{defn}

\begin{rmk} Sometimes the indices on the flags in the diagrams of double Bott-Samelson cells may not start from 0, as we will soon see in the next subsection. 
\end{rmk}

There are canonical maps
\begin{equation}\label{maps}
\xymatrix{\conf^b_d\left(\mathcal{A}^\fr_\sc\right)\ar[r]^e& \conf^b_d\left(\mathcal{A}_\sc\right) \ar[d]^\pi & \\
\conf^b_d(\mathcal{B}) & \conf^b_d\left(\mathcal{A}_\ad\right) \ar[l]^q &}
\end{equation}
where the maps $e$ and $q$ are forgetful maps and  $\pi$ is induced by the projection $\pi:\mathcal{A}_\sc\rightarrow \mathcal{A}_\ad$. 

\begin{prop} Let $\mathcal{A}=\mathcal{A}_\sc$ or $\mathcal{A}_\ad$. The forgetful map $\conf^b_d\left(\mathcal{A}\right)\rightarrow \conf^b_d(\mathcal{B})$ makes $\conf^b_d(\mathcal{A})$ a $\T\times \T$-principal bundle over $\conf^b_d(\mathcal{B})$.  
\end{prop}
\begin{proof} It follows from that $\T$ acts freely and transitively on fibers of the projection  $\mathcal{A}_\pm\rightarrow \mathcal{B}_\pm$. 
\end{proof}

\subsection{Reflections and Transposition}\label{simplereflection}

In this section we introduce five natural biregular morphisms between double Bott-Samelson cells: four reflection maps and a transposition map.

Let us start with the left reflection $_i r:\conf^b_{s_i d}(\mathcal{B})\rightarrow \conf^{s_i b}_d(\mathcal{B})$. Recall that elements of $\conf^b_{s_i d}(\mathcal{B})$ are equivalence classes of collections of flags of the following relative pattern
\[
\begin{tikzpicture}
\node (d1) at (0,0) [] {$\B_0$};
\node (d2) at (2,0) [] {$\B_1$};
\node (u1) at (2,2) [] {$\B^0$};
\node (d3) at (4,0) [] {$\dots$};
\node (u2) at (4,2) [] {$\dots$};
\draw [->] (d1) -- node [below] {$s_i$} (d2);
\draw [->] (d2) -- (d3);
\draw [->] (u1) -- (u2);
\draw (u1) -- (d1);
\end{tikzpicture}
\]
By Proposition \ref{unique}, there is a unique Borel subgroup $\B_{-1}$ such that 
\[
\begin{tikzpicture}
\node (d1) at (0,0) [] {$\B_0$};
\node (d2) at (2,0) [] {$\B_1$};
\node (u1) at (2,2) [] {$\B^0$};
\node (d3) at (4,0) [] {$\dots$};
\node (u2) at (4,2) [] {$\dots$};
\node (u0) at (0,2) [] {$\B^{-1}$};
\draw [->] (d1) -- node [below] {$s_i$} (d2);
\draw [->] (d2) -- (d3);
\draw [->] (u1) -- (u2);
\draw [->] (u0) -- node [above] {$s_i$}(u1);
\draw (u0) -- node [left] {$s_i$} (d1);
\end{tikzpicture}
\]
Note that $(\B^{-1}, \B_1)$ is in general position. Hence we  get a configuration belongs to  $\conf^{s_i b}_d(\mathcal{B})$:
\[
\begin{tikzpicture}
\node (u1) at (2,2) [] {$\B^0$};
\node (d2) at (2,0) [] {$\B_1$};
\node (d3) at (4,0) [] {$\dots$};
\node (u2) at (4,2) [] {$\dots$};
\node (u0) at (0,2) [] {$\B^{-1}$};
\draw [->] (d2) -- (d3);
\draw [->] (u1) -- (u2);
\draw [->] (u0) -- node [above] {$s_i$}(u1);
\draw (u0) -- (d2);
\end{tikzpicture}
\]
This defines a \emph{left reflection} 
\[
_i r:\conf_{s_i d}^b(\mathcal{B})\rightarrow \conf_d^{s_i b}(\mathcal{B}).
\]

One can reverse this construction and get another left reflection that is the inverse of $_i r$:
\[
^i r:\conf^{s_i b}_d(\mathcal{B})\rightarrow \conf_{s_i d}^b(\mathcal{B}).
\]

Let us apply the same construction on the right: the flags $\B^{m+1}$ and $\B_{n+1}$ uniquely determine each other and hence we can define two right reflections that are inverse of each other:
\[
\begin{tikzpicture}[baseline=5ex]
\node (u1) at (0,2) [] {$\dots$};
\node (u2) at (2,2) [] {$\B^m$};
\node (d1) at (0,0) [] {$\dots$};
\node (d2) at (2,0) [] {$\B_n$};
\node (d3) at (4,0) [] {$\B_{n+1}$};
\draw [->] (u1) -- (u2);
\draw [->] (d1) -- (d2);
\draw [->] (d2) -- node[below] {$s_i$} (d3);
\draw (u2) -- (d3);
\end{tikzpicture} \quad \longleftrightarrow \quad \begin{tikzpicture}[baseline=5ex]
\node (u1) at (0,2) [] {$\dots$};
\node (u2) at (2,2) [] {$\B^m$};
\node (d1) at (0,0) [] {$\dots$};
\node (d2) at (2,0) [] {$\B_n$};
\node (u3) at (4,2) [] {$\B^{m+1}$};
\draw [->] (u1) -- (u2);
\draw [->] (d1) -- (d2);
\draw [->] (u2) -- node[above] {$s_i$} (u3);
\draw (u3) -- (d2);
\end{tikzpicture}
\]
\[
r_i:\conf_{ds_i}^b(\mathcal{B})\rightarrow \conf^{bs_i}_d(\mathcal{B}), \quad \quad \text{and} \quad \quad r^i:\conf^{bs_i}_d(\mathcal{B})\rightarrow\conf_{ds_i}^b(\mathcal{B}). 
\]

\begin{notn} We adopt the convention of writing the reflection with the index $i$ at one of the four corners indicating the double Bott-Samelson cell from which the reflection map originates.
\end{notn}

For $b=s_{i_1}s_{i_2}\cdots s_{i_p} \in \Br_+$, we set $b^\circ = s_{i_p}\ldots s_{i_2} s_{i_1}$. By composing reflection maps in the four possible directions we get the following four biregular isomorphisms
\[
\xymatrix{\conf^{d^\circ b}_e(\mathcal{B}) & & \conf^{bd^\circ}_e(\mathcal{B}) \\
& \conf^b_d(\mathcal{B}) \ar[ul]_\cong  \ar[ur]^\cong \ar[dl]^\cong \ar[dr]_\cong & \\
\conf^e_{b^\circ d}(\mathcal{B}) & & \conf^e_{db^\circ} (\mathcal{B})
}
\]

The reflection maps on $\conf^b_d(\mathcal{B})$ can be naturally extended to reflection maps on $\conf^b_d\left(\mathcal{A}\right)$. Below we present the construction by using the left reflection $_i r$ as an example. Let us  start with the configuration on the left. First we reflect $\B_0$ from the lower left hand corner to become $\B^{-1}$ at the upper left hand corner; then we find the unique decorated flag $\A^{-1}$ over $\B^{-1}$ such that  $(\A^{-1}, \A^0)$ is compatible (Lemma \ref{unique compatible}); finally we forget the decoration of $\A^0$ and ``downgrade'' it to $\B^0$.

\[
\quad \quad \quad\begin{tikzpicture}[baseline=5ex]
\node (d1) at (0,0) [] {$\B_0$};
\node (d2) at (2,0) [] {$\B_1$};
\node (u1) at (2,2) [] {$\A^0$};
\node (d3) at (4,0) [] {$\dots$};
\node (u2) at (4,2) [] {$\dots$};
\draw [->] (d1) -- node [below] {$s_i$} (d2);
\draw [->] (d2) -- (d3);
\draw [->] (u1) -- (u2);
\draw (u1) -- (d1);
\end{tikzpicture} \quad \quad \rightsquigarrow \quad \quad \begin{tikzpicture}[baseline=5ex]
\node (u1) at (2,2) [] {$\A^0$};
\node (d2) at (2,0) [] {$\B_1$};
\node (d3) at (4,0) [] {$\dots$};
\node (u2) at (4,2) [] {$\dots$};
\node (u0) at (0,2) [] {$\B^{-1}$};
\draw [->] (d2) -- (d3);
\draw [->] (u1) -- (u2);
\draw [->] (u0) -- node [above] {$s_i$}(u1);
\draw (u0) -- (d2);
\end{tikzpicture}
\]
\[
\rightsquigarrow \quad \quad\begin{tikzpicture}[baseline=5ex]
\node (u1) at (2,2) [] {$\A^0$};
\node (d2) at (2,0) [] {$\B_1$};
\node (d3) at (4,0) [] {$\dots$};
\node (u2) at (4,2) [] {$\dots$};
\node (u0) at (0,2) [] {$\A^{-1}$};
\draw [->] (d2) -- (d3);
\draw [->] (u1) -- (u2);
\draw [->] (u0) -- node [above] {$s_i$}(u1);
\draw (u0) -- (d2);
\end{tikzpicture} \quad \quad \rightsquigarrow \quad \quad \begin{tikzpicture}[baseline=5ex]
\node (u1) at (2,2) [] {$\B^0$};
\node (d2) at (2,0) [] {$\B_1$};
\node (d3) at (4,0) [] {$\dots$};
\node (u2) at (4,2) [] {$\dots$};
\node (u0) at (0,2) [] {$\A^{-1}$};
\draw [->] (d2) -- (d3);
\draw [->] (u1) -- (u2);
\draw [->] (u0) -- node [above] {$s_i$}(u1);
\draw (u0) -- (d2);
\end{tikzpicture}
\]

Below is an illustration of the left reflection $^i r$. It moves the decoration before  reflection.
\[
\quad \quad \quad\begin{tikzpicture}[baseline=5ex]
\node (u1) at (2,2) [] {$\B^1$};
\node (d2) at (2,0) [] {$\B_1$};
\node (d3) at (4,0) [] {$\dots$};
\node (u2) at (4,2) [] {$\dots$};
\node (u0) at (0,2) [] {$\A^0$};
\draw [->] (d2) -- (d3);
\draw [->] (u1) -- (u2);
\draw [->] (u0) -- node [above] {$s_i$}(u1);
\draw (u0) -- (d2);
\end{tikzpicture}
 \quad \quad \rightsquigarrow \quad \quad 
 \begin{tikzpicture}[baseline=5ex]
\node (u1) at (2,2) [] {$\A^1$};
\node (d2) at (2,0) [] {$\B_1$};
\node (d3) at (4,0) [] {$\dots$};
\node (u2) at (4,2) [] {$\dots$};
\node (u0) at (0,2) [] {$\A^0$};
\draw [->] (d2) -- (d3);
\draw [->] (u1) -- (u2);
\draw [->] (u0) -- node [above] {$s_i$}(u1);
\draw (u0) -- (d2);
\end{tikzpicture}
\]
\[
\rightsquigarrow \quad \quad\begin{tikzpicture}[baseline=5ex]
\node (u1) at (2,2) [] {$\A^1$};
\node (d2) at (2,0) [] {$\B_1$};
\node (d3) at (4,0) [] {$\dots$};
\node (u2) at (4,2) [] {$\dots$};
\node (u0) at (0,2) [] {$\B^0$};
\draw [->] (d2) -- (d3);
\draw [->] (u1) -- (u2);
\draw [->] (u0) -- node [above] {$s_i$}(u1);
\draw (u0) -- (d2);
\end{tikzpicture} \quad \quad \rightsquigarrow \quad \quad \begin{tikzpicture}[baseline=5ex]
\node (d1) at (0,0) [] {$\B_{-1}$};
\node (d2) at (2,0) [] {$\B_0$};
\node (u1) at (2,2) [] {$\A^1$};
\node (d3) at (4,0) [] {$\dots$};
\node (u2) at (4,2) [] {$\dots$};
\draw [->] (d1) -- node [below] {$s_i$} (d2);
\draw [->] (d2) -- (d3);
\draw [->] (u1) -- (u2);
\draw (u1) -- (d1);
\end{tikzpicture}
\]

The following Lemma implies that the reflection maps can be restricted to $\conf^b_d\left(\mathcal{A}_\sc^\fr\right)$.

\begin{lem} Let $\xymatrix{\A_0  \ar[r]^{s_i} & \A_1}$ and $\xymatrix{\A^0 \ar[r]^{s_i} & \A^1 }$ be  compatible pairs. If either of
\begin{enumerate}
    \item $\xymatrix{\A_0 \ar@{-}[r]^{s_i} &\A^0}$
    \item $\xymatrix{\A_1 \ar@{-}[r]^{s_i} &\A^1}$
\end{enumerate}
is true, then $\left(\A_1,\A^0\right)$ is a pinning if and only if $\left(\A_0,\A^1\right)$ is a pinning.
\end{lem}
\begin{proof} Due to the symmetry of the cases, it suffices to prove that under (1), $\left(\A_0,\A^1\right)$ is a pinning if $\left(\A_1,\A^0\right)$ is. Note that by Lemma \ref{unique}, we know that $\xymatrix{\A_0\ar@{-}[r] & \A^1}$ and $\xymatrix{\A_1 \ar@{-}[r] & \A^0}$ already. By Lemma \ref{decorated general position} we may assume without loss of generality that $\A_1=\U_-$ and $\A^0=\U_+$. Then Lemma \ref{unique} together with the compatibility condition on $\xymatrix{\A_0\ar[r]^{s_i} & \A_1}$ implies that $\A_0=\U_-\doverline{s}_i$. On the other hand, by Lemma \ref{moduli of tits codist si} we know that $\B^1=e_i(q)\overline{s}_i\B_+$ and hence $\A^1=e_i(q)\overline{s}_i\U_+$. Now to show that $\left(\A_0,\A^1\right)$ is a pinning, we just need to notice that
\[
\U_-\doverline{s}_ie_i(q)\overline{s}_i\U_+=\U_-e_{-i}(-q)\U_+=\U_-\U_+.\qedhere
\]
\end{proof}

By Lemma \ref{undecorated transposition} and Lemma \ref{decorated transposition}, we obtain biregular isomorphisms $\left(-\right)^t$, called the transposition maps,  from the double Bott-Samelson cells associated to $(b, d)$ to the ones associated to $\left(d^\circ, b^\circ\right)$, as illustrated below:  
\[
\xymatrix@=20pt{ \A^0 \ar[r]^{s_{i_1}} \ar@{-}[d] & \B^1 \ar[r]^{s_{i_2}} & \B^2 \ar[r] & \dots \ar[r] & \B^{m-1} \ar[r]^{s_{i_m}} & \B^m \ar@{-}[d] \\
\B_0 \ar[r]_{s_{j_1}} & \B_1 \ar[r]_{s_{j_2}} & \B_2 \ar[r] & \dots \ar[r] & \B_{n-1} \ar[r]_{s_{j_n}} & \A_n}
\]
\vspace{-12pt}
\[
\downarrow \left(-\right)^t
\]
\vspace{-12pt}
\[
\xymatrix@=20pt{ \left(\A_n\right)^t \ar[r]^{s_{j_n}} \ar@{-}[d] & \left(\B_{n-1}\right)^t \ar[r]^{s_{j_{n-1}}} & \left(\B_{n-2}\right)^t \ar[r] & \dots \ar[r] & \left(\B_1\right)^t \ar[r]^{s_{j_1}} & \left(\B_0\right)^t \ar@{-}[d] \\
\left(\B^m\right)^t \ar[r]_{s_{i_m}} & \left(\B^{m-1}\right)^t \ar[r]_{s_{i_{m-1}}} & \left(\B^{m-2}\right)^t \ar[r] & \dots \ar[r] & \left(\B^{1}\right)^t \ar[r]_{s_{i_1}} & \left(\A^0\right)^t}
\]

\begin{prop}\label{trans&ref} We have the following commutation relations:
\begin{enumerate}
    \item $\left(-\right)^t\circ r_i={^ir}\circ \left(-\right)^t$;
    \item $\left(-\right)^t\circ r^i={_ir}\circ \left(-\right)^t$;
    \item $\left(-\right)^t\circ {_ir}={r^i}\circ \left(-\right)^t$;
    \item $\left(-\right)^t\circ {^ir}={r_i}\circ \left(-\right)^t$.
\end{enumerate}
\end{prop}

\begin{proof}
It follows directly from  the definitions of the transposition and the reflection morphisms.
\end{proof}

\subsection{Affineness of Decorated Double Bott-Samelson Cells}

\begin{thm}\label{affineness of conf} Both $\conf^b_d(\mathcal{A})$ and $\conf^b_d\left(\mathcal{A}_\sc^\fr\right)$ are affine varieties. In particular, $\conf^b_d\left(\mathcal{A}\right)$ is the non-vanishing locus of a single function, hence it is smooth.
\end{thm}

\begin{proof}
Since reflection maps are biregular, it suffices to prove the Theorem for the cases when $b=e$. By Corollary \ref{decorated general position} every element in $\conf^e_d(\mathcal{A})$ admits a  unique representative that takes the  form 
\[
\begin{tikzpicture}
\node (u) at (0,2) [] {$\U_+$};
\node (d1) at (-3,0) [] {$\B_0$};
\node (d2) at (-1,0) [] {$\B_1$};
\node (d3) at (1,0) [] {$\dots$};
\node (d4) at (3,0) [] {$\U_-t$};
\draw [->] (d1) -- node [below] {$s_{i_1}$} (d2);
\draw [->] (d2) -- node [below] {$s_{i_2}$} (d3);
\draw [->] (d3) -- node [below] {$s_{i_n}$} (d4);
\draw (d1) -- (u) -- (d4);
\end{tikzpicture}
\]

Let us rebuild the above picture by starting with the right edge $\xymatrix{\U_+ \ar@{-}[r] & \U_-t}$, which is parametrized by the maximal torus $\T$.  We find $\B_{n-1}$ that is $s_{i_n}$ away from $\B_-$. By Corollary \ref{2.30}, the space of $\B_{n-1}$ is parametrized by $\mathbb{A}^1$. Then we proceed to find $\B_{n-2}$ that is $s_{i_{n-1}}$ away from $\B_{n-1}$ and so on until  arriving at $\B_0$. The parameter space is isomorphic to $\T\times \mathbb{A}^n$. Let us propagate the decoration from $\U_- t$  to  $\B_k$ such that every $(\A_k, \A_{k-1})$ is compatible 
\[
\begin{tikzpicture}
\node (u) at (0,2) [] {$\U_+$};
\node (d1) at (-3,0) [] {$\A_0$};
\node (d2) at (-1,0) [] {$\dots$};
\node (d3) at (1,0) [] {$\A_{n-1}$};
\node (d4) at (3,0) [] {$\U_-t$};
\draw [->] (d1) -- node [below] {$s_{i_1}$} (d2);
\draw [->] (d2) -- node [below] {$s_{i_2}$} (d3);
\draw [->] (d3) -- node [below] {$s_{i_n}$} (d4);
\draw (d4) -- (u);
\end{tikzpicture}
\]

Finally we require that $(\A_0, \U_+)$ is in general position.  By Theorem \ref{gaussian} and Remark \ref{gaussian'}, it is equivalent to the common non-vanishing locus of a finite collection of regular functions on $\T\times \mathbb{A}^n$, which is equivalent to a single non-vanishing locus of the product of these regular functions. This concludes the proof of the affineness and the non-vanishing locus nature of $\conf^e_d(\mathcal{A})$.

The space $\conf^e_d\left(\mathcal{A}_\sc^\fr\right)$ can be realized as a subvariety of $\conf^e_d\left(\mathcal{A}_\sc\right)$ by imposing the condition that $(\U_+, \U_-t)$ and $(\U_+, \A_0)$ in the previous picture are pinnings. By Corollary \ref{decorated general position}, it is equivalent to setting certain regular functions to be 1, which implies the affineness of $\conf^e_d\left(\mathcal{A}_\sc^\fr\right)$.
\end{proof}

\section{Cluster Structures on Double Bott-Samelson Cells}

\subsection{Triangulations, String Diagrams, and Seeds} 
\label{sec.tri.comb}
In this section we construct seeds that define the cluster structures on double Bott-Samelson cells. The general definition of a seed can be found in the appendix.
Our main tool is the \emph{amalgamation} procedure introduced by Fock and Goncharov \cite{FGamalgamation}.

We start with the definition of a triangulation.

\begin{defn} Let $(b,d)$ be a pair of positive braids. 
A \emph{triangulation} associated to $(b,d)$ is a trapezoid with bases of lengths $m=l(b)$ and $n=l(d)$ together with
\begin{enumerate}
    \item a word ${\bf i}$ of $b$ whose letters label the unit intervals along the top base of the trapezoid;
    \item a word ${\bf j}$ of $d$ whose letters label the unit intervals along the bottom base of the trapezoid;
    \item a collection of line segments called \emph{diagonals}, each of which connects a marked point on the top to a marked point on the bottom, and they divide the trapezoid into triangles. 
\end{enumerate}
\end{defn}

\begin{exmp}\label{triangulation exmp} Below is a triangulation associated to the pair of positive braids $\left(s_1 s_2 s_3, s_3 s_1 s_1 s_3 s_2\right)$.
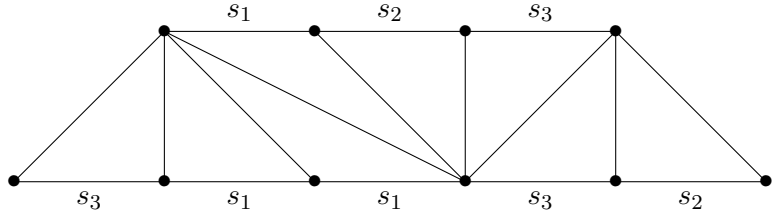
\[
\begin{tikzpicture}[baseline=5ex]
\foreach \i in {0,...,3}
    {
    \coordinate (u\i) at (2*\i+2,2);
    \node at (u\i) [] {$\bullet$};
    }
\foreach \i in {0,...,5}
    {
    \coordinate (d\i) at (2*\i,0);
    \node at (d\i) [] {$\bullet$};
    }
\draw [-] (u0) -- node [above] {$s_1$} (u1);
\draw [-] (u1) -- node [above] {$s_2$} (u2);
\draw [-] (u2) -- node [above] {$s_3$} (u3);
\draw [-] (d0) -- node [below] {$s_3$} (d1);
\draw [-] (d1) -- node [below] {$s_1$} (d2);
\draw [-] (d2) -- node [below] {$s_1$} (d3);
\draw [-] (d3) -- node [below] {$s_3$} (d4);
\draw [-] (d4) -- node [below] {$s_2$} (d5);
\draw (u0) -- (d0);
\draw (u0) -- (d1);
\draw (u0) -- (d2);
\draw (u0) -- (d3);
\draw (u1) -- (d3);
\draw (u2) -- (d3);
\draw (u3) -- (d3);
\draw (u3) -- (d4);
\draw (u3) -- (d5);
\end{tikzpicture}
\]
\end{exmp}

\begin{prop}\label{triangulation move} Any two triangulations associated to the positive braids $(b,d)$ can be transformed into one another via a sequence of moves of the following types:
\begin{enumerate}
    \item flipping a diagonal within a quadrilateral;
    \item changing the labeling of the unit intervals along the bases locally according to a braid relation.
\end{enumerate}
\end{prop}

The construction of seeds involves one more combinatorial gadget called \emph{string diagram}, which is obtained from a triangulation as follows.

\begin{enumerate}
    \item Lay down $\tilde{r}$ horizontal lines across the trapezoid; call the $i$th line from the top the \emph{$i$th level}.
    \item For each triangle of the form $\begin{tikzpicture}[baseline=2ex]
    \node at (-0.5,0) [] {$\bullet$};
    \node at (0.5,0) [] {$\bullet$};
    \node at (0,1) [] {$\bullet$};
    \draw (-0.5,0) -- node [below] {$s_i$} (0.5,0) -- (0,1) -- cycle;
    \end{tikzpicture}$ we put a \emph{node} labeled by $i$ on the $i$th level within the triangle.
    \item For each triangle of the form $\begin{tikzpicture}[baseline=2ex]
    \node at (-0.5,1) [] {$\bullet$};
    \node at (0.5,1) [] {$\bullet$};
    \node at (0,0) [] {$\bullet$};
    \draw (-0.5,1) -- node [above] {$s_i$} (0.5,1) -- (0,0) -- cycle;
    \end{tikzpicture}$ we put a \emph{node} labeled by $-i$ on the $i$th level within the triangle.
    \item The nodes cut the horizontal lines into line segments called \emph{strings}. Strings with nodes at both ends are called \emph{closed strings}. The rest strings are called \emph{open strings}. 
\end{enumerate}

\begin{exmp}\label{3.4} The blue diagram below is the string diagram associated to the triangulation in Example \ref{triangulation exmp} with $r=3$ and $l=1$.
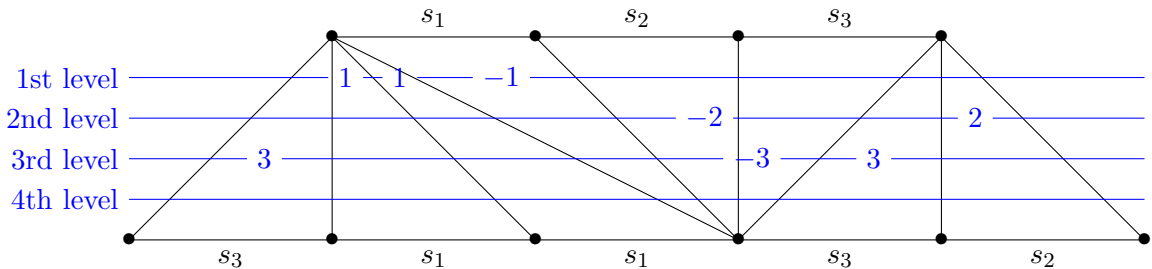
\[
\begin{tikzpicture}[baseline=5ex, scale=0.7]
\foreach \i in {0,...,3}
    {
    \coordinate (u\i) at (3*\i+3,3);
    \node at (u\i) [] {$\bullet$};
    }
\foreach \i in {0,...,5}
    {
    \coordinate (d\i) at (3*\i,0);
    \node at (d\i) [] {$\bullet$};
    }
\draw [-] (u0) -- node [above] {$s_1$} (u1);
\draw [-] (u1) -- node [above] {$s_2$} (u2);
\draw [-] (u2) -- node [above] {$s_3$} (u3);
\draw [-] (d0) -- node [below] {$s_3$} (d1);
\draw [-] (d1) -- node [below] {$s_1$} (d2);
\draw [-] (d2) -- node [below] {$s_1$} (d3);
\draw [-] (d3) -- node [below] {$s_3$} (d4);
\draw [-] (d4) -- node [below] {$s_2$} (d5);
\draw (u0) -- (d0);
\draw (u0) -- (d1);
\draw (u0) -- (d2);
\draw (u0) -- (d3);
\draw (u1) -- (d3);
\draw (u2) -- (d3);
\draw (u3) -- (d3);
\draw (u3) -- (d4);
\draw (u3) -- (d5);
\node [blue] (1) at (3.2,2.4) [] {$1$};
\node [blue] (2) at (4,2.4) [] {$1$};
\node [blue] (3) at (5.5,2.4) [] {$-1$};
\node [blue] (4) at (8.5,1.8) [] {$-2$};
\node [blue] (5) at (12.5,1.8) [] {$2$};
\node [blue] (6) at (2,1.2) [] {$3$};
\node [blue] (7) at (11,1.2) [] {$3$};
\node [blue] (8) at (9.2,1.2) [] {$-3$};
\draw [blue] (0,2.4) node [left] {1st level} --(1) -- (2) -- (3) -- (15,2.4);
\draw [blue] (0,1.8) node [left] {2nd level}-- (4) -- (5) -- (15,1.8);
\draw [blue] (0,1.2) node [left] {3rd level} -- (6) -- (8) -- (7) -- (15,1.2);
\draw [blue] (0,0.6) node [left] {4th level} -- (15,0.6);
\end{tikzpicture}
\] 
\end{exmp}

\begin{notn} We introduce two ways to denote strings in a string diagram: either by a lower case Roman letter starting from $a,b,c,\dots$, or by a symbol $\binom{i}{j}$ with $1\leq i\leq \tilde{r}$ and $j=0,1,2,\dots$. The symbol $\binom{i}{j}$ indicates that it is the $(j+1)$th string on the $i$th level counting from the left. 
\end{notn}

A seed $\vec{s}$ consists of a finite set $I$ of \emph{vertices}, a subset $I^\uf\subset I$ of \emph{unfrozen vertices}, an  $I\times I$ matrix $\epsilon$ called the \emph{exchange matrix}, and  a collection of positive integers  $\{d_a\}_{a\in I}$ called \emph{multipliers}. The entries of $\epsilon$ are integers unless they are in the submatrix $(I\backslash I^{\uf})\times (I\backslash I^{\uf})$. We further require that $\gcd\left(d_a\right)=1$ and  $\hat{\epsilon}_{ab}:=\epsilon_{ab}d_b^{-1}$ is skewsymmetric. This requirement resembles the data of the extended generalized symmetrizable Cartan matrix:  the $\tilde{r}\times \tilde{r}$ matrix $\C$  is equipped with an integral diagonal matrix $\D=\diag\left(\D_1,\dots, \D_r, 1,\dots, 1\right)$ such that $\gcd\left(\D_i\right)=1$ and $\D^{-1}\C$ is  symmetric.

Now we associate a seed to every string diagram (and hence to  every triangulation). Let $I$  be the set of strings and let $I^\uf$ to be the subset of closed strings. For a string $a$ on the $i$th level we set  $d_a:=\D_i$ for $1\leq i\leq r$ and $d_a:=1$ if $i>r$. The exchange matrix $\epsilon$ is 
\[
\epsilon=\sum_{\text{nodes $n$}}\epsilon^{(n)}.
\]
The matrix $\epsilon^{(n)}$  is defined below for $n$ labeled by $1\leq i\leq r$. If $n$ is labeled by $-i$, then $\epsilon^{(n)}$ is defined in the same way except that its entries are given by the opposite numbers. 

\begin{itemize}
    \item Let $a$ and $b$ be the two strings on the $i$th level with the node $n$ as an end point as below
\[
\begin{tikzpicture}
\node (n) at (0,0) [] {$i$};
\draw (-1,0) -- node [above] {$a$} (n) --  node [above] {$b$} (1,0);
\end{tikzpicture}
\]
Then we define
\[
\epsilon_{ab}^{(n)}=-\epsilon_{ba}^{(n)}=-1.
\]
\item Let $c$ be a string on the $j$th level that intersects with the triangle that  contains $n$. We set
\[
\epsilon^{(n)}_{ac}=-\epsilon^{(n)}_{bc}=-\frac{\C_{ji}}{2}, \hskip 12mm
\epsilon^{(n)}_{ca}=-\epsilon^{(n)}_{cb}=\frac{\C_{ij}}{2}.
\]
\item The rest entries are 0.
\end{itemize}

\begin{rmk} Although $\epsilon^{(n)}$ have entries with denominator 2, when we sum up all nodes $n$,  the entries of the resulting matrix $\epsilon$ are all integers except for those between two open strings.
\end{rmk}

When we perform the moves in Proposition \ref{triangulation move}, the corresponding mutations of  string diagrams and seeds  are described in the following proposition. Its  proof is a direct combinatorial check and will be skipped. See \cite[Theorem 3.15]{FGamalgamation}.

\begin{prop}\label{mutation} (1a) If we flip a diagonal inside a quadrilateral  $
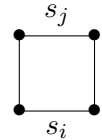
\begin{tikzpicture}[baseline=2ex]
\draw (0,0) -- node [below] {$s_i$} (1,0) -- (1,1) -- node [above] {$s_j$} (0,1) -- cycle;
\node at (0,0) [] {$\bullet$};
\node at (0,1) [] {$\bullet$};
\node at (1,0) [] {$\bullet$};
\node at (1,1) [] {$\bullet$};
\end{tikzpicture}$ with $i\neq j$, the corresponding nodes in the string diagram on different levels will slide across each other and the seed remains the same.  

(1b) If we flip a diagonal inside a quadrilateral $\begin{tikzpicture}[baseline=2ex]
\draw (0,0) -- node [below] {$s_i$} (1,0) -- (1,1) -- node [above] {$s_i$} (0,1) -- cycle;
\node at (0,0) [] {$\bullet$};
\node at (0,1) [] {$\bullet$};
\node at (1,0) [] {$\bullet$};
\node at (1,1) [] {$\bullet$};
\draw (0,1) -- (1,0);
\end{tikzpicture} \quad \longleftrightarrow \quad \begin{tikzpicture}[baseline=2ex]
\draw (0,0) -- node [below] {$s_i$} (1,0) -- (1,1) -- node [above] {$s_i$} (0,1) -- cycle;
\node at (0,0) [] {$\bullet$};
\node at (0,1) [] {$\bullet$};
\node at (1,0) [] {$\bullet$};
\node at (1,1) [] {$\bullet$};
\draw (0,0) -- (1,1);
\end{tikzpicture}$, the corresponding two adjacent nodes in the string diagram are switched and the seed is mutated at the vertex corresponding to the closed string between these two nodes.
\[
\begin{tikzpicture}[baseline=0.5ex]
\node (0) at (0,0) [] {$i$};
\node (1) at (1.5,0) [] {$-i$};
\draw (-1,0) -- (0) -- node [above] {$a$} (1) -- (2.5,0);
\end{tikzpicture} \quad \longleftrightarrow \quad \begin{tikzpicture}[baseline=0.5ex]
\node (0) at (0,0) [] {$-i$};
\node (1) at (1.5,0) [] {$i$};
\draw (-1,0) -- (0) -- node [above] {$a$} (1) -- (2.5,0);
\end{tikzpicture} \quad \quad \text{and} \quad \quad \vec{s} \overset{\mu_a}{\longleftrightarrow} \vec{s}'
\]

(2) If we perform a braid move to the labeling of the intervals along one of the bases of the trapezoid, depending on whether it is of Dynkin type $\mathrm{A}_1\times \mathrm{A}_1$, $\mathrm{A}_2$, $\mathrm{B}_2$, or $\mathrm{G}_2$, the corresponding string diagram undergoes changes as described case by case below and the corresponding seed undergoes a sequence of mutations. We will only depict the cases where a braid move takes place along the bottom base of the trapezoid; the top base cases are completely analogous.
\begin{itemize}
    \item $\C_{ij}=\C_{ji}=0$: corresponding nodes  slide across each other and  seed remains the same.
    
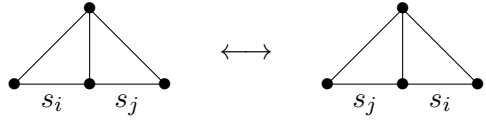
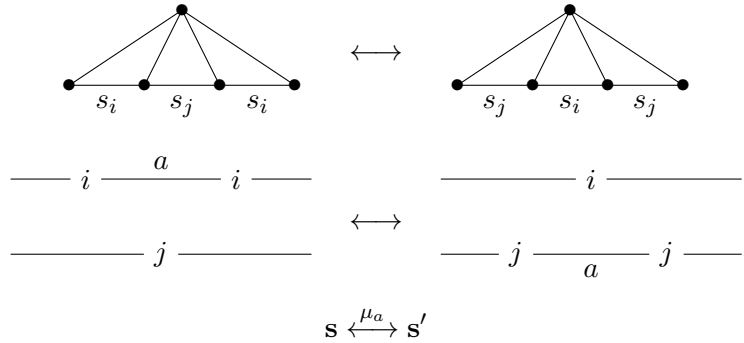
\[
    \begin{tikzpicture}[baseline=2ex]
    \draw (0,0) -- (0,1) -- (-1,0) -- node [below] {$s_i$} (0,0) -- node [below] {$s_j$} (1,0) -- (0,1);
    \node at (0,0) [] {$\bullet$};
    \node at (1,0) [] {$\bullet$};
    \node at (0,1) [] {$\bullet$};
    \node at (-1,0) [] {$\bullet$};
    \end{tikzpicture} \quad \longleftrightarrow \quad \begin{tikzpicture}[baseline=2ex]
    \draw (0,0) -- (0,1) -- (-1,0) -- node [below] {$s_j$} (0,0) -- node [below] {$s_i$} (1,0) -- (0,1);
    \node at (0,0) [] {$\bullet$};
    \node at (1,0) [] {$\bullet$};
    \node at (0,1) [] {$\bullet$};
    \node at (-1,0) [] {$\bullet$};
    \end{tikzpicture}
    \]
    
    \item $\C_{ij}=\C_{ji}=-1$:
    \[
    \begin{tikzpicture}[baseline=2ex]
    \draw (-0.5,0) -- node [below] {$s_i$} (-1.5,0) -- (0,1) -- (-0.5,0) -- node [below] {$s_j$} (0.5,0) -- (0,1) -- (1.5,0) -- node [below] {$s_i$} (0.5,0);
    \foreach \i in {0,...,3} {
    \node at (-1.5+\i,0) [] {$\bullet$};
    }
    \node at (0,1) [] {$\bullet$};
    \end{tikzpicture}\quad \longleftrightarrow \quad \begin{tikzpicture}[baseline=2ex]
    \draw (-0.5,0) -- node [below] {$s_j$} (-1.5,0) -- (0,1) -- (-0.5,0) -- node [below] {$s_i$} (0.5,0) -- (0,1) -- (1.5,0) -- node [below] {$s_j$} (0.5,0);
    \foreach \i in {0,...,3} {
    \node at (-1.5+\i,0) [] {$\bullet$};
    }
    \node at (0,1) [] {$\bullet$};
    \end{tikzpicture}
    \]
    \[
    \begin{tikzpicture}[baseline=2ex]
    \node (0) at (-1,1) [] {$i$};
    \node (1) at (1,1) [] {$i$};
    \node (2) at (0,0) [] {$j$};
    \draw (-2,1) -- (0) -- node [above] {$a$} (1) -- (2,1);
    \draw (-2,0) -- (2) -- (2,0);
    \end{tikzpicture} \quad \longleftrightarrow \quad \begin{tikzpicture}[baseline=2ex]
    \node (0) at (-1,0) [] {$j$};
    \node (1) at (1,0) [] {$j$};
    \node (2) at (0,1) [] {$i$};
    \draw (-2,0) -- (0) -- node [below] {$a$} (1) -- (2,0);
    \draw (-2,1) -- (2) -- (2,1);
    \end{tikzpicture}
    \]
    \[
    \vec{s} \overset{\mu_a}{\longleftrightarrow} \vec{s}'
    \]
    
\item $\C_{ij}=-2$ and $\C_{ji}=-1$: 
\[
\begin{tikzpicture}[baseline=2ex]
\foreach \i in {0,...,4}
{
\draw (0,1) -- (-2+\i,0);
\node at (-2+\i,0) [] {$\bullet$};
}
\draw (-2,0) -- node [below] {$s_i$} (-1,0) -- node [below] {$s_j$} (0,0) -- node [below] {$s_i$} (1,0) -- node [below] {$s_j$} (2,0);
\node at (0,1) [] {$\bullet$};
\end{tikzpicture} \quad \longleftrightarrow \quad \begin{tikzpicture}[baseline=2ex]
\foreach \i in {0,...,4}
{
\draw (0,1) -- (-2+\i,0);
\node at (-2+\i,0) [] {$\bullet$};
}
\draw (-2,0) -- node [below] {$s_j$} (-1,0) -- node [below] {$s_i$} (0,0) -- node [below] {$s_j$} (1,0) -- node [below] {$s_i$} (2,0);
\node at (0,1) [] {$\bullet$};
\end{tikzpicture} 
\]
\[
\begin{tikzpicture}[baseline=2ex]
\node (0) at (0,1) [] {$i$};
\node (1) at (2,1) [] {$i$};
\node (2) at (1,0) [] {$j$};
\node (3) at (3,0) [] {$j$};
\draw (-1,0) -- (2) -- node [below] {$b$} (3) -- (4,0);
\draw (-1,1) -- (0) -- node [above] {$a$} (1) -- (4,1);
\end{tikzpicture}\quad \longleftrightarrow \quad 
\begin{tikzpicture}[baseline=2ex]
\node (0) at (1,1) [] {$i$};
\node (1) at (3,1) [] {$i$};
\node (2) at (0,0) [] {$j$};
\node (3) at (2,0) [] {$j$};
\draw (-1,0) -- (2) -- node [below] {$b$} (3) -- (4,0);
\draw (-1,1) -- (0) -- node [above] {$a$} (1) -- (4,1);
\end{tikzpicture}
\]
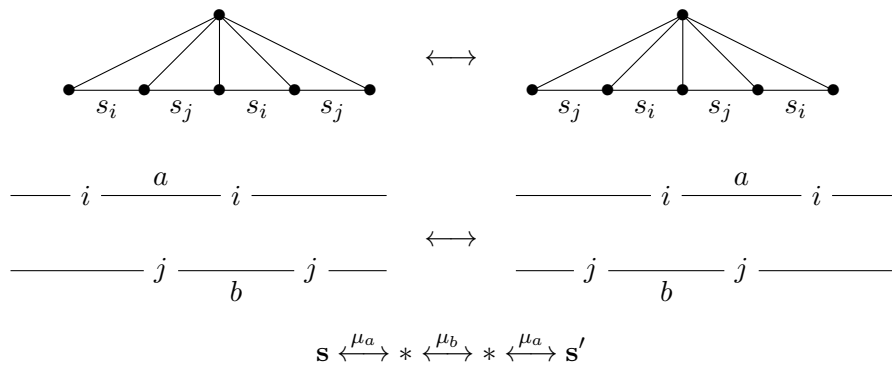
\[
\vec{s} \overset{\mu_a}{\longleftrightarrow} * \overset{\mu_b}{\longleftrightarrow} * \overset{\mu_a}{\longleftrightarrow} \vec{s}'
\]
\item $\C_{ij}=-3$ and $\C_{ji}=-1$:
\[
\begin{tikzpicture}[baseline=2ex,scale=0.8]
\foreach \i in {0,...,6} {
\draw (0,1) -- (-3+\i,0);
\node at (-3+\i, 0) [] {$\bullet$};
}
\draw (-3,0) -- node [below] {$s_i$} (-2,0)-- node [below] {$s_j$} (-1,0) -- node [below] {$s_i$} (0,0) -- node [below] {$s_j$} (1,0) -- node [below] {$s_i$} (2,0) -- node [below] {$s_j$} (3,0);
\node at (0,1) [] {$\bullet$};
\end{tikzpicture}\quad \longleftrightarrow \quad \begin{tikzpicture}[baseline=2ex,scale=0.8]
\foreach \i in {0,...,6} {
\draw (0,1) -- (-3+\i,0);
\node at (-3+\i, 0) [] {$\bullet$};
}
\draw (-3,0) -- node [below] {$s_j$} (-2,0)-- node [below] {$s_i$} (-1,0) -- node [below] {$s_j$} (0,0) -- node [below] {$s_i$} (1,0) -- node [below] {$s_j$} (2,0) -- node [below] {$s_i$} (3,0);
\node at (0,1) [] {$\bullet$};
\end{tikzpicture}
\]
\[
\begin{tikzpicture}[baseline=2ex,scale=0.7]
\node (0) at (0,1) [] {$i$};
\node (1) at (2,1) [] {$i$};
\node (4) at (4,1) [] {$i$};
\node (2) at (1,0) [] {$j$};
\node (3) at (3,0) [] {$j$};
\node (5) at (5,0) [] {$j$};
\draw (-1,0) -- (2) -- node [below] {$c$} (3) -- node [below] {$d$} (5) -- (6,0);
\draw (-1,1) -- (0) -- node [above] {$a$} (1) -- node [above] {$b$} (4) -- (6,1);
\end{tikzpicture}\quad \longleftrightarrow \quad 
\begin{tikzpicture}[baseline=2ex,scale=0.7]
\node (0) at (1,1) [] {$i$};
\node (1) at (3,1) [] {$i$};
\node (4) at (5,1) [] {$i$};
\node (2) at (0,0) [] {$j$};
\node (3) at (2,0) [] {$j$};
\node (5) at (4,0) [] {$j$};
\draw (-1,0) -- (2) -- node [below] {$c$} (3) -- node [below] {$d$} (5) -- (6,0);
\draw (-1,1) -- (0) -- node [above] {$a$} (1) -- node [above] {$b$} (4) -- (6,1);
\end{tikzpicture}
\]
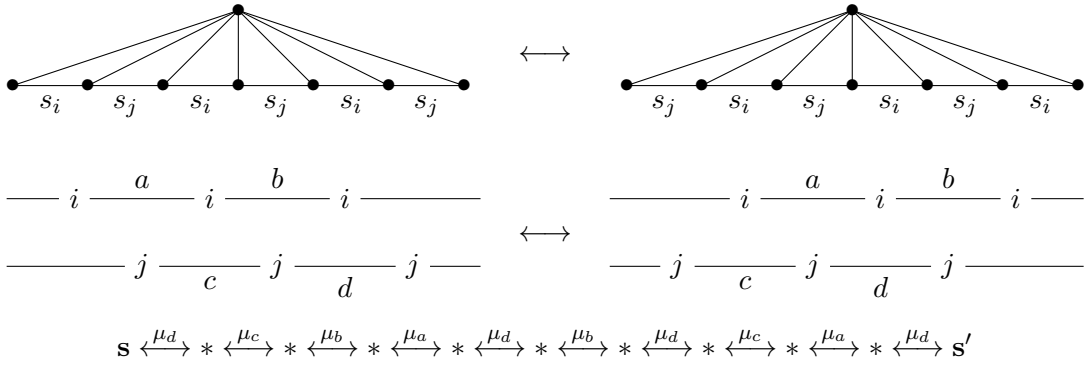
\[
\vec{s} \overset{\mu_d}{\longleftrightarrow} * \overset{\mu_c}{\longleftrightarrow} * \overset{\mu_b}{\longleftrightarrow} * \overset{\mu_a}{\longleftrightarrow} * \overset{\mu_d}{\longleftrightarrow} * \overset{\mu_b}{\longleftrightarrow} * \overset{\mu_d}{\longleftrightarrow} * \overset{\mu_c}{\longleftrightarrow} * \overset{\mu_a}{\longleftrightarrow} * \overset{\mu_d}{\longleftrightarrow} \vec{s}'
\]
\end{itemize}
\end{prop}

The picture at the end of Section 2 shows that on the combinatorial level the transposition map rotates the trapezoid by 180 degrees. The rotation gives rise to a bijection between the triangulations associated to (b,d) and to (d,b). Correspondingly, there is a bijection between the string diagrams, given by a horizontal flip plus a change of sign for every node. 

\begin{exmp} Below are the bijections between triangulations and between string diagrams. Within the two triangulations we point out a pair of corresponding triangles, and within the two string diagrams we point out a pair of corresponding strings.
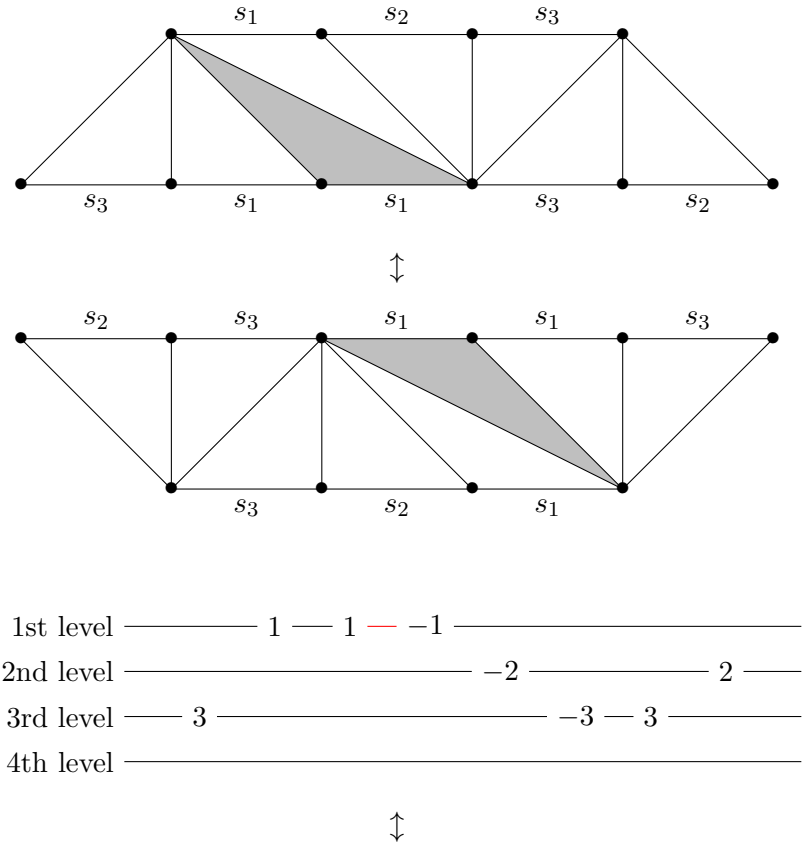
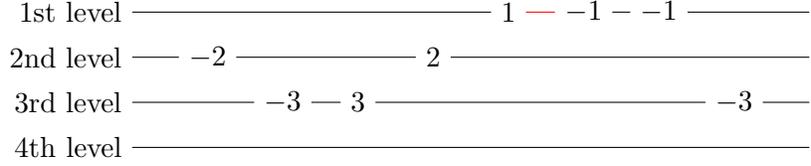
\[
\begin{tikzpicture}[baseline=5ex]
\path [fill=lightgray] (4,0) -- (2,2) -- (6,0) -- cycle;
\foreach \i in {0,...,3}
    {
    \coordinate (u\i) at (2*\i+2,2);
    \node at (u\i) [] {$\bullet$};
    }
\foreach \i in {0,...,5}
    {
    \coordinate (d\i) at (2*\i,0);
    \node at (d\i) [] {$\bullet$};
    }
\draw [-] (u0) -- node [above] {$s_1$} (u1);
\draw [-] (u1) -- node [above] {$s_2$} (u2);
\draw [-] (u2) -- node [above] {$s_3$} (u3);
\draw [-] (d0) -- node [below] {$s_3$} (d1);
\draw [-] (d1) -- node [below] {$s_1$} (d2);
\draw [-] (d2) -- node [below] {$s_1$} (d3);
\draw [-] (d3) -- node [below] {$s_3$} (d4);
\draw [-] (d4) -- node [below] {$s_2$} (d5);
\draw (u0) -- (d0);
\draw (u0) -- (d1);
\draw (u0) -- (d2);
\draw (u0) -- (d3);
\draw (u1) -- (d3);
\draw (u2) -- (d3);
\draw (u3) -- (d3);
\draw (u3) -- (d4);
\draw (u3) -- (d5);
\end{tikzpicture}
\]
\[
\updownarrow
\]
\[
\begin{tikzpicture}[baseline=5ex]
\path [fill=lightgray] (4,2) -- (6,2) -- (8,0) -- cycle;
\foreach \i in {0,...,3}
    {
    \coordinate (u\i) at (8-2*\i,0);
    \node at (u\i) [] {$\bullet$};
    }
\foreach \i in {0,...,5}
    {
    \coordinate (d\i) at (10-2*\i,2);
    \node at (d\i) [] {$\bullet$};
    }
\draw [-] (u0) -- node [below] {$s_1$} (u1);
\draw [-] (u1) -- node [below] {$s_2$} (u2);
\draw [-] (u2) -- node [below] {$s_3$} (u3);
\draw [-] (d0) -- node [above] {$s_3$} (d1);
\draw [-] (d1) -- node [above] {$s_1$} (d2);
\draw [-] (d2) -- node [above] {$s_1$} (d3);
\draw [-] (d3) -- node [above] {$s_3$} (d4);
\draw [-] (d4) -- node [above] {$s_2$} (d5);
\draw (u0) -- (d0);
\draw (u0) -- (d1);
\draw (u0) -- (d2);
\draw (u0) -- (d3);
\draw (u1) -- (d3);
\draw (u2) -- (d3);
\draw (u3) -- (d3);
\draw (u3) -- (d4);
\draw (u3) -- (d5);
\end{tikzpicture}
\]
\vspace{22pt}
\[
\begin{tikzpicture}
\node (a1) at (1,1.2) [] {$1$};
\node (a2) at (2,1.2) [] {$1$};
\node (a3) at (3,1.2) [] {$-1$};
\node (b1) at (4,0.6) [] {$-2$};
\node (b2) at (7,0.6) [] {$2$};
\node (g1) at (0,0) [] {$3$};
\node (g2) at (5,0) [] {$-3$};
\node (g3) at (6,0) [] {$3$};
\draw (-1,-0.6) node [left] {4th level} -- (8,-0.6);
\draw (-1,0) node [left] {3rd level} -- (g1) -- (g2) -- (g3) -- (8,0);
\draw (-1,0.6) node [left] {2nd level} -- (b1) -- (b2) -- (8,0.6);
\draw (-1,1.2) node [left] {1st level} -- (a1) -- (a2);
\draw [red] (a2) -- (a3);
\draw (a3) -- (8,1.2);
\end{tikzpicture}
\]
\[
\updownarrow
\]
\[
\begin{tikzpicture}
\node (a1) at (6,1.2) [] {$-1$};
\node (a2) at (5,1.2) [] {$-1$};
\node (a3) at (4,1.2) [] {$1$};
\node (b1) at (3,0.6) [] {$2$};
\node (b2) at (0,0.6) [] {$-2$};
\node (g1) at (7,0) [] {$-3$};
\node (g2) at (2,0) [] {$3$};
\node (g3) at (1,0) [] {$-3$};
\draw (-1,-0.6) node [left] {4th level} -- (8,-0.6);
\draw (8,0) -- (g1) -- (g2) -- (g3) -- (-1,0)node [left] {3rd level} ;
\draw (8,0.6) -- (b1) -- (b2) -- (-1,0.6) node [left] {2nd level} ;
\draw (8,1.2) -- (a1) -- (a2);
\draw [red] (a2) -- (a3);
\draw (a3) -- (-1,1.2) node [left] {1st level};
\end{tikzpicture}
\]
\end{exmp}

\begin{prop} Transposition induces a seed isomorphism.
\end{prop}
\begin{proof} Note that the local exchange matrix $\epsilon^{(n)}$ is invariant under a simultaneous flip of the strings and changing sign of the nodes. The proposition follows.
\end{proof}

It will be proved in later sections that the transposition also induces cluster isomorphisms between the associated cluster ensembles.

\subsection{\texorpdfstring{Cluster Poisson Structure on $\conf^b_d\left(\mathcal{A}_\ad\right)$ and $\conf^b_d(\mathcal{B})$}{}}
\label{section.ajod}

In this section we associate a coordinate chart of  $\conf^b_d\left(\mathcal{A}_\ad\right)$ to every seed $\vec{s}$ obtained from a triangulation. 
We show that these coordinate charts are related by the cluster Poisson mutations corresponding to the seed mutations in Proposition \ref{mutation}, 
equipping  $\conf^b_d\left(\mathcal{A}_\ad\right)$ with a natural cluster Poisson structure. The space $\conf^b_d(\mathcal{B})$ inherits a cluster Poisson structure from that of $\conf^b_d\left(\mathcal{A}_\ad\right)$ via the projection  $\conf^b_d\left(\mathcal{A}_\ad\right)\rightarrow \conf^b_d(\mathcal{B})$.

Recall that points in $\conf^b_d\left(\mathcal{A}_\ad\right)$ are configurations of flags of the following form.
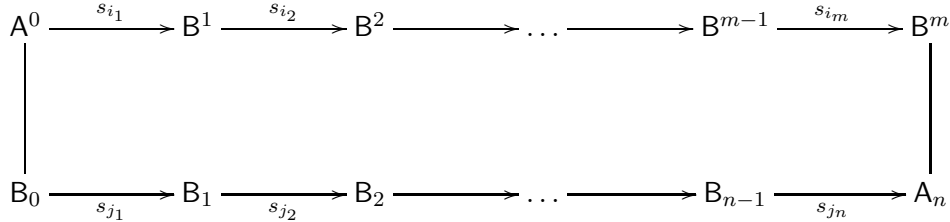
\[
\xymatrix@=10ex{ \A^0 \ar[r]^{s_{i_1}} \ar@{-}[d] & \B^1 \ar[r]^{s_{i_2}} & \B^2 \ar[r] & \dots \ar[r] & \B^{m-1} \ar[r]^{s_{i_m}} & \B^m \ar@{-}[d] \\
\B_0 \ar[r]_{s_{j_1}} & \B_1 \ar[r]_{s_{j_2}} & \B_2 \ar[r] & \dots \ar[r] & \B_{n-1} \ar[r]_{s_{j_n}} & \A_n}
\]

When we draw certain diagonals on the trapezoid to make a triangulation, we can view each of the extra diagonals as imposing a general position condition on the underlying undecorated flags it connects. The coordinate system we construct will depend on a choice of such triangulation.

By Corollary \ref{induced pinning} we know that the left edge $\xymatrix{\B_0\ar@{-}[r] & \A^0}$ is equivalent to a pinning, and by Lemma \ref{decorated general position} we may use the $\G$-action to move the whole configuration to a unique representative with $\B_0=\B_-$ and $\A^0=\U_+$. We call such unique representative the \emph{special representative} and we will use it heavily for the discussion below. Note that in this particular representative, the underlying undecorated flag of $\A^0$ is $\B_+$. The key lemma to define the cluster Poisson coordinates is the following.

\begin{lem}\label{lusztig coord} Fix the special representative with $\A^0=\U_+$ and $\B_0=\B_-$. One can associate a unique unipotent element $e_{-i}(p)$ to each triangle of the form $\begin{tikzpicture}[baseline=2ex] \draw (-0.5,1) -- node[above] {$s_i$} (0.5,1)-- (0,0) -- cycle;
\end{tikzpicture}$ and a unique unipotent element $e_i(q)$ to each triangle of the form $\begin{tikzpicture}[baseline=2ex] \draw (-0.5,0) -- node[below] {$s_i$} (0.5,0) -- (0,1) -- cycle;
\end{tikzpicture}$, with $p,q\neq 0$, such that for any diagonal $\xymatrix{\B_k \ar@{-}[r] & \B^l}$ in the triangulation (including the right most edge $\xymatrix{\B_n \ar@{-}[r] & \B^m}$),
\[
\vcenter{\vbox{\xymatrix{\B^l \ar@{-}[d] \\ \B_k}}}=\vcenter{\vbox{\xymatrix{x\B_+ \ar@{-}[d] \\ x\B_-}}}.
\]
where $x$ is the product of unipotent elements $e_{-i}(p)$ and $e_i(q)$ associated to triangles to the left of the diagonal according to their order from left to right on the triangulation.
\end{lem}

\begin{exmp}\label{3.11} Let us convey the meaning of the above lemma in a concrete example. Consider the following triangulation. Suppose we have fixed the special representative with $\A^0=\U_+$ and $\B_0=\B_-$. Then the lemma claims that there exist unique unipotent elements $e_{-i}(p)$ and $e_i(q)$ associated to the triangles (with $p,q\neq 0$), such that any pair of undecorated flags connected by a diagonal can be obtained by moving the pair $\left(\B_-,\B_+\right)$ by a group element that is the product of those unipotent elements to the left of that diagonal.
\[
\begin{tikzpicture}[scale=0.8]
\node (u0) at (3,3) [] {$\U_+$};
\node (d0) at (0,0) [] {$\B_-$};
\foreach \i in {1,...,3} {
\node (u\i) at (3*\i+3,3) [] {$\B^\i$};
}
\foreach \i in {1,...,4} {
\node (d\i) at (3*\i,0) [] {$\B_\i$};
}
\node (d5) at (15,0) [] {$\A_5$};
\draw (d0) -- (u0) -- (d1);
\draw (d2) -- (u0) -- (d3) -- (u1);
\draw (u2) -- (d3) -- (u3) -- (d4);
\draw (u3) -- (d5);
\draw [->] (u0) -- node [above] {$s_1$} (u1);
\draw [->] (u1) -- node [above] {$s_2$} (u2);
\draw [->] (u2) -- node [above] {$s_3$} (u3);
\draw [->] (d0) -- node [below] {$s_3$} (d1);
\draw [->] (d1) -- node [below] {$s_1$} (d2);
\draw [->] (d2) -- node [below] {$s_1$} (d3);
\draw [->] (d3) -- node [below] {$s_3$} (d4);
\draw [->] (d4) -- node [below] {$s_2$} (d5);
\node at (2,1) [] {$e_3\left(q_1\right)$};
\node at (4,1) [] {$e_1\left(q_2\right)$};
\node at (6,1) [] {$e_1\left(q_3\right)$};
\node at (6,2) [] {$e_{-1}\left(p_1\right)$};
\node at (8,2) [] {$e_{-2}\left(p_2\right)$};
\node at (10,2) [] {$e_{-3}\left(p_3\right)$};
\node at (11,1) [] {$e_3\left(q_4\right)$};
\node at (13,1) [] {$e_2\left(q_5\right)$};
\end{tikzpicture}
\]
For example, $\vcenter{\vbox{\xymatrix{\B^1 \ar@{-}[d] \\ \B_3}}}=\vcenter{\vbox{\xymatrix{x\B_+  \ar@{-}[d] \\ x\B_-}}}$ with $x=e_3\left(q_1\right)e_1\left(q_2\right) e_1\left(q_3\right)e_{-1}\left(p_1\right)$,
and $\vcenter{\vbox{\xymatrix{\B^3 \ar@{-}[d] \\ \B_4}}} =\vcenter{\vbox{\xymatrix{y\B_+\ar@{-}[d] \\ y\B_-}}}$ with $y=e_3\left(q_1\right)e_1\left(q_2\right) e_1\left(q_3\right)e_{-1}\left(p_1\right)e_{-2}\left(p_2\right)e_{-3}\left(p_3\right)e_3\left(q_4\right)$.
\end{exmp}

\noindent\textit{Proof of Lemma \ref{lusztig coord}.} Let us first look at the left most triangle. Without loss of generality we may assume that it is of the following shape (the upside down case is similar).
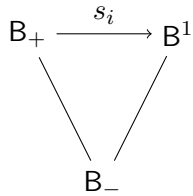
\[
\begin{tikzpicture}
\node (u) at (0,0) [] {$\B_-$};
\node (d0) at (-1,2) [] {$\B_+$};
\node (d1) at (1,2) [] {$\B^1$};
\draw [->] (d0) -- node[above] {$s_i$} (d1);
\draw (d0) -- (u) -- (d1);
\end{tikzpicture}
\]
We would like to show that $\B^1=e_{-i}(p)\B_+$ for some unique $p\neq 0$. Suppose $\B_1=x\B_+$. Then from $\xymatrix{\B_-\ar@{-}[r] & x\B_+}$ we know that $x$ is Gaussian decomposable, i.e., $x=[x]_-[x]_0[x]_+$. In particular $x\B_+=[x]_-\B_+$, so without loss of generality we may replace $x$ by $[x]_-$ and assume that $x\in \U_-$. Now the top edge also tells us that $x\in \B_+s_i\B_+$. But $\U_-\cap \B_+s_i\B_+=\U_{-i}=\left\{e_{-i}(p) \right\}\cong \mathbb{A}^1_p$. This shows that $x=e_{-i}(p)$ for some $p$, and we know that $p\neq 0$ because $\B_1\neq \B_+$. Also note that for different values of $p$, $e_{-i}(p)\B_+$ are distinct flags. Therefore $\B^1=e_{-i}(p)\B_+$ for some unique $p\neq 0$.

Now we move onto the next triangle. But instead of doing a new argument, we can move the whole configuration by $e_{-i}(-p)$ so that $\B^1$ becomes $\B_+$, and then repeat the same argument above. To be more precise, let us suppose that the second triangle looks like the following.
\[
\begin{tikzpicture}[baseline=5ex,scale=0.8]
\node (u0) at (-1,2) [] {$\B_+$};
\node (u1) at (1,2) [] {$e_{-i}(p)\B_+$};
\node (d0) at (0,0) [] {$\B_-$};
\node (d1) at (2,0) [] {$\B_1$};
\draw [->] (u0) -- node [above] {$s_i$} (u1);
\draw [->] (d0) -- node [below] {$s_j$} (d1);
\draw (u0) -- (d0) -- (u1) -- (d1);
\end{tikzpicture}\quad \quad \rightsquigarrow\begin{array}{c}
     \text{move the whole} \\
     \text{configuration} \\
     \text{by $e_{-i}(-p)$}
\end{array} \rightsquigarrow\quad \quad \begin{tikzpicture}[baseline=5ex,scale=0.8]
\node (u0) at (-1,2) [] {$e_{-i}(-p)\B_+$};
\node (u1) at (1,2) [] {$\B_+$};
\node (d0) at (0,0) [] {$\B_-$};
\node (d1) at (2,0) [] {$e_{-i}(-p)\B_1$};
\draw [->] (u0) -- node [above] {$s_i$} (u1);
\draw [->] (d0) -- node [below] {$s_j$} (d1);
\draw (u0) -- (d0) -- (u1) -- (d1);
\end{tikzpicture}
\]
Then by applying the same argument to the second triangle, we get $e_{-i}(-p)\B_1=e_j(q)\B_-$, and hence $\B_1=e_{-i}(p)e_j(q)\B_-$. 

We can repeat such argument again for the third triangle by first moving the special representative by $\left(e_{-i}(p)e_j(q)\right)^{-1}$, and similarly for the fourth triangle and so on. Each step will produce a unique unipotent element as required by the lemma. \qed

\begin{defn} The non-zero numbers $p_i$ and $q_i$ are called \emph{Lusztig factorization coordinates}. 
\end{defn}

From the construction of $\G_\ad$ (see Appendix \ref{appen.a} for more details), we have a basis $\left\{\omega_i^\vee\right\}_{i=1}^{\tilde{r}}$ for the cocharacter lattice of the maximal torus $\T_\ad$, and we can factor the unipotent elements $e_i(q)$ and $e_{-i}(p)$ as
\[
e_i(q)=q^{\omega_i^\vee}e_i q^{-\omega_i^\vee} \quad \quad \text{and} \quad \quad e_{-i}(p)=p^{-\omega_i^\vee}e_{-i} p^{\omega_i^\vee}.
\]
Therefore instead of multiplying the unipotent group elements $e_{-i}\left(p\right)$ and $e_{j}\left(q\right)$ according to their order in the triangulation, we can split each of these unipotent group elements into three factors and write associate them to different parts of the corresponding string diagram. More precisely, we associate the maximal torus elements with strings and associate $e_{\pm i}$ with the nodes $\pm i$. 
\[
\begin{tikzpicture}[baseline=2ex]
    \node at (-0.5,1) [] {$\bullet$};
    \node at (0.5,1) [] {$\bullet$};
    \node at (0,0) [] {$\bullet$};
    \draw (-0.5,1) -- node [above] {$s_i$} (0.5,1) -- (0,0) -- cycle;
    \node at (0,0.7) [] {$p$};
    \end{tikzpicture}  \quad \rightsquigarrow \quad e_{-i}(p) \quad \rightsquigarrow \quad \begin{tikzpicture}[baseline=-0.5ex]
    \node (0) at (0,0) [] {$e_{-i}$};
    \draw (-2,0) -- node [above] {$p^{-\omega_i^\vee}$} (0) -- node [above] {$p^{\omega_i^\vee}$} (2,0);
    \end{tikzpicture} 
\]
\[
\begin{tikzpicture}[baseline=2ex]
    \node at (-0.5,0) [] {$\bullet$};
    \node at (0.5,0) [] {$\bullet$};
    \node at (0,1) [] {$\bullet$};
    \draw (-0.5,0) -- node [below] {$s_i$} (0.5,0) -- (0,1) -- cycle;
    \node at (0,0.3) [] {$q$};
    \end{tikzpicture}  \quad \rightsquigarrow \quad e_{\alpha}(q) \quad \rightsquigarrow \quad \begin{tikzpicture}[baseline=-0.5ex]
    \node (0) at (0,0) [] {$e_i$};
    \draw (-2,0) -- node [above] {$q^{\omega_i^\vee}$} (0) -- node [above] {$q^{-\omega_i^\vee}$} (2,0);
    \end{tikzpicture} 
\]

We then make the following observation. To a closed string $a$ on the $i$th level there are two maximal torus elements attached, one at each end; we can multiply them and get a maximal torus element of the form $X_a^{\omega_i^\vee}$. To an open string $a$ on the $i$th level on the left side of the string diagram there is one maximal torus element attached, and we can put it in the form $X_a^{\omega_i^\vee}$. Note that these numbers $X_a$ are ratios of Lusztig factorization coordinates and hence they are non-zero. 

Moreover, any ordered product of unipotent group elements on the triangulization diagram is equal to an ordered product of maximal torus elements of the form $X_a^{\omega_i^\vee}$ associated to the strings and the Chevalley generators $e_{\pm i}$ associated to the nodes within the corresponding part of the string diagram, according to their order on the string diagram. Note that since $X^{\omega_i^\vee}e_{\pm j}=e_{\pm j}X^{\omega_i^\vee}$ whenever $i\neq j$, ambigous ordering bewteen factors on different levels does not affect the outcome of the products.

However, the Lusztig factorization coordinates are not enough to define the cluster Poisson coordinates. After finding all the Lusztig factorization coordinates, we can multiply all the unipotent elements associated to the triangles together according to the order of the triangulation and get a group element $g$, and by Lemma \ref{lusztig coord}, the right most edge of the special representative is $\vcenter{\vbox{\xymatrix{\B^m\ar@{-}[d] \\ \B_n}}}=\vcenter{\vbox{\xymatrix{g\B_+ \ar@{-}[d] \\ g\B_-}}}$. Since $\A_n$ is a decoration over $\B_n=g\B_-$, there must be some $t\in \T_\ad$ such that $\A_n=\U_-t^{-1}g^{-1}$. Furthermore, since $\left\{\omega_i^\vee\right\}_{i=1}^{\tilde{r}}$ is a basis for the cocharacter lattice of $\T_\ad$, we can write $t=\prod_{i=1}^{\tilde{r}}t_i^{\omega_i^\vee}$ for $t_i\in \mathbb{G}_m$.

Now for an open string $a$ on the $i$th level on the right side of the string diagram, there may be one maximal torus element coming from its left end point, which we can write as $r_i^{\omega_i^\vee}$. We then define $X_a:=r_i t_i\in \mathbb{G}_m$ to be the number we associate to this open string $a$. 

\begin{defn}\label{X def} We call the numbers $X_a$ the \emph{cluster Poisson coordinates} associated to the seed (string diagram/triangulation) on $\conf^b_d\left(\mathcal{A}_\ad\right)$.
\end{defn}

\begin{exmp} Let us demonstrate the construction of the cluster Poisson coordinates associated to the triangulation in Example \ref{3.11}. Remember that we have $r=3$ and $l=1$.
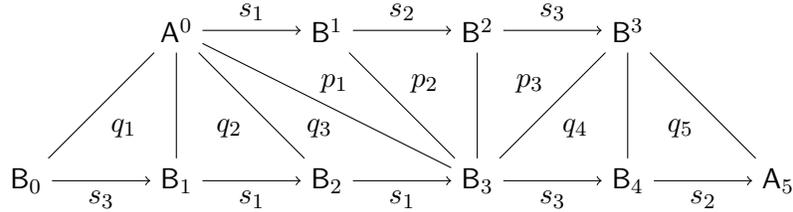
\[
\begin{tikzpicture}
\node (u0) at (2,2) [] {$\A^0$};
\node (d0) at (0,0) [] {$\B_0$};
\foreach \i in {1,...,3} {
\node (u\i) at (2*\i+2,2) [] {$\B^\i$};
}
\foreach \i in {1,...,4} {
\node (d\i) at (2*\i,0) [] {$\B_\i$};
}
\node (d5) at (10,0) [] {$\A_5$};
\draw (d0) -- (u0) -- (d1);
\draw (d2) -- (u0) -- (d3) -- (u1);
\draw (u2) -- (d3) -- (u3) -- (d4);
\draw (u3) -- (d5);
\draw [->] (u0) -- node [above] {$s_1$} (u1);
\draw [->] (u1) -- node [above] {$s_2$} (u2);
\draw [->] (u2) -- node [above] {$s_3$} (u3);
\draw [->] (d0) -- node [below] {$s_3$} (d1);
\draw [->] (d1) -- node [below] {$s_1$} (d2);
\draw [->] (d2) -- node [below] {$s_1$} (d3);
\draw [->] (d3) -- node [below] {$s_3$} (d4);
\draw [->] (d4) -- node [below] {$s_2$} (d5);
\node at (1.3,0.7) [] {$q_1$};
\node at (2.7,0.7) [] {$q_2$};
\node at (3.9,0.7) [] {$q_3$};
\node at (4.1,1.3) [] {$p_1$};
\node at (5.3,1.3) [] {$p_2$};
\node at (6.7,1.3) [] {$p_3$};
\node at (7.3,0.7) [] {$q_4$};
\node at (8.7,0.7) [] {$q_5$};
\end{tikzpicture}
\]
The string diagram of this triangulation is given in Example \ref{3.4}. By factoring the unipotent group elements and taking in the extra factor of $t=\prod_i t_i^{\omega_i^\vee}$ induced by the decoration $\A_n$, we get the following group elements associated to strings and nodes of the string diagram:
\[
\begin{tikzpicture}[scale=0.7]
\node (a1) at (2,3) [] {$e_1$};
\node (a2) at (4,3) [] {$e_1$};
\node (a3) at (6,3) [] {$e_{-1}$};
\node (b1) at (7,1.5) [] {$e_{-2}$};
\node (b2) at (11,1.5) [] {$e_2$};
\node (g1) at (1,0) [] {$e_3$};
\node (g2) at (8,0) [] {$e_{-3}$};
\node (g3) at (10,0) [] {$e_3$};
\draw (-1.5,-1.5) node [left] {4th} -- node [above] {$X_{\binom{4}{0}}^{\omega_4^\vee}$} (13.5,-1.5);
\draw (-1.5,0) node [left] {3rd} -- node [above] {$X_{\binom{3}{0}}^{\omega_3^\vee}$}  (g1) -- node [above] {$X_{\binom{3}{1}}^{\omega_3^\vee}$} (g2) -- node [above] {$X_{\binom{3}{2}}^{\omega_3^\vee}$} (g3) -- node [above] {$X_{\binom{3}{3}}^{\omega_3^\vee}$} (13.5,0);
\draw (-1.5,1.5) node [left] {2nd} -- node [above] {$X_{\binom{2}{0}}^{\omega_2^\vee}$} (b1) -- node [above] {$X_{\binom{2}{1}}^{\omega_2^\vee}$} (b2) -- node [above] {$X_{\binom{2}{2}}^{\omega_2^\vee}$} (13.5,1.5);
\draw (-1.5,3) node [left] {1st} -- node [above] {$X_{\binom{1}{0}}^{\omega_1^\vee}$} (a1) --node [above] {$X_{\binom{1}{1}}^{\omega_1^\vee}$}  (a2) -- node [above] {$X_{\binom{1}{2}}^{\omega_1^\vee}$} (a3) -- node [above] {$X_{\binom{1}{3}}^{\omega_1^\vee}$} (13.5,3);
\end{tikzpicture}
\]
where
\begin{align*}
X_{\binom{1}{0}}=&q_2  & X_{\binom{1}{1}}=& \frac{q_3}{q_2} & X_{\binom{1}{2}}=&\frac{1}{q_3p_1} &  X_{\binom{1}{3}}=& p_1t_1\\
X_{\binom{2}{0}}=&\frac{1}{p_2} & X_{\binom{2}{1}}=& p_2q_5 & X_{\binom{2}{2}}=&\frac{t_2}{q_5} & & \\
X_{\binom{3}{0}}=&q_1 & X_{\binom{3}{1}}=&\frac{1}{q_1p_3} & X_{\binom{3}{2}}=&p_3q_4 &  X_{\binom{3}{3}}=&\frac{t_3}{q_4}\\
X_{\binom{4}{0}}=& t_4. & & & & & &
\end{align*}

Note that we can recover all the Lusztig factorization coordinates from the cluster Poisson coordinates, and hence we can determine all the underlying undecorated flags in the special representative with $\vcenter{\vbox{\xymatrix{\A^0 \ar@{-}[d] \\ \B_0}}}=\vcenter{\vbox{\xymatrix{\A_+ \ar@{-}[d] \\ \B_-}}}$ using the cluster Poisson coordinates. 
\end{exmp}

Recall that the transposition map induces a natural bijection between string diagrams of $(b,d)$ and string diagrams of $\left(d^\circ, b^\circ\right)$ and a natural bijections between strings inside the two corresponding string diagrams, as well as a seed isomorphism between the associated seeds. Now we would like to lift such correspondence to the level of cluster Poisson coordinate charts.

\begin{prop}\label{3.15} Let $\left(X_a\right)$ and $\left(X'_a\right)$ be the cluster Poisson coordinate charts associated to two corresponding string diagrams under transposition. Then under the transposition morphism $t:\conf^b_d\left(\mathcal{A}_\ad\right)\rightarrow \conf^{d^\circ}_{b^\circ}\left(\mathcal{A}_\ad\right)$, $t^*X'_a=X_a$.
\end{prop}
\begin{proof} Since cluster Poisson coordinates are computed using Lusztig factorization coordinates, let us first take a look at how transposition changes Lusztig factorization coordinates. Let $g$ be the group element that is the ordered product of the unipotent elements associated to the triangles and let $t=\prod_{i=1}^{\tilde{r}}t_i^{\omega_i^\vee}$ be the maximal torus element used to make $\A_n=\U_-t^{-1}g^{-1}$. Without loss of generality we may assume that the last triangle in the special representative is of the following form 
\[
\begin{tikzpicture}
\node (u) at (0,2) [] {$g\B_+$};
\node (d0) at (-1.5,0) [] {$ge_i(-q)\B_- $};
\node (d1) at (1.5,0) [] {$\U_-t^{-1}g^{-1}$};
\draw (d0) -- (u) -- (d1);
\draw [->] (d0) -- node [below] {$s_i$} (d1);
\node at (0, 0.7) [] {$e_i(q)$};
\end{tikzpicture}
\]
Then under transposition, this triangle is mapped to the following triangle (here $\overset{tg^t}{\rightsquigarrow}$ means move the whole configuration by $tg^t$).
\[
\begin{tikzpicture}[baseline=5ex, scale=0.7]
\node (d) at (0,0) [] {$\B_-g^t$};
\node (u0) at (-2,2) [] {$\left(g^{-1}\right)^tt^{-1}\U_+$};
\node (u1) at (2,2) [] {$\B_+e_{-i}(-q)g^t$};
\draw (u0) -- (d) -- (u1);
\draw [->] (u0) --  node [above] {$s_i$} (u1);
\end{tikzpicture} \quad \overset{tg^t}{\rightsquigarrow} \quad \begin{tikzpicture}[baseline=5ex, scale=0.7]
\node (d) at (0,0) [] {$\B_-$};
\node (u0) at (-1.5,2) [] {$\U_+$};
\node (u1) at (1.5,2) [] {$\B_+e_{-i}(-q)t^{-1}$};
\draw (u0) -- (d) -- (u1);
\draw [->] (u0) --  node [above] {$s_i$} (u1);
\end{tikzpicture} \]
\[
=\quad \begin{tikzpicture}[baseline=5ex, scale=0.7]
\node (d) at (0,0) [] {$\B_-$};
\node (u0) at (-1.5,2) [] {$\U_+$};
\node (u1) at (1.5,2) [] {$ e_{-i}(qt_i^{-1} )B_+$};
\draw (u0) -- (d) -- (u1);
\draw [->] (u0) --  node [above] {$s_i$} (u1);
\end{tikzpicture}
\]
Therefore the change of Lusztig factorization coordinate for such configuration is $q\mapsto qt_i^{-1} $. By similar computations, it is not hard to find that under transposition, the Lusztig factorization coordinates for corresponding triangles in the two triangulations are related by
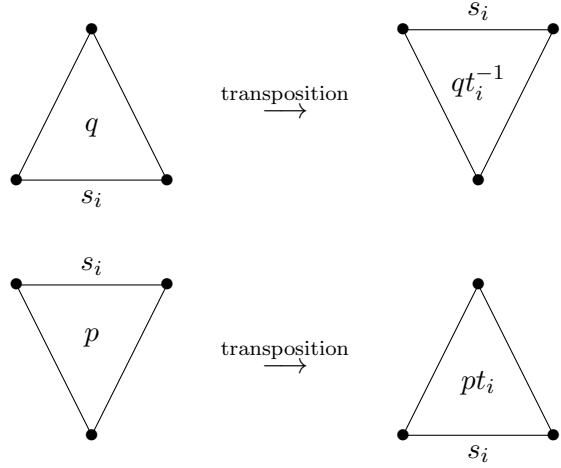
\[
\begin{tikzpicture}[baseline=5ex]
    \node at (-1,0) [] {$\bullet$};
    \node at (1,0) [] {$\bullet$};
    \node at (0,2) [] {$\bullet$};
    \draw (-1,0) -- node [below] {$s_i$} (1,0) -- (0,2) -- cycle;
    \node at (0,0.7) [] {$q$};
    \end{tikzpicture}
 \quad \overset{\text{transposition}}{\longrightarrow} \quad 
    \begin{tikzpicture}[baseline=5ex]
    \node at (-1,2) [] {$\bullet$};
    \node at (1,2) [] {$\bullet$};
    \node at (0,0) [] {$\bullet$};
    \draw (-1,2) -- node [above] {$s_i$} (1,2) -- (0,0) -- cycle;
    \node at (0,1.3) [] {$qt_i^{-1}$};
    \end{tikzpicture}
\]
\[
 \begin{tikzpicture}[baseline=5ex]
    \node at (-1,2) [] {$\bullet$};
    \node at (1,2) [] {$\bullet$};
    \node at (0,0) [] {$\bullet$};
    \draw (-1,2) -- node [above] {$s_i$} (1,2) -- (0,0) -- cycle;
    \node at (0,1.3) [] {$p$};
    \end{tikzpicture}
 \quad \overset{\text{transposition}}{\longrightarrow} \quad 
    \begin{tikzpicture}[baseline=5ex]
    \node at (-1,0) [] {$\bullet$};
    \node at (1,0) [] {$\bullet$};
    \node at (0,2) [] {$\bullet$};
    \draw (-1,0) -- node [below] {$s_i$} (1,0) -- (0,2) -- cycle;
    \node at (0,0.7) [] {$pt_i $};
    \end{tikzpicture}
\]

Since the cluster Poisson coordinates associated to the closed strings are of the form $p/p$, $q/q$, $pq$ or $(pq)^{-1}$, they do not change under transposition as a result of the formulas above.

Let $a'$ be an open string on the $i$th level on the left of the string diagram after transposition. Then $a'$ corresponds to an open string $a$ on the $i$th level on the right of the string diagram before transposition. If the right end point of $a'$ is a node $i$, then the left end point of $a$ is a node $-i$. Let $p$ be the Lusztig factorization coordinate associated to this node $-i$. From the construction of cluster Poisson coordinates and the coordinate transformation formula for Lusztig factorization coordinates above, we see that 
\[
X_a=pt_i =X_{a'}.
\]
On the other hand, if the right end point of $a'$ is a node $-i$, then the left end point of $a$ is a node $i$. Let $q$ be the Lusztig coordinate associated to this node $i$. Then again we have
\[
X_a=q^{-1}t_i=X_{a'}.
\]

Since if we apply transposition twice we get back the identity map, it follows that $t^*X_{a'}=X_a$ also holds for open strings $a'$ on the right. This finishes the proof.
\end{proof}

\begin{cor} The cluster Poisson coordinates associated to closed strings are unchanged under a change of decorations over $\B^0$ and $\B_n$.
\end{cor}
\begin{proof} From the construction of cluster Poisson coordinates it is obvious that the ones associated to closed strings are unaffected by a change of decoration on $\B_n$. But then since the roles of $\B^0$ and $\B_n$ are interchanged under transposition and the cluster Poisson coordinates remain unchanged under transposition, we can deduce that the cluster Poisson coordinates associated to the closed strings are unaffected by a change of decoration on $\B^0$.
\end{proof}

This corollary shows that the subset of cluster Poisson coordinates associated to closed strings actually descend to coordinates on $\conf^b_d(\mathcal{B})$. We call these coordinates the \emph{cluster Poisson coordinates} on $\conf^b_d(\mathcal{B})$ associated to the seed (string diagram/triangulation).

In order to justify the name ``cluster Poisson coordinates'', we need to show that they actually transform as cluster Poisson coordinates, which boils down to showing the cluster Poisson analogue of Proposition \ref{mutation}. 

\begin{prop} \label{x coordinate trans}\begin{enumerate}
    \item If $i\neq j$, then 
    \[
    \begin{tikzpicture}[baseline=2ex]
    \node (a) at (1,1) [] {$i$};
    \node (b) at (0,0) [] {$-j$};
    \draw (-1,0) --node[below] {$X_c$} (b) -- node [below] {$X_d$} (2,0);
    \draw (-1,1) -- node [above] {$X_a$} (a) -- node [above] {$X_b$} (2,1);
    \end{tikzpicture} \quad \longleftrightarrow \quad 
    \begin{tikzpicture}[baseline=2ex]
    \node (a) at (0,1) [] {$i$};
    \node (b) at (1,0) [] {$-j$};
   \draw (-1,0) --node[below] {$X_c$} (b) -- node [below] {$X_d$} (2,0);
    \draw (-1,1) -- node [above] {$X_a$} (a) -- node [above] {$X_b$} (2,1);
    \end{tikzpicture}
    \]
    On the other hand, if we interchange two neighboring nodes of opposite signs on the same level, we get
    \[
    \begin{tikzpicture}[baseline=2ex,scale=0.8]
    \node (a0) at (0,1) [] {$-i$};
    \node (a1) at (2,1) [] {$i$};
    \draw (-1.5,1) node [left] {$i$th level} -- node [above] {$X_a$} (a0) -- node [above] {$X_b$} (a1) -- node [above] {$X_c$} (3.5,1);
    \draw (-1.5,0) node [left] {$j$th level} -- node [below] {$X_d$} (3.5,0);
    \end{tikzpicture} \quad \longleftrightarrow \quad 
    \begin{tikzpicture}[baseline=2ex,scale=0.8]
    \node (a0) at (0,1) [] {$i$};
    \node (a1) at (2,1) [] {$-i$};
    \draw (-1.5,1) node [left] {$i$-level} -- node [above] {$X'_a$} (a0) -- node [above] {$X'_b$} (a1) -- node [above] {$X'_c$} (3.5,1);
    \draw (-1.5,0) node [left] {$j$-level} -- node [below] {$X'_d$} (3.5,0);
    \end{tikzpicture}
    \]
    where
    \[
    X'_a= X_aX_b\left(1+X_b\right)^{-1}, \quad  X'_b=X_b^{-1}, 
    \]
    \[
    X'_c=X_cX_b\left(1+X_b\right)^{-1}, \quad X'_d=X_d\left(1+X_b\right)^{-\C_{ij}}.
    \]
    \item If we perform a braid move to the string diagram, depending on whether it is of Dynkin type $\mathrm{A}_1\times \mathrm{A}_1$, $\mathrm{A}_2$, $\mathrm{B}_2$, or $\mathrm{G}_2$, the cluster Poisson coordinates transform according to the sequences of mutations described in Proposition \ref{mutation} (we will only depict the case where a braid move takes place among nodes labeled by simple roots; the cases where they are labeled by opposite simple roots are completely analogous).
    \begin{itemize}
        \item $\C_{ij}\C_{ji}=0$: 
        \[
    \begin{tikzpicture}[baseline=2ex,scale=0.8]
    \node (a) at (1,1) [] {$i$};
    \node (b) at (0,0) [] {$j$};
    \draw (-1,0) --node[below] {$X_c$} (b) -- node [below] {$X_d$} (2,0);
    \draw (-1,1) -- node [above] {$X_a$} (a) -- node [above] {$X_b$} (2,1);
    \end{tikzpicture} \quad \longleftrightarrow \quad 
    \begin{tikzpicture}[baseline=2ex,scale=0.8]
    \node (a) at (0,1) [] {$i$};
    \node (b) at (1,0) [] {$j$};
   \draw (-1,0) --node[below] {$X_c$} (b) -- node [below] {$X_d$} (2,0);
    \draw (-1,1) -- node [above] {$X_a$} (a) -- node [above] {$X_b$} (2,1);
    \end{tikzpicture}
    \]
    \item $\C_{ij}\C_{ji}=-1$:
    \[
    \begin{tikzpicture}[baseline=2ex,scale=0.8]
    \node (0) at (-1,1) [] {$i$};
    \node (1) at (1,1) [] {$i$};
    \node (2) at (0,0) [] {$j$};
    \draw (-2,1) -- node [above] {$X_a$} (0) -- node [above] {$X_b$} (1) -- node [above] {$X_c$} (2,1);
    \draw (-2,0) -- node [below] {$X_d$} (2) -- node [below] {$X_e$} (2,0);
    \end{tikzpicture} \quad \longleftrightarrow \quad \begin{tikzpicture}[baseline=2ex,scale=0.8]
    \node (0) at (-1,0) [] {$j$};
    \node (1) at (1,0) [] {$j$};
    \node (2) at (0,1) [] {$i$};
    \draw (-2,0) -- node [below] {$X'_d$} (0) -- node [below] {$X'_b$} (1) -- node [below] {$X'_e$} (2,0);
    \draw (-2,1) -- node [above] {$X'_a$} (2) -- node [above] {$X'_c$} (2,1);
    \end{tikzpicture}
    \]
    where
    \[
    X'_a=X_a\left(1+X_b\right), \quad X'_b=X_b^{-1}, \quad X'_c=X_cX_b\left(1+X_b\right)^{-1},
    \]
    \[
    X'_d=X_dX_b\left(1+X_b\right)^{-1}, \quad X'_e=X_e\left(1+X_b\right).
    \]
    \item $\C_{ij}=-2$ and $\C_{ji}=-1$:
    \[
\begin{tikzpicture}[baseline=2ex,scale=0.8]
\node (0) at (0,1) [] {$i$};
\node (1) at (2,1) [] {$i$};
\node (2) at (1,0) [] {$j$};
\node (3) at (3,0) [] {$j$};
\draw (-1,0) -- node [below] {$X_e$} (2) -- node [below] {$X_b$} (3) -- node [below] {$X_f$} (4,0);
\draw (-1,1) -- node [above] {$X_c$} (0) -- node [above] {$X_a$} (1) -- node [above] {$X_d$} (4,1);
\end{tikzpicture}\quad \longleftrightarrow \quad 
\begin{tikzpicture}[baseline=2ex,scale=0.8]
\node (0) at (1,1) [] {$i$};
\node (1) at (3,1) [] {$i$};
\node (2) at (0,0) [] {$j$};
\node (3) at (2,0) [] {$j$};
\draw (-1,0) --  node [below] {$X'_e$} (2) -- node [below] {$X'_b$} (3) -- node [below] {$X'_f$} (4,0);
\draw (-1,1) -- node [above] {$X'_c$} (0) -- node [above] {$X'_a$} (1) -- node [above] {$X'_d$} (4,1);
\end{tikzpicture}
\]
where
\begin{align*}
    X'_a=& X_a\frac{1}{F_b} & X'_b=&\frac{1}{X_a^2X_b}F_a^2 & X'_c=&X_c\frac{F_b}{F_a}\\
    X'_d=& X_dF_a & X'_e=&X_eX_a^2X_b\frac{1}{F_b} & X'_f=&X_fX_b\frac{F_b}{F_a^2}
\end{align*}
and 
\begin{equation}\label{B2 F}
F_a=1+X_b+X_aX_b \quad \quad \quad \quad F_b=1+X_b+2X_aX_b+X_a^2X_b.
\end{equation}
    \item $\C_{ij}=-3$ and $\C_{ji}=-1$:
    \[
\begin{tikzpicture}[baseline=2ex,scale=0.6]
\node (0) at (0,1) [] {$i$};
\node (1) at (2,1) [] {$i$};
\node (4) at (4,1) [] {$i$};
\node (2) at (1,0) [] {$j$};
\node (3) at (3,0) [] {$j$};
\node (5) at (5,0) [] {$j$};
\draw (-1,0) -- node [below] {$X_g$} (2) -- node [below] {$X_c$} (3) -- node [below] {$X_d$} (5) -- node [below] {$X_h$} (6,0);
\draw (-1,1) -- node [above] {$X_e$} (0) -- node [above] {$X_a$} (1) -- node [above] {$X_b$} (4) -- node [above] {$X_f$} (6,1);
\end{tikzpicture}\quad \longleftrightarrow \quad 
\begin{tikzpicture}[baseline=2ex,scale=0.6]
\node (0) at (1,1) [] {$i$};
\node (1) at (3,1) [] {$i$};
\node (4) at (5,1) [] {$i$};
\node (2) at (0,0) [] {$j$};
\node (3) at (2,0) [] {$j$};
\node (5) at (4,0) [] {$j$};
\draw (-1,0) -- node [below] {$X'_g$} (2) -- node [below] {$X'_c$} (3) -- node [below] {$X'_d$} (5) -- node [below] {$X'_h$} (6,0);
\draw (-1,1) -- node [above] {$X'_e$} (0) -- node [above] {$X'_a$} (1) -- node [above] {$X'_b$} (4) -- node [above] {$X'_f$} (6,1);
\end{tikzpicture}
\]
where
\begin{align*}
    X'_a=&X_a\frac{F_d}{F_bF_c} & X'_b=&X_b\frac{F_a}{F_d} & X'_c=&\frac{1}{X_a^3X_b^3X_c^2X_d}\frac{F_a^3}{F_d} & X'_d=&X_c\frac{F_b^3F_c}{F_a^3}\\
    X'_e=&X_e\frac{F_c}{F_a} & X'_f=&X_fF_b & X'_g=&X_a^3X_b^3X_c^2X_dX_g\frac{1}{F_c} &  X'_h=&X_dX_h\frac{F_d}{F_b^3}
\end{align*}
and
\begin{equation}\label{G2 F}
\resizebox{11cm}{!}{\begin{math}\begin{split}
    F_a=&1+X_d+3X_bX_d+3X_b^2X_d+3X_b^2X_cX_d+X_b^3X_d+2X_b^3X_cX_d+X_b^3X_b^2X_d\\
    &+2X_aX_b^2X_cX_d+2X_aX_b^3X_cX_d+2X_aX_b^3X_c^2X_d+X_a^2X_b^3X_c^2X_d,\\
    F_b=&1+X_d+2X_bX_d+X_b^2X_d+X_b^2X_cX_d+X_aX_b^2X_cX_d\\
    F_c=&1+X_d+3X_bX_d+3X_b^2X_d+3X_b^2X_cX_d+X_b^3X_d+2X_b^3X_cX_d+X_b^3X_c^2X_d\\
    &+3X_aX_b^2X_cX_d+3X_aX_b^3X_cX_d
    +3X_aX_b^3X_c^2X_d+3X_a^2X_b^3X_c^2X_d\\
    &+X_a^3X_b^3X_c^2X_d\\
    F_d=&1+2X_d+X_d^2+6X_bX_d+6X_bX_d^2+6X_b^2X_d+15X_b^2X_d^2+3X_b^2X_cX_d+3X_b^2X_cX_d^2\\
    &+2X_b^3X_d+20X_b^3X_d^2+2X_b^3X_cX_d+12X_b^3X_cX_d^2+15X_b^4X_d^2+18X_b^4X_cX_d^2\\
    &+3X_b^4X_c^2X_d^2+6X_b^5X_d^2+12X_b^5X_cX_d^2+6X_b^5X_c^2X_d^2+X_b^6X_d^2+3X_b^6X_cX_d^2\\
    &+3X_b^6X_c^2X_d^2+X_b^6X_c^3X_d^2+3X_aX_b^2X_cX_d+3X_aX_b^2X_cX_d^2+3X_aX_b^3X_cX_d\\
    &+12X_aX_b^3X_cX_d^2+18X_aX_b^4X_cX_d^2+6X_aX_b^4X_c^2X_d^2+12X_aX_b^5X_cX_d^2\\ &+12X_aX_b^5X_c^2X_d^2+3X_aX_b^6X_cX_d^2+6X_aX_b^6X_c^2X_d^2+3X_aX_b^6X_c^3X_d^2+3X_a^2X_b^4X_c^2X_d^2\\
    &+6X_a^2X_b^5X_c^2X_d^2+3X_a^2X_b^6X_c^2X_d^2+3X_a^2X_b^6X_c^3X_d^2+X_a^3X_b^6X_c^3X_d^2.
\end{split}\end{math}}
\end{equation}
    \end{itemize}
The coordinate transformations in the last two cases are written in the form of the factorization formula \eqref{factorization formula x} and are obtained from the corresponding mutation sequences described in Proposition \ref{mutation}.
\end{enumerate}
\end{prop}
\begin{proof}\begin{enumerate}
    \item The case where $i\neq j$ follows from the Lie group identity
    $e_{-j}e_i=e_i e_{-j}.$
    The other case follows from the Lie group identity
    \[
    e_{-i}X_b^{\omega_i^\vee}e_i=\left(\frac{X_b}{1+X_b}\right)^{\omega_i^\vee}\left(e_i\right) \left(\frac{1}{X_b}\right)^{\omega_i^\vee}\left(e_{-i}\right)\left(\frac{X_b}{1+X_b}\right)^{\omega_i^\vee}\prod_{j\neq i}\left(1+X_b\right)^{-\C_{ij}\omega_j^\vee}.
    \]
    \item The case $\C_{ij}=\C_{ji}=0$ follows from the Lie group identity 
    \[
    e_j e_i=e_i e_j.
    \]
    The case $\C_{ij}=\C_{ji}=-1$ follows from the Lie group identity
    \[
    \resizebox{12cm}{!}{\begin{math}e_i X_b^{\omega_i^\vee} e_j e_i=\left(1+X_b\right)^{\omega_i^\vee}\left(\frac{X_b}{1+X_b}\right)^{\omega_j^\vee}\left(e_j\right)\left(\frac{1}{X_b}\right)^{\omega_j^\vee} \left(e_i e_j\right)\left(1+X_b\right)^{\omega_j^\vee}\left(\frac{X_b}{1+X_b}\right)^{\omega_i^\vee}.\end{math}}
    \]
    The other two cases can be proved by a computer check and the technique of cluster folding. See \cite[Section 3.6, 3.7]{FGamalgamation}.
    \qedhere
\end{enumerate}
\end{proof}

\subsection{\texorpdfstring{Cluster $\mathrm{K}_2$ Structure on $\conf^b_d\left(\mathcal{A}_\sc\right)$ and $\conf^b_d\left(\mathcal{A}_\sc^\fr\right)$}{}}
\label{section.ajodad}

In this section we associate a coordinate chart of the decorated double Bott-Samelson cell $\conf^b_d\left(\mathcal{A}_\sc\right)$ to every seed $\vec{s}$ obtained from a triangulation and show that these coordinate charts are related by cluster $\mathrm{K}_2$ mutations corresponding to the seed mutations stated in Proposition \ref{mutation}, equipping $\conf^b_d\left(\mathcal{A}_\sc\right)$ with a natural cluster $\mathrm{K}_2$ structure. By restricting the cluster $\mathrm{K}_2$ structure on $\conf^b_d\left(\mathcal{A}_\sc\right)$ to the framed double Bott-Samelson cell $\conf^b_d\left(\mathcal{A}_\sc^\fr\right)$ we also obtain a cluster $\mathrm{K}_2$ structure on the latter.

Take a point in $\conf^b_d\left(\mathcal{A}_\sc\right)$. We can propagate the decoration over $\B^0$ horizontally across the top of the trapezoid using Lemma \ref{unique compatible}, equipping each flag on the top base with a decoration so that every adjacent pair is compatible. Similarly we can do the same on the bottom base: propagating the decoration over $\B_n$ horizontally across the bottom and equipping each flag on the bottom base with a decoration.
\[
\xymatrix@=10ex{ \A^0 \ar[r]^{s_{i_1}} \ar@{-}[d] & \A^1 \ar[r]^{s_{i_2}} & \A^2 \ar[r] & \dots \ar[r] & \A^{m-1} \ar[r]^{s_{i_m}} & \A^m \ar@{-}[d] \\
\A_0 \ar[r]_{s_{j_1}} & \A_1 \ar[r]_{s_{j_2}} & \A_2 \ar[r] & \dots \ar[r] & \A_{n-1} \ar[r]_{s_{j_n}} & \A_n}
\]

\begin{defn} Fix a triangulation and consider its corresponding string diagram and seed. Let $a$ be a string on the $i$th level in the string diagram (and hence a vertex in the corresponding seed). Let $\xymatrix{\A_l \ar@{-}[r] & \A^k}$ be a diagonal in the triangulation that $a$ intersects. Let us set
\[
A_a:=\Delta_{\omega_i}\left(\A_l,\A^k\right).
\]
For notation convenience we also adopt the shorthard $\Delta_i(l,k)=\Delta_{\omega_i}\left(\A_l,\A^k\right)$.
\end{defn}

Since a string may cross many different diagonals, the first thing we need to check is the well-defined-ness of $A_a$.

\begin{prop}\label{welldefined} The functions $A_a$ are well-defined.
\end{prop}
\begin{proof} It suffices to show that for a string $a$ on the $i$th level that intersects both sides of a triangle, the functions $A_a$ obtained by using each of the two sides of the triangle are equal. Without loss of generality we suppose the triangle looks like the following (after fixing a representative), with $t\in \T_\sc$.
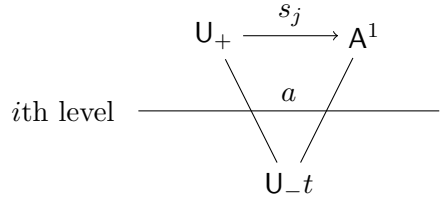
\[
\begin{tikzpicture}
\node (u) at (0,0) [] {$\U_-t$};
\node (d1) at (-1,2) [] {$\U_+$};
\node (d2) at (1,2) [] {$\A^1$};
\draw (d1) -- (u) -- (d2);
\draw [->] (d1) -- node[above] {$s_j$} (d2);
\draw (-2,1) -- node [above] {$a$} (2,1);
\node at (-3,1) [] {$i$th level};
\end{tikzpicture}
\]
By Lemma \ref{moduli of tits codist si}, we know that $\A^1=e_j(p)\overline{s}_j\U_+$. Then 
\begin{align*}
\Delta_{\omega_i}\left(\U_-t,\A^1\right)&=\left(te_j(p)\overline{s}_j\right)^{\omega_i}=\left(te_{-j}\left(p^{-1}\right)p^{\alpha_j^\vee}e_j\left(-p^{-1}\right)\right)^{\omega_i}=\left(tp^{\alpha_j^\vee}\right)^{\omega_i}=t^{\omega_i}\\
&=\Delta_{\omega_i}\left(\U_-t,\U_+\right),
\end{align*}
where $\left(tp^{\alpha_j^\vee}\right)^{\omega_i}=t^{\omega_i}$ uses the assumption that $i\neq j$.
\end{proof}

\begin{cor}\label{A to L} For a triangle of either of the following forms, 
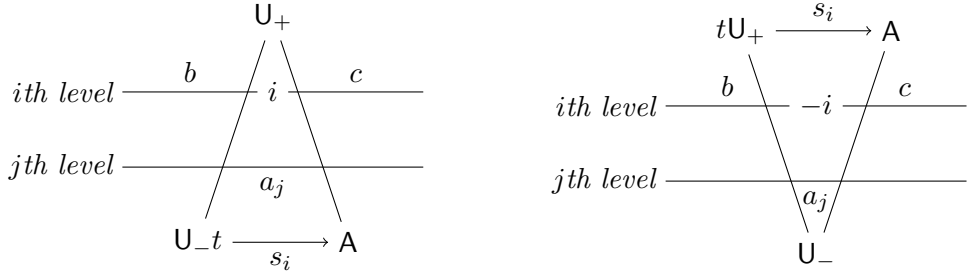
\[
\begin{tikzpicture}
\node (u) at (0,3) [] {$\U_+$};
\node (d0) at (-1,0) [] {$\U_-t$};
\node (d1) at (1,0) [] {$\A$};
\draw (d1) -- (u) -- (d0);
\draw [->] (d0) -- node [below] {$s_i$} (d1);
\node (0) at (0,2) [] {$i$};
\draw (-2,2) node [left] {$i$th level}  -- node [above] {$b$} (0) -- node [above] {$c$}  (2,2);
\draw (-2,1) node [left] {$j$th level} -- node [below] {$a_j$} (2,1);
\end{tikzpicture} \quad \quad \quad \quad 
\begin{tikzpicture}
\node (d) at (0,0) [] {$\U_-$};
\node (u0) at (-1,3) [] {$t\U_+$};
\node (u1) at (1,3) [] {$\A$};
\draw (u0) -- (d) -- (u1);
\draw[->] (u0) -- node [above] {$s_i$} (u1);
\node (0) at (0,2) [] {$-i$};
\draw (-2,2) node [left] {$i$th level}  -- node [above] {$b$} (0) -- node [above] {$c$}  (2,2);
\draw (-2,1) node [left] {$j$th level} -- node [below] {$a_j$} (2,1);
\end{tikzpicture}
\]
The underlying undecorated flag $\B$ associated to $\A$ is
\[
\B=e_{\pm i}\left(\frac{\prod_{j \neq i} A_{a_j}^{-\C_{ji}}}{A_{b}A_{c}}\right)\B_\pm,
\]
where the $\pm$ sign depends on the orientation of the triangle the same way as the sign of the corresponding node in the string diagram.
\end{cor}
\begin{proof} We will only show the computation for the case on the right; the case on the left is completely analogous. By comparison we see that we need to act by $t$ to move the triangle in the last proof into the configuration stated in the corollary. Therefore 
\[
\B= te_i(p)\overline{s}_i\B_+=e_i\left(t^{\alpha_i}p\right)\overline{s}_i\B_+=e_{-i}\left(t^{-\alpha_i}p^{-1}\right)\B_+.
\]
But then since $A_b=t^{\omega_i}$, $A_c=t^{\omega_i}p$, and $A_{a_j}=t^{\omega_j}$ for any $j\neq i$, we have
\[
\frac{\prod_{j\neq i}A_{a_j}^{-\C_{ji}}}{A_bA_c}=\frac{\prod_{j\neq i}t^{-\C_{ji}\omega_j}}{pt^{2\omega_i}}=t^{-\sum_j\C_{ji}\omega_j}p^{-1}=t^{-\alpha_i}p^{-1}.\qedhere
\]
\end{proof}

\begin{prop} An assignment of $\mathbb{G}_m$ values to all the functions $\left\{A_a\right\}$ associated to a triangulation recovers a point in $\conf^b_d\left(\mathcal{A}_\sc\right)$. As a corollary, these functions form a torus chart on $\conf^b_d\left(\mathcal{A}_\sc\right)$.
\end{prop}
\begin{proof} The idea is to mimic the proof of Lemma \ref{lusztig coord} and use the unipotent elements given by Corollary \ref{A to L} to construct the $b$-chain and the $d$-chain of decorated flags. Without loss of generality we may assume that the left most triangle is of the following shape.
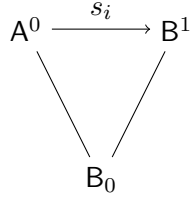
\[
\begin{tikzpicture}
\node (u) at (0,0) [] {$\B_0$};
\node (d0) at (-1,2) [] {$\A^0$};
\node (d1) at (1,2) [] {$\B^1$};
\draw [->] (d0) -- node[above] {$s_i$} (d1);
\draw (d0) -- (u) -- (d1);
\end{tikzpicture}
\]
Then we can set $\A_0=\U_-$ and $\A^0=t\U_+$ with $t:=\prod_{j=1}^{\tilde{r}}A_{\binom{j}{0}}^{\alpha_j^\vee}$. Then by Corollary \ref{A to L} we get $\B^1=e_{-i}\left(p\right)\B_+$ with $p=\frac{\prod_{j\neq i}A_{\binom{j}{0}}^{-\C_{ji}}}{A_{\binom{i}{0}}A_{\binom{i}{1}}}$. Then the compatibility condition requires that the decorated flag $\A^1$ is given by $e_{-i}(p)r\U_+$ with $r:=t\left(t^{\alpha_i}p\right)^{-\alpha_i^\vee}$ as a maximal torus element. 

The rest proceeds similar to the proof of Lemma \ref{lusztig coord}, with one slight caveat: when the triangles switch orientation (between upward pointing and downward pointing), we need to move the whole configuration by, not just a unipotent element, but a product of a unipotent element and a maximal torus element so that we may continue to use Corollary \ref{A to L}. For example, suppose the second triangle looks like the picture on the left below. Then we need to move the whole configuration by $r^{-1}e_{-i}(-p)$.
\[
\begin{tikzpicture}[baseline=5ex,scale=0.8]
\node (u0) at (-1,2) [] {$t\U_+$};
\node (u1) at (1,2) [] {$e_{-i}(p)r\U_+$};
\node (d0) at (0,0) [] {$\U_-$};
\node (d1) at (2,0) [] {$\B_1$};
\draw [->] (u0) -- node [above] {$s_i$} (u1);
\draw [->] (d0) -- node [below] {$s_j$} (d1);
\draw (u0) -- (d0) -- (u1) -- (d1);
\end{tikzpicture}
 \rightsquigarrow\begin{array}{c}
     \text{move the whole} \\
     \text{configuration} \\
     \text{by $r^{-1}e_{-i}(-p)$}
\end{array} \rightsquigarrow \begin{tikzpicture}[baseline=5ex,scale=0.7]
\node (u0) at (-1,2) [] {$r^{-1}e_{-i}(-p)t\U_+$};
\node (u1) at (2,2) [] {$\U_+$};
\node (d0) at (0,0) [] {$\U_-r$};
\node (d1) at (3,0) [] {$r^{-1}e_{-i}(-p)\B_1$};
\draw [->] (u0) -- node [above] {$s_i$} (u1);
\draw [->] (d0) -- node [below] {$s_j$} (d1);
\draw (u0) -- (d0) -- (u1) -- (d1);
\end{tikzpicture}
\]

In the end we will get a $b$-chain and a $d$-chain of compatible decorated flags, with their initial flags and terminal flags in general position. Then to get a point in $\conf^b_d\left(\mathcal{A}_\sc\right)$, we simply forget the decorations everywhere except $\A^0$ and $\A_n$.
\end{proof}

\begin{defn} \label{A def} We call the functions $A_a$ the \emph{cluster $\mathrm{K}_2$ coordinates} associated to the seed (string diagram/triangulation) on $\conf^b_d\left(\mathcal{A}_\sc\right)$.
\end{defn}

Since the framed double Bott-Samelson cell $\conf^b_d\left(\mathcal{A}_\sc^\fr\right)$ has compatible decorated flags around the perimeter, it can be seen as the subset of $\conf^b_d\left(\mathcal{A}_\sc\right)$ with all cluster $\mathrm{K}_2$ coordinates associated to open strings set to 1. It then follows that the cluster $\mathrm{K}_2$ coordinates associated to closed strings restrict to a set of coordinates on $\conf^b_d\left(\mathcal{A}_\sc^\fr\right)$, which we call the \emph{cluster $\mathrm{K}_2$ coordinates} on $\conf^b_d\left(\mathcal{A}_\sc^\fr\right)$ associated to the seed (string diagram/triangulation).

In order to justify the name ``cluster $\mathrm{K}_2$ coordinates'', we need to show that they actually transform as cluster $\mathrm{K}_2$ coordinates. 

\begin{prop} \label{a coordinate trans}\begin{enumerate}
    \item If $i\neq j$, then 
    \[
    \begin{tikzpicture}[baseline=2ex,scale=0.8]
    \node (a) at (1,1) [] {$i$};
    \node (b) at (0,0) [] {$-j$};
    \draw (-1,0) --node[below] {$A_c$} (b) -- node [below] {$A_d$} (2,0);
    \draw (-1,1) -- node [above] {$A_a$} (a) -- node [above] {$A_b$} (2,1);
    \end{tikzpicture} \quad \longleftrightarrow \quad 
    \begin{tikzpicture}[baseline=2ex,scale=0.8]
    \node (a) at (0,1) [] {$i$};
    \node (b) at (1,0) [] {$-j$};
   \draw (-1,0) --node[below] {$A_c$} (b) -- node [below] {$A_d$} (2,0);
    \draw (-1,1) -- node [above] {$A_a$} (a) -- node [above] {$A_b$} (2,1);
    \end{tikzpicture}
    \]
    On the other hand, if we interchange two neighboring nodes of opposite signs on the same level, we get
    \[
    \begin{tikzpicture}[baseline=2ex,scale=0.7]
    \node (a0) at (0,1) [] {$-i$};
    \node (a1) at (2,1) [] {$i$};
    \draw (-1.5,1) node [left] {$i$th level} -- node [above] {$A_a$} (a0) -- node [above] {$A_b$} (a1) -- node [above] {$A_c$} (3.5,1);
    \draw (-1.5,0) node [left] {$j$th level} -- node [below] {$A_j$} (3.5,0);
    \end{tikzpicture} \quad \longleftrightarrow \quad 
    \begin{tikzpicture}[baseline=2ex,scale=0.7]
    \node (a0) at (0,1) [] {$i$};
    \node (a1) at (2,1) [] {$-i$};
    \draw (-1.5,1) node [left] {$i$th level} -- node [above] {$A_a$} (a0) -- node [above] {$A'_b$} (a1) -- node [above] {$A_c$} (3.5,1);
    \draw (-1.5,0) node [left] {$j$th level} -- node [below] {$A_j$} (3.5,0);
    \end{tikzpicture}
    \]
    where
    \[
    A'_b=\frac{1}{A_b}\left(A_aA_c+\prod_{j \neq i}A_j^{-\C_{ji}}\right).
    \]
    \item If we perform a braid move to the string diagram, depending on whether it is of Dynkin type $\mathrm{A}_1\times \mathrm{A}_1$, $\mathrm{A}_2$, $\mathrm{B}_2$, or $\mathrm{G}_2$, the cluster Poisson coordinates transform according to the sequences of mutations described in Proposition \ref{mutation} (we will only depict the case where a braid move takes place among nodes labeled by simple roots; the cases where they are labeled by opposite simple roots are completely analogous).
    \begin{itemize}
        \item $\C_{ij}\C_{ji}=0$: 
        \[
    \begin{tikzpicture}[baseline=2ex,scale=0.8]
    \node (a) at (1,1) [] {$i$};
    \node (b) at (0,0) [] {$j$};
    \draw (-1,0) --node[below] {$A_c$} (b) -- node [below] {$A_d$} (2,0);
    \draw (-1,1) -- node [above] {$A_a$} (a) -- node [above] {$A_b$} (2,1);
    \end{tikzpicture} \quad \longleftrightarrow \quad 
    \begin{tikzpicture}[baseline=2ex,scale=0.8]
    \node (a) at (0,1) [] {$i$};
    \node (b) at (1,0) [] {$j$};
   \draw (-1,0) --node[below] {$A_c$} (b) -- node [below] {$A_d$} (2,0);
    \draw (-1,1) -- node [above] {$A_a$} (a) -- node [above] {$A_b$} (2,1);
    \end{tikzpicture}
    \]
    \item $\C_{ij}\C_{ji}=-1$:
    \[
    \begin{tikzpicture}[baseline=2ex,scale=0.8]
    \node (0) at (-1,1) [] {$i$};
    \node (1) at (1,1) [] {$i$};
    \node (2) at (0,0) [] {$j$};
    \draw (-2,1) -- node [above] {$A_a$} (0) -- node [above] {$A_b$} (1) -- node [above] {$A_c$} (2,1);
    \draw (-2,0) -- node [below] {$A_d$} (2) -- node [below] {$A_e$} (2,0);
    \end{tikzpicture} \quad \longleftrightarrow \quad \begin{tikzpicture}[baseline=2ex,scale=0.8]
    \node (0) at (-1,0) [] {$j$};
    \node (1) at (1,0) [] {$j$};
    \node (2) at (0,1) [] {$i$};
    \draw (-2,0) -- node [below] {$A_d$} (0) -- node [below] {$A'_b$} (1) -- node [below] {$A_e$} (2,0);
    \draw (-2,1) -- node [above] {$A_a$} (2) -- node [above] {$A_c$} (2,1);
    \end{tikzpicture}
    \]
    where
    \[
    A'_b=\frac{1}{A_b}\left(A_aA_e+A_cA_d\right).
    \]
    \item $\C_{ij}=-2$ and $\C_{ji}=-1$:
\[
\begin{tikzpicture}[baseline=2ex,scale=0.8]
\node (0) at (0,1) [] {$i$};
\node (1) at (2,1) [] {$i$};
\node (2) at (1,0) [] {$j$};
\node (3) at (3,0) [] {$j$};
\draw (-1,0) -- node [below] {$A_e$} (2) -- node [below] {$A_b$} (3) -- node [below] {$A_f$} (4,0);
\draw (-1,1) -- node [above] {$A_c$} (0) -- node [above] {$A_a$} (1) -- node [above] {$A_d$} (4,1);
\end{tikzpicture}\quad \longleftrightarrow \quad 
\begin{tikzpicture}[baseline=2ex,scale=0.8]
\node (0) at (1,1) [] {$i$};
\node (1) at (3,1) [] {$i$};
\node (2) at (0,0) [] {$j$};
\node (3) at (2,0) [] {$j$};
\draw (-1,0) --  node [below] {$A_e$} (2) -- node [below] {$A'_b$} (3) -- node [below] {$A_f$} (4,0);
\draw (-1,1) -- node [above] {$A_c$} (0) -- node [above] {$A'_a$} (1) -- node [above] {$A_d$} (4,1);
\end{tikzpicture}
\]
where
\[
    A'_a= \frac{A_aA_f}{A_b}\tilde{F}_a,  \quad \text{and} \quad  A'_b=\frac{A_cA_f}{A_b}\tilde{F}_b 
\]
and $\tilde{F}_a$ and $\tilde{F}_b$ are the polynomials in Equation \eqref{B2 F} with the substitution $X_k=\prod_lA_l^{\epsilon_{kl}}$ where $\epsilon$ is the exchange matrix associated to the seed on the left. 
    \item $\C_{ij}=-3$ and $\C_{ji}=-1$:
     \[
\begin{tikzpicture}[baseline=2ex,scale=0.6]
\node (0) at (0,1) [] {$i$};
\node (1) at (2,1) [] {$i$};
\node (4) at (4,1) [] {$i$};
\node (2) at (1,0) [] {$j$};
\node (3) at (3,0) [] {$j$};
\node (5) at (5,0) [] {$j$};
\draw (-1,0) -- node [below] {$A_g$} (2) -- node [below] {$A_c$} (3) -- node [below] {$A_d$} (5) -- node [below] {$A_h$} (6,0);
\draw (-1,1) -- node [above] {$A_e$} (0) -- node [above] {$A_a$} (1) -- node [above] {$A_b$} (4) -- node [above] {$A_f$} (6,1);
\end{tikzpicture}\quad \longleftrightarrow \quad 
\begin{tikzpicture}[baseline=2ex,scale=0.6]
\node (0) at (1,1) [] {$i$};
\node (1) at (3,1) [] {$i$};
\node (4) at (5,1) [] {$i$};
\node (2) at (0,0) [] {$j$};
\node (3) at (2,0) [] {$j$};
\node (5) at (4,0) [] {$j$};
\draw (-1,0) -- node [below] {$A_g$} (2) -- node [below] {$A'_c$} (3) -- node [below] {$A'_d$} (5) -- node [below] {$A_h$} (6,0);
\draw (-1,1) -- node [above] {$A_e$} (0) -- node [above] {$A'_a$} (1) -- node [above] {$A'_b$} (4) -- node [above] {$A_f$} (6,1);
\end{tikzpicture}
\]
where
\[
    A'_a=\frac{A_aA_h}{A_d}\tilde{F}_a \quad \quad  A'_b=\frac{A_bA_h}{A_d}\tilde{F}_b \quad \quad  A'_c=\frac{A_gA_h}{A_d}\tilde{F}_c\quad \quad  A'_d=\frac{A_cA_h^2}{A_d^{-2}}\tilde{F}_d
\]
and $\tilde{F}_a$, $\tilde{F}_b$, $\tilde{F}_c$, $\tilde{F}_d$ are the polynomials in Equation \eqref{G2 F} with the substitution $X_k=\prod_lA_l^{\epsilon_{kl}}$ where $\epsilon$ is the exchange matrix associated to the seed on the left.
    \end{itemize}
The coordinate transformations in the last two cases are written in the form of the factorization formula \eqref{factorization formula a}. 
\end{enumerate}
\end{prop}
\begin{proof}\begin{enumerate}
    \item The case where $i\neq j$ is clear from the definition of cluster $\mathrm{K}_2$ coordinates. For the other case, we can use Corollary \ref{A to L} to recover the local configuration of decorated flags.
    \[
    \begin{tikzpicture}[scale=0.7]
    \node (u0) at (0,3) [] {$\displaystyle e_{-i}\left(-\frac{\prod_{j\neq i}A_j^{-\C_{ji}}}{A_aA_b}\right)A_a^{\alpha_i^\vee}\prod_{j \neq i} A_j^{\alpha_j^\vee}\U_+$};
    \node (u1) at (7,3) [] {$\displaystyle  A_b^{\alpha_i^\vee}\prod_{j\neq i}A_j^{\alpha_j^\vee} \U_+$};
    \node (d0) at (0,0) [] {$\U_-$};
    \node (d1) at (7,0) [] {$\displaystyle \U_-A_c^{\alpha_i^\vee}\prod_{j\neq i} A_j^{\alpha_j^\vee}e_i\left(-\frac{\prod_{j \neq i}A_j^{-\C_{ji}}}{A_bA_c}\right)A_b^{-\alpha_i^\vee}\prod_{j\neq i}A_j^{-\alpha_j^\vee}$};
    \draw [->] (u0) -- node [above] {$s_i$} (u1);
    \draw [->] (d0) -- node [below] {$s_i$} (d1);
    \draw (u0) -- (d0) -- (u1) -- (d1);
    \end{tikzpicture}  
    \]
    By using the Lie group identity $e_i(q)t=te_i\left(t^{-i}q\right)$ we can simplify the bottom right decorated flag as $\U_-\left(\frac{A_c}{A_b}\right)^{\alpha_i^\vee} e_i\left(-\frac{A_b}{A_c}\right)$. Now we want to flip the diagonal and compute $\Delta_{\omega_i}$ along this new diagonal. Note that
    \begin{align*}
        &\left(\frac{A_c}{A_b}\right)^{\alpha_i^\vee} e_i\left(-\frac{A_b}{A_c}\right)e_{-i}\left(-\frac{\prod_{j\neq i}A_j^{-\C_{ji}}}{A_aA_b}\right)A_a^{\alpha_i^\vee}\prod_{j \neq i} A_j^{\alpha_j^\vee}\\
        =&\left(\frac{A_c}{A_b}\right)^{\alpha_i^\vee}e_{-i}(\cdots)\left(1+\frac{\prod_{j\neq i}A_j^{-\C_{ji}}}{A_aA_c}\right)^{\alpha_i^\vee}e_i(\cdots)A_a^{\alpha_i^\vee}\prod_{j\neq i}A_j^{\alpha_j^\vee}\\
        =&e_{-i}(\cdots)\left(\frac{1}{A_b}\left(A_aA_c+\prod_{j\neq i}A_j^{-\C_{ji}}\right)\right)^{\alpha_i^\vee}\prod_{j\neq i}A_j^{\alpha_j^\vee}e_i(\cdots).
    \end{align*}
    Therefore it follows that $A'_b=\frac{1}{A_b}\left(A_aA_c+\prod_{j\neq i}A_j^{-\C_{ji}}\right)$
    
    \item The case $\C_{ij}=\C_{ji}=0$ is clear from the definition of cluster $\mathrm{K}_2$ coordinates. For the case $\C_{ij}=\C_{ji}=-1$, we again use the computation in the proof of Proposition \ref{welldefined} to recover the local configuration of decorated flags from the string diagram on the left as follows (to save space we write the diagram sideway).
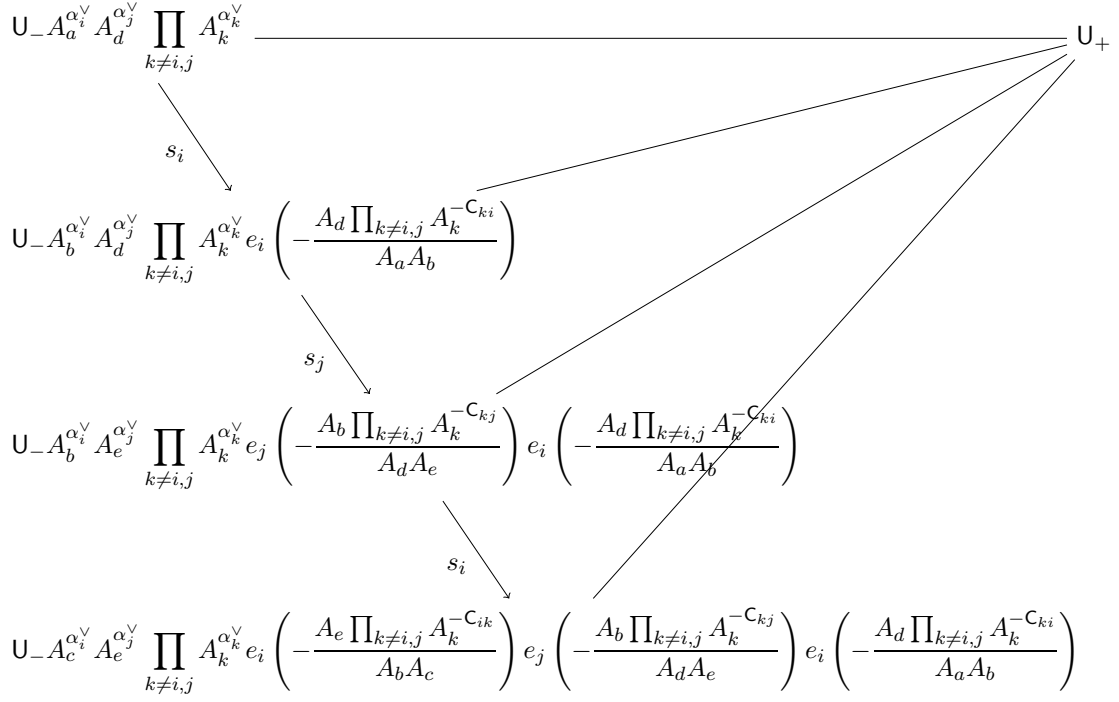
    \[
    \resizebox{.9\hsize}{!}{$\begin{tikzpicture}
    \node (u) at (16,5) [] {$\U_+$};
    \node (d0) at (0,5) [right] {$\displaystyle \U_-A_a^{\alpha_i^\vee}A_d^{\alpha_j^\vee}\prod_{k\neq i,j}A_k^{\alpha_k^\vee}$};
    \node (d1) at (0,2) [right] {$\displaystyle \U_-A_b^{\alpha_i^\vee}A_d^{\alpha_j^\vee}\prod_{k\neq i,j}A_k^{\alpha_k^\vee} e_i\left(-\frac{A_d\prod_{k\neq i,j}A_k^{-\C_{ki}}}{A_aA_b}\right)$};
    \node (d2) at (0,-1) [right] {$\displaystyle \U_-A_b^{\alpha_i^\vee}A_e^{\alpha_j^\vee}\prod_{k\neq i,j}A_k^{\alpha_k^\vee}e_j\left(-\frac{A_b\prod_{k\neq i,j}A_k^{-\C_{kj}}}{A_dA_e}\right)e_i\left(-\frac{A_d\prod_{k\neq i,j}A_k^{-\C_{ki}}}{A_aA_b}\right)$};
    \node (d3) at (0,-4) [right] {$\displaystyle \U_-A_c^{\alpha_i^\vee}A_e^{\alpha_j^\vee}\prod_{k\neq i,j}A_k^{\alpha_k^\vee}e_i\left(-\frac{A_e\prod_{k\neq i,j}A_k^{-\C_{ik}}}{A_bA_c}\right)e_j\left(-\frac{A_b\prod_{k\neq i,j}A_k^{-\C_{kj}}}{A_dA_e}\right)e_i\left(-\frac{A_d\prod_{k\neq i,j}A_k^{-\C_{ki}}}{A_aA_b}\right)$};
    \draw [->] (d0) -- node[below left] {$s_i$} (d1);
    \draw [->] (d1) -- node [below left] {$s_j$} (d2);
    \draw [->] (d2) -- node [below left] {$s_i$} (d3);
    \draw (d0) -- (u) -- (d1);
    \draw (d2) -- (u) -- (d3);
    \end{tikzpicture}$}
    \]
    Similarly from the string diagram on the right we have the following configuration of decorated flags.
    \[
    \resizebox{.9\hsize}{!}{$\begin{tikzpicture}
    \node (u) at (16,5) [] {$\U_+$};
    \node (d0) at (0,5) [right] {$\displaystyle \U_-A_a^{\alpha_i^\vee}A_d^{\alpha_j^\vee}\prod_{k\neq i,j}A_k^{\alpha_k^\vee}$};
    \node (d1) at (0,2) [right] {$\displaystyle \U_-A_a^{\alpha_i^\vee}{A'}_b^{\alpha_j^\vee}\prod_{k\neq i,j}A_k^{\alpha_k^\vee} e_j\left(-\frac{A_a\prod_{k\neq i,j}A_k^{-\C_{kj}}}{A'_bA_d}\right)$};
    \node (d2) at (0,-1) [right] {$\displaystyle \U_-A_c^{\alpha_i^\vee}{A'}_b^{\alpha_j^\vee}\prod_{k\neq i,j}A_k^{\alpha_k^\vee}e_i\left(-\frac{A'_b\prod_{k\neq i,j}A_k^{-\C_{ki}}}{A_aA_c}\right)e_j\left(-\frac{A_a\prod_{k\neq i,j}A_k^{-\C_{kj}}}{A'_bA_d}\right)$};
    \node (d3) at (0,-4) [right] {$\displaystyle \U_-A_c^{\alpha_i^\vee}A_e^{\alpha_j^\vee}\prod_{k\neq i,j}A_k^{\alpha_k^\vee}e_j\left(-\frac{A_c\prod_{k\neq i,j}A_k^{-\C_{kj}}}{A'_bA_e}\right)e_i\left(-\frac{A'_b\prod_{k\neq i,j}A_k^{-\C_{ki}}}{A_aA_c}\right)e_j\left(-\frac{A_a\prod_{k\neq i,j}A_k^{-\C_{kj}}}{A'_bA_d}\right)$};
    \draw [->] (d0) -- node[below left] {$s_j$} (d1);
    \draw [->] (d1) -- node [below left] {$s_i$} (d2);
    \draw [->] (d2) -- node [below left] {$s_j$} (d3);
    \draw (d0) -- (u) -- (d1);
    \draw (d2) -- (u) -- (d3);
    \end{tikzpicture}$}
    \]
    By the uniqueness part of Lemma \ref{unique decorated}, it suffices to show that the last decorated flags from the two diagrams are equal, which boils down to showing that  
    \begin{align*}
         & e_i\left(-\frac{A_e\prod_{k\neq i,j}A_k^{-\C_{ki}}}{A_bA_c}\right)e_j\left(-\frac{A_b\prod_{k\neq i,j}A_k^{-\C_{kj}}}{A_dA_e}\right)e_i\left(-\frac{A_d\prod_{k\neq i,j}A_k^{-\C_{ki}}}{A_aA_b}\right)\\
        =&e_j\left(-\frac{A_c\prod_{k\neq i,j}A_k^{-\C_{kj}}}{A'_bA_e}\right)e_i\left(-\frac{A'_b\prod_{k\neq i,j}A_k^{-\C_{ki}}}{A_aA_c}\right)e_j\left(-\frac{A_a\prod_{k\neq i,j}A_k^{-\C_{kj}}}{A'_bA_d}\right).
    \end{align*}
    But this is precisely the Lie group identity
    \[
    e_i\left(q_1\right)e_j\left(q_2\right)e_i\left(q_3\right)=e_j \left(\frac{q_2q_3}{q_1+q_3}\right)e_i\left(q_1+q_3\right)e_j \left(\frac{q_1q_2}{q_1+q_3}\right).
    \]
    The proof for the other two cases can be found in \cite[Section 7.6]{GS3} 
   \qedhere
\end{enumerate}
\end{proof}

From the proof of the case 1 above we also deduce the following identity, which is useful in the next section.

\begin{cor}\label{identity1} Suppose $\G=\G_\sc$ and $\xymatrix{\A^k \ar[r]^{s_i} & \A^{k+1}}$ and $\xymatrix{\A_l \ar[r]^{s_i} & \A_{l+1}}$ are compatible pairs that are part of a decorated double Bott-Samelson cell. Then
\[
\Delta_j\left(l+1,k\right)\Delta_{i}\left(l,k+1\right)-\Delta_{i}\left(l,k\right)\Delta_{i}\left(l+1,k+1\right)=\prod_{j\neq i}\Delta_{j}\left(*,*\right)^{-\C_{ji}},
\]
where $(*,*)$ means either of $(l,k)$, $(l,k+1)$, $(l+1,k)$, and $(l+1,k+1)$.
\end{cor}
\begin{proof} First we observe that both sides of the equation are regular functions on the decorated double Bott-Samelson cell and therefore it suffices to show the equality with four extra open conditions:
\[
\begin{tikzpicture}
\node (u0) at (0,2) [] {$\A^k$};
\node (u1) at (2,2) [] {$\A^{k+1}$};
\node (d0) at (0,0) [] {$\A_l$};
\node (d1) at (2,0) [] {$\A_{l+1}$};
\draw [->] (u0) -- node [above] {$s_i$} (u1);
\draw [->] (d0) -- node [below] {$s_i$} (d1);
\draw (d0) -- (u0) -- (d1) -- (u1) -- (d0);
\end{tikzpicture}
\]
But then this reduces to the identity
\[
A'_bA_b-A_aA_c=\prod_{j \neq i}A_j^{-\C_{ji}}
\]
in case 1 of the last proposition.
\end{proof}

Lastly, let us investigate how the cluster $\mathrm{K}_2$ coordinates transform under the transposition morphism.

\begin{prop} Let $\left(A_a\right)$ and $\left(A'_a\right)$ be the cluster $\mathrm{K}_2$ coordinate charts associated to two corresponding string diagrams under transposition. Then under the transposition morphism $t:\conf^b_d\left(\mathcal{A}_\sc\right)\rightarrow \conf^{d^\circ}_{b^\circ}\left(\mathcal{A}_\ad\right)$, $t^*A'_a=A_a$.
\end{prop}
\begin{proof} It follows from Lemma \ref{2.23}.
\end{proof}

\subsection{Coordinate Rings as Upper Cluster Algebras}

Decorated Bott-Samelson cells are affine. In this section we show that their coordinate rings are 
\[
\mathcal{O}\left(\conf^b_d\left(\mathcal{A}_\sc\right)\right)\cong \up\left(\mathscr{A}^b_d\right) \quad \text{and} \quad \mathcal{O}\left(\conf^b_d\left(\mathcal{A}_\ad\right)\right)\cong \up\left(\mathscr{X}^b_d\right),
\]
where $\up\left(\mathscr{A}^b_d\right)$ is the upper cluster algebra arising from the family of mutation equivalent seeds associated to $(b,d)$, and $\up\left(\mathscr{X}^b_d\right)$ is the corresponding cluster Poisson algebra.

Let us start with the first claim. We simplify the notation and write $\mathcal{O}\left(\conf^b_d\left(\mathcal{A}_\sc\right)\right)$ as $\mathcal{O}^b_d$. 

\begin{lem} The ring $\mathcal{O}^b_d$ is a unique factorization domain (UFD).
\end{lem}
\begin{proof}
The proof of Theorem \ref{affineness of conf} shows that $\mathcal{O}^b_d$ is the non-vanishing locus of some function on $\T\times\mathbb{A}^n$, which is also the non-vanishing locus of some function on $\mathbb{A}^N$ for  $N=n+\dim \T$.  Note that $\mathcal{O}(\mathbb{A}^N)$ is a polynomial ring. The ring $\mathcal{O}^b_d$ is a localization of $\mathcal{O}(\mathbb{A}^N)$ and therefore is a UFD.
\end{proof}

\begin{lem}\label{codim 1 vanishing locus} The vanishing locus $\left\{A_c=0\right\}$ in $\conf^b_d\left(\mathcal{A}_\sc\right)$ associated with any closed string $c$ is of codimension 1. 
\end{lem}
\begin{proof} Fix a triangulation containing $c$ as a closed string. Suppose $c$ is on the $i$th level. We assign a Tits codistance $s_i$ to a diagonal in the triangulation if this diagonal intersects $c$, and assign a Tits codistance $e$ otherwise. Since $c$ is a closed string, the left most and the right most diagonals (the two slant sides of the trapezoid) must be assigned with Tits codistance $e$. Therefore this configuration describes a subset $W_c$ of $\conf^b_d\left(\mathcal{A}_\sc\right)$. We claim that $W_c$ is of codimension 1 within $\conf^b_d\left(\mathcal{A}_\sc\right)$ and $W_c\subset \left\{A_c=0\right\}$. Note that the lemma follows from this claim.

Let us compute the dimension of $W_c$. Without loss of generality, we can fix the left most pair of flags $\xymatrix{\A^0 \ar@{-}[r] & \B_0}$ for any point in $W_c$ to be $\xymatrix{\U_+\ar@{-}[r] & \B}$, which exhausts the diagonal $\G$-action on the configuration of flags. By Proposition \ref{3.3}, we see that for most triangles in the triangulation, under the Tits codistance assignment, each of them contributes a $\mathbb{G}_m$ factor of $W_c$; the only exceptions are the two triangles containing the end points of $c$, and they contribute a $\{\ast\}$-factor and a $\mathbb{A}^1$-factor respectively. In the end, there is also another decoration over the last flag of the bottom chain, which gives rise to a $\dim \T$-factor of $W_c$. In conclusion, this shows that
\[
W_c\cong \mathbb{G}_m^{l(b)+l(d)-2}\times \mathbb{A}^1\times \T,
\]
where $l(b)$ and $l(d)$ are the lengths of the positive braids $b$ and $d$. On the other hand, we recall from Theorem \ref{affineness of conf} that 
\[
\dim \conf^b_d\left(\mathcal{A}_\sc\right)=\dim \T+l(b)+l(d).
\]
Therefore $W_c$ is indeed of codimension 1. 

To see that $W_c\subset \left\{A_c=0\right\}$, consider the triangle containing the left end point of the closed string $c$. Suppose the triangle is of the following form.
\[
\begin{tikzpicture}
\node (d1) at (0,0) [] {$\A_j$};
\node (d2) at (2,0) [] {$\A_{j+1}$};
\node (u) at (1,2) [] {$\A^k$};
\node (c) at (1,0.7) [] {$i$};
\draw (d1) -- (u);
\draw (d2) -- node [above right] {$s_i$}(u);
\draw [->] (d1) -- node [below] {$s_i$} (d2);
\draw  (-1,0.7) node [left] {$i$th level} -- (c) --  node [below right] {$c$} (3,0.7);
\end{tikzpicture}
\]
By using the $\G$-action, we may assume without loss of generality that $\A^k=t\U_+$ and $\A_j=\U_-$. Then $\A_{j+1}$ must be $\U_-\doverline{s}_i$. Therefore
\[
A_c=\Delta_{\omega_i}\left(\doverline{s}_i\right)=0.
\]
This shows that $W_c\subset \left\{A_c=0\right\}$. The case with an upside-down triangle can be proved in a similar way.
\end{proof}

Fix a triangulation for  $\conf^b_d\left(\mathcal{A}_\sc\right)$. Take the very first triangle on the left and consider the corresponding node in the string diagram; let $a$ be the closed string on the right of this node and suppose $a$ is on the $i$th level. Its associated cluster variable $A_a$ is a regular function on $\conf^b_d\left(\mathcal{A}_\sc\right)$.
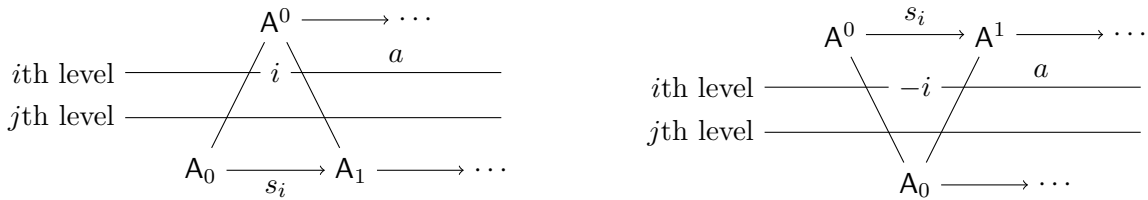
\[
\begin{tikzpicture}[scale=0.7]
\node (d0) at (0,0) [] {$\A_0$};
\node (d1) at (2,0) [] {$\A_1$};
\node (u)  at (1,2) [] {$\A^0$};
\draw [->] (d0) -- node [below] {$s_i$} (d1);
\draw [->] (d1) -- (3.5,0) node [right] {$\cdots$};
\draw [->] (u) -- (2.5,2) node [right] {$\cdots$};
\draw (d0) -- (u) -- (d1);
\node (0) at (1,1.3) [] {$i$};
\draw (-1,1.3) node [left] {$i$th level} -- (0) -- node [above] {$a$} (4,1.3);
\draw (-1,0.7) node [left] {$j$th level} -- (4,0.7);
\end{tikzpicture} \quad \quad \quad \quad 
\begin{tikzpicture}[scale=0.7]
\node (u0) at (0,2) [] {$\A^0$};
\node (u1) at (2,2) [] {$\A^1$};
\node (d)  at (1,0) [] {$\A_0$};
\draw [->] (u0) -- node [above] {$s_i$} (u1);
\draw [->] (u1) -- (3.5,2) node [right] {$\cdots$};
\draw [->] (d) -- (2.5,0) node [right] {$\cdots$};
\draw (u0) -- (d) -- (u1);
\node (0) at (1,1.3) [] {$-i$};
\draw (-1,1.3) node [left] {$i$th level} -- (0) -- node [above] {$a$} (4,1.3);
\draw (-1,0.7) node [left] {$j$th level} -- (4,0.7);
\end{tikzpicture}
\]

\begin{lem}\label{irreducible} The function $A_a$  is an irreducible element of $\mathcal{O}^b_d$.
\end{lem}
\begin{proof} Let us prove the case of the picture on the right; the case of the picture on the left is completely analogous. Without loss of generality we set $\A^0=t\U_+$ and $\A_0=\U_-$. From Corollary \ref{A to L} we know that if we impose the general position condition $\xymatrix{\B^1 \ar@{-}[r] & \B_-}$, then
\[
\B^1=e_{-i}\left(\frac{\prod_{j \neq i} t^{-\C_{ji}\omega_j}}{t^{\omega_i}A_a}\right)\B_+=e_i\left(\frac{t^{\omega_i}A_a}{\prod_{j\neq i}t^{-\C_{ji}\omega_j}}\right)\overline{s}_i\B_+.
\]
On the other hand, we know from Lemma \ref{moduli of tits codist si} that the argument $\frac{t^{\omega_i}A_a}{\prod_{j\neq i}t^{-\C_{ji}\omega_j}}$ parametrizes the moduli space of all the flags that are of Tits distance $s_i$ away from $\B_+$, which is isomorphic to $\mathbb{A}^1$. Since $t^{\omega_i}$ and $t^{\omega_j}$ (the frozen cluster $\mathrm{K}_2$ variables) are never zero on $\conf^b_d\left(\mathcal{A}_\sc\right)$ by definition, it follows that $A_a$ also parametrizes the $\mathbb{A}^1$ moduli space of flags. Recall from the proof of Theorem \ref{affineness of conf} that such parameter is precisely one of the generators of the polynomial ring $\mathcal{O}(\mathbb{A}^N)$. Hence, $A_a$ is a unit or an irreducible element in $\mathcal{O}^b_d$. However, $A_a$ cannot be a unit, because by Lemma \ref{codim 1 vanishing locus}, its vanishing locus $\left\{A_a=0\right\}$ is non-empty and of codimension 1. Therefore $A_a$ must be an irreducible element.
\end{proof}

Next we will use an induction to prove that all non-frozen cluster $\mathrm{K}_2$ coordinates associated to the fixed triangulation are irreducible elements of $\mathcal{O}^b_d$. Since triangles in a triangulation possess a well-defined ordering from left to right, by associating each closed string $c$ to the triangle corresponding to its left node, we get an ordering $<$ on the closed strings. We will perform an induction according to this order.

Let $c$ be a closed string on the $j$th level. Let $b'$ be the positive braid that is the remaining part of the word for $b$ after deleting all the letters occurring before the triangle corresponding to the left node of the closed string $c$, and let $d'$ be the positive braid that is the remaining part of the word for $d$ after deleting all the letters occurring before the triangle corresponding to the left node of the closed string $c$.
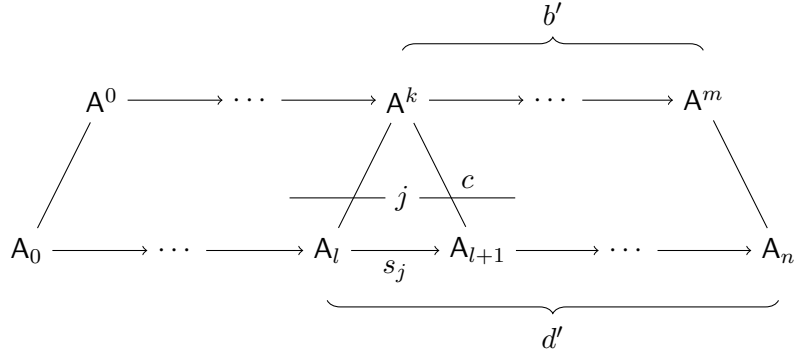
\[
\begin{tikzpicture}
\node (u0) at (1,2) [] {$\A^0$};
\node (u1) at (3,2) [] {$\cdots$};
\node (u2) at (5,2) [] {$\A^k$};
\node (u3) at (7,2) [] {$\cdots$};
\node (u4) at (9,2) [] {$\A^m$};
\node (d0) at (0,0) [] {$\A_0$};
\node (d1) at (2,0) [] {$\cdots$};
\node (d2) at (4,0) [] {$\A_l$};
\node (d3) at (6,0) [] {$\A_{l+1}$};
\node (d4) at (8,0) [] {$\cdots$};
\node (d5) at (10,0) [] {$\A_n$};
\foreach \i in {0,...,3}
    {
    \pgfmathtruncatemacro{\j}{\i+1};
    \draw [->] (u\i) -- (u\j);
    }
\foreach \i in {0,1,3,4}
    {
    \pgfmathtruncatemacro{\j}{\i+1};
    \draw [->] (d\i) -- (d\j);
    }
\draw [->] (d2) --  node [below] {$s_j$} (d3);
\draw (u0) -- (d0);
\draw (d2) -- (u2) -- (d3);
\draw (u4) -- (d5);
\node (0) at (5,0.7) [] {$j$};
\draw (3.5,0.7) -- (0) -- node [above] {$c$} (6.5,0.7);
\draw [decorate,decoration={brace,amplitude=5pt,mirror,raise=4ex}]
  (4,0) -- (10,0) node[midway,yshift=-3em]{$d'$};
\draw [decorate,decoration={brace,amplitude=5pt, raise=4ex}]
  (5,2) -- (9,2) node[midway,yshift=3em]{$b'$};
\end{tikzpicture}
\]

\begin{lem}\label{distinguished} Let $S$ be the set of strings that crosses the diagonal $\xymatrix{\A_l \ar@{-}[r] & \A^k}$. Then as algebras,
\[
\mathcal{O}^b_d\left[\frac{1}{\prod_{e<c}A_e}\right]\cong \mathbb{C}\left[A_e^{\pm} \right]_{e<c}\underset{\mathbb{C}\left[A_f^\pm\right]_{f\in S}}{\otimes}\mathcal{O}^{b'}_{d'}.
\]
\end{lem}
\begin{proof} The statement of the proposition is equivalent to the geometric statement that the distinguished open subset (non-vanishing locus) $U_{\prod_{e<c}A_e}$ of $\conf^b_d\left(\mathcal{A}_\sc\right)$ is biregularly isomorphic to a fiber product $T_{<c}\underset{T_S}{\times}\conf^{b'}_{d'}\left(\mathcal{A}_\sc\right)$ where $T_{<c}$ is a torus with coordinates $\left\{A_e\right\}_{e<c}$ and $T_S$ is a torus with coordinates $\left\{A_f\right\}_{f\in S}$. Note that given a point in $U_{\prod_{e<c}A_e}$ we automatically get a point in $T_{<c}$ using the coordinates $\left\{A_e\right\}_{e<c}$ and a point in $\conf^{b'}_{d'}\left(\mathcal{A}_\sc\right)$ by taking the decorated flags in the truncated part corresponding to the shape $\left(b',d'\right)$, and they are mapped to the same point in the torus $T_S$ since they both have $\left\{A_f\right\}$ as non-zero coordinate functions. On the other hand, Given a point in the fiber product $T_{<c}\underset{T_S}{\times}\conf^{b'}_{d'}\left(\mathcal{A}_\sc\right)$ we can recover a point in $U_{\prod_{e<c}A_e}$ by building some extra decorated flags upon the configuration in $\conf^{b'}_{d'}\left(\mathcal{A}_\sc\right)$ using the non-zero functions $\left\{A_e\right\}_{e<c}$. These two morphisms are obviously regular and inverses of each other; therefore $U_{\prod_{e<c}A_e} \cong T_{<c}\underset{T_S}{\times}\conf^{b'}_{d'}\left(\mathcal{A}_\sc\right) $ and the original statement follows immediately.
\end{proof}

Now we are ready to give the proof of the general statement. 

\begin{prop} $A_c$ are irreducible elements in $\mathcal{O}^b_d$ for all closed strings $c$.
\end{prop}
\begin{proof} We will do an induction based on the order $<$ on closed strings. Lemma \ref{irreducible} takes care of the base case. Now inductively suppose $A_e$ are irreducible for all closed strings $e<c$. By Lemma \ref{irreducible} we know that $A_c$ is an irreducible element of $\mathcal{O}^{b'}_{d'}$. Then $A_c=1\otimes A_c$ is also an irreducible element in $ \mathbb{C}\left[A_e^{\pm} \right]_{e<c}\underset{\mathbb{C}\left[A_f^\pm\right]_{f\in S}}{\otimes}\mathcal{O}^{b'}_{d'}\cong \mathcal{O}^b_d\left[\frac{1}{\prod_{e<c}A_e}\right]$ (which is a UFD as well because it is a localization of a UFD). This means that the factorization of $A_c$ in $\mathcal{O}^b_d$ must be of the form 
\begin{equation}\label{factorization}
A_c=F\prod_{e<c}A_e^{n_e}
\end{equation}
for some irreducible element $F$. Now to prove that $A_c$ is indeed irreducible in $\mathcal{O}^b_d$, it suffices to prove that $n_e=0$ for all $e<c$. 

For each $e<c$, consider the subset $W_e$ associated with the closed string $e$ as constructed in the proof of Lemma \ref{codim 1 vanishing locus}. We claim that $W_e\subset \left\{A_e=0\ \text{but} \ A_c\neq 0\right\}$. If the left end point of $c$ lies on the right of the right end point of $e$, then the claim is obvious since any diagonal crossing $c$ is assigned with a general position condition. On the other hand, if the left end point of $c$ lies on the left of the right end point of $e$, we consider the triangle containing the left end point of $e$. Note that in this case, the level of $c$ must be distinct from the level of $e$. By symmetry, let us assume that it looks like the following.
\[
\begin{tikzpicture}
\node (d1) at (0,0) [] {$\A_l$};
\node (d2) at (2,0) [] {$\A_{l+1}$};
\node (u) at (1,2) [] {$\A^k$};
\node (c) at (1,0.7) [] {$j$};
\draw (d1) -- node [left] {$s_i$} (u);
\draw (d2) -- node [right] {$s_i$}(u);
\draw [->] (d1) -- node [below] {$s_j$} (d2);
\draw  (-1,0.7) node [left] {$j$th level} -- (c) --  node [below right] {$c$} (3,0.7);
\draw (-1,1.3) node [left] {$i$th level} -- node [above] {$e$} (3,1.3);
\end{tikzpicture}
\]
By using the $\G$-action, we can fix $\A^k=\overline{s}_it\U_+$ for some $t\in \T$, $\A_l=\U_-$, and $\A_{l+1}=\U_-\doverline{s}_je_{-j}(q)$ for some $q\neq 0$. Since $i\neq j$, we have
\[
A_c=\Delta_{\omega_j}\left(\doverline{s}_je_{-j}\overline{s}_it\U_+\right)=\Delta_{\omega_j}\left(q^{\alpha_j^\vee}\left(s_i(t)\right)\overline{s}_i\right)\neq 0.
\]
This shows that $W_e\subset \left\{A_e=0 \ \text{but} \ A_c\neq 0\right\}$, which implies that $\left\{A_e=0 \ \text{but} \ A_c\neq 0\right\}$ is non-empty. Therefore we can conclude that $n_e=0$ in \eqref{factorization}.
\end{proof}

Now since $A_c$ associated to closed strings are all irreducible elements in $\mathcal{O}^b_d$, their vanishing loci $D_c:=\left\{A_c=0\right\}$ are all irreducible divisors. Furthermore, we deduce the following corollary from the proposition above.

\begin{cor}\label{order} For any two distinct closed strings $c$ and $e$, $\codim\left(D_c\cap D_e\right)\geq 2$ in $\conf^b_d\left(\mathcal{A}_\sc\right)$ and $\ord_{D_c}A_e=0$.
\end{cor}
\begin{proof} First note that the codimension statement follows from the order statement. This is because $\ord_{D_c}A_e=0$ implies that $A_e$ is invertible along $D_c$ and hence $D_c\setminus D_e$ is open in $D_c$; but then since $D_c$ and $D_e$ are both irreducible, this is equivalent to $\codim\left(D_c\cap D_e\right)\geq 2$.

To compute the order of $A_e$ along $D_c$, note that $\ord_{D_c}A_e\geq 0$ because $A_e$ is a regular function on $\conf^b_d\left(\mathcal{A}_\sc\right)$. On the other hand, $\ord_{D_c}A_e>0$ implies that $A_e$ is a multiple of $A_c$, which is impossible as we saw from the proof of last proposition. 
\end{proof}

Our strategy to show $\mathcal{O}^b_d\cong \up\left(\mathscr{A}^b_d\right)$ is analogous to Berenstein, Fomin, and Zelevinsky's proof that the coordinate ring of double Bruhat cells is an upper cluster algebra, which relies on the following theorem.

\begin{thm}[\cite{BFZ}, Corollary 1.9] Fix an initial seed $\vec{s}_0$ and let $\vec{A}_0:=\left\{A_a\right\}$ be the corresponding initial $\mathrm{K}_2$ cluster. Let $\vec{A}_c$ be the $\mathrm{K}_2$ cluster obtained from $\vec{A}$ via a single mutation in the direction of $c$. If the restriction of the exchange matrix $\left.\epsilon_{;\vec{s}_0}\right|_{I^\uf\times I}$ is full-ranked, then upper cluster algebra $\up(\mathcal{A})$ is equal to the intersection
\[
\mathbb{C}\left[\vec{A}^\pm_0\right]\cap \bigcap_{\text{$c$ non-frozen}} \mathbb{C}\left[\vec{A}_c^\pm\right]\subset \Frac\left(\mathbb{C}\left[\vec{A}_0\right]\right).
\]
\end{thm}

In geometric terms, the full-rank condition on $\left.\epsilon_{;\vec{s}_0}\right|_{I^\uf\times I}$ is equivalent to saying that the canonical map $p:T_{\mathscr{A};\vec{s}_0}\rightarrow T_{\mathscr{X}^\uf;\vec{s}_0^\uf}$ is surjective (see Appendix \ref{app B} for definition), and the intersection of Laurent polynomial rings is precisely the coordinate ring of the union of the corresponding seed tori $\spec\mathbb{C}\left[\mathbf{A}_0^\pm\right]\cup \bigcup_c \spec\mathbb{C}\left[\mathbf{A}_c^\pm\right]$.

Going back to the decorated double Bott-Samelson cell $\conf^b_d\left(\mathcal{A}_\sc\right)$. Since we are free to choose any seed in the mutation equivalent family as the initial seed, for the sake of simplicity let us fix our initial seed to be one that is associated to a triangulation in which all triangles of the form $
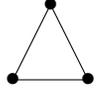
\begin{tikzpicture}[baseline=2ex]
    \node at (-0.5,0) [] {$\bullet$};
    \node at (0.5,0) [] {$\bullet$};
    \node at (0,1) [] {$\bullet$};
    \draw (-0.5,0) -- (0.5,0) -- (0,1) -- cycle;
    \end{tikzpicture}$ come before triangles of the form $\begin{tikzpicture}[baseline=2ex]
    \node at (-0.5,1) [] {$\bullet$};
    \node at (0.5,1) [] {$\bullet$};
    \node at (0,0) [] {$\bullet$};
    \draw (-0.5,1) -- (0.5,1) -- (0,0) -- cycle;
    \end{tikzpicture}$, i.e., a triangulation that looks like the following.

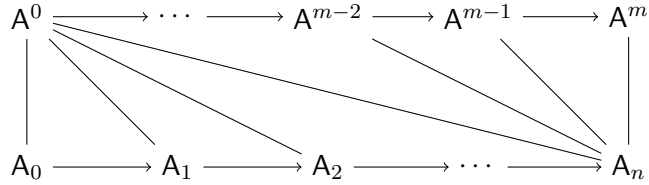
\[
\begin{tikzpicture}
\foreach \i in {0,1,2}
    {
    \node (d\i) at (2*\i,0) [] {$\A_\i$};
    }
\node (d3) at (6,0) [] {$\cdots$};
\node (d4) at (8,0) [] {$\A_n$};
\node (u0) at (0,2) [] {$\A^0$};
\node (u1) at (2,2) [] {$\cdots$};
\node (u2) at (4,2) [] {$\A^{m-2}$};
\node (u3) at (6,2) [] {$\A^{m-1}$};
\node (u4) at (8,2) [] {$\A^m$};
\draw [->] (d0) -- (d1);
\draw [->] (d1) -- (d2);
\draw [->] (d2) -- (d3);
\draw [->] (d3) -- (d4);
\draw [->] (u0) -- (u1);
\draw [->] (u1) -- (u2);
\draw [->] (u2) -- (u3);
\draw [->] (u3) -- (u4);
\draw (u0) -- (d0);
\draw (u0) -- (d1);
\draw (u0) -- (d2);
\draw (u0) -- (d4);
\draw (d4) -- (u2);
\draw (d4) -- (u3);
\draw (d4) -- (u4);
\end{tikzpicture}
\]
Note that any closed string $c$ in the corresponding string diagram lies in a part that is of one of the following three forms.
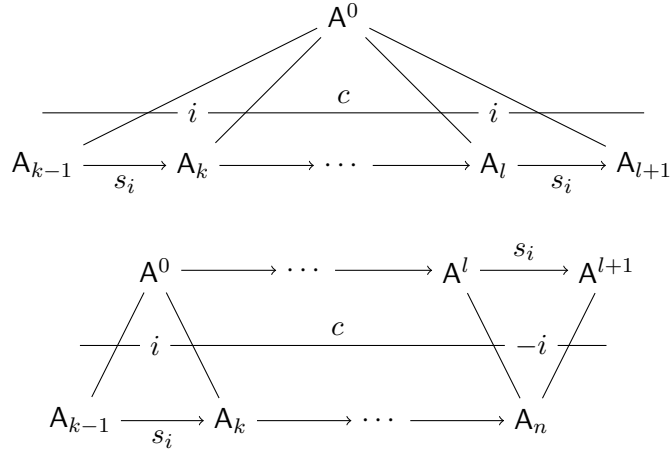
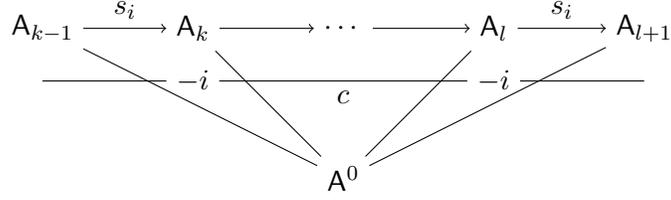
\[
\begin{tikzpicture}
\node (u) at (0,2) [] {$\A^0$};
\node (d0) at (-4,0) [] {$\A_{k-1}$};
\node (d2) at (0,0) [] {$\cdots$};
\node (d1) at (-2,0) [] {$\A_k$};
\node (d3) at (2,0) [] {$\A_l$};
\node (d4) at (4,0) [] {$\A_{l+1}$};
\draw [->] (d0) -- node [below] {$s_i$} (d1);
\draw [->] (d1) -- (d2);
\draw [->] (d2) -- (d3);
\draw (d0) -- (u) -- (d1);
\draw (d3) -- (u) -- (d4);
\draw [->] (d3) -- node [below] {$s_i$} (d4);
\node (0) at (-2,0.7) [] {$i$};
\node (1) at (2,0.7) [] {$i$};
\draw (-4,0.7) -- (0) -- node [above] {$c$} (1) -- (4,0.7);
\end{tikzpicture}
\]
\begin{equation}\label{three cases}
\begin{tikzpicture}[baseline=2ex]
\node (u0) at (1,2) [] {$\A^0$};
\node (d0) at (0,0) [] {$\A_{k-1}$};
\node (d1) at (2,0) [] {$\A_k$};
\node (u1) at (3,2) [] {$\cdots$};
\node (d2) at (4,0) [] {$\cdots$};
\node (d3) at (6,0) [] {$\A_n$};
\node (u2) at (5,2) [] {$\A^l$};
\node (u3) at (7,2) [] {$\A^{l+1}$};
\node (0) at (1,1) [] {$i$};
\node (1) at (6,1) [] {$-i$};
\draw [->] (d0) -- node [below] {$s_i$} (d1);
\draw [->] (d1) -- (d2);
\draw [->] (d2) -- (d3);
\draw [->] (u0) -- (u1);
\draw [->] (u1) -- (u2);
\draw [->] (u2) -- node [above] {$s_i$} (u3);
\draw (d0) -- (u0) -- (d1);
\draw (u2) -- (d3) -- (u3);
\draw (0,1) -- (0) -- node [above] {$c$} (1) -- (7,1);
\end{tikzpicture}
\end{equation}
\[
\begin{tikzpicture}
\node (u) at (0,0) [] {$\A^0$};
\node (d0) at (-4,2) [] {$\A_{k-1}$};
\node (d2) at (0,2) [] {$\cdots$};
\node (d1) at (-2,2) [] {$\A_k$};
\node (d3) at (2,2) [] {$\A_l$};
\node (d4) at (4,2) [] {$\A_{l+1}$};
\draw [->] (d0) -- node [above] {$s_i$} (d1);
\draw [->] (d1) -- (d2);
\draw [->] (d2) -- (d3);
\draw (d0) -- (u) -- (d1);
\draw (d3) -- (u) -- (d4);
\draw [->] (d3) -- node [above] {$s_i$} (d4);
\node (0) at (-2,1.3) [] {$-i$};
\node (1) at (2,1.3) [] {$-i$};
\draw (-4,1.3) -- (0) -- node [below] {$c$} (1) -- (4,1.3);
\end{tikzpicture}
\]

The following lemma is essentially equivalent to Zelevinsky's result on double Bruhat cells \cite[Lemma 3.1 (4)]{Zel}. But since the decorated flag language we use is significantly different from what is in his proof, we rephrase his proofs below for the purpose of completeness.

\begin{lem}\label{oncemutation} For a triangulation chosen as above and any closed string $c$ in the corresponding string diagram, the once-mutated cluster $\mathrm{K}_2$ variable $A'_c$ (as an element in $\Frac\left(\mathcal{O}^b_d\right)$ a priori) belongs to $\mathcal{O}^b_d$ and $\conf^b_d\left(\mathcal{A}_\sc\right)$ contains the seed torus $\spec\mathbb{C}\left[\mathbf{A}_c^\pm\right]$.
\end{lem}
\begin{proof} Let us consider the top case first.
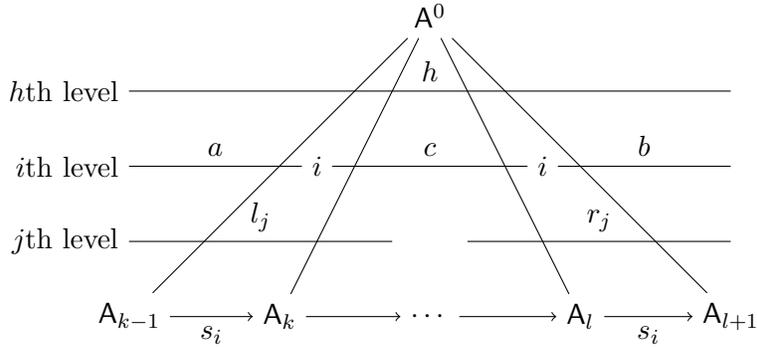
\[
\begin{tikzpicture}
\node (u) at (0,4) [] {$\A^0$};
\node (d0) at (-4,0) [] {$\A_{k-1}$};
\node (d2) at (0,0) [] {$\cdots$};
\node (d1) at (-2,0) [] {$\A_k$};
\node (d3) at (2,0) [] {$\A_l$};
\node (d4) at (4,0) [] {$\A_{l+1}$};
\draw [->] (d0) -- node [below] {$s_i$} (d1);
\draw [->] (d1) -- (d2);
\draw [->] (d2) -- (d3);
\draw (d0) -- (u) -- (d1);
\draw (d3) -- (u) -- (d4);
\draw [->] (d3) -- node [below] {$s_i$} (d4);
\node (0) at (-1.5,2) [] {$i$};
\node (1) at (1.5,2) [] {$i$};
\draw (-4,2) node [left] {$i$th level} -- node [above] {$a$} (0) -- node [above] {$c$} (1) -- node [above] {$b$} (4,2);
\draw (-4,1) node [left] {$j$th level} -- node [above] {$l_j$} (-0.5,1);
\draw (0.5,1) -- node [above] {$r_j$} (4,1);
\draw (-4,3) node [left] {$h$th level} -- node [above] {$h$} (4,3);
\end{tikzpicture}
\]
Let $S$ denote the set of indices $j$ whose corresponding simple reflections $s_j$ occur between the two $s_i$. Let $l_j$ and $r_j$ denote the strings on level $j$ going across the triangles corresponding to these two $s_i$. Then the cluster $\mathrm{K}_2$ mutation formula says that 
\[
A'_c=\frac{1}{A_c}\left(A_a\prod_{j\in S} A_{r_j}^{-\C_{ji}}+A_b\prod_{j\in S}A_{l_j}^{-\C_{ji}}\right)
\]
as a rational function on $\conf^b_d\left(\mathcal{A}_\sc\right)$. Note that $A'_c$ is obviously regular outside of the divisor $D_c$. 

Along the divisor $D_c$ we need to do a small trick. Recall from Corollary \ref{order} that 
\[
\ord_{D_c}\Delta_{\omega_h}\left(\A_k,\A^0\right)=0
\]
for all $h\neq i$. Let us multiply both sides of the cluster $\mathrm{K}_2$ mutation formula by the product $\prod_{h\notin S\cup \{i\}} A_h^{-\C_{hi}}$. Now fix a decorated flag $\A^{-1}$ such that $\xymatrix{\A^{-1}\ar[r]^{s_i} & \A^0}$ is a compatible pair and $\xymatrix{\A_{k-1} \ar@{-}[r] & \A^{-1}}$ are in general position (not necessarily compatible). Such a decorated flag exists because the decorated double Bott-Samelson cell associated to the triangle $\begin{tikzpicture}[baseline=2ex]
\node (u0) at (0,1) [] {$\A^{-1}$};
\node (u1) at (2,1) [] {$\A^0$};
\node (d) at (1,0) [] {$\A_{k-1}$};
\draw [->] (u0) -- node[above] {$s_i$} (u1);
\draw (u0) -- (d) --(u1);
\end{tikzpicture}$ is not empty. Then Corollary \ref{identity1} implies that
\begin{align*}
    A'_c\prod_{h\notin S\cup \{i\}} A_h^{-\C_{hi}}=& \frac{1}{A_c}\left(A_a\left(\Delta_i\left(l+1,-1\right)\Delta_{i}\left(l,0\right)-\Delta_{i}\left(l,-1\right)\Delta_{i}\left(l+1,0\right)\right)\right.\\
    &\left.\quad +A_b\left(\Delta_{i}\left(k,-1\right)\Delta_{i}\left(k-1,0\right)-\Delta_{i}\left(k-1,-1\right)\Delta_{i}\left(k,0\right)\right)\right)\\
    =&\frac{1}{A_c}\left(A_a\Delta_{i}\left(l+1,-1\right) A_c - A_a \Delta_{i}\left(l,-1\right) A_b \right.\\
    &\left.\quad +A_b\Delta_{i}\left(k,-1\right)A_a-A_b\Delta_{i}\left(k-1,-1\right)A_c\right)\\
    =&A_a\Delta_{i}\left(l+1,-1\right)-A_b\Delta_{i}\left(k-1,-1\right).
\end{align*}
The last equality was due to the fact that $\Delta_{i}\left(l,-1\right)=\Delta_{i}\left(k,-1\right)$ because there is no more $s_i$ between the two $s_i$. Therefore we conclude that\footnote{It is worth mentioning that $A'_c$ is independent of the choice of the decorated flag $\A^{-1}$.}
\begin{equation}\label{**}
A'_c=\frac{A_a\Delta_{i}\left(l+1,-1\right)-A_b\Delta_{i}\left(k-1,-1\right)}{\prod_{h\notin S\cup \{i\}} A_h^{-\C_{hi}}}.
\end{equation}
Note that the order of vanishing of the denominator $\prod_{h\notin S\cup \{i\}} A_h^{-\C_{hi}}$ is zero. Therefore $A'_c$ is a regular function by the standard codimension 2 argument.

To show that $\conf^b_d\left(\mathcal{A}_\sc\right)$ contains $\spec\mathbb{C}\left[\mathbf{A}_c^\pm\right]$, it suffices to show that we can construct a configuration in $\conf^b_d\left(\mathcal{A}_\sc\right)$ for any assignment of non-zero numbers to $\left\{A'_c\right\}\cup \left\{A_a\right\}_{a\neq c}$. When the assignment of numbers satisfy
\[
A_a\prod_{j\in S} A_{r_j}^{-\C_{ji}}+A_b\prod_{j\in S}A_{l_j}^{-\C_{ji}}\neq 0,
\]
we can reproduce $A_c$ from these numbers and we get a unique point in the complement of the divisor $D_c$. 

When the above non-vanishing condition is not satisfied, then $A_c=0$ and we need to do a small trick similar to the one we did in proving $A'_c$ is regular. First we observe that by using the cluster $\mathrm{K}_2$ coordinates on the left of $A_c$ (including the ones associated to the left open strings and closed strings that are $<c$) we can build a unique configuration $\begin{tikzpicture}[baseline=2ex]
\node (u) at (0,1) [] {$\A^0$};
\node (d0) at (-1.5,0) [] {$\A_0$};
\node (d1) at (0,0) [] {$\cdots$};
\node (d2) at (1.5,0) [] {$\A_{k-1}$};
\draw (d0) -- (u) -- (d2);
\draw [->] (d0) -- (d1);
\draw [->] (d1) -- (d2);
\end{tikzpicture}$. 

Next we fix a decorated flag $\A^{-1}$ the same way as before. Let us now consider the cluster $\mathrm{K}_2$ coordinate chart associated to the following triangulation.
\begin{equation}\label{***}
\begin{tikzpicture}[baseline=5ex]
\node (u0) at (-1,2) [] {$\A^{-1}$};
\node (d0) at (-4,0) [] {$\A_{k-1}$};
\node (d1) at (-2,0) [] {$\A_k$};
\node (d2) at (0,0) [] {$\cdots$};
\node (d3) at (2,0) [] {$\A_l$};
\draw (u0) -- (d0);
\draw [->] (d0) -- node [below] {$s_i$} (d1);
\draw [->] (d1) -- (d2);
\draw [->] (d2) -- (d3);
\draw (u0) -- (d1);
\draw (u0) -- (d3);
\end{tikzpicture}
\end{equation}
We claim that the non-zero values of $\left\{A'_c\right\}\cup \left\{A_a\right\}_{a\neq c}$ can produce non-zero cluster $\mathrm{K}_2$ coordinates associated to the above triangulation.

First let us apply Corollary \ref{identity1} to $\xymatrix{\A^{-1} \ar[r]^{s_i} & \A^0}$ and $\xymatrix{\A_{k-1} \ar[r]^{s_i} & \A_k}$; the assumption $A_c=\Delta_{i}\left(k,0\right)=0$ reduces the identity in Corollary \ref{identity1} to
\[
\Delta_{i}\left(k,-1\right)A_a=\prod_{j\neq i} A_{l_j}^{-\C_{ji}}.
\]
Since both $A_a$ and $A_{l_j}$ are assumed to be non-zero, we can solve for $\Delta_{i}\left(k,-1\right)$ using the above identity, and the result is still non-zero. This shows $\xymatrix{\A_k\ar@{-}[r] & \A^{-1}}$ and produces non-zero values for the cluster $\mathrm{K}_2$ coordinates along this diagonal.

Let $p$ be an integer with $k\leq p\leq l$. By applying Proposition \ref{welldefined} to the triangle $\begin{tikzpicture}[baseline=2ex]
\node (u0) at (-1,1) [] {$\A^{-1}$};
\node (u1) at (1,1) [] {$\A^0$};
\node (d) at (0,0) [] {$\A_p$};
\draw (u0) -- (d) -- (u1);
\draw [->] (u0) -- node [above] {$s_i$} (u1);
\end{tikzpicture}$, we know that for any $j \neq i$, $\Delta_{j}\left(p,-1\right)=\Delta_{j}\left(p,0\right)\neq 0$. This combined with the fact that $\Delta_{i}\left(p,-1\right)=\Delta_{i}\left(k,-1\right)\neq 0$ proves that $\xymatrix{\A_p\ar@{-}[r] & \A^{-1}}$ and gives non-zero values for all cluster $\mathrm{K}_2$ coordinates along these diagonals. Using these non-zero cluster $\mathrm{K}_2$ coordinates in the triangulation \eqref{***}, we can uniquely construct decorated flags $\A_k, \A_{k+1}, \dots, \A_l$.

Next we rewrite Equation \eqref{**} as
\[
\Delta_{i}\left(l+1,-1\right)=\frac{A'_c\displaystyle\prod_{h\notin S\cup \{i\}}A_h^{-\C_{hi}}+A_b\Delta_{i}\left(k-1,-1\right)}{A_a}.
\]
Note that everything on the right is given already (including $\Delta_{i}\left(k-1,-1\right)$ from the choice of $\A^{-1}$). Therefore we can compute $\Delta_{i}\left(l+1,-1\right)$ using this equation. While this equation does not guarantee that $\Delta_{i}\left(l+1,-1\right)$ is non-zero, it is nevertheless a number in $\mathbb{A}^1$ since the denominator on the right hand side is non-zero. Recall from Lemma \ref{irreducible} that $\Delta_{i}\left(l+1,-1\right)$ parametrizes all the Borel subgroups that are of Tits distance $s_i$ away from $\A_l$; therefore it can be used to uniquely determine $\A_{l+1}$. Furthermore, this decorated flag $\A_{l+1}$ must satisfy $\Delta_{j}\left(l+1,0\right)=\Delta_{j}\left(l+1,-1\right)=A_{r_j}\neq 0$ for all $j \neq i$ and $\Delta_{i}\left(j+1,0\right)=A_b\neq 0$. Therefore we know that $\xymatrix{\A_{l+1}\ar@{-}[r] & \A^0}$.

Once we have the pair of decorated flags $\xymatrix{\A_{l+1}\ar@{-}[r]& \A^0}$, we can then use the cluster $\mathrm{K}_2$ coordinates on the right of $A_c$ to build the remaining decorated flags $\begin{tikzpicture}[baseline=2ex]
\node (u0) at (0,1) [] {$\A^0$};
\node (u1) at (1.5,1) [] {$\cdots$};
\node (u2) at (3,1) [] {$\A^m$};
\node (d0) at (0,0) [] {$\A_{l+1}$};
\node (d1) at (1.5,0) [] {$\cdots$};
\node (d2) at (3,0) [] {$\A_n$};
\draw [->] (u0) -- (u1);
\draw [->] (u1) -- (u2);
\draw [->] (d0) -- (d1);
\draw [->] (d1) -- (d2);
\draw (u0) -- (d0);
\draw (u2) -- (d2);
\end{tikzpicture}$. This finishes the reconstruction of the configuration from the non-zero numerical assignments to the cluster $\mathrm{K}_2$ variables $\left\{A'_c\right\}\cup \left\{A_a\right\}_{a\neq c}$.

Next let us consider the middle case. 
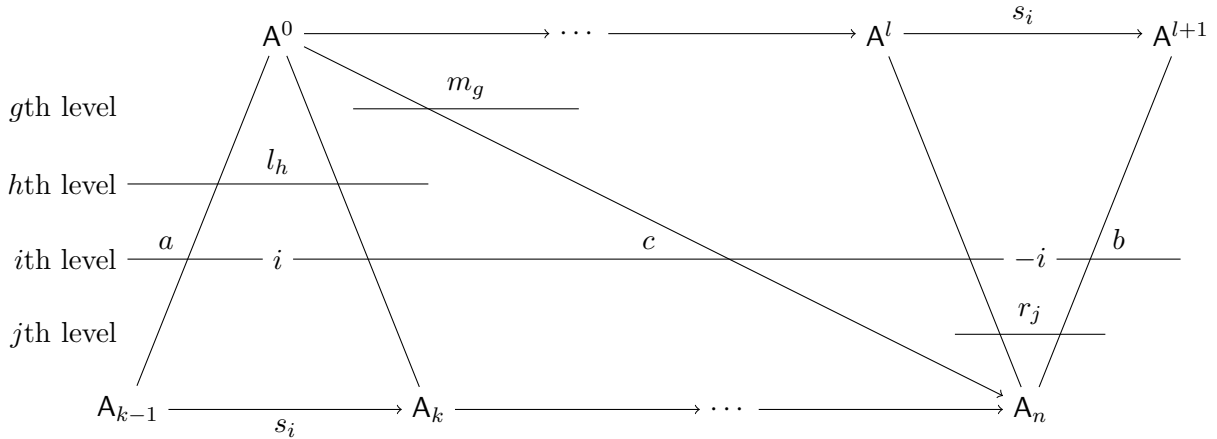
\[
\begin{tikzpicture}[scale=0.7]
\node (u0) at (2,5) [] {$\A^0$};
\node (d0) at (0,0) [] {$\A_{k-1}$};
\node (d1) at (4,0) [] {$\A_k$};
\node (u1) at (6,5) [] {$\cdots$};
\node (d2) at (8,0) [] {$\cdots$};
\node (d3) at (12,0) [] {$\A_n$};
\node (u2) at (10,5) [] {$\A^l$};
\node (u3) at (14,5) [] {$\A^{l+1}$};
\node (0) at (2,2) [] {$i$};
\node (1) at (12,2) [] {$-i$};
\draw [->] (d0) -- node [below] {$s_i$} (d1);
\draw [->] (d1) -- (d2);
\draw [->] (d2) -- (d3);
\draw [->] (u0) -- (u1);
\draw [->] (u1) -- (u2);
\draw [->] (u0) -- (d3);
\draw [->] (u2) -- node [above] {$s_i$} (u3);
\draw (d0) -- (u0) -- (d1);
\draw (u2) -- (d3) -- (u3);
\draw (0,2) node [left] {$i$th level} -- node[above] {$a \quad \quad $} (0) -- node [above] {$c$} (1) -- node [above] {$b$} (14,2);
\draw (0,3) node [left] {$h$th level} -- node [above] {$l_h$} (4,3);
\node at (0,4) [left] {$g$th level};
\draw (3,4) -- node [above] {$m_g$} (6,4);
\node at (0,1) [left] {$j$th level};
\draw (11,1) -- node [above] {$r_j$} (13,1);
\end{tikzpicture}
\]

Let $S_-$ denote the set of indices $h$ whose corresponding simple reflections $s_h$ occur after the last $s_i$ along the bottom edge of the trapezoid, and let $S_+$ denote the set of indices $j$ whose corresponding simple reflections $s_j$ occur before the first $s_i$ along the top edge of the trapezoid. Let $l_h$ denote the string on level $h\in S_-$, going across the triangle corresponding to the left $s_i$; let $r_j$ denote the string on level $j\in S_+$, going across the triangle corresponding to the right $s_i$; let $m_g$ denote the string on level $g\in S_+\cap S_-$, going across the diagonal separating the upward pointing and downward pointing triangles. Then the cluster $\mathrm{K}_2$ mutation formula says that
\begin{equation}\label{****}\resizebox{12cm}{!}{\begin{math}
A'_c=\frac{1}{A_c}\left(A_aA_b\prod_{g\in S_+\cap S_-}A_{m_g}^{-\C_{gi}} +\left(\prod_{h\in S_-}A_{l_h}^{-\C_{hi}}\right)\left(\prod_{j\in S_+}A_{r_j}^{-\C_{ji}}\right)\left(\prod_{g\notin S_-\cup S_+\cup\{i\} }A_{m_g}^{-\C_{gi}}\right)\right)\end{math}}
\end{equation} 
Note that $A'_c$ is obviously regular outside of the divisor $D_c$.

Along the divisor $D_c$ we again need to do a small trick. First note that by Corollary \ref{order}, 
\[
\resizebox{12cm}{!}{\begin{math}
\ord_{D_c}\left(\prod_{g\notin \left(S_+\cap S_-\right)\cup \{i\}}A_{m_g}^{-\C_{gi}}\right)=\ord_{D_c}\left(\left(\prod_{h\notin S_-\cup \{i\}}A_{l_h}^{-\C_{hi}}\right)\left(\prod_{j\in S_-\setminus S_+}A_{r_j}^{-\C_{ji}}\right)\right)=0.\end{math}}
\]
Let us call this product $M$. We multiply both sides of Equation \eqref{****} by $M$ and then try to do some simplification. Fix a decorated flag $\A^{-1}$ such that $\xymatrix{\A^{-1} \ar[r]^{s_i} & \A^0}$ is a compatible pair and $\xymatrix{\A_{k-1} \ar@{-}[r] & \A^{-1}}$ (not necessarily compatible) and fix a decorated flag $\A_{n+1}$ such that $\xymatrix{\A_n \ar[r]^{s_i} & \A_{n+1}}$ is a compatible pair and $\xymatrix{\A^{l+1} \ar@{-}[r] & \A_{n+1}}$ are in general position (not necessarily compatible). Then by Corollary \ref{identity1} we have
\[
M\prod_{g\in S_+\cap S_-}A_{m_g}^{-\C_{gi}}=\Delta_{i}\left(n+1,-1\right) \Delta_{i}\left(n,0\right)-\Delta_{i}\left(n,-1\right)\Delta_{i}\left(n+1,0\right),
\]
\begin{align*}
& M\left(\prod_{h\in S_-}A_{l_h}^{-\C_{hi}}\right)\left(\prod_{j\in S_+}A_{r_j}^{-\C_{ji}}\right)\left(\prod_{g\notin S_-\cup S_+\cup\{i\} }A_{m_g}^{-\C_{gi}}\right)\\
=& M \left(\prod_{h\in S_-}A_{l_h}^{-\C_{hi}}\right)\left(\prod_{j \notin \left(S_-\setminus S_+\right)\cup \{i\}}A_{r_j}^{-\C_{ji}}\right)\\
=&\left(\prod_{h\notin S_-\cup \{i\}}A_{l_h}^{-\C_{hi}}\right)\left(\prod_{j\in S_-\setminus S_+}A_{r_j}^{-\C_{ji}}\right)\left(\prod_{h\in S_-}A_{l_h}^{-\C_{hi}}\right)\left(\prod_{j\notin \left(S_-\setminus S_+\right)\cup \{i\}}A_{r_j}^{-\C_{ji}}\right)\\
=&\left(\prod_{h\neq i}\Delta_{h}\left(k,0\right)^{-\C_{hi}}\right)\left(\prod_{j\neq i}\Delta_{j}\left(n,l\right)^{-\C_{ji}}\right)\\
=&\left(\Delta_{i}\left(k,-1\right)\Delta_{i}\left(k-1,0\right)-\Delta_{i}\left(k-1,-1\right)\Delta_{i}\left(k,0\right)\right)\\
&\left(\Delta_{i}\left(n+1,l\right)\Delta_{i}\left(n,l+1\right)-\Delta_{i}\left(n,l\right)\Delta_{i}\left(n+1,l+1\right)\right).
\end{align*}
Plugging these into $M$ times Equation \eqref{****} and remembering 
\begin{align*}
\Delta_{i}\left(k,-1\right)=&\Delta_{i}\left(n,-1\right),\\
\Delta_{i}\left(n+1,0\right)=&\Delta_{i}\left(n+1,l\right),\\
\Delta_{i}\left(k,0\right)=\Delta_{i}\left(n,0\right)=&\Delta_{i}\left(n,l\right)=A_c,
\end{align*}
we get that
\begin{align*}
MA'_c=&A_aA_b\Delta_{i}(n,0)\Delta_{i}\left(n+1,-1\right)-A_b\Delta_{i}\left(k-1,-1\right)\Delta_{i}\left(n+1,l\right)\\
&-A_a\Delta_{i}\left(k,-1\right)\Delta_{i}\left(n+1,l+1\right)+ A_c\Delta_{i}\left(k-1,-1\right) 
 \Delta_{i}\left(n+1,l+1\right)\\
=& \det \begin{pmatrix} A_a &  \Delta_{i}\left(n+1,l\right) &  A_c \\
 \Delta_{i}\left(k-1,-1\right)&   \Delta_{i}\left(n+1,-1\right) &  \Delta_{i}\left(k,-1\right)\\
0 &  \Delta_{i}\left(n+1,l+1\right) & A_b
\end{pmatrix}.
\end{align*}
Note that along the divisor $D_c$, the upper right hand corner of the last matrix vanishes. Therefore we can conclude that along the divisor $D_c$,
\begin{equation}\label{*}
A'_c=\frac{\det \begin{pmatrix} A_a &  \Delta_{i}\left(n+1,l\right) &  0 \\
 \Delta_{i}\left(k-1,-1\right)&   \Delta_{i}\left(n+1,-1\right) &  \Delta_{i}\left(k,-1\right)\\
0 &  \Delta_{i}\left(n+1,l+1\right) & A_b
\end{pmatrix}}{\displaystyle\prod_{g\notin \left(S_+\cap S_-\right)\cup \{i\}}A_{m_g}^{-\C_{gi}}}
\end{equation}
is a regular function as well.

To show that $\conf^b_d\left(\mathcal{A}_\sc\right)$ contains $\spec\mathbb{C}\left[\mathbf{A}_c^\pm\right]$, it suffices to show that we can construct a configuration in $\conf^b_d\left(\mathcal{A}_\sc\right)$ for any assignment of non-zero numbers to $\left\{A'_c\right\}\cup \left\{A_a\right\}_{a\neq c}$. Most of the arguments are similar to the top case, so we will be brief. There is nothing to show when 
\[
A_aA_b\prod_{g\in S_+\cap S_-}A_{m_g}^{-\C_{gi}} +\left(\prod_{h\in S_-}A_{l_h}^{-\C_{hi}}\right)\left(\prod_{j\in S_+}A_{r_j}^{-\C_{ji}}\right)\left(\prod_{g\notin S_-\cup S_+\cup\{i\} }A_{m_g}^{-\C_{gi}}\right)\neq 0
\]
since we can already recover a non-zero value for $A_c$. When this non-vanishing condition fails, $A_c=0$ and we have to do a small trick again. In fact, we only need to focus on the parallelogram $\begin{tikzpicture}[baseline=2ex]
\node (u0) at (0,1) [] {$\A^0$};
\node (u1) at (1.5,1) [] {$\cdots$};
\node (u2) at (3,1) [] {$\A^{l+1}$};
\node (d0) at (0,0) [] {$\A_{k-1}$};
\node (d1) at (1.5,0) [] {$\cdots$};
\node (d2) at (3,0) [] {$\A_n$};
\draw [->] (u0) -- (u1);
\draw [->] (u1) -- (u2);
\draw [->] (d0) -- (d1);
\draw [->] (d1) -- (d2);
\draw (u0) -- (d0);
\draw (u2) -- (d2);
\end{tikzpicture}$ since everything outside can be constructed from the given values of cluster $\mathrm{K}_2$ variables as usual. Therefore it suffices to show that starting from a given $\xymatrix{\A^0\ar@{-}[r] & \A_{k-1}}$ we can construct the rest of the decorated flags in the parallelogram based on the non-zero values of the given cluster $\mathrm{K}_2$ variables. 

With the decorated flag $\A^{-1}$, we can find unique decorated flags $\A_k, \A_{k+1}, \dots, \A_n$ using the $\mathrm{K}_2$ cluster associated to the following triangle.
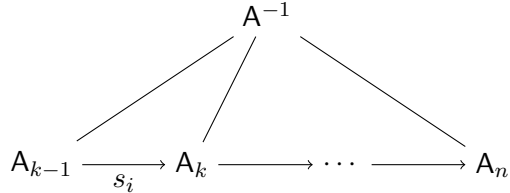
\[
\begin{tikzpicture}
\node (u0) at (-1,2) [] {$\A^{-1}$};
\node (d0) at (-4,0) [] {$\A_{k-1}$};
\node (d1) at (-2,0) [] {$\A_k$};
\node (d2) at (0,0) [] {$\cdots$};
\node (d3) at (2,0) [] {$\A_n$};
\draw (u0) -- (d0);
\draw [->] (d0) -- node [below] {$s_i$} (d1);
\draw [->] (d1) -- (d2);
\draw [->] (d2) -- (d3);
\draw (u0) -- (d1);
\draw (u0) -- (d3);
\end{tikzpicture}
\]
With the decorated flag $\A_{n+1}$, we can find unique decorated flags $\A^1, \A^2, \dots, \A^l$ using the $\mathrm{K}_2$ cluster associated to the following triangle.
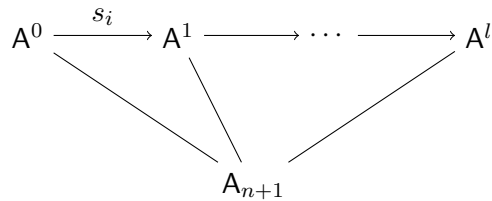
\[
\begin{tikzpicture}
\node (u0) at (-1,-2) [] {$\A_{n+1}$};
\node (d0) at (-4,0) [] {$\A^0$};
\node (d1) at (-2,0) [] {$\A^1$};
\node (d2) at (0,0) [] {$\cdots$};
\node (d3) at (2,0) [] {$\A^l$};
\draw (u0) -- (d0);
\draw [->] (d0) -- node [above] {$s_i$} (d1);
\draw [->] (d1) -- (d2);
\draw [->] (d2) -- (d3);
\draw (u0) -- (d1);
\draw (u0) -- (d3);
\end{tikzpicture}
\]

To determine the last decorated flag $\A^{l+1}$, we need to compute $\Delta_{i}\left(n+1,l+1\right)$ using Equation \eqref{*}; note that the numerical value of everything else in that equation is already given, and the coefficient of $\Delta_{i}\left(n+1,l+1\right)$ is $A_a\Delta_{i}\left(k,-1\right)$, which is non-zero. Therefore we get $\Delta_{i}\left(n+1,l+1\right)$ as an $\mathbb{A}^1$ number, which combined with $\xymatrix{\A_{n+1} \ar@{-}[r] & \A^l}$ determines the decorated flag $\A^{l+1}$ uniquely.

We will omit the proof for the bottom case because it is completely analogous to the top case.
\end{proof}

\begin{prop}\label{codimension 2} The union $U:=\spec\mathbb{C}\left[\mathbf{A}_0^\pm\right]\cup \bigcup_c\spec\mathbb{C}\left[\mathbf{A}_c^\pm\right]$ is of codimension at least 2 in $\conf^b_d\left(\mathcal{A}_\sc\right)$.
\end{prop}
\begin{proof} Let $U_0:=\spec \mathbb{C}\left[\mathbf{A}_0^\pm\right]$ and let $U_c:=\spec\mathbb{C}\left[\mathbf{A}_c^\pm\right]$. Note that
\begin{align*}
    \conf^b_d\left(\mathcal{A}_\sc\right)\setminus U=&\left(\bigcup_{\text{$c$ unfrozen}} D_c\right) \cap \left(\left(\bigcup_{\text{$e$ unfrozen}} U_e\right)^c\right)\\
    =&\left(\bigcup_{\text{$c$ unfrozen}} D_c\right)\cap \left(\bigcap_{\text{$e$ unfrozen}}\left(U_e\right)^c\right)\\
    =&\bigcup_{\text{$c$ unfrozen}}\left(D_c\cap \left(\bigcap_{\text{$e$ unfrozen}}\left(U_e\right)^c\right)\right)\\
    \subset &\bigcup_{\text{$c$ unfrozen}} \left(D_c\cap\left(U_c\right)^c\right).
    \end{align*}
From the proof of the Lemma \ref{oncemutation} we see that $D_c\cap U_c$ is a non-empty open subset of $D_c$. Since $D_c$ is irreducible, $D_c\cap\left(U_c\right)^c$ must be at least codimension 1 inside $D_c$ and hence at least codimension 2 inside $\conf^b_d\left(\mathcal{A}_\sc\right)$.
\end{proof}

To prove $\mathcal{O}^b_d\cong \up\left(\mathscr{A}^b_d\right)$, we still need to verify the surjectivity condition on the canonical map $p$. We will do so by realizing the map $p:T_{\mathscr{A};\vec{s}_0}\rightarrow T_{\mathscr{X}^\uf;\vec{s}_0^\uf}$ as a restriction of the composition $\conf^b_d\left(\mathcal{A}_\sc\right)\rightarrow \conf^b_d\left(\mathcal{A}_\ad\right)\rightarrow \conf^b_d(\mathcal{B})$. 

\begin{prop}\label{3.38} Consider the surjective map $\pi:\conf^b_d\left(\mathcal{A}_\sc\right)\rightarrow \conf^b_d\left(\mathcal{A}_\ad\right)$. For any closed string $c$, 
\[
\pi^*\left(X_c\right)=\prod_a A_a^{\epsilon_{ca}}
\]
where $\epsilon_{ca}$ is the exchange matrix for the given seed.
\end{prop}
\begin{proof} Note that it suffices to prove this statement on one seed. Let us again use the seed as in Lemma \ref{oncemutation}.
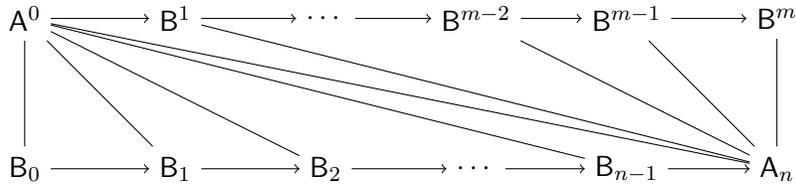
\[
\begin{tikzpicture}
\foreach \i in {0,1,2}
    {
    \node (d\i) at (2*\i,0) [] {$\B_\i$};
    }
\node (d3) at (6,0) [] {$\cdots$};
\node (d4) at (8,0) [] {$\B_{n-1}$};
\node (d5) at (10,0) [] {$\A_n$};
\node (u0) at (0,2) [] {$\A^0$};
\node (u1) at (2,2) [] {$\B^1$};
\node (u2) at (4,2) [] {$\cdots$};
\node (u3) at (6,2) [] {$\B^{m-2}$};
\node (u4) at (8,2) [] {$\B^{m-1}$};
\node (u5) at (10,2) [] {$\B^m$};
\draw [->] (d0) -- (d1);
\draw [->] (d1) -- (d2);
\draw [->] (d2) -- (d3);
\draw [->] (d3) -- (d4);
\draw [->] (d4) -- (d5);
\draw [->] (u0) -- (u1);
\draw [->] (u1) -- (u2);
\draw [->] (u2) -- (u3);
\draw [->] (u3) -- (u4);
\draw [->] (u4) -- (u5);
\draw (u0) -- (d0);
\draw (u0) -- (d1);
\draw (u0) -- (d2);
\draw (u0) -- (d4);
\draw (u0) -- (d5);
\draw (d5) -- (u1);
\draw (d5) -- (u3);
\draw (d5) -- (u4);
\draw (d5) -- (u5);
\end{tikzpicture}
\]
Then the closed strings again come in three types, as described in \eqref{three cases}. Let us first look at the top case. 
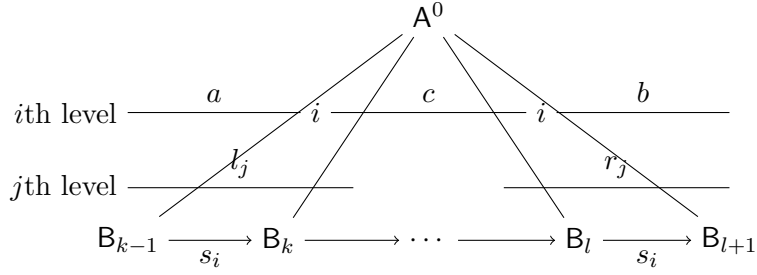
\[
\begin{tikzpicture}
\node (u) at (0,3) [] {$\A^0$};
\node (d0) at (-4,0) [] {$\B_{k-1}$};
\node (d2) at (0,0) [] {$\cdots$};
\node (d1) at (-2,0) [] {$\B_k$};
\node (d3) at (2,0) [] {$\B_l$};
\node (d4) at (4,0) [] {$\B_{l+1}$};
\draw [->] (d0) -- node [below] {$s_i$} (d1);
\draw [->] (d1) -- (d2);
\draw [->] (d2) -- (d3);
\draw (d0) -- (u) -- (d1);
\draw (d3) -- (u) -- (d4);
\draw [->] (d3) -- node [below] {$s_i$} (d4);
\node (0) at (-1.5,1.7) [] {$i$};
\node (1) at (1.5,1.7) [] {$i$};
\draw (-4,1.7) node[left] {$i$th level} -- node[above] {$a$} (0) -- node [above] {$c$} (1) -- node[above] {$b$} (4,1.7);
\draw (-4,0.7) node [left] {$j$th level} -- node [above] {$l_j$} (-1,0.7);
\draw (1,0.7) -- node [above] {$r_j$} (4,0.7);
\end{tikzpicture}
\]
To work out the cluster Poisson coordinate $X_c$, we need to first compute the two Lusztig factorization coordinates at the two ends of the closed string $c$. Consider the triangle containing the left end point. By acting by some $u\in \U_+$ we can move the chosen representative configuration into a configuration with $\B_{k-1}=\B_-$ and $\A^0=\U_+$. Since the cluster $\mathrm{K}_2$ coordinates are invariant under the $\G$-action, by Corollary \ref{A to L} we see that the unipotent element that moves $\B_{k-1}$ to $\B_k$ is $e_i\left(\displaystyle\frac{\prod_{j\neq i}A_{l_j}^{-\C_{ji}}}{A_aA_c}\right)$ and hence the Lusztig factorization coordinate on the left is $q=\frac{\prod_{j\neq i}A_{l_j}^{-\C_{ji}}}{A_aA_c}$. By a similar argument, we get that the Lusztig factorization coordinate on the right is $q'=\frac{\prod_{j\neq i}A_{r_j}^{-\C_{ji}}}{A_bA_c}$. Therefore we obtain
\[
\pi^*\left(X_c\right)=\frac{q'}{q}=\frac{A_a\prod_{j\neq i}A_{r_j}^{-\C_{ji}}}{A_b\prod_{j\neq i}A_{l_j}^{-\C_{ji}}}=\prod_a A_a^{\epsilon_{ca}}.
\]

For the middle case, the Lusztig factorization coordinate on the left can be computed the same as above, which is $q=\frac{\prod_{j\neq i}A_{l_j}^{-\C_{ji}}}{A_aA_c}$.
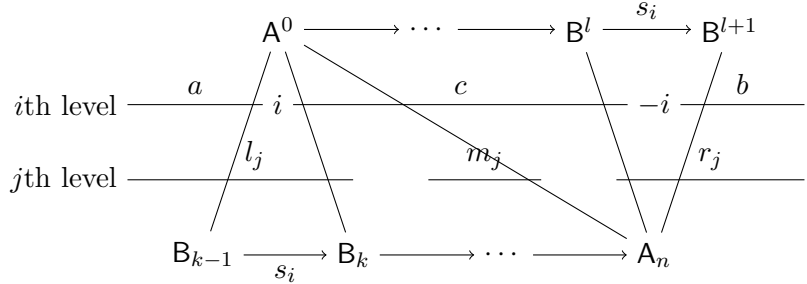
\[
\begin{tikzpicture}[baseline=2ex]
\node (u0) at (1,3) [] {$\A^0$};
\node (d0) at (0,0) [] {$\B_{k-1}$};
\node (d1) at (2,0) [] {$\B_k$};
\node (u1) at (3,3) [] {$\cdots$};
\node (d2) at (4,0) [] {$\cdots$};
\node (d3) at (6,0) [] {$\A_n$};
\node (u2) at (5,3) [] {$\B^l$};
\node (u3) at (7,3) [] {$\B^{l+1}$};
\node (0) at (1,2) [] {$i$};
\node (1) at (6,2) [] {$-i$};
\draw [->] (d0) -- node [below] {$s_i$} (d1);
\draw [->] (d1) -- (d2);
\draw [->] (d2) -- (d3);
\draw [->] (u0) -- (u1);
\draw [->] (u1) -- (u2);
\draw [->] (u2) -- node [above] {$s_i$} (u3);
\draw (d0) -- (u0) -- (d1);
\draw (u2) -- (d3) -- (u3);
\draw (u0) -- (d3);
\draw (-1,2) node[left] {$i$th level} -- node[above] {$a$} (0) -- node [above] {$c$} (1) -- node[above] {$b$} (8,2);
\draw (-1,1) node [left] {$j$th level} --  node [above] {$\quad l_j$} (2,1);
\draw (3,1)  -- node [above] {$m_j$} (4.5,1);
\draw (5.5,1) -- node [above] {$r_j$} (8,1);
\end{tikzpicture}
\]
However, the Lusztig factorization coordinate on the right is slightly more complicated to compute. Recall from Corollary \ref{A to L}, when triangles switch orientation, one also needs to move the whole configuration by the maximal torus element $t:=\prod_{j=1}^{\tilde{r}}\A_{m_j}^{\alpha_j^\vee}$. But then in the computation of cluster Poisson coordinates on $\conf^b_d\left(\mathcal{A}_\ad\right)$, we do not allow maximal torus elements appearing in the middle; therefore we need to push this extra factor of $t$ all the way to the right, resulting in an extra $t^{\alpha_i}$ factor (the same technique was also used in the proof of Proposition \ref{3.15}). Therefore the Lusztig factorization coordinate $p$ is
\[
p=\frac{\prod_{j\neq i}A_{r_j}^{-\C_{ji}}}{A_bA_c}t^{\alpha_i}=\frac{\prod_{j\neq i}A_{r_j}^{-\C_{ji}}}{A_bA_c}\frac{A_c^2}{\prod_{j\neq i}A_{m_j}^{-\C_{ji}}}=\frac{A_c\prod_{j\neq i}A_{r_j}^{-\C_{ji}}}{A_b\prod_{j\neq i}A_{m_j}^{-\C_{ji}}}.
\]
Now by the definition of cluster Poisson coordinates, we find that
\[
\pi^*\left(X_c\right)=\frac{1}{pq}=\frac{A_aA_b\prod_{j\neq i}A_{m_j}^{-\C_{ji}}}{\left(\prod_{j\neq i}A_{l_j}^{-\C_{ji}}\right)\left(\prod_{j\neq i}A_{r_j}^{-\C_{ji}}\right)}=\prod_a A_a^{\epsilon_{ca}}.
\]

The bottom case can be either computed analogous to the top case or obtained from the top case using the transposition morphism as an intertwiner.
\end{proof}

\begin{prop}\label{3.39} Consider the seed as in Lemma \ref{oncemutation}. Let $m_j$ be the string that passes through the diagonal $\xymatrix{\A_n \ar@{-}[r] & \A^0}$. For an open string $f$ on the $i$th next to a node $n$, denote the string that crosses the vertical line at $n$ on the $j$th level by $a_j$ and denote the string on the other side of $f$ by $b$. Then we have
\[
\pi^*\left(X_f\right)=\left\{\begin{array}{ll} 
\frac{\prod_{j\neq i}A_{a_j}^{-\C_{ji}}}{A_b A_f} & \text{if $\xymatrix{\ar@{-}[r]^f & i \ar@{-}[r]^b & }$ or $\xymatrix{\ar@{-}[r]^b & -i \ar@{-}[r]^f &}$};\\
& \\ 
\frac{A_b\prod_{j\neq i}A_{m_j}^{-\C_{ji}}}{A_f\prod_{j\neq i}A_{a_j}^{-\C_{ji}}} & \text{if $\xymatrix{\ar@{-}[r]^f & -i \ar@{-}[r]^b & }$ or $\xymatrix{\ar@{-}[r]^b & i \ar@{-}[r]^f &}$}.
\end{array} \right.
\]
If $f$ is a string on the $i$th level that is open on both ends (which typically happens for $i>r$), then we have
\[
\pi^*\left(X_f\right)=\prod_{j=1}^{\tilde{r}} A_{m_j}^{\C_{ji}}
\]
In particular, these formulas are all in the form
\[
\pi^*\left(X_f\right)=\left(\prod_{\text{$c$ closed}} A_c^{\epsilon_{fc}}\right)\cdot \left(\text{Laurent monomial in $A_g$ for open string $g$}\right)
\]
\end{prop}
\begin{proof} The cases where $f$ is on the left side of the string diagram with a node attached can be obtained the same way as in the proof of the last proposition. The cases where $f$ is on the right side of the string diagram with a node attached can then be obtained from the cases on the left via the transposition morphism. It remains to show the case where $f$ is open on both ends, for which
\[
\pi^*\left(X_f\right)=t^{\alpha_i}=t^{\sum_{j=1}^{\tilde{r}}\C_{ji}\omega_j }=\prod_{j=1}^{\tilde{r}} A_{m_j}^{\C_{ji}}.\qedhere
\]
\end{proof}

Our chosen triangulation determines a cluster Poisson seed torus $T_\mathscr{X}$ on $\conf^b_d\left(\mathcal{A}_\ad\right)$ a cluster $\mathrm{K}_2$ seed torus $T_\mathscr{A}$ on $\conf^b_d\left(\mathcal{A}_\sc\right)$. Since these two seed tori are both cut out by the same general position conditions on the underlying undecorated flags, the surjective map $\pi:\conf^b_d\left(\mathcal{A}_\sc\right)\rightarrow \conf^b_d\left(\mathcal{A}_\ad\right)$ restricts to a surjective map $\pi:T_\mathscr{A}\rightarrow T_\mathscr{X}$. 

Moreover, the formulas from Proposition \ref{3.38} and Proposition \ref{3.39} showed that $\pi$ fits in the following commutative diagram (where the maps $e,f,p$, and $q$ are defined in Appendix \ref{app B}):
\[
\xymatrix{T_{\mathscr{A}^\uf} \ar[r]^e \ar[dr]_f & T_\mathscr{A} \ar[d]^\pi \ar[dr]^p & \\
& T_\mathscr{X} \ar[r]_q & T_{\mathscr{X}^\uf}}
\]
In particular, this diagram is just a restriction of the maps in Diagram \eqref{maps} to some open subsets corresponding to our choice of triangulation.

Since $p=q\circ \pi$ and both $q$ and $\pi$ are surjective, we know that $p$ is surjective as well. By combining the surjectivity of $p$ and the codimension 2 condition (Proposition \ref{codimension 2}, we finally prove our first claim at the beginning of the section.

\begin{thm}\label{3.42} $\mathcal{O}^b_d\cong \up\left(\mathscr{A}^b_d\right)$.
\end{thm}

Let us look at the other claim at the beginning of this subsection, $\mathcal{O}\left(\conf^b_d\left(\mathcal{A}_\ad\right)\right)\cong \up \left(\mathscr{X}^b_d\right)$. We will use a technique similar to the proof of \cite[Theorem 2.15]{SWgrass}. 

First note that from the commutation diagram above we get an injective algebra homomorphism $\pi^*:\up\left(\mathscr{X}^b_d\right)\rightarrow \up\left(\mathscr{A}^b_d\right)$ (see Proposition \ref{pi} and Propositoin \ref{surjective pi}). In particular, an element $F\in \Frac\left(\up\left(\mathscr{X}^b_d\right)\right)$ is in $\up\left(\mathscr{X}^b_d\right)$ if and only if $\pi^*(F)$ is in $\up\left(\mathscr{A}^b_d\right)$.

\begin{thm}\label{3.43} $\mathcal{O}\left(\conf^b_d\left(\mathcal{A}_\ad\right)\right)\cong \up\left(\mathscr{X}^b_d\right)$.
\end{thm}
\begin{proof} Since $\conf^b_d\left(\mathcal{A}_\ad\right)$ contains a cluster Poisson seed torus $T_\mathscr{X}$ as an open dense subset, we have $\Frac\left(\mathcal{O}\left(\conf^b_d\left(\mathcal{A}_\ad\right)\right)\right)\cong \Frac\left(\mathcal{O}\left(T_\mathscr{X}\right)\right)=\Frac\left(\up\left(\mathscr{X}^b_d\right)\right)$, which contains both $\mathcal{O}\left(\conf^b_d\left(\mathcal{A}_\ad\right)\right)$ and $\up\left(\mathscr{X}^b_d\right)$ as subalgebras. Now consider the following commutative diagram of algebras.
\[
\xymatrix{\Frac\left(\mathcal{O}\left(\conf^b_d\left(\mathcal{A}_\ad\right)\right)\right) \ar@{^(->}[d]_{\pi^*}\ar[r]^(0.6){\chi}_(0.6){\cong} &\Frac\left(\up\left(\mathscr{X}^b_d\right)\right)\ar@{^(->}[d]_{\pi^*}\\
\Frac\left(\mathcal{O}\left(\conf^b_d\left(\mathcal{A}_\sc\right)\right)\right) \ar[r]^(0.6){\cong}_(0.6){\alpha} & \Frac\left(\up\left(\mathscr{A}^b_d\right)\right)}
\]

Suppose $f$ is a regular function on $\conf^b_d\left(\mathcal{A}_\ad\right)$. Then we know that $\chi(f)\in \Frac\left(\up\left(\mathscr{X}^b_d\right)\right)$. To show that $\chi(f)$ is actually in $\up\left(\mathscr{X}^b_d\right)$, we only need to show $\pi^*\circ\chi(f)\in \up\left(\mathscr{A}^b_d\right)$ by Proposition \ref{pi}. But this is true because $\pi^*\circ \chi(f)=\alpha\circ \pi^*(f)$ is in the image of the composition 
\[
\mathcal{O}\left(\conf^b_d\left(\mathcal{A}_\ad\right)\right) \overset{\pi^*}{\longrightarrow} \mathcal{O}\left(\conf^b_d\left(\mathcal{A}_\sc\right)\right) \overset{\alpha}{\longrightarrow} \up\left(\mathscr{A}^b_d\right).
\]

On the other hand, suppose $F$ is an element in $\up\left(\mathscr{X}^b_d\right)$. Then we know that $\chi^{-1}(F)$ is a rational function on $\conf^b_d\left(\mathcal{A}_\ad\right)$. Since $\pi:\conf^b_d\left(\mathcal{A}_\sc\right)\rightarrow \conf^b_d\left(\mathcal{A}_\ad\right)$ is surjective, to show the regularity of $\chi^{-1}(F)$, it suffices to show that $\pi^*\circ \chi^{-1}(F)$ is a regular function on $\conf^b_d\left(\mathcal{A}_\sc\right)$. But this is true because $\pi^*\circ \chi^{-1}(F)=\alpha^{-1}\circ \pi^*(F)$ is in the image of the composition
\[
\up\left(\mathscr{X}^b_d\right)\overset{\pi^*}{\longrightarrow}
 \up\left(\mathscr{A}^b_d\right)  \overset{\alpha^{-1}}{\longrightarrow} \mathcal{O}\left(\conf^b_d\left(\mathcal{A}_\sc\right)\right). \qedhere
\]
\end{proof}

\section{Donaldson-Thomas Transformation of Bott-Samelson Cells}

In this section we will show that the cluster Donaldson-Thomas transformation exists on the unfrozen Poisson cluster algebra $\up\left(\mathscr{X}^b_d\right)^\uf$, and realize it as a biregular morphism on the undecorated double Bott-Samelson cell $\conf^b_d(\mathcal{B})$. We will first show it for the case $\left(\mathscr{X}^e_b\right)^\uf$ by constructing a maximal green sequence, whose existence implies the existence of cluster Donaldson-Thomas sequence (see Appendix \ref{app B} for more details), and then deduce the general cases $\left(\mathscr{X}^b_d\right)^\uf$ by using the reflection maps defined in Subsection \ref{simplereflection}.

\subsection{A Maximal Green Sequence for \texorpdfstring{$\left(\mathscr{X}^e_b\right)^\uf$}{}} Let us consider the cluster Poisson algebra $\up\left(\mathscr{X}^e_b\right)^\uf$ associated to the undecorated double Bott-Samelson cell $\conf^e_b(\mathcal{B})$. Our goal is to construct a maximal green sequence.

We should first fix an initial seed. Due to the fact that the positive braid upstairs is the identity $e$, the trapezoid is in fact a triangle, and there is only one available triangulation in this case, which is the following.
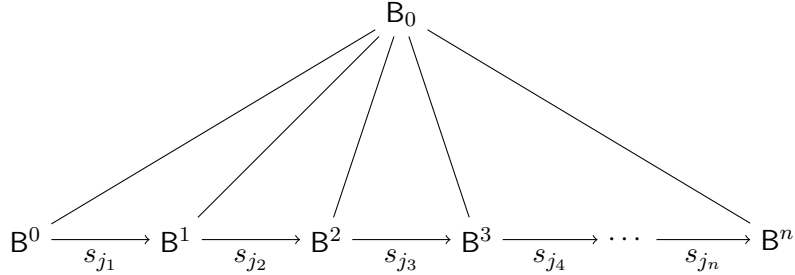
\[
\tikz{
\foreach \i in {0,...,3}
    {
    \node (\i) at (2*\i,0) [] {$\B^\i$};
    }
\node (4) at (8,0) [] {$\cdots$};
\node (5) at (10,0) [] {$\B^n$};
\foreach \i in {0,...,3}
    {
    \pgfmathtruncatemacro{\j}{\i+1};
    \draw [->] (\i) -- node [below] {$s_{j_\j}$} (\j);
    }
\draw [->] (4) -- node [below] {$s_{j_n}$} (5);
\node (u) at (5,3) [] {$\B_0$};
\draw (0) -- (u) -- (1);
\draw (2) -- (u) -- (3);
\draw (5) -- (u);
}
\]

Since all the triangles are of the shape $\begin{tikzpicture}[baseline=2ex]\draw (0,1) -- (-0.5,0) --  node [below] {$s_j$} (0.5,0) -- cycle;\end{tikzpicture}$, the corresponding string diagram only has simple roots at the nodes. To better demonstrate the idea, let us use the word 
\[
(2,1,3,2,1,3,1,3,2,2,1)
\]
as a running example. The string diagram looks like the following (we will only draw the closed strings since open strings are not part of the unfrozen seed.
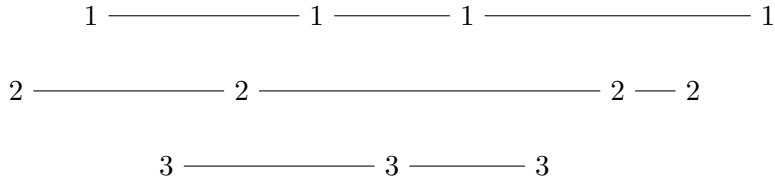
\[
\tikz{
\node (a1) at (1,2) [] {$1$};
\node (b1) at (0,1) [] {$2$};
\node (g1) at (2,0) [] {$3$};
\node (b2) at (3,1) [] {$2$};
\node (a2) at (4,2) [] {$1$};
\node (g2) at (5,0) [] {$3$};
\node (a3) at (6,2) [] {$1$};
\node (g3) at (7,0) [] {$3$};
\node (b3) at (8,1) [] {$2$};
\node (b4) at (9,1) [] {$2$};
\node (a4) at (10,2) [] {$1$};
\draw (a1) -- (a2) -- (a3) --(a4);
\draw (b1) -- (b2) -- (b3) -- (b4);
\draw (g1) -- (g2) -- (g3);
}
\]

Recall that the closed strings become vertices of the seed $\vec{s}^\uf$, and the associated exchange matrix $\epsilon$ has two types of non-vanishing entries: $\pm 1$ for neighboring vertices on the same horizontal level, and $\pm \C_{ji}$ between nearby vertices on the $i$th and $j$th levels. Since the magnitude of these entries are determined by the horizontal levels of the vertices, the only extra data we need to record are the signs. We will hence use arrows of the form $\begin{tikzpicture}[baseline=-0.5ex]\node (a) at (0,0) [] {$a$}; \node (b) at (2,0) [] {$b$}; \draw [->] (a) -- (b);\end{tikzpicture}$ to denote the first case with $\epsilon_{ab}=-\epsilon_{ba}=1$, and use arrow $\begin{tikzpicture}[baseline=-0.5ex]\node (a) at (0,0) [] {$a$}; \node (b) at (2,0) [] {$b$}; \draw [->,decorate,decoration={snake,amplitude=.4mm,segment length=2mm,post length=1mm}] (a) -- (b); \end{tikzpicture}$ to denote the second case with $\epsilon_{ab}>0$ and $\epsilon_{ba}<0$. 

Using such notation, our running example gives rise to a seed that can be described as follows.
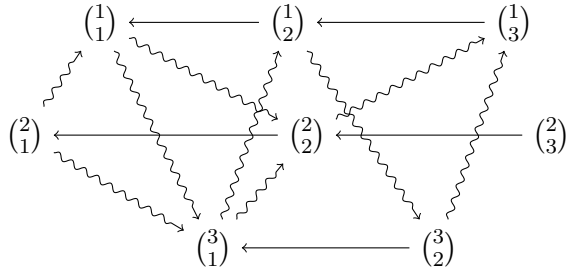
\[
\tikz{
\node (a1) at (2.5,3) [] {$\binom{1}{1}$};
\node (a2) at (5,3)  [] {$\binom{1}{2}$};
\node (a3) at (8,3) [] {$\binom{1}{3}$};
\node (b1) at (1.5,1.5) [] {$\binom{2}{1}$};
\node (b2) at (5.25,1.5) [] {$\binom{2}{2}$};
\node (b3) at (8.5,1.5) [] {$\binom{2}{3}$};
\node (c1) at (4, 0) [] {$\binom{3}{1}$};
\node (c2) at (7,0) [] {$\binom{3}{2}$};
\draw [<-] (a1) -- (a2);
\draw [<-] (a2) -- (a3);
\draw [<-] (b1) -- (b2);
\draw [<-] (b2) -- (b3);
\draw [<-] (c1) -- (c2);
\draw [->,decorate,decoration={snake,amplitude=.4mm,segment length=2mm,post length=1mm}] (b1) -- (a1);
\draw [->,decorate,decoration={snake,amplitude=.4mm,segment length=2mm,post length=1mm}] (a1) -- (b2);
\draw [->,decorate,decoration={snake,amplitude=.4mm,segment length=2mm,post length=1mm}] (b2) -- (a3);
\draw [->,decorate,decoration={snake,amplitude=.4mm,segment length=2mm,post length=1mm}] (b1) -- (c1);
\draw [->,decorate,decoration={snake,amplitude=.4mm,segment length=2mm,post length=1mm}] (c1) -- (b2);
\draw [->,decorate,decoration={snake,amplitude=.4mm,segment length=2mm,post length=1mm}] (a1) -- (c1);
\draw [->,decorate,decoration={snake,amplitude=.4mm,segment length=2mm,post length=1mm}] (c1) -- (a2);
\draw [->,decorate,decoration={snake,amplitude=.4mm,segment length=2mm,post length=1mm}] (a2) -- (c2);
\draw [->,decorate,decoration={snake,amplitude=.4mm,segment length=2mm,post length=1mm}] (c2) -- (a3);
}
\]

We are now ready to describe a maximal green sequence for seeds of such form. First for each entry $i_k$ of the word $\vec{i}=\left(i_1,i_2,\dots, i_n\right)$, we define 
\[
t_k:=\#\left\{l \ \middle| \ \text{$k<l\leq n$ and $i_l=i_k$}\right\}.
\]
Note that if $i_k$ is the first occurrence of the letter $i$ in the word $\vec{i}$, then there are exactly $t_k$ many closed strings on level $i$.

Next for each $i_k$ we define a mutation sequence 
\[
L_k:=\mu_{\binom{i_k}{t_k}}\circ \dots \circ \mu_{\binom{i_k}{2}}\circ \mu_{\binom{i_k}{1}}.
\]
In particular, if $i_k$ is the last occurrence of the letter $i$ in the word $\vec{i}$, then $L_k$ is trivial. We then claim the following.

\begin{thm}\label{mgs} The mutation sequence
\[
L_n\circ L_{n-1}\circ \dots \circ L_2\circ L_1
\]
is a maximal green sequence.
\end{thm}

In the running example, the mutation sequence described in Theorem \ref{mgs} is 
\begin{align*}& L_{11}\circ L_{10}\circ L_9\circ L_8\circ L_7\circ L_6\circ L_5\circ L_4\circ L_3\circ L_2\circ L_1\\
=& L_9 \circ L_7\circ L_6\circ L_5\circ L_4\circ L_3\circ L_2\circ L_1\\
=&\left(\mu_{\binom{2}{1}}\right)\circ\left(\mu_{\binom{1}{1}}\right)\circ \left(\mu_{\binom{3}{1}}\right)\circ\left(\mu_{\binom{1}{2}}\circ \mu_{\binom{1}{1}}\right)\circ \left(\mu_{\binom{2}{2}}\circ \mu_{\binom{2}{1}}\right)\circ\left(\mu_{\binom{3}{2}}\circ \mu_{\binom{3}{1}}\right)\\
&\circ \left(\mu_{\binom{1}{3}}\circ \mu_{\binom{1}{2}}\circ \mu_{\binom{1}{1}}\right)\circ \left(\mu_{\binom{2}{3}}\circ \mu_{\binom{2}{2}}\circ \mu_{\binom{2}{1}}\right).
\end{align*}

\begin{prop}\label{procedure} In terms of triangles in the triangulation, the mutation sequence $L_k$ is equivalent to a reflection map $_{i_k}r$ followed by a change of coordinates corresponding a sequence of mutations that moves the new triangle $\begin{tikzpicture}[baseline=2ex]\draw (0,0) -- (-0.5,1) -- node [above] {$s_i$} (0.5,1) -- cycle;\end{tikzpicture}$ to the right across all triangles of the shape $\begin{tikzpicture}[baseline=2ex] \draw (-0.5,0) -- (0.5,0) -- (0,1) -- cycle;
\end{tikzpicture}$
\end{prop}
\begin{proof} Recall that if we start with a cluster Poisson coordinate chart on the left, then the reflection morphism maps it to the cluster Poisson coordinate chart on the right.
\[
\begin{tikzpicture}[baseline=4ex]
\node (u0) at (0,1.5) [] {$\B^0$};
\node (u1) at (1,1.5) [] {$\cdots$};
\node (d0) at (-0.5,0) [] {$\B_0$};
\node (d1) at (0.5,0) [] {$\B_1$};
\node (d2) at (1.5,0) [] {$\cdots$};
\draw [->] (u0) -- (u1);
\draw [->] (d0) -- node [below] {$s_i$} (d1);
\draw (d0) -- (u0) -- (d1);
\draw [->] (d1) -- (d2);
\end{tikzpicture} \quad \quad \overset{_ir}{\longrightarrow} \quad \quad 
\begin{tikzpicture}[baseline=4ex]
\node (u) at (-1,1.5) [] {$\B^{-1}$};
\node (u0) at (0,1.5) [] {$\B^0$};
\node (u1) at (1,1.5) [] {$\cdots$};
\node (d0) at (-0.5,0) [] {$\B_1$};
\node (d1) at (0.5,0) [] {$\cdots$};
\draw [->] (u0) -- (u1);
\draw [->] (d0) -- (d1);
\draw (u) -- (d0) -- (u0);
\draw [->] (u) -- node [above] {$s_i$} (u0);
\end{tikzpicture}
\]
In particular, these two cluster Poisson coordinates have isomorphic unfrozen seeds because the sign change of the left most node of the string diagram does not change the exchange matrix entries between closed strings. Now it remains to show that cluster Poisson coordinates pull back to their corresponding cluster Poisson coordinates via $_ir^*$. But by the definition of cluster Poisson coordinates, it suffices to show that if the triangle in the picture on the left is $
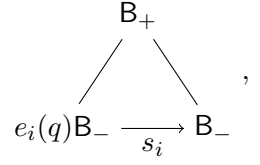
\begin{tikzpicture}[baseline=4ex]
\node (u0) at (0,1.5) [] {$\B_+$};
\node (d0) at (-1,0) [] {$e_i(q)\B_-$};
\node (d1) at (1,0) [] {$\B_-$};
\draw [->] (d0) -- node [below] {$s_i$} (d1);
\draw (d0) -- (u0) -- (d1);
\end{tikzpicture}$, then the triangle in the picture on the right is
$
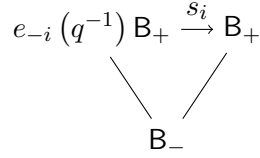
\begin{tikzpicture}[baseline=4ex]
\node (u) at (-2,1.5) [] {$e_{-i}\left(q^{-1}\right)\B_+$};
\node (u0) at (0,1.5) [] {$\B_+$};
\node (d0) at (-1,0) [] {$\B_-$};
\draw [->] (u) -- node [above] {$s_i$} (u0);
\draw (u) -- (d0) -- (u0);
\end{tikzpicture}$. By the definition of the reflection map $_ir$, it suffices to show that $\xymatrix{e_i(q)\B_- \ar@{-}[r]^{s_i} & e_{-i}\left(q^{-1}\right)B_+}$. But this follows from the following computation:
\begin{align*}
\B_-e_i(q)^{-1}e_{-i}\left(q^{-1}\right)\B_+=&\B_-e_i(-q)e_{-i}\left(q^{-1}\right)\B_+\\
=&\B_-\overline{s}_i q^{-\alpha_i^\vee}e_i(q)\B_+\\
=&\B_-\overline{s}_i\B_+.\qedhere
\end{align*}
\end{proof}

The last proposition shows that the mutation sequence in Theorem \ref{mgs} defines a biregular map that is a composition of left reflections:
\[
\left({_{j_n}r}\right)\circ \left({_{j_{n-1}}r}\right)\circ \dots\circ  \left({_{j_1}r}\right):\conf_b^e(\mathcal{B})\rightarrow\conf^{b^\circ}_e(\mathcal{B}),
\]
where $s_{j_1}s_{j_2}\dots s_{j_n}=b$.

Now it remains to prove Theorem \ref{mgs}. For simplicity, let us first investigate what happens when $b=\left(s_1, s_1, s_1, s_1\right)$.

It is not hard to see that the initial seed looks like the following.
\[
\begin{tikzpicture}
\node (1) at (0,0) [] {$a$};
\node (2) at (2,0) [] {$b$};
\node (3) at (4,0) [] {$c$};
\draw [<-] (1) -- (2);
\draw [<-] (2) -- (3);
\end{tikzpicture}
\]

Recall from Appendix \ref{app B} that we can use auxiliary frozen vertices (principal coefficients) to keep track of the coloring of the vertices. We adopt the convention of labeling the auxiliary frozen vertex connecting to vertex $a$ in the initial seed by $a'$. 
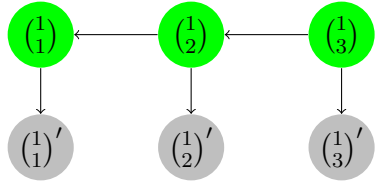
\[
\begin{tikzpicture}
\node [circle, fill=green, inner sep=0pt,minimum size=25pt] (1) at (0, 0) [] {$\binom{1}{1}$}; 
\node [circle, fill=green, inner sep=0pt,minimum size=25pt] (2) at (2, 0) [] {$\binom{1}{2}$};
\node [circle, fill=green, inner sep=0pt,minimum size=25pt] (3) at (4, 0) [] {$\binom{1}{3}$}; 
\node  [circle, fill=lightgray, inner sep=0pt,minimum size=25pt] (f1) at (0, -1.5) [] {$\binom{1}{1}'$}; 
\node [circle, fill=lightgray, inner sep=0pt,minimum size=25pt] (f2) at (2, -1.5) [] {$\binom{1}{2}'$};
\node [circle, fill=lightgray, inner sep=0pt,minimum size=25pt] (f3) at (4, -1.5) [] {$\binom{1}{3}'$}; 
\foreach \i in {1,2,3}
    {
    \draw [->] (\i) -- (f\i);
    }
\draw [<-] (1) -- (2);
\draw [<-] (2) -- (3);
\end{tikzpicture}
\]

The sequence of mutations for such a seed described in Theorem \ref{mgs} is then $\mu_a\circ\mu_b\circ\mu_a\circ\mu_c\circ\mu_b\circ \mu_a$, and the seed after each mutation looks like the following (going from left to right across each row, and go from the top row to the bottom row). 
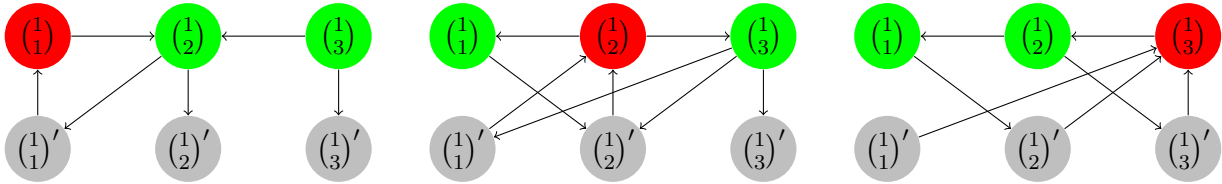
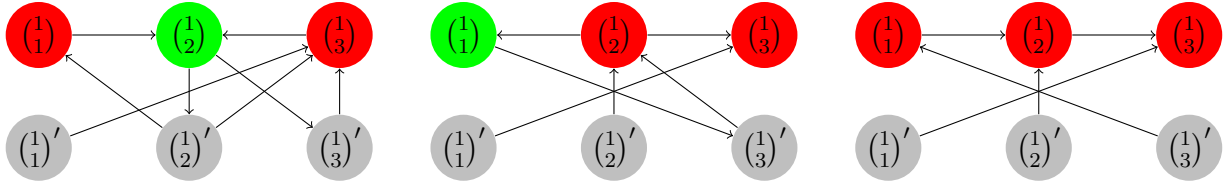
\[
\begin{tikzpicture}
\foreach \i in {1,2,3}
    {
    \node [circle, fill=lightgray, inner sep=0pt,minimum size=25pt] (f\i) at (2*\i,-1.5) [] {$\binom{1}{\i}'$};
    }
\node [circle, fill=red, inner sep=0pt,minimum size=25pt] (1) at (2,0) [] {$\binom{1}{1}$};
\node [circle, fill=green, inner sep=0pt,minimum size=25pt] (2) at (4,0) [] {$\binom{1}{2}$};
\node [circle, fill=green, inner sep=0pt,minimum size=25pt] (3) at (6,0) [] {$\binom{1}{3}$};
\draw [->] (1) -- (2);
\draw [->] (3) -- (2);
\draw [->] (3) -- (f3);
\draw [->] (2) -- (f1);
\draw [->] (2) -- (f2);
\draw [->] (f1) -- (1);
\end{tikzpicture}
\quad \quad 
\tikz{
\foreach \i in {1,2,3}
    {
    \node [circle, fill=lightgray, inner sep=0pt,minimum size=25pt] (f\i) at (2*\i,-1.5) [] {$\binom{1}{\i}'$};
    }
\node [circle, fill=green, inner sep=0pt,minimum size=25pt] (1) at (2,0) [] {$\binom{1}{1}$};
\node [circle, fill=red, inner sep=0pt,minimum size=25pt] (2) at (4,0) [] {$\binom{1}{2}$};
\node [circle, fill=green, inner sep=0pt,minimum size=25pt] (3) at (6,0) [] {$\binom{1}{3}$};
\draw [->] (2) -- (1);
\draw [->] (2) -- (3);
\draw [->] (3) -- (f2);
\draw [->] (f2) -- (2);
\draw [->] (f1) -- (2);
\draw [->] (3) -- (f3);
\draw [->] (1) -- (f2);
\draw [->] (3) -- (f1);
} 
\]
\[
\tikz{
\foreach \i in {1,2,3}
    {
    \node [circle, fill=lightgray, inner sep=0pt,minimum size=25pt] (f\i) at (2*\i,-1.5) [] {$\binom{1}{\i}'$};
    }
\node [circle, fill=green, inner sep=0pt,minimum size=25pt] (1) at (2,0) [] {$\binom{1}{1}$};
\node [circle, fill=green, inner sep=0pt,minimum size=25pt] (2) at (4,0) [] {$\binom{1}{2}$};
\node [circle, fill=red, inner sep=0pt,minimum size=25pt] (3) at (6,0) [] {$\binom{1}{3}$};
\draw [<-] (1) -- (2);
\draw [<-] (2) -- (3);
\draw [->] (f1) -- (3);
\draw [->] (f2) -- (3);
\draw [->] (f3) -- (3);
\draw [->] (2) -- (f3);
\draw [->] (1) -- (f2);
} \quad \quad
\tikz{
\foreach \i in {1,2,3}
    {
    \node [circle, fill=lightgray, inner sep=0pt,minimum size=25pt] (f\i) at (2*\i,-1.5) [] {$\binom{1}{\i}'$};
    }
\node [circle, fill=red, inner sep=0pt,minimum size=25pt] (1) at (2,0) [] {$\binom{1}{1}$};
\node [circle, fill=green, inner sep=0pt,minimum size=25pt] (2) at (4,0) [] {$\binom{1}{2}$};
\node [circle, fill=red, inner sep=0pt,minimum size=25pt] (3) at (6,0) [] {$\binom{1}{3}$};
\draw [->] (1) -- (2);
\draw [->] (3) -- (2);
\draw [->] (f1) -- (3);
\draw [->] (f2) -- (3);
\draw [->] (f3) -- (3);
\draw [->] (2) -- (f3);
\draw [->] (2) -- (f2);
\draw [->] (f2) -- (1);
}
\]
\[
\tikz{
\foreach \i in {1,2,3}
    {
    \node [circle, fill=lightgray, inner sep=0pt,minimum size=25pt] (f\i) at (2*\i,-1.5) [] {$\binom{1}{\i}'$};
    }
\node [circle, fill=green, inner sep=0pt,minimum size=25pt] (1) at (2,0) [] {$\binom{1}{1}$};
\node [circle, fill=red, inner sep=0pt,minimum size=25pt] (2) at (4,0) [] {$\binom{1}{2}$};
\node [circle, fill=red, inner sep=0pt,minimum size=25pt] (3) at (6,0) [] {$\binom{1}{3}$};
\draw [<-] (1) -- (2);
\draw [<-] (3) -- (2);
\draw [->] (f1) -- (3);
\draw [->] (f3) -- (2);
\draw [->] (f2) -- (2);
\draw [->] (1) -- (f3);
}\quad \quad 
\tikz{
\foreach \i in {1,2,3}
    {
    \node [circle, fill=lightgray, inner sep=0pt,minimum size=25pt] (f\i) at (2*\i,-1.5) [] {$\binom{1}{\i}'$};
    }
\node [circle, fill=red, inner sep=0pt,minimum size=25pt] (1) at (2,0) [] {$\binom{1}{1}$};
\node [circle, fill=red, inner sep=0pt,minimum size=25pt] (2) at (4,0) [] {$\binom{1}{2}$};
\node [circle, fill=red, inner sep=0pt,minimum size=25pt] (3) at (6,0) [] {$\binom{1}{3}$};
\draw [<-] (2) -- (1);
\draw [<-] (3) -- (2);
\draw [->] (f3) -- (1);
\draw [->] (f2) -- (2);
\draw [->] (f1) -- (3);
}
\]

The case of a longer positive braid $b=\left(s_1,s_1,\dots,s_1\right)$ can be done in a similar way. Before we go into cases with more than one letter, let us first make the following observation for this case.

\begin{prop}\label{observation} If $b=\left(s_1,s_1,\dots, s_1\right)$, then the following is true for the sequence of mutations given by Theorem \ref{mgs}
\begin{itemize}
    \item $\epsilon_{ab'}$ is $-1$, $0$, or $1$ for any two vertices $a$ and $b$.
    \item In each cycle, the first mutation we do changes the color of the first green vertex to red (counting from left to right), and then each mutation we do moves this red color to the right while restoring the green color for the vertices behind.
    \item At the end of each cycle, all red vertices are to the right of all green vertices, and the separation point is the last node with a simple root labelling (counting from left to right).
    \[
    \tikz{
    \node (0) at (0,0) [] {$1$};
    \node (1) at (2,0) [] {$\cdots$};
    \node (2) at (4,0) [] {$1$};
    \node (3) at (6,0) [] {$-1$};
    \node (4) at (8,0) [] {$\cdots$};
    \node (5) at (10,0) [] {$-1$};
    \draw [green] (0) -- (1) -- (2);
    \draw [red] (2) -- (3) -- (4) -- (5);
    }
    \]
    \item After going through the mutation sequence in Theorem \ref{mgs} completely, the final seed looks like the following.
    
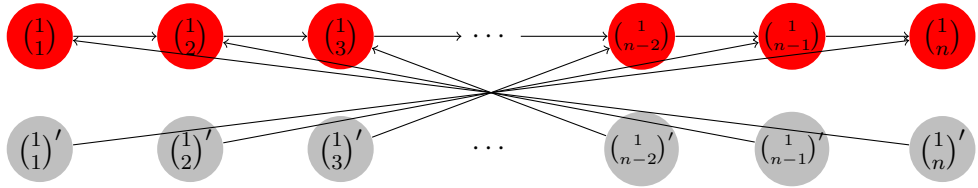
\[
    \tikz[scale=0.8]{
    \node (4) at (6,1.5) [] {$\cdots$};
    \node at (6,0) [] {$\cdots$};
    \foreach \i in {1,2,3}
        {
        \node [circle, fill=red, inner sep=0pt,minimum size=25pt] (\i) at (2*\i-2,1.5) [] {$\binom{1}{\i}$};
        \node [circle, fill=lightgray, inner sep=0pt,minimum size=25pt] (\i') at (2*\i-2,0) [] {$\binom{1}{\i}'$};
        }
    \node [circle, fill=red, inner sep=0pt,minimum size=25pt] (5) at (8,1.5) [] {\footnotesize{$\binom{1}{n-2}$}};
    \node [circle, fill=red, inner sep=0pt,minimum size=25pt] (6) at (10,1.5) [] {\footnotesize{$\binom{1}{n-1}$}};
    \node [circle, fill=red, inner sep=0pt,minimum size=25pt] (7) at (12,1.5) [] {$\binom{1}{n}$};
    \node [circle, fill=lightgray, inner sep=0pt,minimum size=25pt] (5') at (8,0) [] {\footnotesize{$\binom{1}{n-2}'$}};
    \node [circle, fill=lightgray, inner sep=0pt,minimum size=25pt] (6') at (10,0) [] {\footnotesize{$\binom{1}{n-1}'$}};
    \node [circle, fill=lightgray, inner sep=0pt,minimum size=25pt] (7') at (12,0) [] {$\binom{1}{n}'$};
    \foreach \i in {1,...,6}
        {
        \pgfmathtruncatemacro{\j}{\i+1};
        \draw [<-] (\j) -- (\i);
        }
    \draw [->] (1') -- (7);
    \draw [->] (2') -- (6);
    \draw [->] (3') -- (5);
    \draw [->] (5') -- (3);
    \draw [->] (6') -- (2);
    \draw [->] (7') -- (1);
    }
    \]
\end{itemize}
\end{prop}
\begin{proof} It follows from an induction on the number of vertices in the seed.
\end{proof}

Now let us turn to the more complicated cases with more than one letter. We start with the following proposition.

\begin{prop}\label{cycle} If a seed comes from a triangulation, then between the $i$th level and the $j$th level, the exchange matrix can always be depicted locally by oriented cycle using the arrow notation as below (plus some degenerate cases).
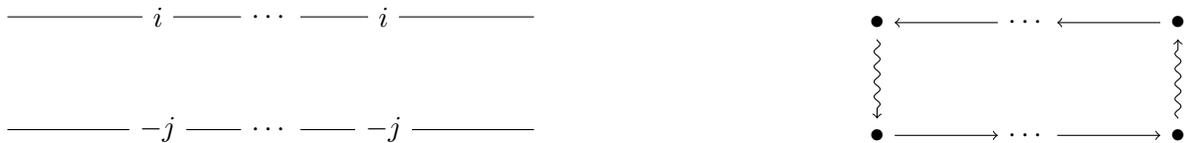
\[
\tikz[scale=0.5]{
\node (d1) at (2,0) [] {$-j$};
\node (d2) at (3.5,0) [] {$\cdots$};
\node (d3) at (5,0) [] {$-j$};
\node (u1) at (2,1.5) [] {$i$};
\node (u2) at (3.5,1.5) [] {$\cdots$};
\node (u3) at (5,1.5) [] {$i$};
\draw (0,0) -- (d1) -- (d2) -- (d3) -- (7,0);
\draw (0,1.5) -- (u1) -- (u2) -- (u3) -- (7,1.5);
} \quad \quad \quad \quad 
\tikz[scale=0.5]{
\node (0) at (0,1.5) [] {$\bullet$};
\node (1) at (2,1.5) [] {$\cdots$};
\node (2) at (4,1.5) [] {$\bullet$};
\node (3) at (4,0) [] {$\bullet$};
\node (4) at (2,0) [] {$\cdots$};
\node (5) at (0,0) [] {$\bullet$};
\draw [<-] (0) -- (1);
\draw [<-] (1) -- (2);
\draw [<-] (3) -- (4);
\draw [<-] (4) -- (5);
\draw [->,decorate,decoration={snake,amplitude=.4mm,segment length=2mm,post length=1mm}] (3) -- (2);
\draw [->,decorate,decoration={snake,amplitude=.4mm,segment length=2mm,post length=1mm}] (0) -- (5);
}
\]
\end{prop}
\begin{proof} This statement follows from seed amalgamation. 
\end{proof}

\begin{rmk} Please be aware that degenerate cases include situations like the following two.
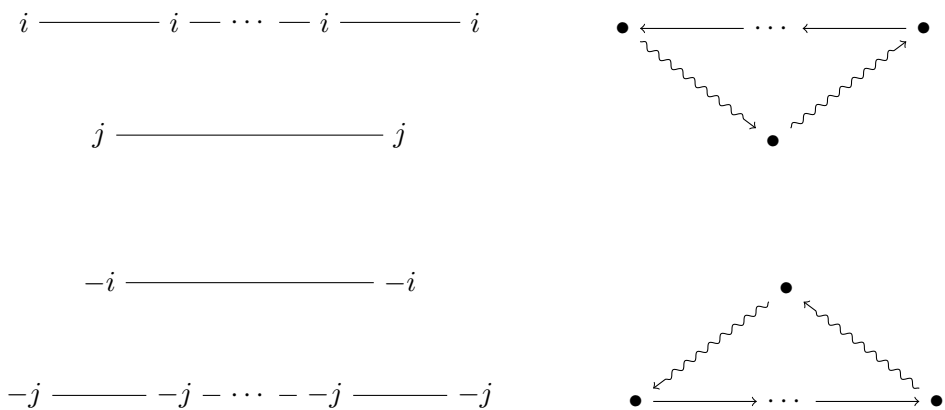
\[
\tikz[scale=0.7]{
\node (2) at (3,1.5) [] {$\cdots$};
\foreach \i in {0,1}
    {
    \node (\i) at (\i*2,1.5) [] {$i$};
    \pgfmathtruncatemacro{\j}{\i+3};
    \node (\j) at (\j*2-2,1.5) [] {$i$};
    }
\foreach \i in {0,...,3}
    {
    \pgfmathtruncatemacro{\j}{\i+1};
    \draw (\i) -- (\j);
    }
\node (b1) at (1,0) [] {$j$};
\node (b2) at (5,0) [] {$j$};
\draw (b1) -- (b2);
}
\quad \quad \quad \quad 
\tikz[scale=0.7]{
\node (u1) at (0,1.5) [] {$\bullet$};
\node (u2) at (2,1.5) [] {$\cdots$};
\node (u3) at (4,1.5) [] {$\bullet$};
\node (d) at (2,0) [] {$\bullet$};
\draw [<-] (u1) -- (u2);
\draw [<-] (u2) -- (u3);
\draw [->,decorate,decoration={snake,amplitude=.4mm,segment length=2mm,post length=1mm}] (d) -- (u3);
\draw [->,decorate,decoration={snake,amplitude=.4mm,segment length=2mm,post length=1mm}] (u1) -- (d);
}
\]
\vspace{0.5cm}
\[
\tikz[scale=0.7]{
\node (2) at (3,0) [] {$\cdots$};
\foreach \i in {0,1}
    {
    \node (\i) at (\i*2,0) [] {$-j$};
    \pgfmathtruncatemacro{\j}{\i+3};
    \node (\j) at (\j*2-2,0) [] {$-j$};
    }
\foreach \i in {0,...,3}
    {
    \pgfmathtruncatemacro{\j}{\i+1};
    \draw (\i) -- (\j);
    }
\node (b1) at (1,1.5) [] {$-i$};
\node (b2) at (5,1.5) [] {$-i$};
\draw (b1) -- (b2);
}
\quad \quad \quad \quad 
\tikz[scale=0.7]{
\node (u1) at (0,0) [] {$\bullet$};
\node (u2) at (2,0) [] {$\cdots$};
\node (u3) at (4,0) [] {$\bullet$};
\node (d) at (2,1.5) [] {$\bullet$};
\draw [->] (u1) -- (u2);
\draw [->] (u2) -- (u3);
\draw [->,decorate,decoration={snake,amplitude=.4mm,segment length=2mm,post length=1mm}] (d) -- (u1);
\draw [->,decorate,decoration={snake,amplitude=.4mm,segment length=2mm,post length=1mm}] (u3) -- (d);
}
\]
\end{rmk}

\begin{prop}\label{general} If one mutate according to Theorem \ref{mgs}, then the followings are true.
\begin{itemize}
    \item $\epsilon_{ab'}=0$ for vertices $a$ and $b$ on different levels.
    \item $\epsilon_{ab'}$ is either $-1$, $0$, or $1$ for vertices $a$ and $b$ on the same level.
    \item Within each cycle, the mutations only change colors of vertices within the level (say the $i$th) on which the mutations are taking places, and they change the colors the same way as the case of positive braids $\left(s_i,s_i,\dots, s_i\right)$ (Proposition \ref{observation}).
\end{itemize} 
\end{prop}
\begin{proof} We do an induction on the number of mutations. The base case is trivial. For the inductive step, let us first prove the claim ``$\epsilon_{ab'}=0$ for vertices $a$ and $b$ on different levels''. At the beginning of each cycle, the vertex we pick for step (1) always looks like the following.
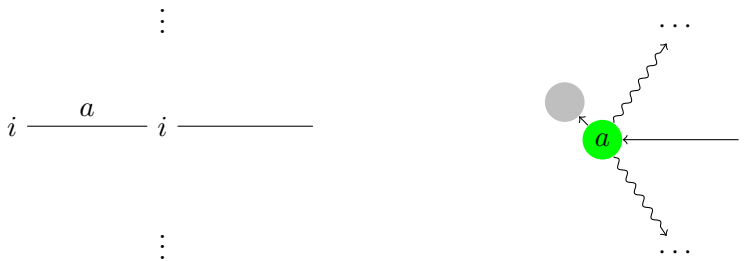
\[
\tikz[scale=0.7]{
\node at (-2,3) [] {$\vdots$};
\node (a2) at (-2,1.5) [] {$i$};
\node (a1) at (-4,1.5) [] {$i$};
\node at (-2,0) [] {$\vdots$};
\draw (0,1.5) -- (a2) -- node [above] {$a$} (a1);
} \quad \quad \quad \quad
\tikz[scale=0.7]{
\node [circle, fill=green, inner sep=0pt,minimum size=15pt] (a) at (0,0) {$a$};
\node [circle, fill=lightgray, inner sep=0pt,minimum size=15pt] (r) at (-1,0.5) {};
\node (l) at (2,0) [] {$\vdots$};
\node (u) at (1,1.5) [] {$\cdots$};
\node (d) at (1,-1.5) [] {$\cdots$};
\draw [->] (l) -- (a);
\draw [->,decorate,decoration={snake,amplitude=.4mm,segment length=2mm,post length=1mm}] (a) -- (u);
\draw [->,decorate,decoration={snake,amplitude=.4mm,segment length=2mm,post length=1mm}] (a) -- (d);
\draw [->] (a) -- (r);
}
\]
Note that the squiggly arrows going across different levels always go away from vertex $a$. Therefore when we mutate at $a$, no new arrow will be added between the auxiliary frozen vertices connected to $a$ and other vertices on other levels.

When we are in the middle of a cycle, we know from Proposition \ref{cycle} that the part of the seed we are mutating must look like the following.
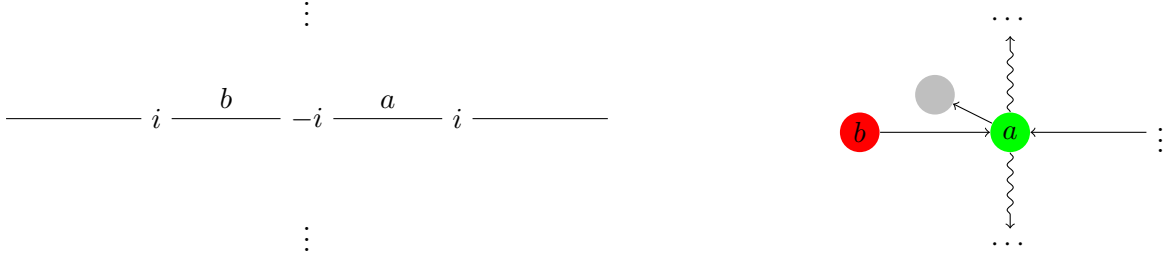
\[
\tikz[scale=0.7]{
\node (a1) at (-2,0) [] {$i$};
\node (a2) at (-4,0) [] {$-i$};
\node (a3) at (-6,0) [] {$i$};
\node at (-4,1.5) [] {$\vdots$};
\node at (-4,-1.5) [] {$\vdots$};
\draw (0,0) -- (a1) -- node [above] {$a$} (a2) -- node [above] {$b$} (a3) -- (-8,0);
} \quad \quad \quad \quad
\tikz[scale=0.7]{
\node [circle, fill=green, inner sep=0pt,minimum size=15pt] (a) at (-2,0) [] {$a$};
\node [circle, fill=red, inner sep=0pt,minimum size=15pt] (b) at (-4,0) [] {$b$};
\node [circle, fill=lightgray, inner sep=0pt,minimum size=15pt] (f) at (-3,0.5) [] {};
\node (l) at (0,0) [] {$\vdots$};
\node (u) at (-2,1.5) [] {$\cdots$};
\node (d) at (-2,-1.5) [] {$\cdots$};
\draw [->] (l) -- (a);
\draw [->] (b) -- (a);
\draw [->,decorate,decoration={snake,amplitude=.4mm,segment length=2mm,post length=1mm}] (a) -- (u);
\draw [->,decorate,decoration={snake,amplitude=.4mm,segment length=2mm,post length=1mm}] (a) -- (d);
\draw [->] (a) -- (f);
}
\]
Again, the squiggly arrows going across different levels always point away from vertex $a$. Therefore when we mutate at $a$, still no new arrow will be added between auxiliary frozen vertices connected to $a$ and other vertices on other levels. Therefore the claim that $\epsilon_{cd'}=0$ for vertices $c$ and $d$ on different levels remains true.

But then since the color of the vertex $c$ is recorded by $\epsilon_{cd'}$, we know that the color change must only occur at the $i$th level. The other two claims then follow immediately Proposition \ref{observation}.
\end{proof}

\textit{Proof of Theorem \ref{mgs}}. From our last proposition we know that on each horizontal level the colors of the vertices change the same way as the single letter case; therefore it is true that Theorem \ref{mgs} mutates at green vertices only and eventually all green vertices turn red. \qed

\subsection{Cluster Donaldson-Thomas Transformation on \texorpdfstring{$\conf^b_d(\mathcal{B})$}{}} Now we have obtained a maximal green sequence for the seed $\vec{i}_\tau$ associated to one triangulation of $\conf^e_b(\mathcal{B})$, so we are very close to getting a cluster Donaldson-Thomas transformation; all we need in addition is an appropriate seed isomorphism which permutes the rows of the $c$-matrix, making it into $-\mathrm{id}$. 

Recall from Proposition \ref{procedure} that the maximal green sequence given in Theorem \ref{mgs} can be thought of as a sequence of left reflections. So after completing the mutation sequence in Theorem \ref{mgs}, we get a triangulation for $\conf^{b^\circ}_e(\mathcal{B})$. This implies that we need some cluster isomorphism that maps $\conf^{b^\circ}_e(\mathcal{B})$ back to $\conf_b^e(\mathcal{B})$. One obvious choice is the transposition map $(-)^t$ from Section \ref{simplereflection}. We know from Proposition \ref{3.15} that the transposition map is indeed a cluster isomorphism, and combinatorially it is indeed a seed isomorphism that maps the seed after the maximal green sequence back to the initial seed with the property that the $c$-matrix is permuted into $-\mathrm{id}$.

Combining the transposition map with the maximal green sequence from last subsection we can now describe the cluster Donaldson-Thomas transformation on the double Bott-Samelson cell $\conf_b^e(\mathcal{B})$ geometrically, which is the composition
\[
\begin{tikzpicture}
\node (1) at (0,0) [] {$\DT:\conf_b^e(\mathcal{B})$};
\node (2) at (5,0) [] {$\conf^{b^\circ}_e(\mathcal{B})$};
\node (3) at (10,0) [] {$\conf_b^e(\mathcal{B}).$};
\draw [->] (1) -- node [above] {\footnotesize{$\begin{array}{c} \text{a sequence of} \\ \text{left reflections}\end{array}$}} (2);
\draw [->] (2) --  node [above] {\footnotesize{transposition}} (3);
\end{tikzpicture}
\]

\begin{exmp} Let us do another example to state in more details how the cluster Donaldson-Thomas transformation works. Take $b=s_1 s_2 s_3$ and the cluster Donaldson-Thomas transformation on $\conf_b^e(\mathcal{B})$ is given by the following operations.
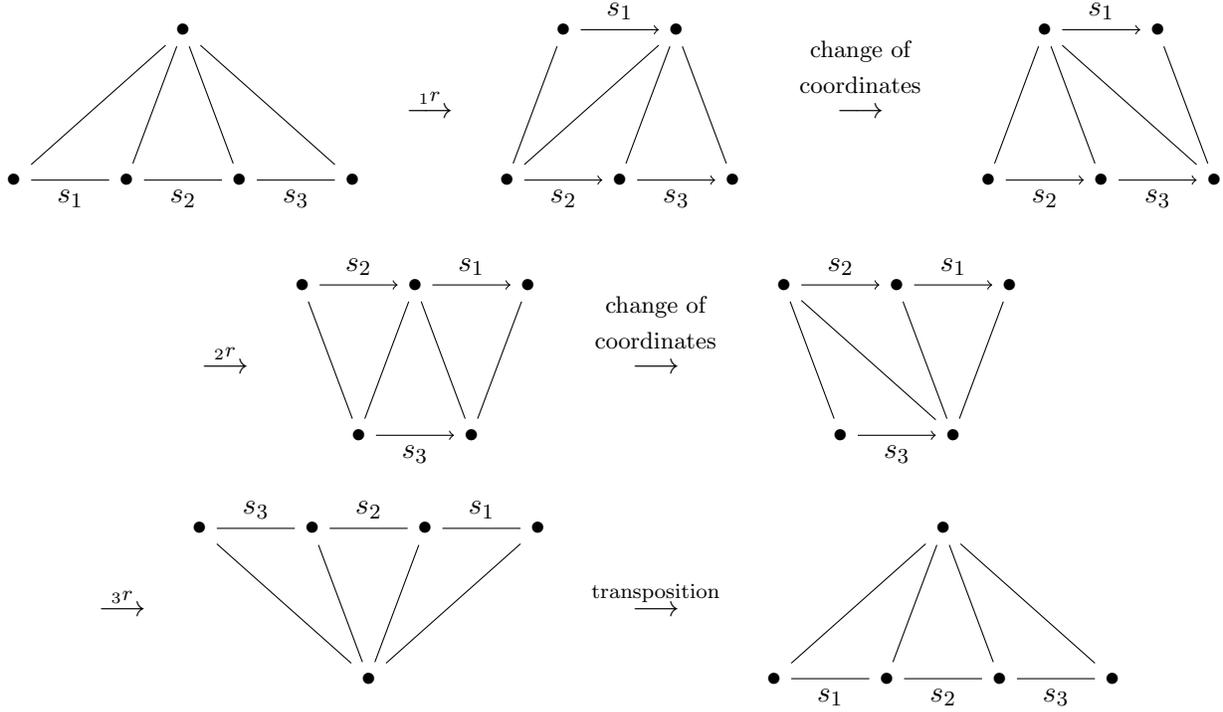
\[
\begin{tikzpicture}[baseline=5ex,scale=0.7]
\foreach \i in {0,...,3}
    {
    \node (\i) at (1.5*\i,0) [] {$\bullet$};
    }
\draw (0) -- node [below] {$s_1$} (1);
\draw (1) -- node [below] {$s_2$} (2);
\draw (2) -- node [below] {$s_3$} (3);
\node (u) at (2.25,2) [] {$\bullet$};
\draw (0) -- (u) -- (1);
\draw (2) -- (u) -- (3);
\end{tikzpicture}
\quad \overset{_1 r}{\longrightarrow} \quad 
\begin{tikzpicture}[baseline=5ex,scale=0.7]
\node (u0) at (-0.75,2) [] {$\bullet$};
\node (u1) at (0.75,2) [] {$\bullet$};
\node (d0) at (-1.5,0) [] {$\bullet$};
\node (d1) at (0,0) [] {$\bullet$};
\node (d2) at (1.5,0) [] {$\bullet$};
\draw [->] (d0) -- node [below] {$s_2$} (d1);
\draw [->] (d1) -- node [below] {$s_3$} (d2);
\draw [->] (u0) -- node [above] {$s_1$} (u1);
\draw (u0) -- (d0) -- (u1) -- (d1);
\draw (u1) -- (d2);
\end{tikzpicture}\quad  \overset{\begin{array}{c}\text{\footnotesize{change of}}\\ \text{\footnotesize{coordinates}}\end{array}}{\longrightarrow} \quad  
\begin{tikzpicture}[baseline=5ex,scale=0.7]
\node (u0) at (-0.75,2) [] {$\bullet$};
\node (u1) at (0.75,2) [] {$\bullet$};
\node (d0) at (-1.5,0) [] {$\bullet$};
\node (d1) at (0,0) [] {$\bullet$};
\node (d2) at (1.5,0) [] {$\bullet$};
\draw [->] (d0) -- node [below] {$s_2$} (d1);
\draw [->] (d1) -- node [below] {$s_3$} (d2);
\draw [->] (u0) -- node [above] {$s_1$} (u1);
\draw (d2) -- (u0) -- (d1);
\draw (u0) -- (d0);
\draw (u1) -- (d2);
\end{tikzpicture}
\]
\[
\overset{_2 r}{\longrightarrow} \quad \begin{tikzpicture}[baseline=5ex,scale=0.7]
\node (u0) at (-0.75,0) [] {$\bullet$};
\node (u1) at (0.75,0) [] {$\bullet$};
\node (d0) at (-1.5,2) [] {$\bullet$};
\node (d1) at (0,2) [] {$\bullet$};
\node (d2) at (1.5,2) [] {$\bullet$};
\draw [->] (d0) -- node [above] {$s_2$} (d1);
\draw [->] (d1)-- node [above] {$s_1$} (d2);
\draw [->] (u0) -- node [below] {$s_3$} (u1);
\draw (u0) -- (d1) -- (u1);
\draw (u0) -- (d0);
\draw (u1) -- (d2);
\end{tikzpicture}\quad \overset{\begin{array}{c}\text{\footnotesize{change of}}\\ \text{\footnotesize{coordinates}}\end{array}}{\longrightarrow} \quad
\begin{tikzpicture}[baseline=5ex,scale=0.7]
\node (u0) at (-0.75,0) [] {$\bullet$};
\node (u1) at (0.75,0) [] {$\bullet$};
\node (d0) at (-1.5,2) [] {$\bullet$};
\node (d1) at (0,2) [] {$\bullet$};
\node (d2) at (1.5,2) [] {$\bullet$};
\draw [->] (d0) -- node [above] {$s_2$} (d1);
\draw [->] (d1) -- node [above] {$s_1$} (d2);
\draw [->] (u0) -- node [below] {$s_3$} (u1);
\draw (d1) -- (u1) -- (d0);
\draw (u0) -- (d0);
\draw (u1) -- (d2);
\end{tikzpicture}
\]
\[
\overset{_3 r}{\longrightarrow} \quad 
\begin{tikzpicture}[baseline=5ex,scale=0.7]
\foreach \i in {0,...,3}
    {
    \node (\i) at (1.5*\i,2) [] {$\bullet$};
    }
\draw (0) -- node [above] {$s_3$} (1);
\draw (1) -- node [above] {$s_2$} (2);
\draw (2) -- node [above] {$s_1$} (3);
\node (u) at (2.25,0) [] {$\bullet$};
\draw (0) -- (u) -- (1);
\draw (2) -- (u) -- (3);
\end{tikzpicture} \quad \overset{\text{transposition}}{\longrightarrow} \quad 
\begin{tikzpicture}[baseline=5ex,scale=0.7]
\foreach \i in {0,...,3}
    {
    \node (\i) at (1.5*\i,0) [] {$\bullet$};
    }
\draw (0) -- node [below] {$s_1$} (1);
\draw (1) -- node [below] {$s_2$} (2);
\draw (2) -- node [below] {$s_3$} (3);
\node (u) at (2.25,2) [] {$\bullet$};
\draw (0) -- (u) -- (1);
\draw (2) -- (u) -- (3);
\end{tikzpicture}
\]
\end{exmp}

Let us now look at the general case for double Bott-Samelson cell $\conf^b_d(\mathcal{B})$. By a similar argument as Proposition \ref{procedure} one can prove that reflection maps $r^i$ are more than just biregular maps: they are cluster isomorphisms that preserve the cluster structures on the undecorated double Bott-Samelson cells. To be more precise, suppose $b=s_{i_1}s_{i_2}\dots s_{i_m}$. Consider the biregular map $\eta:=r^{i_1}\circ r^{i_2}\circ \dots\circ r^{i_m}$. This is not only a biregular map from $\conf^b_d(\mathcal{B})$ onto $\conf^e_{db^{\circ}}(\mathcal{B})$, but also a cluster isomorphism between $\up\left(\mathscr{X}^b_d\right)^\uf$ and $\up\left(\mathscr{X}^e_{db^\circ}\right)^\uf$, whose cluster Donaldson-Thomas transformation has been found in the previous sections.

Since $\eta$ is a cluster isomorphism, by a result of Goncharov and Shen (Theorem \ref{central}), we know that the cluster Donaldson-Thomas transformation on $\conf^b_d(\mathcal{B})$ can be obtained from that of $\conf^e_{db^\circ}(\mathcal{B})$ as
\[
\DT_{\conf^b_d(\mathcal{B})}=\eta^{-1}\circ \DT_{\conf^e_{db^\circ}(\mathcal{B})}\circ \eta.
\]

Let us draw a picture to describe what such composition really does. We begin with a triangulation associated to the pair $(b,d)$ where the first set of triangles are of the shape $\tikz[baseline=2ex]{\draw (0,1) -- (-0.5,0) --  (0.5,0) -- cycle;}$ (corresponding to letters of $d$) and the second set of triangles are of the shape $\tikz[baseline=2ex]{\draw (0,0) -- (-0.5,1) --  (0.5,1) -- cycle;}$ (corresponding to letters of $b$).
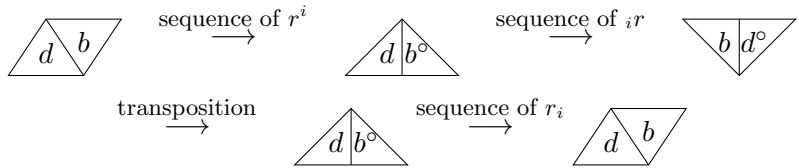
\[
\begin{tikzpicture}[baseline=2.5ex]
\draw (0,0) -- (1,0) -- (1.5,0.75) -- (0.5,0.75) -- cycle;
\draw (0.5,0.75) -- (1,0);
\node at (0.5,0.3) [] {$d$};
\node at (1, 0.45) [] {$b$};
\end{tikzpicture} \quad \overset{\text{\footnotesize{sequence of $r^i$}}}{\longrightarrow}\quad \begin{tikzpicture}[baseline=2.5ex]
\draw (0,0) -- (1.5,0) -- (0.75,0.75) -- cycle;
\draw (0.75,0.75) -- (0.75,0);
\node at (0.55,0.3) [] {$d$};
\node at (0.95,0.3) [] {$b^\circ$};
\end{tikzpicture}\quad 
\overset{\text{\footnotesize{sequence of $_ir$}} }{\longrightarrow} \quad
\begin{tikzpicture}[baseline=2.5ex]
\draw (0,0.75) -- (1.5,0.75) -- (0.75,0) -- cycle;
\draw (0.75,0.75) -- (0.75,0);
\node at (0.55,0.45) [] {$b$};
\node at (0.95,0.45) [] {$d^\circ$};
\end{tikzpicture}
\]
\[
\overset{\text{\footnotesize{transposition}}}{\longrightarrow}\quad 
\begin{tikzpicture}[baseline=2.5ex]
\draw (0,0) -- (1.5,0) -- (0.75,0.75) -- cycle;
\draw (0.75,0.75) -- (0.75,0);
\node at (0.55,0.3) [] {$d$};
\node at (0.95,0.3) [] {$b^\circ$};
\end{tikzpicture}
\overset{\text{\footnotesize{sequence of $r_i$}} }{\longrightarrow}\begin{tikzpicture}[baseline=2.5ex]
\draw (0,0) -- (1,0) -- (1.5,0.75) -- (0.5,0.75) -- cycle;
\draw (0.5,0.75) -- (1,0);
\node at (0.5,0.3) [] {$d$};
\node at (1, 0.45) [] {$b$};
\end{tikzpicture}
\]

There are some further simplifications that we can perform: first, we can use Proposition \ref{trans&ref} to turn each $r_i \circ (-)^t$ into $(-)^t\circ {^ir}$; then we can cancel the left reflections $_ir$ and $^ir$, which yields the following simplified version.
\[
\begin{tikzpicture}[baseline=2.5ex]
\draw (0,0) -- (1,0) -- (1.5,0.75) -- (0.5,0.75) -- cycle;
\draw (0.5,0.75) -- (1,0);
\node at (0.5,0.3) [] {$d$};
\node at (1, 0.45) [] {$b$};
\end{tikzpicture} \quad \overset{\begin{array}{c} \text{\footnotesize{sequence of $r^i$}} \\  \text{\footnotesize{on the $b$ part}} \\ \text{\footnotesize{and sequence of $_ir$}} \\
\text{\footnotesize{on the $d$ part}}\end{array}}{\longrightarrow}\quad \begin{tikzpicture}[baseline=2.5ex]
\draw (0,0.75) -- (1,0.75) -- (1.5,0) -- (0.5,0) -- cycle;
\draw (0.5,0) -- (1,0.75);
\node at (0.5,0.45) [] {$d^\circ$};
\node at (1,0.3) [] {$b^\circ$};
\end{tikzpicture}
\]
\[
= \quad 
\begin{tikzpicture}[baseline=2.5ex]
\draw (0,0) -- (1,0) -- (1.5,0.75) -- (0.5,0.75) -- cycle;
\draw (0.5,0.75) -- (1,0);
\node at (0.5,0.3) [] {$b^\circ$};
\node at (1, 0.45) [] {$d^\circ$};
\end{tikzpicture} \quad \overset{\text{\footnotesize{transposition}}}{\longrightarrow}
\quad \begin{tikzpicture}[baseline=2.5ex]
\draw (0,0) -- (1,0) -- (1.5,0.75) -- (0.5,0.75) -- cycle;
\draw (0.5,0.75) -- (1,0);
\node at (0.5,0.3) [] {$d$};
\node at (1, 0.45) [] {$b$};
\end{tikzpicture}
\]

In conclusion, we have proved the following theorem.

\begin{thm}\label{4.9} The cluster Donaldson-Thomas transformation on $\conf^b_d(\mathcal{B})$ is the composition of a series of reflection maps on both sides and the transposition map as described above.
\end{thm}

\begin{rmk} When we compute $\DT_{\conf^b_d(\mathcal{B})}=\eta^{-1}\circ \DT_{\conf^e_{db^\circ}(\mathcal{B})}\circ \eta$ we chose to work with the cluster transformation $\eta:\conf^b_d(\mathcal{B})\rightarrow\conf^e_{db^\circ}(\mathcal{B})$ defined as a composition of reflections $r^i$; alternatively one can work with the cluster transformation $\eta':\conf^b_d(\mathcal{B})\rightarrow\conf^e_{b^\circ d}(\mathcal{B})$ defined as a composition of reflections $^ir$. The resulting cluster Donaldson-Thomas transformation $\DT_{\conf^b_d(\mathcal{B})}$ would be the same due to the uniqueness of cluster Donaldson-Thomas transformation. Moreover, we could have studied $\DT_{\conf^b_e(\mathcal{B})}$ instead of $\DT_{\conf^e_b(\mathcal{B})}$ and use it to compute the general case; the result would also be the same.
\end{rmk}

The existence of the cluster Donaldson-Thomas transformation is a part of a sufficient condition (Theorem \ref{sufficient condition}) of the Fock-Goncharov cluster duality conjecture (Conjecture \ref{duality conjecture}). In the cluster ensemble $\left(\conf^b_d\left(\mathcal{A}_\sc\right), \conf^b_d\left(\mathcal{A}_\ad\right)\right)$, we already proved the surjectivity of the map $p:\conf^b_d\left(\mathcal{A}_\sc\right)\rightarrow \conf^b_d(\mathcal{B})$ from the proof of Theorem \ref{3.42}, and we just constructed the cluster Donaldson-Thomas transformation explicitly on $\conf^b_d(\mathcal{B})$. Therefore we obtain the following result.

\begin{thm}\label{4.11} The Fock-Goncharov cluster duality conjecture \ref{duality conjecture} holds for the cluster ensemble $\left(\conf^b_d\left(\mathcal{A}_\sc\right), \conf^b_d\left(\mathcal{A}_\ad\right)\right)$.
\end{thm}

\subsection{Reflection Maps as Quasi-Cluster Transformations}

So far we are considering the Donaldson-Thomas Transformation on the undecorated double Bott-Samelson cell $\conf^b_d(\mathcal{B})$, which can be seen as the unfrozen part of a cluster Poisson variety $\mathscr{X}^b_d$ (up to codimension 2). In this section we would like to investigate the lift of the Donaldson-Thomas transformation to the decorated double Bott-Samelson cell $\conf^b_d(\mathcal{A})$. Since the Donaldson-Thomas transformation on $\conf^b_d(\mathcal{B})$ can be broken down into a sequence of reflection maps followed by a transposition, we can define the lift of the Donaldson-Thomas transformation to $\conf^b_d(\mathcal{B})$ to $\conf^b_d(\mathcal{A})$ to be the same composition of reflection maps and transposition on $\conf^b_d(\mathcal{A})$ according to Section \ref{simplereflection}.

The next question is whether such lift is a cluster transformation. Unfortunately the answer is no: in general the reflection maps between decorated double Bott-Samelson cells are not cluster isomorphisms because the underlying seeds are not isomorphic. However, the following weaker statement is true (see Definition \ref{quasi cluster trans} for the definition of quasi-cluster transformation).

\begin{prop} Reflection maps are quasi-cluster transformations.
\end{prop}
\begin{proof} Since left and right reflections maps are intertwined by a transposition map, which is a cluster isomorphism even on the decorated double Bott-Samelson cell level, it suffices to only show the claim for right reflection maps. But then since $r^i$ and $r_i$ are inverses of each other and inverse of a quasi-cluster transformation is also a quasi-cluster transformation, it suffices to only consider $r^i$.

Let us first consider the adjoint case $r^i:\conf^{bs_i}_d\left(\mathcal{A}_\ad\right)\rightarrow \conf^b_{ds_i}\left(\mathcal{A}_\ad\right)$. Suppose we start with the following cluster coordinate chart. Then $r^i$ should produce the last configuration according to its definition.
\begin{align*}
&\begin{tikzpicture}[baseline=2ex]
    \node (u0) at (1,1.5) [] {$\cdots$};
    \node (d0) at (1,0) [] {$\cdots$};
    \node (u1) at (2,1.5) [] {$\B_+$};
    \node (d1) at (2,0) [] {$\U_-t^{-1}$};
    \node (u2) at (4,1.5) [] {$e_{-i}(p)\B_+$};
    \draw [->] (u0) -- (u1);
    \draw [->] (u1) --  node [above] {$s_i$} (u2);
    \draw [->] (d0) -- (d1);
    \draw (u1) -- (d1) -- (u2);
\end{tikzpicture} \quad \quad \overset{\begin{array}{c} \text{\footnotesize{reflect the}}\\\text{\footnotesize{underlying}} \\ \text{\footnotesize{flag}}\end{array}}{\longrightarrow} \quad \quad \begin{tikzpicture}[baseline=2ex]
    \node (u0) at (1,1.5) [] {$\cdots$};
    \node (d0) at (1,0) [] {$\cdots$};
    \node (u1) at (2,1.5) [] {$\B_+$};
    \node (d1) at (2,0) [] {$\U_-t^{-1}$};
    \node (d2) at (4,0) [] {$e_{i}\left(p^{-1}\right)\B_-$};
    \draw [->] (u0) -- (u1);
    \draw [->] (d1) --  node [below] {$s_i$} (d2);
    \draw [->] (d0) -- (d1);
    \draw (d1) -- (u1) -- (d2);
\end{tikzpicture}
\\
& \overset{\begin{array}{c} \text{\footnotesize{find}} \\ \text{\footnotesize{compatible}} \\ \text{\footnotesize{decoration}}\end{array}}{\longrightarrow} \quad 
 \begin{tikzpicture}[baseline=2ex]
    \node (u0) at (1,1.5) [] {$\cdots$};
    \node (d0) at (1,0) [] {$\cdots$};
    \node (u1) at (2,1.5) [] {$\B_+$};
    \node (d1) at (2,0) [] {$\U_-t^{-1}$};
    \node (d2) at (6,0) [] {$\U_-t^{-1}\left(pt^{\alpha_i}\right)^{\alpha_i^\vee}e_{i}\left(-p^{-1}\right)$};
    \draw [->] (u0) -- (u1);
    \draw [->] (d1) --  node [below] {$s_i$} (d2);
    \draw [->] (d0) -- (d1);
    \draw (d1) -- (u1) -- (d2);
\end{tikzpicture}\\
&\overset{\begin{array}{c} \text{\footnotesize{forget the}} \\ \text{\footnotesize{decoration}} \\ \text{\footnotesize{on $\U_-$}}\end{array}}{\longrightarrow} \quad \quad  \begin{tikzpicture}[baseline=2ex]
    \node (u0) at (1,1.5) [] {$\cdots$};
    \node (d0) at (1,0) [] {$\cdots$};
    \node (u1) at (2,1.5) [] {$\B_+$};
    \node (d1) at (2,0) [] {$\B_-$};
    \node (d2) at (6,0) [] {$\U_-t^{-1}\left(pt^{\alpha_i}\right)^{\alpha_i^\vee}e_{i}\left(-p^{-1}\right)$};
    \draw [->] (u0) -- (u1);
    \draw [->] (d1) --  node [below] {$s_i$} (d2);
    \draw [->] (d0) -- (d1);
    \draw (d1) -- (u1) -- (d2);
\end{tikzpicture}
\end{align*}
According to the construction of the Lusztig factorization coordinates and the cluster Poisson coordinates, the last picture above indicates that the Lusztig factorization corresponding to the given triangulation before $r^i$ ends at $\dots e_{-i}(p)t$, which corresponds to the following cluster Poisson coordinates; note that the dots represent the factors that possibly come from the other ends of the strings (note that the squiggly arrow only denotes the sign of the exchange matrix).
\[
\begin{tikzpicture}
    \node (0) at (2,1.5) [] {$-i$};
    \draw (0,1.5) node [left] {$i$th level} -- node (1) [above] {$\dots p^{-1}$} (0) --  node (2) [above] {$pt^{\alpha_i}$} (4,1.5);
    \draw (0,0) node [left] {$j$th level} -- node (3) [above] {$\dots t^{\alpha_j}$} (4,0);
    \draw [red, ->] (1) to [bend left] (2);
    \draw [red, ->,decorate,decoration={snake,amplitude=.4mm,segment length=2mm,post length=1mm}] (2) -- (3);
    \draw [red, ->,decorate,decoration={snake,amplitude=.4mm,segment length=2mm,post length=1mm}] (3) -- (1);
\end{tikzpicture}
\]
After applying $r^i$, the Lusztig factorization ends at $\dots e_i\left(p^{-1}\right)t\left(pt^{\alpha_i}\right)^{-\alpha_i^\vee}$, which corresponds to the following cluster Poisson coordinates.
\[
\begin{tikzpicture}
    \node (0) at (2,1.5) [] {$i$};
    \draw (0,1.5) node [left] {$i$th level} -- node (1) [above] {$\dots p^{-1}$} (0) --  node (2) [above] {$\left(pt^{\alpha_i}\right)^{-1}$} (4,1.5);
    \draw (0,0) node [left] {$j$th level} -- node (3) [above] {$\dots t^{\alpha_j}\left(pt^{\alpha_i}\right)^{-\C_{ij}}$} (4,0);
    \draw [red, <-] (1) to [bend left] (2);
    \draw [red, <-,decorate,decoration={snake,amplitude=.4mm,segment length=2mm,post length=1mm}] (2) -- (3);
    \draw [red, <-,decorate,decoration={snake,amplitude=.4mm,segment length=2mm,post length=1mm}] (3) -- (1);
\end{tikzpicture}
\]
By comparing the two, we see that the pull-back of most cluster Poisson coordinates under $r^i$ remains unchanged; the only cluster Poisson coordinates that are different are the ones associated to the open strings on the right side of the string diagram. In particular, the last frozen variable $X_i$ on the $i$th level gets inverted:
\[
r^{i*}\left(X'_i\right)=X_i^{-1},
\]
and the last frozen variable $X_j$ on any other level (say $j$th) are rescaled:
\[
r^{i*}\left(X'_j\right)=X_jX_i^{-\C_{ij}}.
\]
Since $\left.r^i\right|_{\conf^{bs_i}_d(\mathcal{B})}$ is already known to be a cluster transformation, to show that $r^i:\conf^{bs_i}_d\left(\mathcal{A}_\ad\right)\rightarrow \conf^b_{ds_i}\left(\mathcal{A}_\ad\right)$ is a quasi-cluster transformation, we only need to show that $r^i$ is a Poisson map, which boils down to showing 
\[
\left\{ r^{i*}X'_a, r^{i*}X'_b\right\}=\left\{ X'_a, X'_b\right\}
\] 
for any two strings $a$ and $b$. From the string diagram of $\conf^{bs_i}_d\left(\mathcal{A}_\ad\right)$ we see that $X_i$ only has non-vanishing Poisson bracket with the variable associated to the closed string to the left of it (call it $X_c$) and the frozen variables $X_j$ on the other levels. Therefore any change to Poisson brackets of cluster Poisson variables can only occur on $\left\{X'_c,X'_i\right\}$, $\left\{X'_c,X'_j\right\}$, and $\left\{X'_i,X'_j\right\}$. Since the Poisson structure on a cluster Poisson variety is known to be log canonical, we apply log to simplify the computation:
\begin{align*}
\left\{\log \left(r^{i*}X'_c\right),\log \left(r^{i*}X'_i\right)\right\}&=\left\{\log X_c,\log X_i^{-1}\right\}=-\left\{\log X_c,\log X_i\right\}=-\frac{1}{\D_i}\\
&=\left\{\log X'_c,\log X'_i\right\}\\
\left\{\log \left(r^{i*}X'_c\right),\log \left(r^{i*}X'_j\right)\right\}&=\left\{\log X_c,\log \left(X_jX_i^{-\C_{ij}}\right)\right\}\\
&=\left\{\log X_c,\log X_j\right\}-\C_{ij}\left\{\log X_c,\log X_i\right\}\\
&=\frac{\C_{ji}}{2\D_j}-\frac{\C_{ij}}{\D_i}+\E=\frac{\C_{ji}}{2\D_j}-\frac{\C_{ji}}{\D_j}+\E=-\frac{\C_{ji}}{2\D_j}+\E\\
&=\left\{\log X'_c,\log X'_j\right\}\\
\left\{\log \left(r^{i*}X'_i\right),\log \left(r^{i*}X'_j\right)\right\}&=\left\{\log X_i^{-1}, \log \left(X_jX_i^{-\C_{ij}}\right)\right\}=-\left\{\log X_i,\log X_j\right\}\\
&=-\left(-\frac{\C_{ji}}{\D_j}\right)=\frac{\C_{ji}}{\D_j}=\left\{\log X'_i,\log X'_j\right\}.
\end{align*}
Here $\E$ accounts for the contribution from the earlier parts of the diagram, and they are invariant under right reflections. This finishes proving that the reflection map $r^i$ is a quasi-cluster transformation in the adjoint case.

Next let us consider the simply connected case. By definition, the action of a quasi-cluster transformation on a cluster $\mathrm{K}_2$ variety is induced from its action on a cluster Poisson variety. If we express the action of $r^i$ on our chosen seed above using the character lattice $N$ of $T_\mathscr{X}$, it can be expressed as the following matrix, which consists of an identity matrix for the most part except for the submatrix with rows and columns corresponding to the frozen variables on the right.
\[
\begin{pmatrix}
\mathrm{id} & 0 & 0 & \cdots & 0 \\
0 & -1 & 0 &  \cdots & 0 \\
0 & -\C_{ij} & 1 & \cdots & 0\\
\vdots & \vdots & \vdots & \ddots & 0\\
0 & -\C_{ik} & 0 & \cdots & 1
\end{pmatrix}
\begin{matrix} \\ \text{$i$th} \\ \text{$j$th} \\ \vdots \\ \text{$k$th}\end{matrix}
\]
By Lemma \ref{lem A.32}, the induced quasi-cluster transformation action on the cluster $\mathrm{K}_2$ variety should then be acting on the character lattice $M$ of $T_\mathscr{A}$ by the matrix that is the transpose inverse of the one above with $\C_{ij}$ replaced by $\C_{ji}$, which is
\[
\begin{pmatrix}
\mathrm{id} & 0 & 0 & \cdots & 0 \\
0 & -1 & -\C_{ji} &  \cdots & -\C_{ki} \\
0 & 0 & 1 & \cdots & 0\\
\vdots & \vdots & \vdots & \ddots & 0\\
0 & 0 & 0 & \cdots & 1
\end{pmatrix}
\begin{matrix} \\ \text{$i$th} \\ \text{$j$th} \\ \vdots \\ \text{$k$th}\end{matrix}
\]
Therefore to show that $r^i:\conf^{bs_i}_d\left(\mathcal{A}_\sc\right)\rightarrow \conf^{b}_{ds_i}\left(\mathcal{A}_\sc\right)$ is a quasi-cluster transformation, it suffices to show that all cluster $\mathrm{K}_2$ coordinates remain the same except the last frozen variable $A'_i$ on the $i$th level, which should transform by
\[
r^{i*}\left(A'_i\right)=A_i^{-1}\prod_{j\neq i} A_j^{-\C_{ji}}.
\]

So let us verify this transformation formula geometrically. By a similar computation as in the adjoint case, the reflection map turns the following configuration on the left to the one on the right in the simply connected case.
\[
\begin{tikzpicture}[baseline=4ex,scale=0.75]
    \node (u0) at (0.5,1.5) [] {$\cdots$};
    \node (d0) at (0.5,0) [] {$\cdots$};
    \node (u1) at (2,1.5) [] {$\U_+$};
    \node (d1) at (2,0) [] {$\U_-t^{-1}$};
    \node (u2) at (5,1.5) [] {$e_{-i}(p)p^{-\alpha_i^\vee}\U_+$};
    \draw [->] (u0) -- (u1);
    \draw [->] (u1) --  node [above] {$s_i$} (u2);
    \draw [->] (d0) -- (d1);
    \draw (u1) -- (d1) -- (u2);
\end{tikzpicture}  \overset{r^i}{\longrightarrow} 
\begin{tikzpicture}[baseline=4ex,scale=0.75]
    \node (u0) at (0.5,1.5) [] {$\cdots$};
    \node (d0) at (0.5,0) [] {$\cdots$};
    \node (u1) at (2,1.5) [] {$\U_+$};
    \node (d1) at (2,0) [] {$\U_-t^{-1}$};
    \node (d2) at (6,0) [] {$\U_-t^{-1}\left(pt^{\alpha_i}\right)^{\alpha_i^\vee}e_{i}\left(-p^{-1}\right)$};
    \draw [->] (u0) -- (u1);
    \draw [->] (d1) --  node [below] {$s_i$} (d2);
    \draw [->] (d0) -- (d1);
    \draw (d1) -- (u1) -- (d2);
\end{tikzpicture}
\]
Since the cluster $\mathrm{K}_2$ coordinates are computed by evaluating the $\left(\U_-,\U_+\right)$-invariant function $\Delta_{\omega_j}$ and $\inprod{\alpha_i^\vee}{\omega_j}=\delta_{ij}$, we see that the only cluster $\mathrm{K}_2$ coordinate that changes under $r^i$ is $A'_i$:
\[
r^{i*}\left(A'_i\right)=\Delta_{\omega_i}\left(t^{-1}\left(pt^{\alpha_i}\right)^{\alpha_i^\vee}\right)=pt^{\alpha_i-\omega_i}.
\]
On the other hand, by computation we get
\[
\resizebox{12cm}{!}{\begin{math}
A_i^{-1}\prod_{j\neq i}A_j^{-\C_{ji}}=\left(p^{-\alpha_i^\vee}t^{-1}\right)^{-\omega_i}\prod_{j\neq i}\left(\left(p^{-\alpha_i^\vee}t^{-1}\right)^{\omega_j}\right)^{-\C_{ji}}=pt^{\omega_i}t^{\sum_{j\neq i}\C_{ji}\omega_j}=pt^{\alpha_i-\omega_i},\end{math}}
\]
which agrees with $r^{i*}\left(A'_i\right)$, as predicted by the transformation formula.
\end{proof}

Since reflection maps are quasi-cluster transformations on the decorated double Bott-Samelson cells and the transposition map is a cluster isomorphism, it follows that the Donaldson-Thomas transformation on decorated double Bott-Samelson cells, as a composition of such maps, is also a quasi-cluster transformation.

\begin{cor}\label{4.13} The Donaldson-Thomas transformation on any decorated double Bott-Samelson cell is a quasi-cluster transformation. It then follows that the Donaldson-Thomas transformation on $\conf^b_d\left(\mathcal{A}_\ad\right)$ is a Poisson map.
\end{cor}

Using the fact that reflection maps are quasi-cluster transformations, we can also prove the following sufficient condition on the equality between upper cluster algebras and cluster algebras (Theorem \ref{upper=ordinary}).

\begin{thm}\label{up=ord} The upper cluster algebra $\mathcal{O}\left(\conf^b_d\left(\mathcal{A}_\sc\right)\right)$ coincides with its cluster algebra.
\end{thm}
\begin{proof} By Theorem \ref{3.42}, it suffices to show that $\mathcal{O}\left(\conf^b_d\left(\mathcal{A}_\sc\right)\right)$ is generated by cluster variables (with invertible frozen cluster variables). From Lemma \ref{lem A.32} we know that a quasi-cluster transformation modifies unfrozen cluster $\mathrm{K}_2$ variables by Laurent monomials in the frozen cluster $\mathrm{K}_2$ variables, and maps the frozen cluster $\mathrm{K}_2$ variables to Laurent monomials in the frozen cluster $\mathrm{K}_2$ variables. Therefore we can apply a sequence of reflection maps to reduce the question to proving $\mathcal{O}\left(\conf^e_d\left(\mathcal{A}_\sc\right)\right)$ is generated by cluster variables for any positive braid $d$. 
\[
\begin{tikzpicture}[scale=.8]
\node (u) at (0,2) [] {$\U_+$};
\node (d1) at (-3,0) [] {$\B_0$};
\node (d2) at (-1,0) [] {$\B_1$};
\node (d3) at (1,0) [] {$\dots$};
\node (d4) at (3,0) [] {$\U_-t$};
\draw [->] (d1) -- node [below] {$s_{i_1}$} (d2);
\draw [->] (d2) -- node [below] {$s_{i_2}$} (d3);
\draw [->] (d3) -- node [below] {$s_{i_n}$} (d4);
\draw (d1) -- (u) -- (d4);
\end{tikzpicture}
\]
Recall from Theorem \ref{affineness of conf} that any point in $\conf^b_d\left(\mathcal{A}_\sc\right)$ can be represented by a unique representative as shown above, and $\conf^b_d\left(\mathcal{A}_\sc\right)$ is the non-vanishing locus of a single function $f$ in $\T\times \mathbb{A}^n$, where $\T$ captures the freedom of the maximal torus element $t$ and the coordinates of $\mathbb{A}^n$ are parametrizing the relative positions between each pair of adjacent flags in the bottom chain. Note that by definition, the function $f$ is the product of the frozen cluster $\mathrm{K}_2$ variables on the left, and the torus factor $\T$ is parametrized by the frozen cluster $\mathrm{K}_2$ variables on the right. By assumption, the frozen cluster $\mathrm{K}_2$ variables are invertible in the cluster algebra, too. Therefore it remains to show that the coordinates of the affine space factor $\mathbb{A}^n$ are polynomials in unfrozen cluster $\mathrm{K}_2$ variables and Laurent polynomials in frozen cluster $\mathrm{K}_2$ variables.

First, from the proof of Lemma \ref{irreducible} we know that the first affine parameter of the pair $\xymatrix{\B_0\ar[r]^{s_{i_1}} & \B_1}$ is a monomial in unfrozen cluster $\mathrm{K}_2$ variables and a Laurent monomial in frozen cluster $\mathrm{K}_2$ variables. Next we can apply the quasi-cluster transformation $r^{i_1}\circ {_{i_1}r}$ to move first letter $s_{i_1}$ to the end of the $d$-chain. Then Lemma \ref{irreducible} implies again that the affine parameter of the pair $\xymatrix{\B_1\ar[r]^{s_{i_2}} & \B_2}$ is a monomial in unfrozen cluster $\mathrm{K}_2$ variables (not necessarily of the initial seed) and a Laurent monomial in frozen cluster $\mathrm{K}_2$ variables. The same argument can be applied to the rest of the affine parameters and we can conclude that $\mathcal{O}\left(\conf^e_d\left(\mathcal{A}_\sc\right)\right)$ is generated by cluster $\mathrm{K}_2$ variables (with invertible frozen variables).
\end{proof}

\section{Periodicity of DT transformations and Zamolodchikov's Periodicity Conjecture} \label{PeriodicityConje}

In this section we prove the periodicity of $\DT$ transformations for a family of double Bott-Samelson cells. We give a new geometric proof of Zamolodchikov's periodicity conjecture. 

Below $\G$ is a semisimple algebraic group. The longest Weyl group element $w_0$ can be uniquely lifted to a positive braid. Its square $\Omega:=w_0^2$ is a central element in the braid group $\Br$.  Let $b=s_{i_1}s_{i_2}\ldots s_{i_n}$ be a positive braid. Recall that $b^\circ= s_{i_n}\ldots s_{i_2}s_{i_1}$.

\begin{thm}\label{5.1} Let $b$ and $d$ be positive braids such that $\left( db^\circ\right)^m=\Omega^{n}$. The cluster Donaldson-Thomas transformation of ${\conf^b_d(\mathcal{B})}$ is of a finite order dividing $2(m+n)$.
\end{thm}

\begin{exmp} Let $\G=\PGL_7$. Its Weyl group is the symmetric group $S_7$. Take the element $u=(1~2~3~7~6~5)\in S_7$ in cycle notation.
The length of $u$ is 7. Its lift to  $\Br$ satisfies $u^3 =w_0$. Let $d=u^n \in \Br$. Then $d^6= \Omega^{n}$. The order of the DT transformation of ${\conf^e_{d}(\mathcal{B})}$ divides $12+2n$.
\end{exmp}

\begin{proof} The cluster Donaldson-Thomas transformations on $\conf^b_d(\mathcal{B})$ and $\conf^e_{ db^\circ}(\mathcal{B})$ are intertwined by a sequence of reflection maps on the right, hence they share the same order.  It suffices to consider the cases $\conf^e_b(\mathcal{B})$ with $b^m=\Omega^{n}$. 
Let $b=s_{i_1}s_{i_2}\dots s_{i_l}$. We start with a configuration
\begin{equation}
\label{DT2.pattern18}
\begin{gathered}
\begin{tikzpicture}[scale=.8]
\node (u) at (0,2) [] {$\B_0$};
\node (0) at (-4,0) [] {$\B^0$};
\node (1) at (-2,0) [] {$\B^1$};
\node (2) at (0,0) [] {$\dots$};
\node (3) at (2,0) [] {$\B^{l-1}$};
\node (4) at (4,0) [] {$\B^l$};
\draw [->] (0) -- node [below] {$s_{i_1}$} (1);
\draw [->] (1) -- node [below] {$s_{i_2}$} (2);
\draw [->] (2) -- node [below] {$s_{i_{l-1}}$} (3);
\draw [->] (3) -- node [below] {$s_{i_l}$} (4);
\draw (0) -- (u) -- (4); 
\end{tikzpicture} 
\end{gathered}
\end{equation}
The cluster DT transformation of $\conf_b^e(\mathcal{B})$  is a sequence of left reflections followed by transposition. Concretely, we obtain a collection of Borel subgroups $\B_1, \dots, \B_l$ by the reflections $^ir$ 
\[
\begin{tikzpicture}[scale=.8]
\node (0) at (-4,2) [] {$\B_l$};
\node (1) at (-2,2) [] {$\B_{l-1}$};
\node (2) at (0,2) [] {$\dots$};
\node (3) at (2,2) [] {$\B_1$};
\node (4) at (4,2) [] {$\B_0$};
\node (d) at (0,0) [] {$\B^l$};
\draw [->] (0) -- node [above] {$s_{i_l}$} (1);
\draw [->] (1) -- node [above] {$s_{i_{l-1}}$} (2);
\draw [->] (2) -- node [above] {$s_{i_2}$} (3);
\draw [->] (3) -- node [above] {$s_{i_1}$} (4);
\draw (0) -- (d) -- (4);
\end{tikzpicture}
\]
and then apply the transposition map to get a new configuration in $\conf_b^e(\mathcal{B})$
\begin{equation}
\label{DT1.pattern19}
    \begin{gathered}
\begin{tikzpicture}[scale=.8]
\node (u) at (0,2) [] {$\left(\B^l\right)^t$};
\node (0) at (-4,0) [] {$\B_0^t$};
\node (1) at (-2,0) [] {$\B_1^t$};
\node (2) at (0,0) [] {$\dots$};
\node (3) at (2,0) [] {$\B_{l-1}^t$};
\node (4) at (4,0) [] {$\B_l^t$};
\draw [->] (0) -- node [below] {$s_{i_1}$} (1);
\draw [->] (1) -- node [below] {$s_{i_2}$} (2);
\draw [->] (2) -- node [below] {$s_{i_{l-1}}$} (3);
\draw [->] (3) -- node [below] {$s_{i_l}$} (4);
\draw (0) -- (u) -- (4); 
\end{tikzpicture}
\end{gathered}
\end{equation}
Let us  apply the cluster DT transformation to \eqref{DT1.pattern19} again, obtaining
\begin{equation}
\label{DT2.pattern20}
    \begin{gathered}
\begin{tikzpicture}[scale=.8]
\node (u) at (0,2) [] {$\B_l$};
\node (0) at (-4,0) [] {$\B^l$};
\node (1) at (-2,0) [] {$\B_{l+1}^t$};
\node (2) at (0,0) [] {$\dots$};
\node (3) at (2,0) [] {$\B_{2l-1}^t$};
\node (4) at (4,0) [] {$\B_{2l}^t$};
\draw [->] (0) -- node [below] {$s_{i_1}$} (1);
\draw [->] (1) -- node [below] {$s_{i_2}$} (2);
\draw [->] (2) -- node [below] {$s_{i_{l-1}}$} (3);
\draw [->] (3) -- node [below] {$s_{i_l}$} (4);
\draw (0) -- (u) -- (4); 
\end{tikzpicture}
\end{gathered}
\end{equation}

To better describe the patterns, we introduce the following notations for Borel subgroups in $\mathcal{B}_+$ \[\B_{(0)}:=\tau(\B^0), \hskip 7mm \B_{(1)}:=\B_0, \hskip 7mm \B_{(2)}:=\tau(\B^l), \hskip 7mm \B_{(3)}:=\B_l, 
\hskip 5mm \B_{(4)}:=\B_{2l},\] 
where $\tau$ is a natural isomorphism from $\mathcal{B}_-$ to $\mathcal{B}_+$
that takes $\B_-g$ to $g^{-1}\overline{w}_0\B_+$.
Recall the automorphism $*$ on $\mathcal{B}_+$ by
$\B \mapsto \B^*:=\tau(\B^t).$ 
The configurations \eqref{DT2.pattern18}-\eqref{DT2.pattern20} can be rewritten as
\begin{equation}
\label{dt20191203}
\begin{gathered}
\begin{tikzpicture}[baseline=5ex]
\node (u) at (0,2) [] {$\B_{(1)}$};
\node (0) at (-1,0) [] {$\B_{(0)}$};\
\node (1) at (1,0) [] {$\B_{(2)}$};
\draw [->] (u) -- node [above left] {$w_0$} (0);
\draw [dashed, ->] (0) -- node [below] {$b$} (1);
\draw [->] (1) -- node [above right] {$w_0$} (u);
\end{tikzpicture}  \quad \overset{\DT}{\longmapsto} \quad 
\begin{tikzpicture}[baseline=5ex]
\node (u) at (0,2) [] {$\B_{(2)}^*$};
\node (0) at (-1,0) [] {$\B_{(1)}^*$};
\node (2) at (1,0) [] {$\B_{(3)}^*$};
\draw [->] (u) -- node [above left] {$w_0$} (0);
\draw [dashed,->] (0) -- node [below] {$b$} (1);
\draw [->] (1) -- node [above right] {$w_0$} (u);
\end{tikzpicture}  \quad \overset{\DT}{\longmapsto} \quad 
\begin{tikzpicture}[baseline=5ex]
\node (u) at (0,2) [] {$\B_{(3)}$};
\node (0) at (-1,0) [] {$\B_{(2)}$};
\node (2) at (1,0) [] {$\B_{(4)}$};
\draw [->] (u) -- node [above left] {$w_0$} (0);
\draw [dashed,->] (0) -- node [below] {$b$} (1);
\draw [->] (1) -- node [above right] {$w_0$} (u);
\end{tikzpicture}
\end{gathered}
\end{equation}
Here in each configuration we abbreviate the bottom $b$-chain into a single dashed arrow. We adopt the same convention in the rest of the proof.

Let us cut the first and the last triangles in \eqref{dt20191203} at $\B_{(2)}$. We claim that the obtained chains \[\xymatrix{\B_{(2)} \ar[r]^{w_0} & \B_{(1)} \ar[r]^{w_0} & \B_{(0)} \ar@{-->}[r]^b & \B_{(2)}}, \hskip 11mm 
\xymatrix{\B_{(2)} \ar@{-->}[r]^b & \B_{(4)} \ar[r]^{w_0} & \B_{(3)} \ar[r]^{w_0} & \B_{(2)}}.\] are equivalent under the braid moves from $\Omega b$ to $b \Omega$.
Indeed, let $\B_1$ be the unique flag such that $\xymatrix{\B_{(1)} \ar[r]^{s_{i_1}^*} & \B_1 \ar[r]^{w_0s_{i_1}} & \B_{(0)}}$. The first left reflection in the first $\DT$ is 
\[
^{i_1}r:\quad \begin{tikzpicture}[baseline=5ex,scale=0.8] 
\node (u) at (0,2) [] {$\B_{(1)}$};
\node (0) at (-3,0) [] {$\B_{(0)}$};
\node (1) at (-1,0) [] {$\tau(\B^1)$};
\node (2) at (1,0) [] {$\dots$};
\node (3) at (3,0) [] {$\B_{(2)}$};
\draw [->] (u) -- node [above left] {$w_0$}(0);
\draw [->] (0) -- node [below] {$s_{i_1}$} (1);
\draw [->] (1) -- node [below] {$s_{i_2}$} (2);
\draw [->] (2) -- node [below] {$s_{i_l}$} (3);
\draw [->] (3) -- node [above right] {$w_0$} (u);
\end{tikzpicture}\quad \longmapsto\quad
\begin{tikzpicture}[baseline=5ex,scale=0.8]
\node (u1) at (-1,2) [] {$\B_1$};
\node (u2) at (1,2) [] {$\B_{(1)}$};
\node (d1) at (-2,0) [] {$\tau(\B^1)$};
\node (d2) at (0,0) [] {$\dots$};
\node (d3) at (2,0) [] {$\B_{(2)}$};
\draw [->] (u2) -- node [above] {$s_{i_1}^*$} (u1);
\draw [->] (u1) -- node [above left] {$w_0$} (d1);
\draw [->] (d1) -- node [below] {$s_{i_2}$} (d2);
\draw [->] (d2) -- node [below] {$s_{i_l}$} (d3);
\draw [->] (d3) -- node [above right] {$w_0$} (u2);
\end{tikzpicture}
\]
Let us cut them at $\B_{(2)}$. The obtained chains are equivalent under the braid moves  $w_0 s_{i_1}=s_{i_1}^\ast w_0$:  
\begin{align*}
&\xymatrix{\B_{(2)} \ar[r]^{w_0} & \B_{(1)} \ar[r]^{w_0} & \B_{(0)} \ar@{-->}[r]^b & \B_{(2)}}\\
=\quad &\xymatrix{\B_{(2)} \ar[r]^{w_0} & \B_{(1)}   \ar[r]^{w_0} & \B_{(0)} \ar[r]^{s_{i_1}} & \tau(\B^1)  \ar@{-->}[r]^{s_{i_1}^{-1}b} & \B_{(2)}}\\
=\quad &\xymatrix{\B_{(2)} \ar[r]^{w_0} & \B_{(1)} \ar[r]^{s_{i_1}^*} & \B_1 \ar[r]^{w_0} & \tau(\B^1) \ar@{-->}[r]^{s_{i_1}^{-1}b} & \B_{(2)}}
\end{align*}
Repeat the same procedure for the rest of the left reflections. Eventually the braid moves from $w_0b$ to $b^\ast w_0$ turn the chain $\xymatrix{\B_{(2)} \ar[r]^{w_0} & \B_{(1)} \ar[r]^{w_0} & \B_{(0)} \ar@{-->}[r]^b & \B_{(2)}}$ 
into $\xymatrix{\B_{(2)} \ar[r]^{w_0}& \B_{(1)} \ar@{-->}[r]^{b^*} & \B_{(3)} \ar[r]^{w_0} & \B_{(2)}}$. 
Similarly, the second DT turns the latter into $\xymatrix{\B_{(2)} \ar@{-->}[r]^b & \B_{(4)} \ar[r]^{w_0} & \B_{(3)} \ar[r]^{w_0} & \B_{(2)}}$ by braid moves.

Applying the transformation ${\rm DT}^2$ recursively, we obtain the configurations
\[
\begin{tikzpicture}[baseline=5ex]
\node (u) at (0,2) [] {$\B_{(2k-1)}$};
\node (0) at (-1,0) [] {$\B_{(2k-2)}$};
\node (1) at (1,0) [] {$\B_{(2k)}$};
\draw [->] (u) -- node [above left] {$w_0$} (0);
\draw [->, dashed] (0) -- node [below] {$b$} (1);
\draw [->] (1) -- node [above right] {$w_0$} (u);
\end{tikzpicture}  \quad \overset{\DT^2}{\longrightarrow} \quad 
\begin{tikzpicture}[baseline=5ex]
\node (u) at (0,2) [] {$\B_{(2k+1)}$};
\node (0) at (-1.5,0) [] {$\B_{(2k)}$};
\node (2) at (1.5,0) [] {$\B_{(2k+2)}$};
\draw [->] (u) -- node [above left] {$w_0$} (0);
\draw [->,dashed] (0) -- node [below] {$b$} (1);
\draw [->] (1) -- node [above right] {$w_0$} (u);
\end{tikzpicture}  
\]
Let us concatenate $\max\{m,n\}+1$ copies of the chain $\xymatrix@=12pt{\B_{(2m)} \ar@{-->}[r]^b &\B_{(2m+2)} \ar[r]^{w_0} & \B_{(2m+1)} \ar[r]^{w_0} & \B_{(2m)}}$. We apply $\DT^2$ to move $\Omega=w_0^2$ to the right. The braid moves $(b\Omega)^{n+1}= b^{n+1} \Omega^{1+n}$ give rise to
\begin{align*}
&\xymatrix@=12pt{\dots   \ar@{-->}[r]^b & \B_{(2m+2)} \ar[r]^{w_0} & \B_{(2m+1)} \ar[r]^{w_0}& \B_{(2m)} \ar@{-->}[r]^b & \B_{(2m+2)} \ar[r]^{w_0} & \B_{(2m+1)} \ar[r]^{w_0} & \B_{(2m)}}\\
=& \xymatrix@=12pt{\dots \ar[r]^{w_0} & \B_{(2m+2)} \ar@{-->}[r]^b & \B_{(2m+4)} \ar[r]^{w_0} & \B_{(2m+3)} \ar[r]^{w_0} & \B_{(2m+2)} \ar[r]^{w_0} & \B_{(2m+1))} \ar[r]^{w_0} & \B_{(2m)} }\\
=& \dots\\
=& \xymatrix@=12pt{ \dots \ar@{-->}[r]^(0.4)b & \B_{(2(m+n))} \ar@{-->}[r]^b & \B_{(2(m+n)+2)} \ar[r]^{w_0} & \B_{(2(m+n)+1)} \ar[r]^{w_0} & \B_{(2(m+n))} \ar@{-->}[r]^{\Omega^{n}} & \B_{(2m)}} .
\end{align*}
Conversely, we apply $\DT^{-2}$ to move $b$ to the right, and obtain 
\begin{align*}
&\xymatrix@=12pt{\dots \ar@{-->}[r]^{b} & \B_{(2m+2)} \ar[r]^{w_0} & \B_{(2m+1)} \ar[r]^{w_0} & \B_{(2m)} \ar@{-->}[r]^b & \B_{(2m+2)} \ar[r]^{w_0} & \B_{(2m+1)} \ar[r]^{w_0} & \B_{(2m)} }\\
=& \xymatrix@=12pt{\dots \ar[r]^{w_0} & \B_{(2m-1)} \ar[r]^{w_0} & \B_{(2m-2)} \ar@{-->}[r]^b & \B_{(2m)} \ar[r]^{w_0} & \B_{(2m-1)} \ar[r]^{w_0} & \B_{(2m-2)} \ar@{-->}[r]^b & \B_{(2m)} }\\
=& \xymatrix@=12pt{\dots \ar[r]^{w_0} & \B_{(2m-4)} \ar@{-->}[r]^b & \B_{(2m-2)} \ar[r]^{w_0} & \B_{(2m-3)} \ar[r]^{w_0} & \B_{(2m-4)} \ar@{-->}[r]^b & \B_{(2m-2)} \ar@{-->}[r]^b & \B_{(2m)} }\\
=& \dots\\
=&\xymatrix@=12pt{ \dots \ar[r]^{w_0} & \B_{(0)} \ar@{-->}[r]^b & \B_{(2)} \ar[r]^{w_0} & \B_{(1)} \ar[r]^{w_0} & \B_{(0)} \ar@{-->}[r]^b & \B_{(2)} \ar@{-->}[r]^b & \B_{(4)} \ar@{-->}[r]^b & \dots \ar@{-->}[r]^b & \B_{(2m)}}\\
= & \xymatrix@=12pt{ \dots \ar[r]^{w_0} & \B_{(0)} \ar@{-->}[r]^b & \B_{(2)} \ar[r]^{w_0} & \B_{(1)} \ar[r]^{w_0} & \B_{(0)} \ar@{-->}[r]^{\Omega^{n}} & \B_{(2m)}}.
\end{align*}
Here the last step uses the condition that $b^m=\Omega^{n}$.

Let us compare the two final chains under  the above braid moves. By Theorem \ref{bott-samelson uniqueness}, we get 
\begin{align*}
&\xymatrix@=12pt{ \B_{(0)} \ar@{-->}[r]^b & \B_{(2)} \ar[r]^{w_0} & \B_{(1)} \ar[r]^{w_0} & \B_{(0)}}\\
=&\xymatrix@=12pt{ \B_{(2(m+n))} \ar@{-->}[r]^b & \B_{(2(m+n)+2)} \ar[r]^{w_0} & \B_{(2(m+n)+1)} \ar[r]^{w_0} & \B_{(2(m+n))}},
\end{align*}
which concludes the proof.
\end{proof}

\paragraph{Zamolodchikov's periodicity conjecture.}
A $Y$-system is a system of algebraic recurrence equations associated with a pair of Dynkin diagrams. The periodicity conjecture, first formulated by Zamolodchikov in his study of thermodynamic Bethe ansatz, asserts that all solutions to this system are of period dividing the double of the sum of the Coxeter numbers of underlying Dynkin diagrams. This periodicity property plays a significant rule in conformal field theory and statistical mechanics. It was first settled by Keller \cite{Kelperiod} in full generality by using highly nontrivial techniques including cluster algebras and their additive categorification. We refer to \textit{loc. cit.} for an introduction to the periodicity conjecture. 

Let us reformulate the periodicity conjecture in terms of cluster mutations. Let $\Delta$ be a Dynkin diagram with the Cartan Matrix $\C$. A bipartite coloring on $\Delta$ gives rise to a seed with the same vertex set $I$, multipliers $\left\{d_a\right\}$, and the exchange matrix 
\[
\epsilon_{ab}=\left\{\begin{array}{ll}
c(a)\C_{ba} & \text{if $a\neq b$},\\
0 & \text{otherwise}.\end{array}\right.
\]
where $c(a)=1$ if the vertex $a$ is colored black and $c(a)=-1$ if it is colored white.

Given two bipartite Dynkin diagrams $\Delta$ and $\Delta'$, their square product $\Delta\square\Delta'$ is defined to be a seed with vertex set  $I\times I'$ and exchange matrix
\[
\epsilon_{\left(i,i'\right),\left(j,j'\right)}=\left\{\begin{array}{ll}
    -c\left(i'\right)\epsilon_{ij} & \text{if $i'=j'$},  \\
    c(i)\epsilon_{i'j'} & \text{if $i=j$}, \\
    0 & \text{otherwise}.
\end{array}\right.
\]

\begin{exmp} The quiver $\mathrm{D}_4\square \mathrm{A}_3$ is as follows.
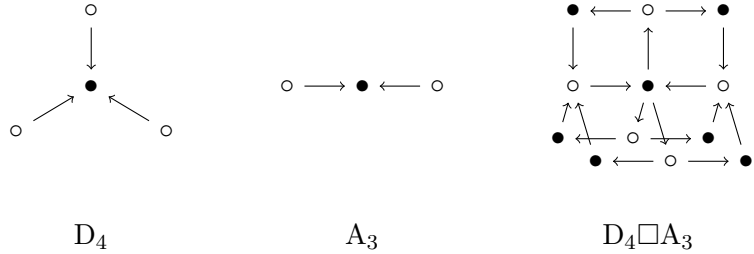
\[
\begin{tikzpicture}
\node (0) at (0,1) [] {$\bullet$};
\node (1) at (0,2) [] {$\circ$};
\node (2) at (-1,0.4) [] {$\circ$};
\node (3) at (1,0.4) [] {$\circ$};
\draw [<-] (0) -- (1);
\draw [<-] (0) -- (2);
\draw [<-] (0) -- (3);
\node at (0,-1) [] {$\mathrm{D}_4$};
\end{tikzpicture}
\quad \quad \quad 
\begin{tikzpicture}
\node (0) at (0,1) [] {$\circ$};
\node (1) at (1,1) [] {$\bullet$};
\node (2) at (2,1) [] {$\circ$};
\draw [->] (0) -- (1);
\draw [->] (2) -- (1);
\node at (1,-1) [] {$\mathrm{A}_3$};
\end{tikzpicture}
\quad \quad \quad
\begin{tikzpicture}
    \node (0) at (0.3,0) [] {$\bullet$};
    \node (1) at (1.3,0) [] {$\circ$};
    \node (2) at (2.3,0) [] {$\bullet$};
    \node (3) at (0,1) [] {$\circ$};
    \node (4) at (1,1) [] {$\bullet$};
    \node (5) at (2,1) [] {$\circ$};
    \node (6) at (0,2) [] {$\bullet$};
    \node (7) at (1,2) [] {$\circ$};
    \node (8) at (2,2) [] {$\bullet$};
    \node (9)  at (-0.2,0.3) [] {$\bullet$};
    \node (10) at (0.8,0.3) [] {$\circ$};
    \node (11) at (1.8,0.3) [] {$\bullet$};
    \draw [->] (1) -- (0);
    \draw [->] (1) -- (2);
    \draw [<-] (3) -- (0);
    \draw [<-] (1) -- (4);
    \draw [<-] (5) -- (2);
    \draw [->] (3) -- (4);
    \draw [<-] (3) -- (6);
    \draw [->] (5) -- (4);
    \draw [->] (7) -- (6);
    \draw [<-] (7) -- (4);
    \draw [->] (7) -- (8);
    \draw [<-] (5) -- (8);
    \draw [<-] (3) -- (9);
    \draw [<-] (10) -- (4);
    \draw [<-] (5) -- (11);
    \draw [->] (10) -- (9);
    \draw [->] (10) -- (11);
    \node at (1,-1) [] {$\mathrm{D}_4\square \mathrm{A}_3$};
    \end{tikzpicture}
\]
\end{exmp}

Let $\tau=\tau_-\circ\tau_+$, where $\tau_+$ is a sequence of mutations at the black vertices of $\Delta\square\Delta'$, and $\tau_-$ is a sequence of mutations at white ones. Note that $\tau$ preserves the quiver $\Delta\square\Delta'$. Following \cite{FZysystem}, the mutation sequence $\tau$ gives rise to the {\it Zamolodchikov transformation} $\Za$ on the cluster Poisson variety $\mathcal{X}_{\Delta\square\Delta'}$. Let $h$ and $h'$ be the Coxeter numbers of $\Delta$ and $\Delta'$ respectively. By \cite[Lemma 2.4]{Kelperiod}, Zamolodchikov's periodicity conjecture is equivalent to the identity 
\begin{equation}
\label{ZaPeriodicity}
\Za^{h+h'}=\Id
\end{equation}

Below we give a new geometric proof of \eqref{ZaPeriodicity} for $\Delta\square \mathrm{A}_n$.
Let $\G$ be a group of type $\Delta$ and let $\mathcal{B}$ be its flag variety. We fix a bipartite coloring on $\Delta$ and  set
\[
b:=\underbrace{s_{b_1}s_{b_2}\dots s_{b_l}}_\text{black vertices} \quad \quad \text{and} \quad \quad w:=\underbrace{s_{w_1}s_{w_2}\dots s_{w_m}}_\text{white vertices}.
\]
Let 
 \[
 p= \underbrace{wbw\dots }_\text{$n+1$ factors} \quad \text{and} \quad q= \underbrace{bwb\dots }_\text{$n+1$ factors}
 \]
Recall the construction of seeds for double Bott-Samelson cells. The quiver associated to the following triangulation of $\conf^{p}_{q}(\mathcal{B})$ is  $\Delta\square \mathrm{A}_n$.
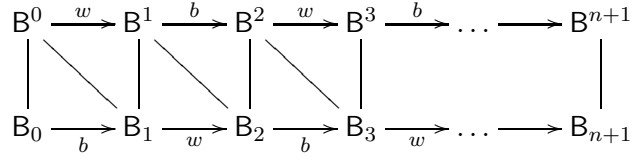
\[
\xymatrix{\B^0 \ar[r]^w \ar@{-}[d] \ar@{-}[dr] & \B^1 \ar@{-}[d] \ar@{-}[dr] \ar[r]^b &  \B^2 \ar[r]^w  \ar@{-}[d] \ar@{-}[dr] & \B^3 \ar[r]^b \ar@{-}[d] & \dots \ar[r] & \B^{n+1} \ar@{-}[d]\\
\B_0 \ar[r]_b & \B_1 \ar[r]_w & \B_2 \ar[r]_b & \B_3 \ar[r]_w & \dots \ar[r] & \B_{n+1}
}
\]
Therefore $\conf^p_q(\mathcal{B})$ is birationally isomorphic to the cluster Poisson variety $\mathcal{X}_{\Delta\square \mathrm{A}_n}.$
\begin{lem} The transformation $\Za$ acts on $\conf^p_q(\mathcal{B})$ as
\[
\vcenter{\vbox{\xymatrix{ \B^0 \ar[r]^w \ar@{-}[d] & \B^1 \ar[r]^b & \B^2 \ar[r]^w & \B^3 \ar[r]^b & \dots \ar[r] & \B^{n-1} \ar[r] & \B^n \ar[r] & \B^{n+1} \ar@{-}[d] \\
 \B_0 \ar[r]_b  & \B_1 \ar[r]_w & \B_2 \ar[r]_b & \B_3 \ar[r]_w & \dots \ar[r] &\B_{n-1} \ar[r] & \B_n \ar[r] & \B_{n+1}
}}} \quad \mapsto 
\]
\[
\vcenter{\vbox{\xymatrix{ \B^{-2} \ar[r]^w \ar@{-}[d] & \B^{-1} \ar[r]^b & \B^0 \ar[r]^w & \B^1 \ar[r]^b & \dots \ar[r] & \B^{n-3} \ar[r] & \B^{n-2} \ar[r] & \B^{n-1} \ar@{-}[d] \\
 \B_2 \ar[r]_b  & \B_3 \ar[r]_w & \B_4 \ar[r]_b & \B_5 \ar[r]_w & \dots \ar[r] &\B_{n+1} \ar[r] & \B_{n+2} \ar[r] & \B_{n+3}
}}}
\]
where $\B^{-2}$ and $\B^{-1}$ are the unique flags obtained by reflecting $\B_1$ and $\B_0$ on the left, and $\B_{n+2}$ and $\B_{n+3}$ are  obtained by reflecting $\B^{n+1}$ and $\B^n$ on the right. In other words, $\Za$ is equivalent to the composition of a Coxeter sequence of reflections on the left from bottom to top with another Coxeter sequence of reflections on the right from top to bottom.
\end{lem}
\begin{proof} The colors of the closed strings (vertices of the seed $\Delta\square \mathrm{A}_n$) look like the following
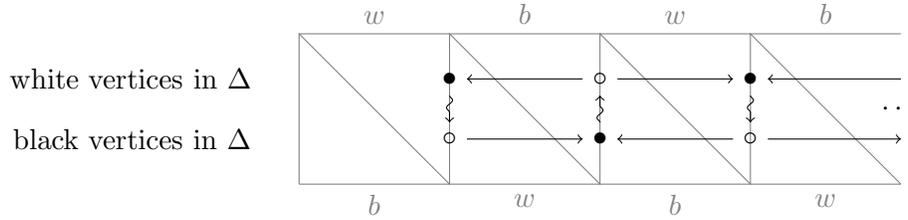
\[
\begin{tikzpicture}
    \draw[gray] (0,0) -- (0,2) -- (2,0) -- (2,2) -- (4,0) -- (4,2) -- (6,0) -- (6,2) -- (8,0) ;
    \foreach \i in {0,2}
        {
        \draw [gray] (\i*2,0) -- node [below] {$b$} (\i*2+2,0);
        \draw [gray] (\i*2,2) -- node [above] {$w$} (\i*2+2,2);
        \draw [gray] (\i*2+2,0) -- node [below] {$w$} (\i*2+4,0);
        \draw [gray] (\i*2+2,2) -- node [above] {$b$} (\i*2+4,2);
        }
    \node at (-0.5,0.6) [left] {black vertices in $\Delta$};
    \node at (-0.5,1.4) [left] {white vertices in $\Delta$};
    \node (b0) at (2,1.4) [] {$\bullet$};
    \node (w0) at (2,0.6) [] {$\circ$};
    \node (b1) at (4,1.4) [] {$\circ$};
    \node (w1) at (4,0.6) [] {$\bullet$};
    \node (b2) at (6,1.4) [] {$\bullet$};
    \node (w2) at (6,0.6) [] {$\circ$};
    \draw [<-] (b0) -- (b1);
    \draw [<-] (b2) -- (b1);
    \draw [<-] (b2) -- (8,1.4);
    \draw [->,decorate,decoration={snake,amplitude=.4mm,segment length=2mm,post length=1mm}] (b0) -- (w0);
    \draw [->,decorate,decoration={snake,amplitude=.4mm,segment length=2mm,post length=1mm}] (w1) -- (b1);
    \draw [->,decorate,decoration={snake,amplitude=.4mm,segment length=2mm,post length=1mm}] (b2) -- (w2);
    \draw [<-] (w1) -- (w0);
    \draw [<-] (w1) -- (w2);
    \draw [<-] (8,0.6) -- (w2);
    \node at (8,1) [] {$\cdots$};
\end{tikzpicture}
\]
Note that black vertices in the quiver $\Delta\square \mathrm{A}_n$ correspond to strings cut out by triangles with the same base labelings. Mutations at all the black vertices of $\Delta\square \mathrm{A}_n$ corresponds to a collection of diagonal flips within a quadrilateral with the same base labelings (either $b$ or $w$). This turns the above triangulation to the left one below.
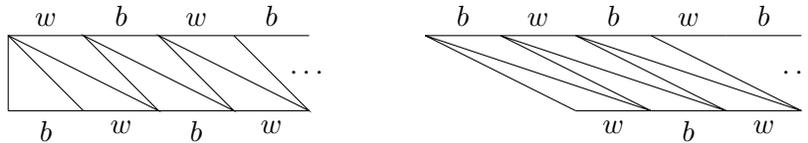
\[
\begin{tikzpicture}
    \draw (0,0) -- (0,1) -- (1,0);
    \draw (0,1) -- (2,0) -- (1,1) -- (3,0) -- (2,1) -- (4,0) -- (3,1);
    \foreach \i in {0,2}
        {
        \draw  (\i,0) -- node [below] {$b$} (\i+1,0);
        \draw  (\i,1) -- node [above] {$w$} (\i+1,1);
        \draw  (\i+1,0) -- node [below] {$w$} (\i+2,0);
        \draw  (\i+1,1) -- node [above] {$b$} (\i+2,1);
        }
    \node at (4,0.5) [] {$\cdots$};
\end{tikzpicture}\quad \quad \quad
\begin{tikzpicture}
    \draw (1,0) -- (-1,1) -- (2,0) -- (0,1) -- (3,0) -- (1,1) -- (4,0) -- (2,1);
    \draw  (-1,1) -- node[above] {$b$}(0,1);
    \draw  (2,0) -- node [below] {$b$} (3,0);
    \foreach \i in {0,2}
        {
        \draw  (\i,1) -- node [above] {$w$} (\i+1,1);
        \draw  (\i+1,0) -- node [below] {$w$} (\i+2,0);
        \draw  (\i+1,1) -- node [above] {$b$} (\i+2,1);
        }
    \node at (4,0.5) [] {$\cdots$};
\end{tikzpicture}
\]
Now reflect the leftmost collection of triangles with base labeled by a $b$-chain to the top and the rightmost collection of triangles (with base labeled by either a $b$-chain or a $w$-chain) to the bottom, and then flip the rest of the non-flipped diagonals.
The resulting triangulation is depicted by the above right picture. Since these reflections and the diagonal flips only involve triangles that are labeled differently, no mutation is involved. By shifting the top chain to the right by two units we get the following triangulation on the left.
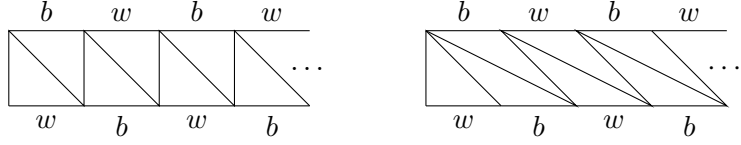
\[
\begin{tikzpicture}
    \draw (0,0) -- (0,1) -- (1,0) -- (1,1) -- (2,0) -- (2,1) -- (3,0) -- (3,1) -- (4,0) ;
    \foreach \i in {0,2}
        {
        \draw  (\i,0) -- node [below] {$w$} (\i+1,0);
        \draw  (\i,1) -- node [above] {$b$} (\i+1,1);
        \draw  (\i+1,0) -- node [below] {$b$} (\i+2,0);
        \draw  (\i+1,1) -- node [above] {$w$} (\i+2,1);
        }
    \node at (4,0.5) [] {$\dots$};
\end{tikzpicture} \quad \quad \quad \begin{tikzpicture}
    \draw (0,0) -- (0,1) -- (1,0);
    \draw (0,1) -- (2,0) -- (1,1) -- (3,0) -- (2,1) -- (4,0) -- (3,1);
    \foreach \i in {0,2}
        {
        \draw  (\i,0) -- node [below] {$w$} (\i+1,0);
        \draw  (\i,1) -- node [above] {$b$} (\i+1,1);
        \draw  (\i+1,0) -- node [below] {$b$} (\i+2,0);
        \draw  (\i+1,1) -- node [above] {$w$} (\i+2,1);
        }
    \node at (4,0.5) [] {$\cdots$};
\end{tikzpicture}
\]
Similarly, mutating at the white vertices of $\Delta\square \mathrm{A}_n$ corresponds to another collection of diagonal flips, which turns the above left triangulation to the above right one.
Let us again reflect the leftmost collection of triangles to the top and the rightmost collection of triangles to the bottom, and flip the rest of the diagonals. The resulting triangulation coincides with the initial one. 

In the whole process there are two sequences of reflections on the left from bottom to top (one for a $b$-chain and the other for a $w$-chain) and two similar sequences of reflections on the right from top to bottom, which is the action described in the lemma.
\end{proof}

\begin{cor}\label{5.9} $\Za^{h+n+1}=\Id$.
\end{cor}
\begin{proof} Let $R$ be the isomorphism from $\conf^p_q(\mathcal{B})$ to $\conf^e_{qp^\circ}(\mathcal{B})$ by reflections on the right. Element in $\conf^e_{qp^\circ}(\mathcal{B})$ are described as
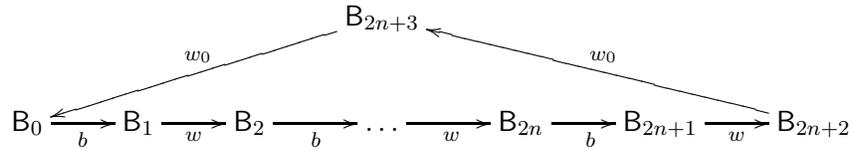
\[
\xymatrix{& & &\B_{2n+3} \ar@{->}[dlll]_{w_0} & & \\
\B_0 \ar[r]_b & \B_1 \ar[r]_w & \B_2 \ar[r]_b & \dots \ar[r]_w & \B_{2n} \ar[r]_b& \B_{2n+1} \ar[r]_w & \B_{2n+2} \ar@{->}[ulll]_{w_0}}
\]
Consider the Coxeter element $c=bw$. We have $c^h=w_0^2$ as positive braids. Therefore the chain $\xymatrix{\B_{2n+2} \ar[r]^{w_0} & \B_{2n+3} \ar[r]^{w_0} & \B_0}$ can be written as $h$ copies of $c$ chains as below, where $\B_{(i)}=\B_{2i}$ for $0\leq i\leq n+1$.
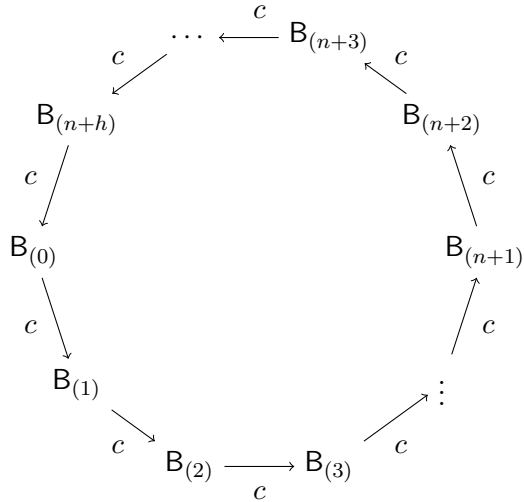
\[
\begin{tikzpicture}
    \foreach \i in {0,...,3}
    {
    \node (\i) at +(180+36*\i:3) [] {$\B_{(\i)}$};
    }
    \node (4) at +(180+36*4:3) [] {$\vdots$};
    \node (5) at +(180+36*5:3) [] {$\B_{(n+1)}$};
    \node (6) at +(180+36*6:3) [] {$\B_{(n+2)}$};
    \node (7) at +(180+36*7:3) [] {$\B_{(n+3)}$};
    \node (8) at +(180+36*8:3) [] {$\cdots$};
    \node (9) at +(180+36*9:3) [] {$\B_{(n+h)}$};
    \foreach \i in {0,...,8}
    {
    \pgfmathtruncatemacro{\j}{\i+1};
    \draw [->] (\i) -- (\j);
    }
    \draw [->] (9) -- (0);
    \foreach \i in {0,...,9}
    {
    \node at +(198+36*\i:3.2) [] {$c$};
    }
\end{tikzpicture}
\]
After conjugation by $R$, the transformation $\Za$ acts on $\conf^e_{qp^\circ}(\mathcal{B})$ by rotating the above circle clockwisely by one step. Therefore $\Za^{h+n+1}=\Id$.
\end{proof}

As in \cite[Sec.5.7]{Keldilog}, Zamolodchikov's periodicity implies that
\[
\left(\DT_{\Delta\square \Delta'}\right)^{m}=\mathrm{Id},\hskip 10mm \mbox{where } m=\frac{2(h+h')}{{\rm gcd}(h, h')}.
\]
We close this section by  presenting a direct proof of the above identity for $\Delta\square \mathrm{A}_n$. 
 
\begin{cor}\label{5.3}$(\DT_{\Delta\square \mathrm{A}_n})^m=\Id$ where $m=\frac{2(h+n+1)}{\gcd(h,n+1)}$.
\end{cor}
\begin{proof}
In the braid group we have
\[
(qp^\circ)^{\frac{h}{{\rm gcd}(h, n+1)}}= \left(\underbrace{bwbwbw\dots bw}_\text{$n+1$ copies of $bw$}\right)^{\frac{h}{{\rm gcd}(h, n+1)}}=c^{\frac{h(n+1)}{{\rm gcd}(h, n+1)}}= \Omega^{\frac{(n+1)}{{\rm gcd}(h, n+1)}}
\]
By Theorem \ref{5.1}, the order of the DT transformation of $\conf^{p}_{q}(\mathcal{B})$ divides $\frac{2(h+n+1)}{{\rm gcd}(h, n+1)}$.
\end{proof}

\section{Points Counting and Positive Braid Closures}

\subsection{Points Counting over Finite Field}\label{computation}

The proof of Theorem \ref{affineness of conf} has shown that the space $\conf^b_d(\mathcal{A})$ can be realized as the non-vanishing locus of an integral polynomial given by Theorem \ref{gaussian}.  Therefore  $\conf^b_d(\mathcal{A})$  is well-defined over any finite field $\mathbb{F}_q$.  
In this section we present an algorithm\footnote{This algorithm has been suggested to us by J.H. Lu.} counting its $\mathbb{F}_q$ points
\[
f^b_d(q):=\left|  \conf^b_d(\mathcal{A})(\mathbb{F}_q)\right|.
\]

Let $u$ and $v$ be Weyl group elements and  $s_i$ a simple reflection. We set
\begin{equation}\label{eq 6.1}
P_i^{u, v}=\left\{\begin{array}{ll}
    q & \text{if $v=s_iu$ and $l\left(s_iu\right)<l(u)$,} \\
    q-1 & \text{if $v=u$ and $l\left(s_iu\right)>l(u)$,} \\
    1 & \text{if $v=s_iu$ and $l\left(s_iu\right)>l(u)$,}\\
    0 & \text{otherwise.}
\end{array}\right.
\end{equation}

\begin{thm} 
\label{vqbiboqef.counting}
Let $\vec{i}=(i_1, \ldots, i_n)$ be a word of $db^\circ$. Then 
\begin{equation}
\label{foruma.vqbiboqef.counting}
f^b_d(q)= (q-1)^{\tilde{r}} \sum_{(u_1, \ldots, u_{n-1})\in \W^{n-1}} \prod_{k=1}^{n} P_{i_k}^{u_{k-1}, u_{k}},
\end{equation}
where $u_0=u_n=e$ and $\tilde{r}=\dim \T$.
\end{thm}

\begin{proof} Since $\conf_d^b(\mathcal{A})\stackrel{\sim}{=}\conf_{db^\circ}^e(\mathcal{A})$, it suffices to count  $\mathbb{F}_q$-points of the latter space. Note that every configuration in $\conf_{db^\circ}^e(\mathcal{A})$ has a unique representative such that the pair of flags associated to its left side is $\xymatrix{\U_+\ar@{-}[r] & \B_-}$. Then let us fix a triangulation by taking diagonals from the top vertex to all the bottom vertices, and
label all the diagonals by Weyl group elements indicating the Tits codistances between the top flag and the bottom flags. It gives rise to a decomposition of the space $\conf_{db^\circ}^e(\mathcal{A})$.  Moving across the triangulation from left to right, we associate the number of possible configurations to each triangle based on Lemma \ref{3.3}: $q$ for $\mathbb{A}^1$, $q-1$ for $\mathbb{G}_m$, $1$ for $\{*\}$, and $0$ for $\emptyset$. In the end there is another decoration $\A_n$ over the last flag $\B_n$, which gives another multiple of $(q-1)^{\tilde{r}}=\left| \T\left(\mathbb{F}_q\right)\right|$. Taking the summation over all possible cases, we get \eqref{foruma.vqbiboqef.counting} 
\end{proof}

Inspired by a recent result of Galashin and Lam \cite{GL} on computing the $\mathbb{F}_q$-point count of positroid cells, we relate the computation of $f^b_d(q)$ with Hecke algebras.

Let us first briefly recall the definition of Hecke algebra $\mathcal{H}$ associated with a generalized Cartan matrix $\C$. As a non-commutative algebra, $\mathcal{H}$ is generated over $\mathbb{Z}\left[q^{\pm 1}\right]$ by the elements $\left\{T_i\right\}_{i\in \S}$, where $\S$ is the set of Coxeter generators. The generators $T_i$ satisfy the usual braid relations \eqref{braid rel} imposed by the Cartan matrix as well as the following identity
\begin{equation}\label{hecke def relation}
\left(T_i+q\right)\left(T_i-1\right)=0 \quad \quad \forall i\in \S.
\end{equation}
Note that \eqref{hecke def relation} implies that
\begin{equation}\label{T_s inverse}
T_i^{-1}=q^{-1}T_i+\left(1-q^{-1}\right) \quad \quad \forall i\in \S.
\end{equation}

For any positive braid $b\in \Br^+$ with braid word $\left(i_1,i_2,\dots, i_l\right)$, we define 
\[
T_b:=T_{i_1}T_{i_2}\cdots T_{i_l}.
\]
Note that this is well-defined since the generators $T_i$ also satisfy the braid relations. Since any Weyl group element $w\in \W$ defines a unique positive braid via any reduced word of $w$, we define $T_w$ to be the corresponding product of the generators $T_i$ according to the reduced word. In particular, as a consequence of \eqref{hecke def relation}, for any Weyl group elements $u,v\in \W$ with $v=s_iu$, we have
\begin{equation} \label{Hecke identity}
T_iT_u=\left\{\begin{array}{ll}
    T_v & \text{if $l(v)>l(u)$}, \\
    (1-q)T_u+qT_v & \text{if $l(v)<l(u)$}. 
\end{array}\right.
\end{equation}

It is known that $\left\{T_w\right\}_{w\in \W}$ forms a linear basis of $\mathcal{H}$, i.e., $\mathcal{H}\cong \bigoplus_{w\in \W}\mathbb{Z}\left[q^{\pm 1}\right]T_w$. We define a $\mathbb{Z}\left[q^{\pm 1}\right]$-linear map $\epsilon:\mathcal{H}\rightarrow \mathbb{Z}\left[q^{\pm 1}\right]$ by
\[
\epsilon\left(T_w\right)=\left\{\begin{array}{ll}
    1 & \text{if $w=e$},  \\
    0 & \text{otherwise}. 
\end{array}\right.
\]

\begin{cor} Let $\vec{i}=(i_1, \ldots, i_n)$ be a word of $db^\circ$ and let $\tilde{r}:=\dim \T$. Then
\[
f^b_d(q)=(q-1)^{\tilde{r}}q^n\epsilon\left(T_{db^\circ}^{-1}\right).
\]
\end{cor}
\begin{proof} By comparison with \eqref{foruma.vqbiboqef.counting}, it suffices to show that 
\[
q^n\epsilon\left(T_{db^\circ}^{-1}\right)=\sum_{(u_1, \ldots, u_{n-1})\in \W^{n-1}} \prod_{k=1}^{n} P_{i_k}^{u_{k-1}, u_{k}}.
\]
Without loss of generality let us assume that $b=e$. Let $\left(i_1,i_2,\dots, i_n\right)$ be a braid word for $d$. Then $T_d^{-1}=T_{i_n}^{-1}T_{i_{n-1}}^{-1}\cdots T_1^{-1}$. By substituting in \eqref{T_s inverse}, we get
\[
q^n\left(T_d^{-1}\right)=\left(T_{i_n}+(q-1)\right)\left(T_{i_{n-1}}+(q-1)\right)\cdots \left(T_{i_1}+(q-1)\right).
\]

For each $0\leq k\leq n$, define $h^k_w(q)$ to be the coefficients in the expansion 
\[
\left(T_{i_k}+(q-1)\right)\cdots \left(T_{i_1}+(q-1)\right)=\sum_{w\in \W}h^k_w(q)T_w
\] 
We claim that for any $u_k\in \W$, 
\begin{equation}\label{induction claim}
\sum_{(u_1,\dots, u_{k-1})\in \W^{k-1}}\prod_{l=1}^kP_{i_l}^{u_l,u_{l-1}}=h^k_{u_k}(q),
\end{equation}
and the theorem will follow from this claim and the definition of the linear map $\epsilon$.

We will do an induction on $k$. The claim is trivial for the base case $k=0$. For the inductive step, we define $u:=u_k$ and $v:=s_{i_{k+1}}u$. If $l(v)>l(u)$, we have
\[
    \left(T_{i_{k+1}}+(q-1)\right)h^k_u(q)T_u=h^k_u(q)T_v+(q-1)h^k_u(q)T_u,
\]
which covers the middle two cases in \eqref{eq 6.1}. On the other hand, if $l(v)<l(u)$, then by \eqref{Hecke identity} we have
\[
\left(T_{i_{k+1}}+(q-1)\right)h^k_u(q)T_u=h^k_u(q)\left((1-q)T_u+qT_v+(q-1)T_u\right)=h^k_u(q)qT_v,
\]
which covers the first case in \eqref{eq 6.1}. By combining these cases, we see that \eqref{induction claim} remains true for $k+1$ and hence the induction is finished.
\end{proof}

The group $\T\times \T$ acts on $\conf^b_d(\mathcal{A})$ by altering the decorations $\A^0$ and $\A_n$. It induces a transitive $\T\times \T$-action on each fiber of the projection $\pi:\conf^b_d(\mathcal{A})\rightarrow \conf^b_d(\mathcal{B})$. As stacks we have the isomorphism
\[
\conf^b_d(\mathcal{A}) \left/\left( \T\times \T\right)\right. \stackrel{\sim}{=} \conf^b_d(\mathcal{B})
\]
Therefore the number of $\mathbb{F}_q$-points of $\conf^b_d(\mathcal{B})$ as a stack is
\[
g^b_d(q):=\left|\conf^b_d(\mathcal{B})(\mathbb{F}_q)\right| = \frac{\left| \conf^b_d(\mathcal{A})(\mathbb{F}_q)\right|}{\left| \T\times \T(\mathbb{F}_q)\right|}= \frac{f^b_d(q)}{(q-1)^{2\tilde{r}}}
\]
Note that in general $g^b_d(q)$ is a rational function with possible poles are at $q=1$.
We include the code of a python program that computes $g^b_d(q)$ for positive braid closures in Appendix \ref{A.4}.

\begin{exmp} Let $\G=\SL_2$, $b=e$, and $d=s_1s_1s_1$. Let us fix a triangulation given by drawing diagonals from the top vertex to all the vertices at the bottom. Below are all the possible cases of different Tits codistances.
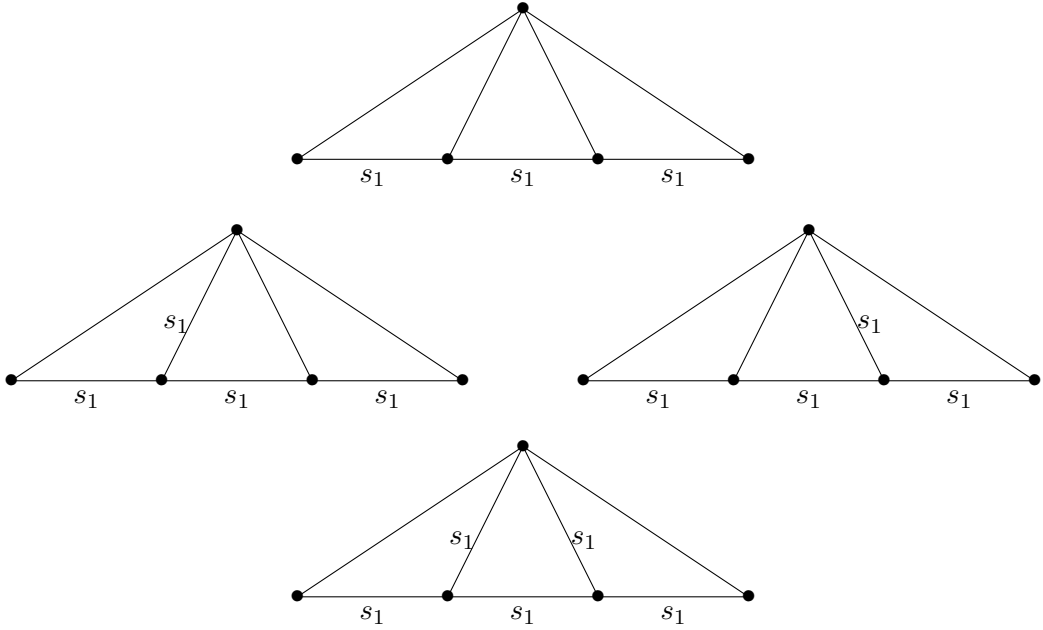
\[
\begin{tikzpicture}[scale=0.8]
\draw (-3,0) -- (0,2) -- (3,0) -- node [below] {$s_1$} (1,0) -- node [below] {$s_1$} (-1,0) -- node [below] {$s_1$} (-3,0);
\draw (0,2) -- (-1,0);
\draw (0,2) -- (1,0);
\node at (0,2) [] {$\bullet$};
\foreach \i in {-3,-1,1,3}
    {
    \node at (\i,0) [] {$\bullet$};
    }
\end{tikzpicture}
\]
\[
\begin{tikzpicture}[scale=0.8]
\draw (-3,0) -- (0,2) -- (3,0) -- node [below] {$s_1$} (1,0) -- node [below] {$s_1$} (-1,0) -- node [below] {$s_1$} (-3,0);
\draw (0,2) -- node [below left] {$s_1$} (-1,0);
\draw (0,2) -- (1,0);
\node at (0,2) [] {$\bullet$};
\foreach \i in {-3,-1,1,3}
    {
    \node at (\i,0) [] {$\bullet$};
    }
\end{tikzpicture}
\quad \quad \quad 
\begin{tikzpicture}[scale=0.8]
\draw (-3,0) -- (0,2) -- (3,0) -- node [below] {$s_1$} (1,0) -- node [below] {$s_1$} (-1,0) -- node [below] {$s_1$} (-3,0);
\draw (0,2) -- (-1,0);
\draw (0,2) -- node [below right] {$s_1$} (1,0);
\node at (0,2) [] {$\bullet$};
\foreach \i in {-3,-1,1,3}
    {
    \node at (\i,0) [] {$\bullet$};
    }
\end{tikzpicture}
\]
\[
\begin{tikzpicture}[scale=0.8]
\draw (-3,0) -- (0,2) -- (3,0) -- node [below] {$s_1$} (1,0) -- node [below] {$s_1$} (-1,0) -- node [below] {$s_1$} (-3,0);
\draw (0,2) -- node [below left] {$s_1$} (-1,0);
\draw (0,2) -- node [below right] {$s_1$} (1,0);
\node at (0,2) [] {$\bullet$};
\foreach \i in {-3,-1,1,3}
    {
    \node at (\i,0) [] {$\bullet$};
    }
\end{tikzpicture}
\]
The first case gives $(q-1)^3$, the fourth case gives $0$, and each of the rest  gives $q(q-1)$. Therefore 
\begin{align*}
f^b_d(q)&=\left((q-1)^3+2q(q-1)\right)(q-1)=q^4-2q^3+2q^2-2q+1.\\
g^b_d(q)&=\frac{f^b_d(q)}{(q-1)^2}=q^2+1.
\end{align*}
\end{exmp}

\begin{rmk} This Example coincides with Example 6.38 in \cite{STZ}. In next section,  we show that $\conf^b_d(\mathcal{B})$ is isomorphic to the moduli space $\mathcal{M}_1\left(\Lambda^b_d\right)$  of microlocal rank one sheaves in \textit{loc.cit.}.
\end{rmk}

\subsection{Legendrian links and Microlocal Rank-1 Sheaves}
\label{cadsbhvo}
Let us briefly recall some basic definitions about Legendrian links. The space $\mathbb{R}^3$ is equipped with the standard contact structure from the 1-form $\alpha= y {\rm d}x -{\rm d}z$. A {\it Legendrian link} in $\mathbb{R}^3$ is a link $\Lambda$ such that the restriction of $\alpha$ to $\Lambda$ vanishes. Two links are {\it Legendrian isotopic} if there is an isotopy between them that preserves the property of being Lengendrian at every stage. A Legendrian link $\Lambda$ can be visualized by its image under the {\it front projection} $\pi_F$ from $\mathbb{R}^3$  to the $xz$-plane. The constraint $\alpha |_\Lambda=0$ implies that the $y$ coordinate of $\Lambda$ is determined by the slope of its front projection. 

Shende, Treumann, and Zaslow \cite{STZ} have associated to every Legendrian link $\Lambda$ a category ${\bf Sh}_{\Lambda}^\bullet(\mathbb{R}^2)$ of constructible sheaves on the $xz$ plane with singular support controlled by the front projection of $\Lambda$. Using a theorem of Guillermou-Kashiwara-Schapira \cite{GKS}, they prove that the category ${\bf Sh}_{\Lambda}^\bullet(\mathbb{R}^2)$ is invariant under Legendrian isotopies. As a consequence, the moduli space $\mathcal{M}_1(\Lambda)$ of {\it microlocal rank one} sheaves in ${\bf Sh}_{\Lambda}^\bullet(\mathbb{R}^2)$ are Legendrian link invariants. 

\vskip 2mm

In this section we investigate Legendrian links arising from a pair $(b,d)$ of positive braids of Dynkin type $\mathrm{A}_r$. Let $\vec{i}$ and $\vec{j}$ be reduced words of $b$ and $d$ respectively. Associated to $(\vec{i}, \vec{j})$ is a Legendrian link $\Lambda_{\vec{j}}^{\vec{i}}$, whose front projection is described by the following steps.
\begin{enumerate}
\item We draw $2r+2$ many horizontal strands on the $xz$-plane. 
\item The top $r+1$ strands have crossings encoded by $\vec{i}$, and the bottom $r+1$ stands have crossings encoded by $\vec{j}$.
\item We close up both ends of the strands by cusps. 
\end{enumerate}

\begin{exmp} Let $r=2$. Let $\vec{i}=(1,2)$ and $\vec{j}=(1)$. The front projection $\pi_F(\Lambda_{\vec{j}}^{\vec{i}})$ of $\Lambda_{\vec{j}}^{\vec{i}}$ is 
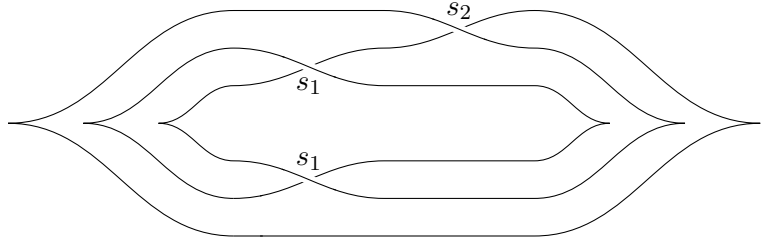
\[
\begin{tikzpicture}
\begin{knot}[
consider self intersections, 
ignore endpoint intersections=false, 
]
\strand (0,0) -- (4,0) to [out=0,in=180] (7,1.5) to [out=180,in=0] (4,3) to [out=180,in=0] (2,2.5) to [out=180, in=0] (0,2) to [out=180, in=0] (-1,1.5) to [out=0,in=180] (0,1) to [out=0,in=180] (2,0.5) to [out=0, in=180] (4,0.5) to [out=0, in =180] (6,1.5) to [out=180,in=0] (4,2.5) to [out=180, in=0] (2,3) to [out=180, in =0] (0,3) to [out=180, in=0] (-3,1.5) to [out=0,in=180] (0,0);
\strand (0,0.5) to [out=0,in=180] (2,1) to (4,1) to [out=0, in=180] (5,1.5) to [out=180, in=0] (4,2) to (2,2) to [out=180, in=0] (0,2.5) to [out=180, in=0] (-2,1.5) to [out=0,in=180] (0,0.5);
\flipcrossings{2,3};
\end{knot}
\node at (1,0.75) [above] {$s_1$};
\node at (1,2.25) [below] {$s_1$};
\node at (3,2.75) [above] {$s_2$};
\end{tikzpicture}
\]

\end{exmp}

 Connected components of the compliment of  $\pi_F(\Lambda_{\vec{j}}^{\vec{i}})$  are called {\it faces}. We use $f_{\rm in}$ to denote the face enclosed by the $(r+1)$-th and the $(r+2)$-th strands, and  $f_{\rm out}$ to denote the unbounded face.  Crossings and cusps of $\pi_F(\Lambda_{\vec{j}}^{\vec{i}})$ cut its strands into segments called {\it edges}.  Two faces are said to be \emph{neighboring} if they are separated by an edge $e$. Note that one of the neighboring faces is \emph{above} $e$ and the other is \emph{below} $e$ with respect to the  $z$-direction in the $xz$ plane.

Let us present an equivalent working definition of microlocal rank one sheaves associated to  $\Lambda_{\vec{j}}^{\vec{i}}$. See \cite{STZ} for the original definition.

\begin{defn}\label{microlocal} A \emph{microlocal rank-1 sheaf} $\mathcal{F}$ associated to $\Lambda^\vec{i}_\vec{j}$ consists of the following data 
\begin{itemize}
    \item assigned to every face $f$ is a finite dimensional vector space $V_f$ is over a field $\mathds{k}$; 
    \item assigned to every edge $e$ is a full-rank linear map $\phi_e:V_f\rightarrow V_g$, where $f$ and $g$ are neighboring faces separated by $e$, with $f$ sitting below $e$
\end{itemize}
such that 
\begin{itemize}
    \item the dimensions of vector spaces assigned to any neighboring faces differ by $1$;
    \item $\dim V_{f_\text{in}}=r+1$ and $\dim V_{f_\text{out}}=0$;
    
    \item for every crossing illustrated on the left below, the following sequence is exact
    \begin{equation}
    \label{vfdhoqwd}
    \xymatrix{V_s\ar[rr]^{\left(\phi_{sw},\phi_{se}\right)} & & V_w\oplus V_e \ar[rr]^{\phi_{wn}-\phi_{en}} & & V_n}
    \end{equation}
    \item for every cusp illustrated on the right below, $\phi_{fg}\circ \phi_{gf}=\mathrm{id}_{V_f}$.
\end{itemize}
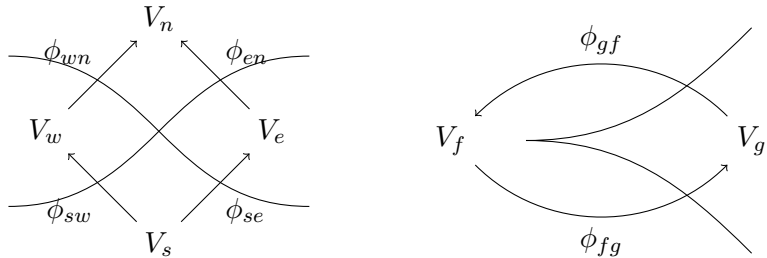
\[
\begin{tikzpicture}
\draw (-2,-1) to [out=0,in=-135] (0,0) to [out=45,in=180] (2,1);
\draw (-2,1) to [out=0,in =135] (0,0) to [out=-45,in=180] (2,-1);
\node (s) at (0,-1.5) [] {$V_s$};
\node (e) at (1.5,0) [] {$V_e$};
\node (w) at (-1.5,0) [] {$V_w$};
\node (n) at (0,1.5) [] {$V_n$};
\draw [->] (s) -- node [below right] {$\phi_{se}$} (e);
\draw [->] (s) -- node [below left] {$\phi_{sw}$} (w);
\draw [->] (e) -- node [above right] {$\phi_{en}$} (n);
\draw [->] (w) --  node [above left] {$\phi_{wn}$} (n);
\end{tikzpicture} \quad \quad \quad \quad
\begin{tikzpicture}
\draw (2,1.5) to [out=-135, in =0] (-1,0) to [out=0, in =135] (2,-1.5);
\node (f) at (-2,0) [] {$V_f$};
\node (g) at (2,0) [] {$V_g$};
\draw [->] (f) to [out=-45, in=-135] node [below] {$\phi_{fg}$} (g);
\draw [->] (g) to [out=135,in =45] node [above] {$\phi_{gf}$} (f);
\end{tikzpicture}
\]
\end{defn}

One can think of a microlocal rank-1 sheaf as a quiver representation, and two microlocal rank-1 sheaves are \emph{isomorphic} if they are isomorphic as quiver representations. Let $\mathcal{M}_1(\Lambda^{\vec{i}}_{\vec{j}})$  be the moduli space of isomorphism classes of microlocal rank-1 sheaves associated to $\Lambda^\vec{i}_\vec{j}$. The space $\mathcal{M}_1(\Lambda^{\vec{i}}_{\vec{j}})$ is invariant under Legendrian isotopies.

\begin{figure}
\begin{center}
\begin{tikzpicture}
\draw (0,0) -- (2,0);
\draw (0,1) -- (2,1);
\draw (0,2) -- (2,2);
\draw (0,3) -- (2,3);
\node (0) at (1,-0.5) [] {$V_0=V_{f_{\text{in}}}$};
\node (1) at (1,0.5) [] {$V_1$};
\node (2) at (1,1.5) [] {$\vdots$};
\node (3) at (1,2.5) [] {$V_r$};
\node (4) at (1,3.5) [] {$V_{r+1}=V_{f_{\text{out}}}$};
\draw [->] (0) -- (1);
\draw [->] (1) -- (2);
\draw [->] (2) -- (3);
\draw [->] (3) -- (4);
\end{tikzpicture}\quad \quad \quad \quad
\begin{tikzpicture}
\draw (0,0) -- (2,0);
\draw (0,1) -- (2,1);
\draw (0,2) -- (2,2);
\draw (0,3) -- (2,3);
\node (0) at (1,-0.5) [] {$V_{r+1}=V_{f_{\text{out}}}$};
\node (1) at (1,0.5) [] {$V_r$};
\node (2) at (1,1.5) [] {$\vdots$};
\node (3) at (1,2.5) [] {$V_1$};
\node (4) at (1,3.5) [] {$V_0=V_{f_{\text{in}}}$};
\draw [->] (0) -- (1);
\draw [->] (1) -- (2);
\draw [->] (2) -- (3);
\draw [->] (3) -- (4);
\end{tikzpicture}
\quad \quad \quad \quad
\begin{tikzpicture}
\draw (4,4) to [out =180, in=0] (0,2)  to [out=0,in=180] (4,0);
\draw (4,3.5) to [out=180,in=0] (1.5,2) to [out=0,in=180] (4,0.5);
\draw (4,3) to [out=180, in=0] (3,2) to [out=0,in=180] (4,1);
\node (0) at (-0.5,2) [] {$V_0$};
\node (1) at (1.25,2) [] {$V_1$};
\node (2) at (2.5,2) [] {$\cdots$};
\node (3) at (4,2) [] {$V_{r+1}$};
\draw [->] (0) to [out=-45,in=-135] (1); 
\draw [->] (1) to [out=-45,in=-135] (2); 
\draw [->] (2) to [out=-45,in=-135] (3); 
\draw [->] (3) to [out=135,in=45] (2); 
\draw [->] (2) to [out=135,in=45] (1); 
\draw [->] (1) to [out=135,in=45] (0); 
\end{tikzpicture}
\end{center}
\caption{Flags obtained from microlocal rank one sheaves.}
\label{dnciweqcn}
\end{figure}
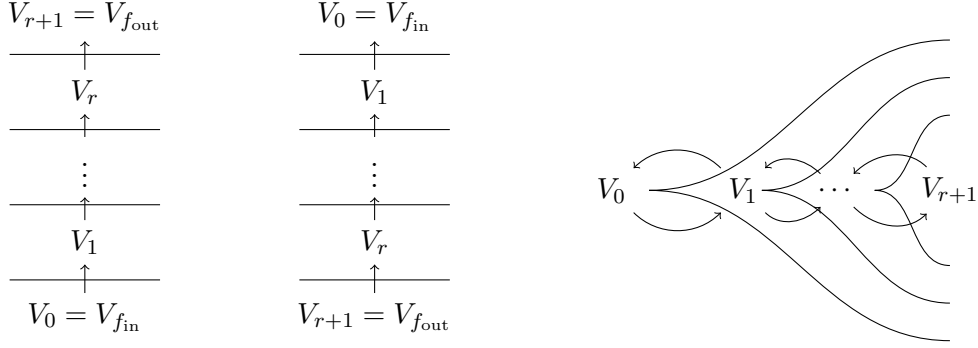

\begin{thm} 
\label{caosdbvco}
Let $\vec{i}$ and $\vec{j}$ be reduced words of positive braids $b$ and $d$ respectively. There exists a natural isomorphism from $\mathcal{M}_1(\Lambda^{\vec{i}}_{\vec{j}})$ to $\conf_d^b(\mathcal{B})$.
\end{thm}

\begin{proof} Let $\mathcal{F}$ be a a microlocal rank one sheaf. Let us align the cusps of $\pi_F(\Lambda^\vec{i}_\vec{j})$ on a horizontal line, which separates the $xz$ plane into two halves. We slice up the bottom half vertically at the crossings. 
For each slice, the sheave $\mathcal{F}$ gives rise to $r+2$ vector spaces with full-ranked linear maps connecting them. See the first graph in Figure \ref{dnciweqcn}. 
Let 
\[
U_i:={\rm Im}(V_i\rightarrow \ldots \rightarrow V_{r+1}) 
\]
By Definition \ref{microlocal}, every local linear map is a codimension 1 inclusion. 
Therefore we obtain a complete flag $\B =\left(0=U_0\subset U_1\subset \dots \subset U_{r+1}=V_{f_{\rm in}}\right)$.

Let $\B=\left(U_0\subset U_1\subset \dots\subset U_r \subset U_{r+1}\right)$ and $\B'=\left(U'_0\subset U'_1\subset \dots \subset U'_r\subset U_{r+1}'\right)$ be  flags associated to  two adjacent slices sharing a crossing on the $j$th level. By the exactness condition \eqref{vfdhoqwd}, we get $U_j\neq U_j'$ and $U_i=U'_i$ if $i\neq j$. Therefore $\xymatrix{\B\ar[r]^{s_j}&\B'}$ in  $\mathcal{B}_-$. Since the crossings of the strands in the bottom half are encoded by $\vec{j}$, by associating flags to the slices, we get a $\vec{j}$-chain 
\[
\xymatrix{\B_0 \ar[r]^{s_{j_1}} & \B_1\ar[r]^{s_{j_2}} & \dots \ar[r]^{s_{j_n}} & \B_n}.
\]

Similarly let us slice up the top half vertically at the crossings. For each slice, we get the data as shown on the  the second graph in Figure \ref{dnciweqcn}). The dimension condition forces $\dim V_i=i$ and the linear maps are all surjective. Let 
\[
W_i:=\ker\left(V_{r+1}\rightarrow V_r\rightarrow \dots \rightarrow V_i\right).
\]
It determines a complete flag $\B=\left(0=W_{r+1}\subset W_r\subset \dots W_0=V_{f_{\text{in}}} \right)$. Similarly, a crossing on the $i$th level at the top half imposes the condition that two adjacent flags are of Tits codistance $s_i$ in the flag variety $\mathcal{B}_+$. Therefore  the top half of a microlocal rank-1 sheaf $\mathcal{F}$  gives rise to an $\vec{i}$-chain 
\[
\xymatrix{\B^0\ar[r]^{s_{i_1}} & \B^1\ar[r]^{s_{i_2}} & \dots \ar[r]^{s_{i_m}} & \B^m}.
\]

Now we have obtained  an $\vec{i}$-chain and a $\vec{j}$-chain from  $\mathcal{F}$. From the construction of the braid closure, we know that the left cusps are nested as in the third graph of Figure \ref{dnciweqcn}. By Definition \ref{microlocal}, the composition $V_i\rightarrow V_{r+1}\rightarrow V_i$ is the identity map $\mathrm{id}_{V_i}$. Therefore
\[
U_i \cap W_i := {\rm im}(V_i\rightarrow V_{r+1}) \cap {\rm ker}(V_{r+1}\rightarrow V_i)=0
\]
Therefore the flags $\B^0$ and $\B_0$ are in general position. Similarly the flags $\B^m$ and $\B_n$ are in general position. Putting them together, we get a configuration 
\[
\xymatrix{\B^0 \ar[r]^{s_{i_1}} \ar@{-}[d] &\B^1 \ar[r]^{s_{i_2}} & \dots \ar[r]^{s_{i_m}} &   \B^m \ar@{-}[d] \\ \B_0 \ar[r]_{s_{j_1}} & \B_1 \ar[r]_{s_{j_2}} & \dots \ar[r]_{s_{j_n}} & \B_n}
\]
This defines a map 
\[
\mathcal{M}_1(\Lambda^\vec{i}_\vec{j})\rightarrow \conf^b_d(\mathcal{B}),
\]
It is not hard to show that this construction can be reversed to get an isomorphism class of microlocal rank one sheaves from a point in $\conf^b_d(\mathcal{B})$. 
\end{proof}

Theorem \ref{caosdbvco} is a slight generalization of \cite[Prop 1.5]{STZ}. It shows that $\mathcal{M}_1(\Lambda^{\vec{i}}_{\vec{j}})$ is equipped with a natural cluster Poisson structure. As a direct consequence, we obtain that 

\begin{cor}\label{link_inv'} The space $\conf^b_d(\mathcal{B})$ (as an algebraic stack) and $g^b_d(q)$ are Legendrian link invariants for closures of positive braids $(b,d)$.
\end{cor}

\renewcommand{\thesection}{A}

\section{Appendix}

\subsection{Basics about Kac-Peterson Groups}
\label{appen.a}

In this appendix we collect the necessary information on Kac-Peterson Groups (a.k.a. \emph{minimal} Kac-Moody groups). We will mostly follow S. Kumar \cite{Kum}.

A \emph{generalized Cartan matrix} is a matrix $\C=(\C_{ij})$ whose diagonal entries are all 2, and whose off-diagonal entries are non-positive integers, such that $\C_{ij}=0$ if and only if $\C_{ji}=0$. 

A \emph{realization} of an $r\times r$ generalized Cartan matrix $\C$ is a quadruple $\left(\mathfrak{h},\mathfrak{h}^*,\Pi^\vee, \Pi\right)$ such that 
\begin{itemize}
    \item $\mathfrak{h}$ and $\mathfrak{h}^*$ are dual complex vector spaces of dimension $\tilde{r}:=r+l$, where $l$ is the corank of $\C$;
    \item $\Pi=\left\{\alpha_1,\dots, \alpha_r\right\}\subset \mathfrak{h}^*$ and $\Pi^\vee=\left\{\alpha_1^\vee,\dots, \alpha_r^\vee\right\}\subset \mathfrak{h}$ are linearly independent subsets;
    \item $\inprod{\alpha_i^\vee}{\alpha_j}=\C_{ij}$ for $i, j=1, \ldots, r$.
\end{itemize}
Every $\C$ admits a unique up to isomorphism realization (\cite[Prop. 1.1]{Kac}).

The Kac-Moody algebra $\mathfrak{g}_\C$ associated to $\C$  is a  Lie algebra, with the generators  $E_{i}, E_{-i}~(i=1, \ldots, r)$ and $\mathfrak{h}$, and the  relations
\begin{equation}
\label{rel.kac-moody}
   \left\{\begin{array}{ll}
    \left[H,H'\right]=0 &  (H, H'\in \mathfrak{h}), \\
    \left[H,E_{i}\right]=\inprod{H}{\alpha_i}E_{i}, & \\
    \left[H,E_{- i}\right]=-\inprod{H}{\alpha_i}E_{i} & (i=1,\ldots r;~ H\in \mathfrak{h}),\\
\left[E_{i}, E_{- j}\right]= \delta_{ij}\alpha_i^\vee& (i,j=1,\ldots, r),\\
    \ad_{E_{ i}}^{1-\C_{ij}}E_{ j}=0, &\\
    \ad_{E_{- i}}^{1-\C_{ij}}E_{-j}=0 & (i\neq j). 
\end{array}\right.
\end{equation}

From now on, let $\C$ be a \emph{symmetrizable} generalized Cartan matrix, that is, there is an invertible diagonal matrix $\D$ such that $\D^{-1}\C$ is symmetric.
The matrix $\D$ may be chosen such that its diagonal entries are positive integers with $\gcd=1$. Let $(\mathfrak{h}, \mathfrak{h}^\ast, \Pi^\vee, \Pi)$ be a realization of $\C$. We further fix once and for all a lattice $\P\subset \mathfrak{h}^*$ with a basis $\left\{\omega_1,\dots, \omega_{\tilde{r}}\right\}$ such that $\Pi\subset \P$ and \[\inprod{\alpha_i^\vee}{\omega_j}=\delta_{ij} \hskip 7mm \text{for $i=1,\ldots, r$ and $j=1, \ldots, \tilde{r}$.}
\]
The lattice $\P$ is called the \emph{weight lattice}. A weight $\lambda\in \P$ is \emph{dominant} if $\inprod{\alpha_i^\vee}{\lambda}\geq 0$ for every  $\alpha_i^\vee\in \Pi$. Denote by $\P_+$ the set of dominant weights. 

The elements  $\omega_1, \ldots, \omega_{\tilde{r}}$ are called the \emph{fundamental} weights. 
They extend  $\C$  to a $\tilde{r}\times \tilde{r}$ matrix 
\[
\tilde{\C}=(\C_{ij})=\begin{pmatrix}
\C &\D\A\\
\A & 0\\
\end{pmatrix}
\]
such that $\alpha_j=\sum_{i=1}^{\tilde{r}}{\C}_{ij}\omega_i$ for $j=1, \ldots, r$.

\begin{lem} The matrix $\tilde{\C}$ is invertible.
\end{lem}
\begin{proof} 
The matrix $\C$ is of corank $l$. We may apply elementary column transformations to the first $r$ columns of $\tilde{\C}$, and obtain a matrix
\[
\tilde{\C'}=\begin{pmatrix}
0 & \C' &\D\A\\
\L & \ast & 0\\
\end{pmatrix}.
\]
 The first $r$ column vectors of $\tilde{\C}$ are linearly independent because $\Pi$ is a linearly independent subset of $\mathfrak{h}^\ast$. Therefore the $l\times l$ submatrix $\L$ is invertible. Meanwhile the matrix $\tilde{\C}$ is symmetrizable. Therefore the first $r$ row vectors   of $\tilde{\C}$ are linearly independent, and the $r\times r$ submatrix $(\C'~~\D\A)$ of $\tilde{\C}'$ is invertible. It follows then $\tilde{\C'}$ (and hence $\tilde{\C}$) is invertible. 
\end{proof}

Using the matrix $\tilde{\C}$ we can extend $\Pi$ to a basis $\left\{\alpha_i\right\}_{i=1}^{\tilde{r}}$ of $\mathfrak{h}^*$ such that
\[
\alpha_j=\sum_{i=1}^{\tilde{r}}\C_{ij}\omega_i.
\]
Let $\left\{\alpha_i^\vee\right\}_{i=1}^{\tilde{r}}$ and $\left\{\omega_i^\vee\right\}_{i=1}^{\tilde{r}}$ be, respectively, the dual basis of $\left\{\omega_i\right\}_{i=1}^{\tilde{r}}$ and  $\left\{\alpha_i\right\}_{i=1}^{\tilde{r}}$. Then 
\[
\alpha_i^\vee=\sum_{j=1}^{\tilde{r}}\C_{ij} \omega_j^\vee
\hskip 7mm 
\text{and}
\hskip 7mm
\inprod{\alpha_i^\vee}{\alpha_j}=\C_{ij} \quad (i, j =1, \ldots, \tilde{r}).
\]
Define $\Q:=\bigoplus_{i=1}^{\tilde{r}} \mathbb{Z}\alpha_i \subset \P$.  The quotient group  $\P/\Q$ is a finite abelian group of order $|\det(\tilde{\C})|$.

We define two algebraic tori
\[
\T_\sc:=\Hom\left(\P,\mathbb{G}_m\right) \quad \text{and} \quad \T_\ad:=\Hom\left(\Q,\mathbb{G}_m\right).
\]
Both $\T_\sc$ and $\T_\ad$ have $\mathfrak{h}$  as their Lie algebras. The embedding $\Q\subset \P$ induces a surjective homomorphism from $\T_\sc$ to $\T_\ad$, whose kernel $\Z$ is isomorphic to  $\P/\Q$.

 The \emph{Kac-Peterson group} $\G_\sc$ (resp. $\G_\ad$) is generated by $\T_\sc$ (resp. $\T_\ad$) and the one-parameter groups
\[
\U_i:=\left\{\exp\left(pE_{i}\right) \ \middle| \ p\in \mathbb{G}_a\right\}, \hskip 7mm i\in \{\pm1, \ldots, \pm r\},
\]
 with relations determined by \eqref{rel.kac-moody}.  They are also known as the \emph{minimal} Kac-Moody groups, in the sense that they are constructed by only exponentiating the \emph{real} root spaces of $\mathfrak{g}_\C$.
The group $\Z$ is contained in the center of $\G_\sc$. The surjection from $\T_\sc$ to $\T_\ad$ induces an $|\Z|$-to-1 covering  map
\[
\pi: \G_\sc\longrightarrow \G_\ad.
\]
 We refer the reader to \cite[Section 7.4]{Kum} for more details on Kac-Peterson groups. 
 
\begin{notn}
We will write $e_i(p)$ instead of $\exp\left(pE_{i}\right)$ and omit the argument $p$ if $p=1$.
 Let $\T$ be either $\T_\sc$ or $\T_\ad$. For a character $\lambda$ of $\T$ and $t\in \T$, we set $t^\lambda:=\lambda(t)$. For a cocharacter $\lambda^\vee$ of $\T$ and $p\in \mathbb{G}_m$, we set $p^{\lambda^\vee}:=\lambda^\vee(p)$.
\end{notn}

Let $\G$ be either $\G_\sc$ or $\G_\ad$. 
Let $\N$ be the normalizer of $\T$ in $\G$.  The \emph{Weyl group} 
$\W:= \N/ \T$ is generated by $\S:=\left\{s_i\right\}_{i=1}^r$ with the relations $s_i^2=1$ for all $i$ together with the \emph{braid relations}
\begin{equation}\label{braid rel}
\underbrace{s_is_j\dots}_{m_{ij}}=\underbrace{s_js_i\dots}_{m_{ij}} \hskip 10mm \forall i\neq j
\end{equation}
where $m_{ij}=2,3,4,6$ or $\infty$ according to whether $\C_{ij}\C_{ji}$ is $0,1,2,3$ or $\geq 4$. The elements
\begin{equation}\label{s_i}
\overline{s}_i:=e_i^{-1}e_{-i}e_i^{-1} \quad \text{and} \quad \doverline{s}_i:=e_ie_{-i}^{-1}e_i.
\end{equation}
are both coset representatives of $s_i\in \N/\T$.
They satisfy the braid relations, and therefore determine two natural representatives for every $w\in \N/\T$, which are denoted as $\overline{w}$ and $\doverline{w}$. 

Let $\U_+$ (resp. $\U_-$) be the subgroup of $\G$ generated by $\U_1, \ldots, \U_r$ (resp. $\U_{-1}, \ldots, \U_{-r}$). Define the \emph{Borel} subgroups 
\[
\B_+:= \U_+\T, \hskip 10mm \B_-:=\U_-\T.
\]
The \emph{transposition} $g \mapsto g^t$ is an involutive anti-automorphism of $\G$ such that
\[
e_{i}(p)^t:=e_{-i}(p) \hskip 7mm \forall i \in {\pm1, \ldots, \pm r}; \hskip 10mm \quad h^t=h \quad \forall h \in \T.
\]
The transposition swaps $\B_+$ and $\B_-$.

Recall the definition of \emph{Tits system} in \cite[Section 5.1]{Kum}. The tuple $\left(\G, \B_+, \B_-, \N, \S\right)$ form a \emph{twin Tits system}, that is,
\begin{itemize}
    \item the quadruples $(\G, \B_+, \N, \S)$  and $(\G, \B_-, \N, \S)$ are Tits systems;
    \item if $l(ws_i)<l(w)$ then $\B_-w\B_+s_i\B_+=\B_-ws_i\B_+$;
    \item $\B_-s_i\cap \B_+=\emptyset$ for $i=1, \ldots, r$.
\end{itemize}
From the twin Tits system we obtain two Borel decompositions and a Birkhoff decomposition
\[
\G=\bigsqcup_{u\in \W}\B_+ u\B_+= \bigsqcup_{v\in \W}\B_- v\B_-,
\hskip 14mm 
\G=\bigsqcup_{w\in \W}\B_-w\B_+.
\]
Recall the flag varieties $\mathcal{B}_+=\G/\B_+$ and $\mathcal{B}_-=\G/\B_-$. As in Section \ref{flags}, the above decompositions induce Tits distance function $d_\pm:\mathcal{B}_\pm\times \mathcal{B}_\pm\rightarrow \W$ and a Tits codistance function $d:\mathcal{B}_+\times\mathcal{B}_-\rightarrow \W$. The quintuple $\left(\B_\pm, d_\pm, d\right)$ is an example of \emph{twin buildings}. 

\vspace{11pt}

\noindent \textit{Proof of Lemma \ref{unique}.} 
For the first case, without loss of generality let us assume that $\B, \B', \B''\in \mathcal{B}_+$. From the assumption $uv=w$ and $l(u)+l(v)=l(w)$ we get 
\[
\B_+w\B_+=\B_+u\B_+v\B_+.
\]
Therefore if $\xymatrix{\B \ar[r]^u & \B' \ar[r]^v &\B''}$ then $\xymatrix{\B\ar[r]^w & \B''}$. Conversely, if $\xymatrix{\B \ar[r]^w & \B''}$ then there exists a flag $\B'$ such that $\xymatrix{\B \ar[r]^u & \B' \ar[r]^v &\B''}$. It remains to show the uniqueness of $\B'$.

Assume $v=s_i$. Let $\B'''$ satisfy $\xymatrix{\B \ar[r]^u & \B''' \ar[r]^{s_i} &\B''}$. Note that $\B_+s_i\B_+s_i\B_+=\B_+\sqcup \B_+s_i\B_+$. If $\B'\neq \B'''$, then we get $\xymatrix{\B' \ar[r]^{s_i} & \B'''}$. Putting all the flags together, we get
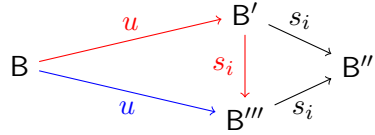
\[
\begin{tikzpicture}
\node (0) at (-1,0) [] {$\B$};
\node (1) at (2,0.7) [] {$\B'$};
\node (2) at (3.5,0) [] {$\B''$};
\node (3) at (2,-0.7) [] {$\B'''$};
\draw [->, red] (0) --  node [above] {$u$} (1);
\draw [->, blue] (0) -- node [below] {$u$} (3);
\draw [->] (1) -- node [above] {$s_i$} (2);
\draw [->] (3) -- node [below] {$s_i$} (2);
\draw [->, red] (1) -- node [left] {$s_i$} (3);
\end{tikzpicture}
\]
Since $l(us_i)=l(u)+1$,  from the red arrows we get $\xymatrix{\B \ar[r]^{us_i} & \B'''}$, which contradicts with the blue arrow $\xymatrix{\B \ar[r]^u & \B'''}$. Therefore $\B'=\B'''$.
For general $v$, we first fix a reduced word $v=s_{i_1}s_{i_2} \dots s_{i_l}$. By the above discussion there is a unique $\B'''$ such that $\xymatrix{\B \ar[r]^{ws_{i_l}} & \B''' \ar[r]^{s_{i_l}} & \B''}$. Then we focus on $\xymatrix{\B \ar[r]^{ws_{i_l}} & \B'''}$. The uniqueness of $\B'$ follows by induction on the length of $v$.

For the third case, using the second condition of twin Tits system recursively, we get 
\[
\B_-u\B_+=\B_-w\B_+v^{-1}\B_+.
\]
Therefore if $\xymatrix{\B_0 \ar@{-}[r]^w & \B^{-1} \ar[r]^{v^{-1}}& \B^0}$ then $\xymatrix{\B_0 \ar@{-}[r]^u & \B^0}$. Conversely,  if $\xymatrix{\B_0 \ar@{-}[r]^u & \B^0}$ then there exists a $\B^{-1}$ such that $\xymatrix{\B_0 \ar@{-}[r]^w & \B^{-1} \ar[r]^{v^{-1}}& \B^0}$. The uniqueness of $\B^{-1}$ follows from the same inductive method as in the proof of the first. 

All other cases are analogous to the third case.
\qed

\vspace{11pt}

Next let us investigate the space of flags that are of Tits distance $s_i$ from a fixed flag.

\begin{lem}\label{moduli of tits codist si}  If $\xymatrix{\B_+\ar[r]^{s_i}& \B}$, then $\B=e_i(q)\overline{s}_i\B_+$ for some $q\in \mathbb{A}^1$.
\end{lem}
\begin{proof} The space of flags of Tits distance $s_i$ to $\B_+$ is the quotient $\left.\left(\B_+s_i\B_+\right)\right/\B_+$. 
By Lemma 6.1.3 of \cite{Kum}, we get $\B_+=\U_i\Q_i$, where the subgroup $\Q_i=\B_+\cap {s}_i\B_+s_i$. Therefore 
\[
\left.\left(\B_+s_i\B_+\right)\right/\B_+=\left.\left(\U_i\Q_is_i\B_+\right)\right/\B_+=\left.\left(\U_is_i\Q_i\B_+\right)\right/\B_+=\left.\left(\U_is_i\B_+\right)\right/\B_+=\left\{e_i(q)\overline{s}_i\B_+\right\}. \qedhere
\]
\end{proof}

\begin{cor}\label{2.30} Let $\B$ be in either $\mathcal{B}_+$ or $\mathcal{B}_-$. The space of flags of Tits distance $s_i$ away from $\B$ is isomorphic to $\mathbb{A}^1$.
\end{cor}

\begin{prop}\label{3.3} Let $u, v\in \W$  and let  $s_i$ be a simple reflection.  
Fix a pair $\xymatrix{\B_0 \ar@{-}[r]^u & \B^0}$. Then the space of flags $\B$ that fits into either of the triangles
\[
\begin{tikzpicture}[baseline=20]
\node (u) at (0,2) [] {$\B^0$};
\node (d0) at (-1,0) [] {$\B_0$};
\node (d1) at (1,0) [] {$\B$};
\draw (d0) -- node [above left] {$u$} (u) -- node [above right] {$v$} (d1);
\draw [->] (d0) -- node [below] {$s_i$} (d1); 
\end{tikzpicture} \quad \quad \text{and} \quad \quad \begin{tikzpicture}[baseline=20]
\node (u) at (0,0) [] {$\B_0$};\node (d0) at (-1,2) [] {$\B$};\node (d1) at (1,2) [] {$\B^0$};\draw (d0) -- node [below left] {$v$} (u) -- node [below right] {$u$} (d1);\draw [->] (d0) -- node [above] {$s_i$} (d1); \end{tikzpicture}
\]
is isomorphic to 
\[
\left\{\begin{array}{ll}
    \mathbb{A}^1 & \text{if $v=s_iu$ and $l\left(s_iu\right)<l(u)$,} \\
    \mathbb{G}_m & \text{if $v=u$ and $l\left(s_iu\right)>l(u)$,} \\
    \{*\} & \text{if $v=s_iu$ and $l\left(s_iu\right)>l(u)$,}\\
    \emptyset & \text{otherwise.}
\end{array}\right.
\]
\end{prop}
\begin{proof} 
By symmetry, we will only prove the first case. Without loss of generality let $\B^0=u\B_+$ and $\B_0=\B_-$. The set of flags of Tits distance $s_i$ to $\B_-$ is  
\begin{equation}
\label{wfrge}
\left\{\B_-e_i(p) \mid p\in \mathbb{G}_m \right\} \sqcup \left\{  \B_-s_i\right\}
\end{equation}

If $l\left(s_iu\right)< l(u)$, then we obtain the the first case by  Lemma \ref{unique} and Corollary \ref{2.30}.

If $l\left(s_iu\right)> l(u)$, then the root $\alpha:=u^{-1}\left(\alpha_i\right)$ is positive. Therefore
\[
\B_-e_i(p)u\B_+= \B_-u e_\alpha(p')\B_+ = \B_- u \B_+.
\] 
Among all the flags in \eqref{wfrge}, only $\B_-s_i$ is of Tits codistance $s_iu$ away from $\B_-u$, from which we arrive at the third case.  The rest are of Tits codistance $u$, which proves the second case.
\end{proof}

Now let us focus on $\G_\sc$. Let $V_\lambda$ denote the irreducible representation of $\G_\sc$ of highest weight $\lambda\in \P_+$. Let $\mathcal{O}\left[\G_\sc\right]$ be the algebra generated by the matrix coefficients of $V_\lambda$, $\lambda\in \P_+$. 

\begin{thm}[Kac-Peterson, {\cite[Theorem 1]{KP}}] \label{Kac-Peterson} Consider the $\G_\sc\times \G_\sc$-action on $\mathcal{O}\left[\G_\sc\right]$ by \[\left(\left(g_1,g_2\right).f\right)(g):=f\left(g_1^{-1}gg_2\right).
\]
Then as $\G_\sc\times \G_\sc$-modules, \[\mathcal{O}\left[\G_\sc\right]\cong \bigoplus_{\lambda\in \P_+}V^*_\lambda\otimes V_\lambda.
\]
\end{thm}

Let $^{\U_-}\mathcal{O}\left[\G_\sc\right]$ be the subring of left $\U_-$ invariant functions. By Theorem \ref{Kac-Peterson}, we get  \[^{\U_-}\mathcal{O}\left[\G_\sc\right]\cong \bigoplus_{\lambda\in \P_+}V_\lambda.\]
Let  $\Delta_\lambda\in V_\lambda$ be the unique highest weight vector such that $\Delta_\lambda(e)=1$. 
Given $\lambda, \mu\in \P_+$, the product $\Delta_\lambda\Delta_\mu$ is a highest weight vector in $V_{\lambda+\mu}$ and satisfy the normalization condition. Therefore
\[
\Delta_{\lambda+\mu}=\Delta_\lambda\Delta_\mu.
\]

\begin{thm}[Geiss-Leclerc-Schr\"{o}er \cite{GLS}, 7.2]\label{gls} An element $x\in \G_\sc$ is Gaussian decomposable if and only if $\Delta_{\omega_i}\left(x\right)\neq 0$ for all fundamental weights $\omega_i$. 
\end{thm}

\subsection{Basics about Cluster Algebras}\label{app B}

We include here the basic facts about cluster algebras that will be needed.

\begin{defn} A \emph{seed} is a quadruple $\vec{s}=\left(I, I^\uf, \epsilon_{ab}, \left\{d_a\right\}_{a\in I}\right)$ satisfying the following properties:
\begin{enumerate}
\item $I$ is a finite set and $I^\uf\subset I$;
\item $\epsilon_{ab}$ is a $\mathbb{Q}$-valued matrix with $\epsilon_{ab}\in \mathbb{Z}$ unless $(a,b)\in I^{uf}\times I^{uf}$; 
\item $\left\{d_a\right\}$ is a collection of positive integers with $\gcd\left(d_a\right)=1$ such that the matrix
$
\hat{\epsilon}_{ab}:=\epsilon_{ab}d_b^{-1}
$
is skewsymmetric.
\end{enumerate}
Elements of $I$ are called \emph{vertices}, elements of $I^\uf$ are called \emph{unfrozen vertices}, and elements of $I\setminus I^\uf$ are called \emph{frozen vertices}. The matrix $\epsilon$ is called the \emph{exchange matrix} and  $d_a$ are called \emph{multipliers}.
\end{defn}

\begin{defn} Given a seed $\vec{s}$ and an unfrozen vertex $c\in I^\uf$, a \emph{mutation} at $c$ produces a new seed $\vec{s}'=\mu_c\vec{s}=\left(I', I'^\uf, \epsilon_{ab}', \{d'_a\}\right)$ with $I'=I$, $I'^\uf=I^\uf$, $d'_a=d_a$, and 
\[
\epsilon'_{ab}=\left\{\begin{array}{ll} -\epsilon_{ab} & \text{if $c\in \{a,b\}$},\\
\epsilon_{ab}+\left[\epsilon_{ac}\right]_+\left[\epsilon_{cb}\right]_+-\left[-\epsilon_{ac}\right]_+\left[-\epsilon_{cb}\right]_+ & \text{if $c\notin\{a,b\}$},\end{array}\right.
\]
where $[x]_+:=\max\{x,0\}$.
 Seeds obtained by a sequence of mutations on $\vec{s}$ are said to be \emph{mutation equivalent} to $\vec{s}$.
\end{defn}

Let $\mathbb{T}$ be an $\left|I^\uf\right|$-regular tree. The edges of $\mathbb{T}$ are labeled by elements of $I^\uf$ such that the labeling of edges connecting to the same vertex are distinct. Every mutation is involutive: $\mu_c^2\vec{s}=\vec{s}$. Therefore we can associate the vertices of $\mathbb{T}$ with seeds from a mutation equivalent family, such that any two vertices associated to a pair of seeds related by a mutation $\mu_c$ are joint by an edge labeled by $c$.  The decorated tree $\mathbb{T}$ is called the \emph{mutation tree} of $\vec{s}$. 

We assign to each vertex $\vec{s}$ of $\mathbb{T}$  two split $|I|$-dimensional algebraic tori: a \emph{$\mathrm{K}_2$ seed torus} $T_{\mathscr{A};\vec{s}}$ and a \emph{Poisson seed torus} $T_{\mathscr{X};\vec{s}}$. The tori $T_{\mathscr{A};\vec{s}}$ and $T_{\mathscr{X};\vec{s}}$ are equipped with coordinate systems $\left\{A_{a;\vec{s}}\right\}_{a\in I}$ and  $\left\{X_{a;\vec{s}}\right\}_{a\in I}$ respectively. We  often drop the subscript $;\vec{s}$ if it is obvious or not important.
The torus $T_{\mathscr{A};\vec{s}}$ admits a canonical 2-form 
\[
\Omega=\sum_{a,b}\hat{\epsilon}_{ab;\vec{s}} \frac{dA_{a;\vec{s}}}{A_{a;\vec{s}}}\wedge \frac{dA_{b;\vec{s}}}{A_{b;\vec{s}}},
\]
The torus $T_{\mathscr{X};\vec{s}}$ admits a canonical Poisson structure determined by the bivector field
\[
\Pi=\sum_{a,b}\hat{\epsilon}_{ab;\vec{s}}X_{a;\vec{s}}X_{b;\vec{s}}\frac{\partial}{\partial X_{a;\vec{s}}}\wedge \frac{\partial }{\partial X_{b;\vec{s}}}.
\]

For every edge $\xymatrix{\vec{s} \ar@{-}[r]^c & \vec{s}'}$ in $\mathbb{T}$,  the associated  seed tori are related by the transition maps:
\[
\begin{tikzpicture}[baseline=-0.5ex] 
\node (s) at (0,0) [] {$T_{\mathscr{A};\vec{s}}$};
\node (s') at (2,0) [] {$T_{\mathscr{A};\vec{s}'}$};
\draw [dashed, <->] (s) --  node [above] {$\mu_c$} (s');
\end{tikzpicture} \quad \text{and} \quad \begin{tikzpicture}[baseline=-0.5ex] 
\node (s) at (0,0) [] {$T_{\mathscr{X};\vec{s}}$};
\node (s') at (2,0) [] {$T_{\mathscr{X};\vec{s}'}$};
\draw [dashed, <->] (s) --  node [above] {$\mu_c$} (s');
\end{tikzpicture}
\]
In terms of the cluster coordinates, the transition maps are expressed as:
\begin{equation}\label{a mutation}
\mu_c^*\left(A_{a;\vec{s}'}\right):=\left\{\begin{array}{ll}
 A_{c;\vec{s}}^{-1}\left(\prod_b A_{b;\vec{s}}^{\left[-\epsilon_{cb;\vec{s}}\right]_+}\right)\left(1+\prod_b A_{b;\vec{s}}^{\epsilon_{cb;\vec{s}}}\right) & \text{if $a=c$,}\\
 A_{a;\vec{s}} & \text{if $a\neq c$,}
 \end{array}\right.
\end{equation}
\begin{equation}\label{x mutation}
\mu_c^*\left(X_{a;\vec{s}'}\right):=\left\{\begin{array}{ll} X_{c;\vec{s}}^{-1} & \text{if $a=c$,}\\
X_{a;\vec{s}}X_{c;\vec{s}}^{\left[\epsilon_{ac;\vec{s}}\right]_+}\left(1+X_{c;\vec{s}}\right)^{-\epsilon_{ac;\vec{s}}}& \text{if $a\neq c$.}\end{array}\right.
\end{equation}
The maps $\mu_c$ preserve the  2-form $\Omega$ and the  bivector field $\Pi$.

Let $\vec{s}$ and $\vec{s}'$ be any two not necessarily adjacent vertices on $\mathbb{T}$. By composing the transition maps along the unique path connecting them, we get birational maps $\mu_{\vec{s}\rightarrow \vec{s}'}:T_{\mathscr{A};\vec{s}}\dashrightarrow T_{\mathscr{A};\vec{s}'}$ and $\mu_{\vec{s}\rightarrow \vec{s}'}:T_{\mathscr{X};\vec{s}}\dashrightarrow T_{\mathscr{X};\vec{s}'}$. Fomin and Zelevinsky \cite{FZIV} proved the following \emph{factorization formulas} for the pull-backs of cluster coordinates:
\begin{align}
\mu_{\vec{s}\rightarrow \vec{s}'}^*\left(A_{a;\vec{s}'}\right)=&\left(\prod_b A_{b;\vec{s}}^{g_{ab;\vec{s}\rightarrow \vec{s}'}}\right)\left(\left.F_{a;\vec{s}\rightarrow \vec{s}'}\right|_{X_{b;\vec{s}}=\prod_cA_{c;\vec{s}}^{\epsilon_{bc;\vec{s}}}}\right), \label{factorization formula a}\\
\label{factorization formula x}
\mu_{\vec{s}\rightarrow \vec{s}'}^*\left(X_{a;\vec{s}'}\right)=&\left(\prod_b X_{b;\vec{s}}^{c_{ab;\vec{s}\rightarrow \vec{s}'}}\right)\left(\prod_b \left(F_{b;\vec{s}\rightarrow \vec{s}'}\right)^{\epsilon_{ab;\vec{s}'}}\right),
\end{align}
where  $c_{ab;\vec{s}\rightarrow \vec{s}'}$ and $g_{ab;\vec{s}\rightarrow \vec{s}'}$ are  $I\times I$ matrices with integer entries, and $F_{a;\vec{s}\rightarrow \vec{s}'}$ are polynomials in the initial cluster Poisson coordinates $\left\{X_{a;\vec{s}}\right\}$. They are called the \emph{$c$-matrix}, the \emph{$g$-matrix}, and the \emph{$F$-polynomials} associated to the mutation map $\mu_{\vec{s}\rightarrow \vec{s}'}$ respectively. We have the following properties:
\begin{enumerate}
    \item (matrix identities) $\epsilon_{;\vec{s}'}g_{;\vec{s}\rightarrow \vec{s}'}=c_{;\vec{s}\rightarrow \vec{s}'}\epsilon_{;\vec{s}}$;
    \item (sign coherence) row vectors of $c$-matrices and column vectors of $g$-matrices are \emph{sign coherent}, that is, their entries are either all non-negative or all non-positive.
    \item (constant term) $F$-polynomials all have a constant term 1;
    \item (positivity) $F$-polynomials all have positive integer coefficients;
\end{enumerate}
In addition, the following properties are immediate consequences of the factorization formulas:
\begin{enumerate}
    \item[(5)] (Laurent phenomenon) $\mu_{\vec{s}\rightarrow \vec{s}'}\left(A_{a;\vec{s}'}\right)$ is a Laurent polynomial;
    \item[(6)] $c_{ab;\vec{s}\rightarrow \vec{s}'}=\ord_{X_{b;\vec{s}}}\mu_{\vec{s}\rightarrow \vec{s}'}^*\left(X_{a;\vec{s}'}\right)$, where $\ord_xf$ yields the lowest degree of $x$ in $f$ if $f$ is a polynomial and $\ord_x\left(\frac{f}{g}\right):=\ord_xf-\ord_xg$.
\end{enumerate}

Let us fix an initial seed $\vec{s}$.
The sign coherence of $c$-vectors allows us to assign a color to each vertex $a\in I$ in a seed $\vec{s}'$: we say $a$ is green if $c_{ab;\vec{s}\rightarrow \vec{s}'}\geq 0$ for all $b$, and red  otherwise. Note that a mutation at the vertex $c$ changes its color, but it may change the colors of other vertices as well. 

From the above definition all vertices of the initial seed $\vec{s}$ are green. A sequence of mutations that turns all vertices red is called a \emph{reddening sequence}, and a reddening sequence consisting of mutations only in the direction of green vertices is called a \emph{maximal green sequence}.

There is a combinatorial way to compute the $c$-matrix using principal coefficients. Given a seed $\vec{s}=\left(I, I^\uf, \epsilon_{;\vec{s}}, \left\{d_a\right\}\right)$, we define the corresponding \emph{seed with principal coefficients} as 
\[
\vec{s}_\prin=\left(I\sqcup I, I^\uf\sqcup \emptyset, \begin{pmatrix} \epsilon_{;\vec{s}} & \mathrm{id} \\ -\mathrm{id} & 0 \end{pmatrix}, \left\{d_a\right\}\sqcup \left\{d_a\right\}\right).
\]
By applying the sequence of mutations $\mu_{\vec{s}\rightarrow \vec{s}'}$ to $\vec{s}_\prin$, we obtain a seed with exchange matrix $\begin{pmatrix}\epsilon_{;\vec{s}'} & c_{;\vec{s}\rightarrow \vec{s}'} \\ * & *\end{pmatrix}$, whose upper right hand corner is precisely the $c$-matrix we need.

\begin{defn} 
\label{cdvneq0}
Fix an initial seed $\vec{s}$ in  $\mathbb{T}$. We define the \emph{upper cluster algebra} to be
\[
\up \left(\mathscr{A}\right):=\bigcap_\vec{s'}\mu_{\vec{s}\rightarrow \vec{s'}}^*\left(\mathcal{O}\left(T_{\mathscr{A};\vec{s'}}\right)\right)\subset \Frac \left(\mathcal{O}\left(T_{\mathscr{A};\vec{s}}\right)\right),
\]
and  define the \emph{cluster Poisson algebra} to be
\[
\up \left(\mathscr{X}\right):=\bigcap_\vec{s}\mu_{\vec{s}\rightarrow \vec{s'}}^*\left(\mathcal{O}\left(T_{\mathscr{X};\vec{s'}}\right)\right)\subset \Frac \left(\mathcal{O}\left(T_{\mathscr{X};\vec{s}}\right)\right).
\]
Note that $\mathcal{O}\left(T_{\mathscr{X};\vec{s'}}\right)$ is a Poisson algebra for each seed $\vec{s'}$ and the mutation maps preserve the Poisson structure. Therefore $\up\left(\mathscr{X}\right)$ is naturally a Poisson algebra.
\end{defn}

The algebras $\up \left(\mathscr{A}\right)$ and $\up \left(\mathscr{X}\right)$ do not depend on the choice of an initial seed, because all mutation maps are algebra isomorphisms on the fields of fractions and $\mu_c^2=\mathrm{id}$.

\begin{defn} The geometric counterparts of $\up(\mathscr{A})$ and $\up(\mathscr{X})$ are the \emph{cluster $\mathrm{K}_2$ variety} $\mathscr{A}$ and the \emph{cluster Poisson variety}\footnote{The cluster Poisson variety $\mathscr{X}$ is not a variety in the traditional sense as it often fails to be separated.} $\mathscr{X}$,  obtained by gluing the seed tori via the transition maps $\mu$:
\[
\mathscr{A}=\left.\bigsqcup_\vec{s} T_{\mathscr{A};\vec{s}}\right/\left\{\mu_c\right\} \quad \text{and} \quad \mathscr{X}=\left.\bigsqcup_\vec{s} T_{\mathscr{X};\vec{s}}\right/\left\{\mu_c\right\}
\]
In particular, $\mathcal{O}(\mathscr{A})= \up(\mathscr{A})$ and $\mathcal{O}(\mathscr{X})=\up(\mathscr{X})$.
\end{defn}

Since the mutation maps between $T_{\mathscr{A};\vec{s}}$ preserve the canonical 2-forms $\Omega$, these 2-forms can be glued into a canonical 2-form $\Omega$ on the cluster $\mathrm{K}_2$ variety $\mathscr{A}$. Similarly, since the mutation maps between $T_{\mathscr{X};\vec{s}}$ are Poisson maps, the cluster Poisson variety $\mathscr{X}$ is naturally a Poisson variety.

\begin{defn} Let $\vec{s}$ and $\vec{s}'$ be two seeds on $\mathbb{T}$. A \emph{seed isomorphism} $\sigma^*:\vec{s}\rightarrow \vec{s}'$ is a bijection $\sigma^*:I\rightarrow I$ such that $\sigma^*\left(I^\uf\right)=I^\uf$, $d_a=d_{\sigma^*(a)}$, and $\epsilon_{ab;\vec{s}}=\epsilon_{\sigma^*(a)\sigma^*(b);\vec{s}'}$.
\end{defn}

\begin{defn}
 A seed isomorphism $\sigma^*:\vec{s}\rightarrow \vec{s}'$ induces algebra automorphisms $\sigma^*$ on $\up(\mathscr{A})$ and $\up(\mathscr{X})$ defined by
\[
\sigma^*\left(A_{a;\vec{s}}\right):=A_{\sigma^*(a);\vec{s}'} \quad \text{and} \quad \sigma^*\left(X_{a;\vec{s}}\right):=X_{\sigma^*(a);\vec{s}'}.
\]
Abusing notation, we still denote by $\sigma$ the induced biregular automorphisms of the corresponding cluster varieties. Such automorphisms are called \emph{cluster transformations}. 
\end{defn}

A cluster transformation pulls back cluster variables to cluster variables according to the factorization formulas \eqref{factorization formula a} and \eqref{factorization formula x}. Therefore it makes sense to define the $c$-matrix, the $g$-matrix, and the $F$-polynomials of a cluster transformation $\sigma$ with respect to a choice of initial seed $\vec{s}$; we denote them by $c_{;\sigma;\vec{s}}$, $g_{;\sigma;\vec{s}}$ and $F_{;\sigma;\vec{s}}$ respectively. 

The following criteria  determine when a cluster transformation is trivial.

\begin{thm}[\cite{GS2,GHKK,CHL}]\label{trivial cluster trans}
Let $\sigma$ be a cluster transformation on $\mathscr{A}$ and $\mathscr{X}$.
The followings statements are equivalent:
\begin{itemize}
    \item $\sigma$ acts trivially on $\mathscr{A}$;
    \item $\sigma$ acts trivially on $\mathscr{X}$;
    \item the $c$-matrix with respect to one (and equivalently any) seed is the identity matrix;
    \item the $g$-matrix with respecct to one (and equivalently any) seed is the identity matrix.
\end{itemize}
\end{thm}

By Theorem \ref{trivial cluster trans}, the group of cluster transformations on $\mathscr{A}$ coincides with the group of cluster transformations on $\mathscr{X}$. We call this group the \emph{cluster modular group} and denote it by $\mathscr{G}$.

\begin{defn} Fix a seed $\vec{s}=\left(I,I^\uf,\epsilon_{ab},\left\{d_a\right\}\right)$. Let $d=\lcm\left\{d_a\right\}$. We define the \emph{Langlands dual seed} $\vec{s}^\vee:=\left(I,I^\uf, -\epsilon_{ba}, \left\{d/d_a\right\}\right)$, and the \emph{chiral dual seed}  $\vec{s}^\circ:=\left(I,I^\uf,-\epsilon_{ab},\left\{d_a\right\}\right)$.
\end{defn}

The following facts are easy to check:
\begin{enumerate}
    \item[(a)] $\vec{s}^{\vee\vee}=\vec{s}$ and $\vec{s}^{\circ \circ}=\vec{s}$;
    \item[(b)] $\mu_c\vec{s}^\vee=\left(\mu_c\vec{s}\right)^\vee$ and $\mu_c\vec{s}^\circ=\left(\mu_c\vec{s}\right)^\circ$;
    \item[(c)] the followings are equivalent for a bijection $\sigma^*:I\rightarrow I$ 
    \begin{itemize}
        \item $\sigma^*:\vec{s}\rightarrow \vec{s}'$ is a seed isomorphism;
        \item $\sigma^*:\vec{s}^\vee\rightarrow \vec{s}'^\vee$ is a seed isomrophism;
        \item $\sigma^*:\vec{s}^\circ\rightarrow \vec{s}'^\circ$ is a seed isomorphism.
    \end{itemize}  
\end{enumerate}

Thanks to (b), the mutation trees of $\vec{s}$,  $\vec{s}^\vee$ and $\vec{s}^\circ$ are naturally isomorphic. We can define the Langlands dual versions and chiral dual versions of cluster algebras and cluster varieties the same way as before, and we will denote them with superscripts $^\vee$ and $^\circ$ respectively.  Nakanishi and Zelevinsky \cite{NZ} proved the following  \emph{tropical duality} relating the $c$-matrix and $g$-matrix associated to the same sequence of mutations when applied to Langlands dual seeds:
\begin{equation}\label{tropical duality}
c_{;\vec{s}\rightarrow \vec{s}'}^t=g_{;\vec{s}^\vee\rightarrow \vec{s}'^\vee}^{-1};
\end{equation}
they also proved the following identity relating the $c$-matrices associated to the opposite sequence of mutations when applied to chiral dual seeds:
\begin{equation}\label{reverse chiral}
c_{;\vec{s}\rightarrow \vec{s}'}^{-1}=c_{;\vec{s}'^\circ\rightarrow \vec{s}^\circ}
\end{equation}

It follows from (c) and the tropical duality that a cluster transformation is trivial if and only if the corresponding cluster transformation on the Langlands dual (resp. chiral dual) is trivial. Thus the Langlands dual (resp. chiral dual) cluster modular groups are isomorphic, i.e., $\mathscr{G}\cong \mathscr{G}^\vee\cong \mathscr{G}^\circ$.

\vskip 2mm

Cluster varieties $\mathscr{A}$ and $\mathscr{X}$ are positive spaces, i.e., they are equipped with a semifield of positive rational functions $\mathbb{Q}_+(\mathscr{A})$ and  $\mathbb{Q}_+(\mathscr{X})$ respectively. Given a semifield $S$,  we define the set of $S$-points:
\[
\mathscr{A}(S):=\Hom_\text{semifield}\left(\mathbb{Q}_+(\mathscr{A}),S\right), \hskip 10mm  \mathscr{X}(S):=\Hom_\text{semifield}\left(\mathbb{Q}_+(\mathscr{X}),S\right).
\]

Let $\mathbb{Z}^t=(\mathbb{Z}, \min, +)$ is the semifield of tropical integers.
Fock and Goncharov  proposed the following conjecture on Langlands dual cluster varieties.

\begin{conj}[{\cite[Conj. 4.1]{FGensemble}}]\label{duality conjecture} The coordinate ring $\mathcal{O}(\mathscr{A})$ admits a basis $\mathscr{G}$-equivariantly parametrized by $\mathscr{X}^\vee\left(\mathbb{Z}^t\right)$, and $\mathcal{O}(\mathscr{X})$ admits a basis $\mathscr{G}$-equivariantly parametrized by $\mathscr{A}^\vee\left(\mathbb{Z}^t\right)$.
\end{conj}

\begin{defn} For a seed $\vec{s}=\left(I,I^\uf,\epsilon,\left\{d_a\right\}\right)$ with $\gcd\left\{d_a\right\}_{a\in I^\uf}=1$, we define its associated \emph{unfrozen seed} to be $\vec{s}^\uf=\left(I^\uf, I^\uf, \epsilon|_{I^\uf\times I^\uf}, \left\{d_a\right\}_{a\in I^\uf}\right)$.
\end{defn}

Let $\mathscr{A}^\uf$ and $\mathscr{X}^\uf$ be the cluster varieties constructed from an unfrozen seed associated to a seed that defines $\mathscr{A}$ and $\mathscr{X}$. Then among their seed tori we can define four maps
\begin{align*}
    e:T_{\mathscr{A}^\uf;\vec{s}^\uf} & \rightarrow T_{\mathscr{A};\vec{s}} & f:T_{\mathscr{A}^\uf;\vec{s}^\uf}&\rightarrow T_{\mathscr{X};\vec{s}}\\
    e^*\left(A_{a;\vec{s}}\right)&:=\left\{\begin{array}{ll} A_{a;\vec{s}^\uf} & \text{if $a\in I^\uf$},\\
1 & \text{if $a\notin I^\uf$,}
\end{array}\right. & f^*\left(X_{a;\vec{s}}\right)&:=\prod_{b\in I^\uf}A_{b;\vec{s}^\uf}^{\epsilon_{ab;\vec{s}}}\\
    p:T_{\mathscr{A};\vec{s}}& \rightarrow T_{\mathscr{X}^\uf;\vec{s}^\uf} & q:T_{\mathscr{X};\vec{s}}& \rightarrow T_{\mathscr{X}^\uf;\vec{s}^\uf}\\
     p^*\left(X_{a;\vec{s}^\uf}\right)&:=\prod_{b\in I}A_{b;\vec{s}}^{\epsilon_{ab;\vec{s}}} & q^*\left(X_{a;\vec{s}^\uf}\right)&:=X_{a;\vec{s}}.
\end{align*}
By direct computation one can verify that these four maps commute with the mutation maps $\mu$. Therefore we can glue them together and obtain four regular maps in the following commutative diagram.
\[
\xymatrix{\mathscr{A}^\uf \ar[r]^e \ar[dr]_f & \mathscr{A} \ar[dr]^p & \\
& \mathscr{X} \ar[r]_q & \mathscr{X}^\uf}
\]

Let $\mathscr{G}$ be the cluster modular group associated to the cluster varieties $\mathscr{A}$ and $\mathscr{X}$, and let $\mathscr{G}^\uf$ be the cluster modular group associated to the cluster varieties $\mathscr{A}^\uf$ and $\mathscr{X}^\uf$. It then follows from the above formulas that $\mathscr{G}$ is a subgroup of $\mathscr{G}^\uf$.

\begin{prop}\label{pi} Let $\pi:T_{\mathscr{A};\vec{s}}\rightarrow T_{\mathscr{X};\vec{s}}$ be a group homomorphism of algebraic tori for some fixed seed $\vec{s}$ such that $f=\pi\circ e$ and $p=q\circ \pi$. Then $\pi$ induces group homomorphisms of algebraic tori between seed tori associated to any other seed in the mutation equivalent family. Moreover, they glue into a well-defined regular map $\pi:\mathscr{A}\rightarrow \mathscr{X}$.
\end{prop}
\begin{proof} The conditions $p=\pi\circ e$ and $f=q\circ \pi$ imply that the only freedom in defining $\pi$ is the frozen factor (the first factor) in 
\[
\pi^*\left(X_{a;\vec{s}}\right):=\left(\prod_{b\in I\setminus I^\uf} A_{b;\vec{s}}^{\cdots} \right)\left(\prod_{c\in I^\uf}A_{c;\vec{s}}^{\epsilon_{ac;\vec{s}}}\right)=\left(\prod_{b\in I\setminus I^\uf} A_{b;\vec{s}}^{\cdots} \right)f^*\left(X_{a;\vec{s}}\right)
\]
for each frozen vertex $a$. But such factor need not change under mutation: we can just define $\pi^*\left(X_{a;\vec{s}'}\right):=\left(\prod_{b\in I\setminus I^\uf} A_{b;\vec{s}'}^{\cdots} \right)f^*\left(X_{a;\vec{s}'}\right)$ for any other seed $\vec{s}'$ and such maps $\pi$ automatically commute with the mutation maps.
\end{proof}

\begin{rmk}\label{A26} There is another way to describe the map $\pi$. Let $M$ and $N$ be the character lattice of $T_{\mathscr{A};\vec{s}}$ and $T_{\mathscr{X};\vec{s}}$ respectively. Then $\pi$ induces a linear map of character lattices $\pi^*:N\rightarrow M$. The cluster Poisson coordinates $\left\{X_{a;\vec{s}}\right\}$ correspond to a basis $\left\{e_{a;\vec{s}}\right\}$ of $N$, and the skewsymmetric matrix $\hat{\epsilon}_{ab}$ defines a skewsymmetric form $\{\cdot,\cdot\}$ on $N$. We can view a seed mutation $\vec{s}'=\mu_c\vec{s}$ as a change of basis on the lattice $N$ given by
\[
e_{a;\vec{s}'}=\left\{\begin{array}{ll}
    -e_{c;\vec{s}} & \text{if $a=c$}, \\
    e_{a;\vec{s}}+\left[\epsilon_{ac;\vec{s}}\right]_+e_{c;\vec{s}} & \text{if $a\neq c$}.
\end{array}\right.
\]
By viewing the seed tori $T_{\mathscr{A};\vec{s}'}$ and $T_{\mathscr{X};\vec{s}'}$ as algebraic tori with character lattices $M$ and $N$, we can redefine the mutation map $\mu_c$ between the corresponding seed tori by the following coordinate-free formula:
\[
\mu_c^*\left(X^n\right)=X^n\left(1+X^{e_{c;\vec{s}}}\right)^{-\left\{n, d_ce_{c;\vec{s}}\right\}},
\]
\[
\mu_c^*\left(A^m\right)=A^m\left(1+A^{p^*\left(e_{c;\vec{s}}\right)}\right)^{-\inprod{m}{d_ce_{c;\vec{s}}}},
\]
where $X^n$ is the character function corresponding to $n\in N$ and $A^m$ is the character function corresponding to $m\in M$. By using this description, it is not hard to verify by computation that the following diagram commutes
\[
\xymatrix{T_{\mathscr{A};\vec{s}}\ar@{-->}[r]^{\mu_c}\ar[d]_\pi & T_{\mathscr{A};\vec{s}}\ar[d]^\pi \\
T_{\mathscr{X};\vec{s}} \ar@{-->}[r]_{\mu_c} & T_{\mathscr{X};\vec{s}}}
\]
where the two vertical maps $\pi$ are induced by the same linear map $\pi^*:N\rightarrow M$.
\end{rmk}

\begin{prop}\label{surjective pi} If in addition to the assumption of Proposition \ref{pi}, the map $\pi:T_{\mathscr{A};\vec{s}}\rightarrow T_{\mathscr{X};\vec{s}}$ is surjective, then the resulting regular map $\pi:\mathscr{A}\rightarrow \mathscr{X}$ is also surjective (and equivalently the induced algebra homomorphism $\pi^*:\up(\mathscr{X})\rightarrow \up(\mathscr{A})$ is injective), and $F\in \Frac\left(\up(\mathscr{X})\right)$ is in $\up(\mathscr{X})$ if and only if $\pi^*(F)$ is in $\up(\mathscr{A})$.
\end{prop}
\begin{proof} If $\pi:T_{\mathscr{A};\vec{s}}\rightarrow T_{\mathscr{X};\vec{s}}$ is surjective for one seed $\vec{s}$, then $\pi^*:N\rightarrow M$ is injective, which implies that the induced maps $\pi:T_{\mathscr{A};\vec{s}'}\rightarrow T_{\mathscr{X};\vec{s}'}$ are surjective for any seed $\vec{s}'$ in the mutation equivalent family by the above observation. Then it follows that $\pi:\mathscr{A}\rightarrow \mathscr{X}$ is also surjective.

Lastly, since $\pi^*:T_{\mathscr{A},\vec{s}}\rightarrow T_{\mathscr{X},\vec{s}}$ maps monomials to monomials for every $\vec{s}$, $F$ is a Laurent polynomial on $T_{\mathscr{X},\vec{s}}$ for all $\vec{s}$ if and only if $\pi^*(F)$ is a Laurent polynomial on $T_{\mathscr{A},\vec{s}}$ for all $\vec{s}$, which implies that $F\in \up\left(\mathscr{X}\right)$ if and only if $\pi^*(F)\in \left(\mathscr{A}\right)$.
\end{proof}

Next we will make use of the alternative description of a seed in Remark \ref{A26} to define  quasi-cluster transformations. A more detailed discussion can be found in \cite{Fra,GS3}.

\begin{defn} Let $\vec{s}$ and $\vec{s}'$ be two seeds on $\mathbb{T}$. Let $N$ be the lattice as described in Remark \ref{A26}, with a skewsymmetric form $\{\cdot, \cdot\}$. Let $\left\{e_{a;\vec{s}}\right\}$ and $\left\{e_{a;\vec{s}'}\right\}$ be the bases of $N$ corresponding to the seeds $\vec{s}$ and $\vec{s}'$ respectively. Then a lattice isomorphism $\sigma^*:N\rightarrow N$ is said to be a \emph{seed quasi-isomorphism} if 
\begin{enumerate}
    \item there exists a seed isomorphism $\sigma^*:\vec{s}^\uf\rightarrow \vec{s}'^\uf$ between the unfrozen seeds such that $\sigma^*\left(e_{a;\vec{s}}\right)=e_{\sigma(a);\vec{s}'}$ for all $a\in I^\uf$;
    \item $\sigma$ preserves the skewsymmetric form $\{\cdot,\cdot\}$ on $N$.
\end{enumerate}
Note that a seed quasi-isomorphism $\sigma^*$ also induces a lattice isomorphism $M\rightarrow M$ because $M$ is dual to $N$ up to a rescaling\footnote{The rescaling is non-trivial only in the skew-symmetrizable cases; see \cite[Appendix A]{GHKK} for more details.}. We abuse notation and denote the induced automorphism on $M$ by $\sigma^*$.
\end{defn}

\begin{defn}\label{quasi cluster trans} A seed quasi-isomorphism $\sigma^*$ naturally induces an automorphism $\sigma^*$ on $\up(\mathscr{A})$ and $\up(\mathscr{X})$ defined by
\[
\sigma^*\left(A_{;\vec{s}}^m\right):=A_{;\vec{s}'}^{\sigma^*(m)} \quad\quad \text{and}\quad \quad \sigma^*\left(X_{;\vec{s}}^n\right):=X_{;\vec{s}'}^{\sigma^*(n)};
\]
we call such automorphisms \emph{quasi-cluster transformations}. In turn such automorphisms also define biregular automorphisms on the corresponding cluster varieties, which we also call \emph{quasi-cluster transformations}. Since a seed quasi-isomorphism preserves the skewsymmetric form $\{\cdot,\cdot\}$, a quasi-cluster transformation on $\up(\mathscr{X})$ is automatically a Poisson automorphism.
\end{defn}

\begin{lem}\label{lem A.32} Let $\sigma$ be a quasi-cluster transformation corresponding to a lattice isomorphism $\sigma^*$ between two seeds $\vec{s}$ and $\vec{s}'$. Suppose $m$ and $n$ are two integer matrices such that
\[
\sigma^*\left(A_{i;\vec{s}}\right)=\prod_jA_{j;\vec{s}'}^{m_{ij}} \quad \quad \text{and} \quad \quad \sigma^*\left(X_{i;\vec{s}}\right)=\prod_jX_{j;\vec{s}'}^{n_{ij}}.
\]
Then $md'=d(n^t)^{-1}$, where $d$ and $d'$ are the multiplier matrices associated with the seed $\vec{s}$ and $\vec{s}'$, respectively, and $t$ denotes the transposition of a matrix.
\end{lem}
\begin{proof} By definition, we have $\sigma^*\left(e_i\right)=\sum_jn_{ij}e'_j$ for the linear isomorphism $\sigma^*:N\rightarrow N$. On the other hand, we know from \cite[Appendix A]{GHKK} that the two $f$-bases of $M$ are given by $f_i=e_i^*/d_i$ and $f'_j=e'^*_j/d'_j$, respectively. It follows from linear algebra that the induced map $\sigma^*:M\rightarrow M$ is given by
\[
\sigma^*(f_i)=\sum_jd_i^{-1}(n^{-1})_{ji}d'_jf'_j.\qedhere
\]
\end{proof}

Since every cluster transformation is a quasi-cluster transformation, the quasi-cluster transformations form a group that contains the cluster modular group $\mathscr{G}$.

\begin{exmp} Below is a simple example that demonstrates the difference between a cluster transformation and a quasi-cluster transformation. Consider the following quivers which differ by one single mutation.
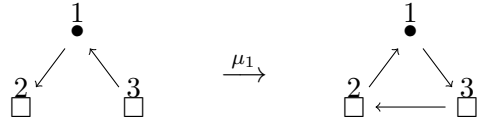
\[
\begin{tikzpicture}[baseline=2ex]
\node (1) at (0,1) [] {$\bullet$};
\node (2) at (-0.75,0) [] {$\Box$};
\node (3) at (0.75,0) [] {$\Box$};
\foreach \i in {1,2,3}
    {
    \node at (\i) [above] {$\i$};
    }
\draw [->] (1) -- (2);
\draw [->] (3) -- (1);
\end{tikzpicture}
\quad \quad \overset{\mu_1}{\longrightarrow}\quad \quad
\begin{tikzpicture}[baseline=2ex]
\node (1) at (0,1) [] {$\bullet$};
\node (2) at (-0.75,0) [] {$\Box$};
\node (3) at (0.75,0) [] {$\Box$};
\foreach \i in {1,2,3}
    {
    \node at (\i) [above] {$\i$};
    }
\draw [->] (2) -- (1);
\draw [->] (1) -- (3);
\draw [->] (3) -- (2);
\end{tikzpicture}
\]
Note that these two quivers are not isomorphic so they cannot possibly define a cluster transformation. On the other hand, if we define 
\[
\sigma^*\left(e_1\right):=e'_1, \quad \sigma^*\left(e_2\right):=e'_1+e'_3, \quad \sigma^*\left(e_3\right):=e'_2,
\]
we see that $\sigma^*$ defines a seed quasi-isomorphism. Therefore $\sigma^*$ induces a quasi-cluster transformation on $\up(\mathscr{A})$ and $\up(\mathscr{X})$ which acts by
\begin{align*}
    \sigma^*\left(X_1\right)&=X'_1=\frac{1}{X_1} & \sigma^*\left(A_1\right)&=\frac{A'_1}{A'_3}=\frac{A_2+A_3}{A_1A_3}\\
    \sigma^*\left(X_2\right)&=X'_1X'_3=\frac{X_3}{1+X_3} & \sigma^*\left(A_2\right)&=A'_3=A_3\\
    \sigma^*\left(X_3\right)&=X'_2=X_2\left(1+X_1\right) &\sigma^*\left(A_3\right)&=A'_2=A_2.
\end{align*}
\end{exmp}

The cluster Poisson algebra $\up\left(\mathscr{X}^\uf\right)$ admits a natural extension into a Poisson algebra of formal series. Goncharov and Shen defined a unique Poisson automorphism on such extension called the \emph{Donaldson-Thomas transformation}. In all known cases the Donaldson-Thomas transformation preserves $\up\left(\mathscr{X}^\uf\right)$ and hence descends to a Poisson automorphism on $\up\left(\mathscr{X}^\uf \right)$. In most known cases the Donaldson-Thomas transformation is a central element in the cluster modular group $\mathscr{G}^\uf$, and when this happens we say that the Donaldson-Thomas transformation is \emph{cluster}. 

In the reverse direction, there is an easy way to check whether a cluster transformation is the Donaldson-Thomas transformation by using the $c$-matrix, and we will use it as the working definition in this paper.

\begin{defn} A cluster transformation in $\mathscr{G}^\uf$ is the cluster Donaldson-Thomas transformation if its $c$-matrix (with respect to any choice of seed) is $-\mathrm{id}$.
\end{defn}

The following theorem justifies the omission of the phrase ``with respect to any choice of seed''.

\begin{thm}[{\cite[Theorem 3.6]{GS2}}]\label{central} If $\sigma$ is a cluster transformation with $c_{;\sigma;\vec{s}}=-\mathrm{id}$, then $c_{;\sigma;\vec{s}'}=-\mathrm{id}$ for any seed $\vec{s}'$ mutation equivalent to $\vec{s}$.  
\end{thm}

Recall that the $c$-matrix of a sequence of mutations can be computed using principal coefficients. It follows by definition of the cluster Donaldson-Thomas transformation that any two isomorphic seeds that give rise to the cluster transformation are related by a reddening sequence. On the the other hand, it is known that a reddening sequence always produces a seed that is isomorphic to the original seed and one can choose an isomorphism such that the associated $c$-matrix is $-\mathrm{id}$. Therefore we have the following implications:
\begin{equation}\label{max green}
\begin{tikzpicture}[baseline=-5ex]
\node (0) at (0,0) [] {$\begin{array}{c} \text{a maximal green}\\ \text{sequence exists}\end{array}$};
\node (1) at (7,0) [] {$\begin{array}{c} \text{a reddening}\\ \text{sequence exists}\end{array}$};
\node (2) at (7,-2) [] {$\begin{array}{c} \text{ Donaldson-Thomas }\\ \text{transformation is  cluster}\end{array}$};
\node (3) at (0,-2) [] {$\begin{array}{c} \text{Donaldson-Thomas}\\ \text{transformation is rational}\end{array}$};
\draw [-implies,double equal sign distance] (0) -- (1);
\draw [implies-implies,double equal sign distance] (1) -- (2);
\draw [-implies,double equal sign distance] (2) -- (3);
\end{tikzpicture}
\end{equation}

Let us explain the importance of Donaldson-Thomas transformations. First, in the skewsymmetric case (where all $d_a=1$),  the Donaldson-Thomas transformation encodes the Donaldson-Thomas invariants of an associated 3d Calabi-Yau category of dg modules (see \cite{GS2}). Second, by combining results of Goncharov and Shen \cite{GS2} and results of Gross, Hacking, Keel, and Kontsevich \cite{GHKK}, one obtain the following sufficient conditions for the cluster duality conjecture.

\begin{thm}[{\cite[Prop.8.28]{GHKK}}]\label{sufficient condition} Conjecture \ref{duality conjecture} holds if the following two conditions hold:
\begin{enumerate}
    \item The Donaldson-Thomas transformation of $\mathscr{X}^\uf$ is cluster;
    \item The algebra homomorphism $p^*:\up\left(\mathscr{X}^\uf\right)\rightarrow \up\left(\mathscr{A}\right)$ is injective.
\end{enumerate}
\end{thm}

\subsection{Python Code for Computing \texorpdfstring{$g^b_d(q)$ for Positive Braid Closures}{}}\label{A.4}

\begin{verbatim}
from sympy import *
q = symbols('q')

# The Weyl group associated to Dynkin type A_(n-1) is the symmetric 
# group S_n. We denote elements of S_n by an n-tuple u consisting of
# integers from 0 to n. Below are left/right multiplication of u by 
# a simple reflection s_k

def left_mult(k,u):
    u[k-1],u[k] = u[k],u[k-1]
    return(u)

def right_mult(k,u):
    for i in range(len(u)):
        if u[i] == k-1:
            u[i] = k
        elif u[i] == k:
            u[i] = k-1
    return(u)

# Below is an algorithm that converts an integer 0<=m<2^l into an l 
# tuple of 0s and 1s corresponding to the binary expansion of m.

def binary(m,l):
    output = []
    while l != 0:
        a = m//2**(l-1)
        m = m-a*2**(l-1)
        l -= 1
        output.append(a)
    return(output)

# Now we are ready to define the function g^b_d(q). 
# The computation needs the input of a triangulation t. However, 
# the output should not depend on t.
# Let l=l(b)+l(d)
# The input of b and d will specify the choice of words. We then 
# represent a triangulation t by an l-tuple of 0s and 1s, with 0 
# representing a nabla-shaped triangle and 1 representing a Delta-
# shaped triangle.

def g(n,b,d,t):
    
# First we should test that t is actually a triangulation for (b,d).

    if t.count(0) != len(b) or t.count(1) != len(d):
        return("Error: Wrong Triangulation!")

# Next we combine b and d into a l-tuple "braid" of pairs (s,k), 
# where s=0,1 indicates the orientation of the triangle 
# and k indicates the simple reflection.

    braid = []
    for i in range(len(t)):
        if t[i] == 0:
            braid.append([t[i],b.pop(0)])
        else:
            braid.append([t[i],d.pop(0)])

# The variable "polynomial" will be the final output g^b_d(q)*(q-1)^(n-1)
    
    polynomial=0
    
# When we go across the triangulation from left to right, at each
# step we can either multiply or not multiply s_k.
# Therefore at most we only need to consider 2^l number of cases.
# We go through all these cases using a parameter 0<=m<2^l
# By converting this parameter m into a binary expression of length l, 
# we exhaust all possible cases.
# 1 means the multiplication takes place, 
# and 0 means the multiplication does not take place.

    for m in range(2**len(t)):
        occurrence = binary(m,len(t))
        
# Start with the left most edge being of Tits codistance u=e.

        u = [j for j in range(n)]
    
# The l-tuple "factors" collects the factors coming from each 
# triangle in the triangulation.
        
        factor = []

# We now go from left to right across the triangulation, and
# impose Tits codistance conditions u according to the binary 
# expression of m.
# Note that letters of b correspond to right multiplications
# and letters of d correspond to left multiplications.
# The computation is governed by Lemma 6.1.

        for i in range(len(t)):
            if braid[i][0] == 1:
                if occurrence[i] == 0:
                    if u[braid[i][1]-1] > u[braid[i][1]]:
                        break
                elif occurrence[i] == 1:
                    if u[braid[i][1]-1] > u[braid[i][1]]:
                        factor.append(q)
                    else:
                        factor.append(1)
                    u = left_mult(braid[i][1],u)
                else:
                    factor.append(q-1)
            else:
                if occurrence[i] == 0:
                    if u.index(braid[i][1]-1) > u.index(braid[i][1]):
                        break
                elif occurrence[i] == 1:
                    if u.index(braid[i][1]-1) > u.index(braid[i][1]):
                        factor.append(q)
                    else:
                        factor.append(1)
                    u = right_mult(braid[i][1],u)
                else:
                    factor.append(q-1)
        
                    
        if u == [j for j in range(n)]:
            term = 1
            for p in factor:
                term *= p
            polynomial += term
                    
    return(cancel(polynomial/(q-1)**(n-1)))
\end{verbatim}

\begin{Backmatter}

\paragraph{Acknowledgements} We are grateful to A.B. Goncharov for introducing us to cluster ensembles and the amalgamation of cluster varieties. We thank J.H. Lu for suggesting a method to count $\mathbb{F}_q$-points of double Bott-Samelson cells and for helpful conversations on generalized double Bruhat cells. We  thank  H. Gao and E. Zaslow for valuable discussions on  Legendrian links and microlocal rank-1 sheaves. L.S. is supported by the Collaboration Grant for Mathematicians from the Simons Foundation, \#711926.

\paragraph{Conflicts of interest}

None.

\bibliographystyle{alphaurl-a}

\bibliography{biblio}

\end{Backmatter}

\end{document}